\documentclass[english]{article}
\usepackage[T1]{fontenc}
\usepackage[latin9]{inputenc}
\usepackage{geometry}
\usepackage{color}
\usepackage{amsmath}
\usepackage{amsthm}
\usepackage{amssymb}
\usepackage{stmaryrd}
\usepackage[all]{xy}

\usepackage{hyperref}
\hypersetup{
    colorlinks=true,
    linkcolor=blue,
    citecolor=black, 
    }
\makeatletter
\theoremstyle{plain}

\theoremstyle{plain}

\usepackage[english]{babel}
\usepackage{pgf}
\usepackage{tikz}\usepackage{color}
\usepackage[all]{xy}
\usetikzlibrary{arrows,shapes,positioning}

\usepackage{pgf}\usepackage{tikz}\usepackage{color}

\usepackage{amsfonts}\usepackage{color}\usepackage{enumitem}\usepackage{bm}\usepackage{thmtools}

\def\theenumi{\arabic{enumi}}

\def\theenumii{\alph{enumii}}
\def\p@enumii{\theenumi.}

\def\theenumiii{\arabic{enumiii}}
\def\p@enumiii{(\theenumi)(\theenumii)}

\def\p@enumiv{\p@enumiii.\theenumiii}

\newtheorem{theorem}{Theorem}[section]
\newtheorem{assumption}[theorem]{Assumption}\newtheorem{claim}[theorem]{Claim}\newtheorem{corollary}[theorem]{Corollary}\newtheorem{fact}[theorem]{Fact}\newtheorem{example}[theorem]{Example}\newtheorem{lemma}[theorem]{Lemma}\newtheorem{proposition}[theorem]{Proposition}\newtheorem{theo}[theorem]{Theorem}

\theoremstyle{definition}
\newtheorem{definition}[theorem]{Definition}\newtheorem{notation}[theorem]{Notation}\newtheorem{remark}[theorem]{Remark}

\usepackage{babel}

\makeatother

\usepackage{babel}

\providecommand{\theoremname}{Theorem}

\begin{document}
\title{Furstenberg entropy spectra of stationary actions of semisimple Lie
groups}

\author{J\'er\'emie Brieussel and Tianyi Zheng}
\date{}

\maketitle
\begin{abstract}
We determine Furstenberg entropy spectra of ergodic stationary actions
of $SL(d,\mathbb{R})$ and its lattices. The constraints on entropy
spectra are derived from a refinement of the Nevo-Zimmer projective
factor theorem. The realisation part is achieved by means of building
Poisson bundles over stationary random subgroups. 
\end{abstract}

\section{Introduction }

A compact metrizable space $X$ acted upon continuously by a locally
compact group $G$ equipped with a probability measure $\mu$ admits
a stationary probability measure $\eta$, which means a fixed point
of the convolution $\mu\ast\eta=\eta$. The system $G\curvearrowright(X,\eta)$
is referred to as a stationary action of $G$. The study of stationary actions
of semisimple Lie groups was initiated by Furstenberg~\cite{Furstenberg1,Furstenberg2}.
He showed in particular that stationary measures on $X$ are in bijection
with measures invariant under a minimal parabolic subgroup~$P$.
He also introduced a numerical invariant $h_{\mu}(X,\eta)$ nowadays
referred to as the Furstenberg entropy:

\begin{equation}
h_{\mu}(X,\eta)=-\int_{G}\int_{X}\log\frac{dg^{-1}\eta}{d\eta}(x)d\eta(x)d\mu(g).\label{eq:h-def}
\end{equation}
It is $0$ if and only if $\eta$ is a $G$-invariant measure.

A systematic study of ergodic stationary actions of semisimple Lie
groups was developed by Nevo and Zimmer in a series of articles. They
established that any such action admits a maximal projective factor
$(G/Q,\nu_{Q})$, where $Q$ is a parabolic subgroup of $G$; and
when $G$ is a higher rank simple Lie group,
this factor is trivial if and only if the stationary measure $\eta$
is actually invariant~\cite{NZ2}. Under a further mixing assumption
(called $P$-mixing), they proved that the stationary system is a relative 
measure-preserving extension of the maximal projective factor~\cite{NZ1}. This implies
in particular that $h_{\mu}(X,\eta)=h_{\mu}(G/Q,\nu_{Q})$ and it follows 
that the Furstenberg
entropy of $P$-mixing stationary $(G,\mu)$-spaces can take on only
finitely many values~\cite{NZ3}. These results no longer hold without higher rank hypothesis, as $PSL(2,\mathbb{R})$
admits infinitely many $P$-mixing stationary systems (in fact can
be taken to be smooth manifolds) with distinct entropy. Nor without
the $P$-mixing hypothesis, as groups with a parabolic subgroup mapping
onto $PSL(2,\mathbb{R})$ may have infinite entropy spectrum~\cite{NZ3}.

The purpose of the present article is to give a complete description
of the entropy spectrum of $SL(d,\mathbb{R})$ and of its lattices,
equipped with appropriate measures, see Theorem \ref{th:spectrum-sl}
below.

\begin{definition}

We say a step distribution $\mu$ on $G$ has \emph{finite boundary
entropy} if the Furstenberg entropy of the Poisson boundary of $(G,\mu)$
is finite. For such a measure $\mu$, we refer to the range of possible
Furstenberg entropy values over all ergodic $\mu$-stationary systems
as the\emph{ Furstenberg entropy spectrum} of $\left(G,\mu\right)$:

\[
{\rm EntSp}\left(G,\mu\right):=\left\{ h_{\mu}(X,\nu):(X,\nu)\mbox{ is an ergodic }(G,\mu)\mbox{-stationary system}\right\} .
\]

\end{definition}

Note that by \cite[Corollary 2.7]{Bader-Shalom}, $(X,\nu)$ is an
ergodic $\left(G,\mu\right)$-stationary system if and only if $\nu$
is extremal in the set of $\mu$-stationary measures on $X$.

\subsection{Structure of stationary systems and constraints on entropy values}

We will use:

\begin{notation}\label{notationf}

Let $\Delta$ denote simple roots of a semisimple Lie group $G$.
For $I\subseteq\Delta$, denote by $P_{I}$ the standard parabolic
subgroup that corresponds to $I$ (see Section~\ref{subsec:structure-of-parabolic}). 

Denote by $\mathsf{f}:2^{\Delta}\to2^{\Delta}$ the map where given
$I\subseteq\Delta$, $\mathsf{f}(I)$ is the largest subset $I'\subseteq I$
such that the Levi subgroup $L_{I'}$ has no $\mathbb{R}$-rank 1
noncompact simple factors. 

\end{notation}

From the proof of Nevo-Zimmer projective factor theorem in \cite{NZ2},
we extract the following statement. Plausibility of such a formulation is hinted in the remarks after \cite[Theorem 11.4]{NZ2}.

\begin{theorem}[A refinement of {\cite[Theorem 3]{NZ2}}]\label{invariant2}
Let $G$ be a connected semisimple real Lie group with finite center,
$\mu$ an admissible measure on $G$. Suppose $(X,\nu)$ is an ergodic
$(G,\mu)$-system where $\nu$ is not $G$-invariant. Let $\lambda$
be the corresponding $P$-invariant measure on $X$ provided by the
Furstenberg isomorphism. Suppose $\left(G/Q,\nu_{Q}\right)$ is the
maximal standard projective factor of $(X,\nu)$, where $Q=P_{I}$,
then the measure $\lambda$ is invariant under the parabolic subgroup
$P_{\mathsf{f}(I)}$.
\end{theorem}

A probability measure $\mu$ on $G$ is called \emph{admissible} if
${\rm supp}\mu$ generates $G$ as a semigroup, and some convolution
power $\mu^{\ast k}$ is absolutely continuous with respect to Haar
measure on $G$. The Furstenberg isomorphism between $\mu$-stationary and $P$-invariant probability measures on X is described in Section~\ref{subsubsec:Furstenberg}. 

When the Levi subgroup $L_{I}$ has no rank one noncompact
factors, $\mathsf{f}(I)=I$, Theorem \ref{invariant2} implies the
following corollary. Compared to~\cite[Theorem 3 and 9.1]{NZ2}:
the $\lambda$-ergodicity assumption of $S$-action is dropped; instead
it is assumed that the Levi subgroup of $Q$ has no rank one factors.

\begin{corollary}\label{mp-ext}

Let $G$ be a connected semisimple real Lie group with finite center,
$\mu$ an admissible measure on $G$. Suppose $(X,\nu)$ is an ergodic
$(G,\mu)$-system where $\nu$ is not $G$-invariant. Let $\left(G/Q,\nu_{Q}\right)$
be the maximal standard projective factor of $(X,\nu)$. If the Levi
subgroup of $Q$ has no $\mathbb{R}$-rank $1$ non-compact simple
factors, then $\left(X,\nu\right)\to\left(G/Q,\nu_{Q}\right)$ is
a relative measure-preserving extension.

\end{corollary}

The conclusion of Theorem \ref{invariant2} can be formulated in terms
of the boundary map. Denote by $\mathcal{P}(X)$ the space of probability
measures on $X$ and let $\boldsymbol{\beta}_{v}:G/P\to\mathcal{P}(X)$
be the \emph{boundary map} associated with the stationary measure
$\nu$ (its definition is reviewed in Subsection \ref{subsec:boundary-map}).
Then the measure $\lambda$ is invariant under $P_{\mathsf{f}(I)}$
if and only if $\boldsymbol{\beta}_{v}$ factors through the
projection $G/P\to G/P_{\mathsf{f}(I)}$. In this formulation, we
can derive an analogous statement for lattices equipped with Furstenberg
measures, through an induction procedure for stationary actions in
\cite{BH21,BBHP}.

\begin{definition}\label{furstenbergmeas}

Let $\Gamma$ be a lattice in a semi-simple Lie group $G$. Denote
by $P$ a minimal parabolic subgroup of $G$. Let $\mu$ be a probability measure such that the Poisson boundary of $(G,\mu)$ can be identified with $(G/P,\nu_P)$ for some measure $\nu_P$ is in the same measure class as 
$\bar{m}_K$, the unique $K$-invariant probability measure on $G/P$.  We say a non-degenerate
measure $\mu_0$ on $\Gamma$ is a \emph{Furstenberg discretisation} of $(G,\mu)$ if
 the Poisson boundary of $(\Gamma,\mu_0)$ can be identified with $(G/P,\nu_P)$.
Such measures on $\Gamma$ exist by \cite{Furstenberg3}.

\end{definition}

In this setting, a $(\Gamma,\mu_{0})$-stationary system $(X,\nu)$
gives rise to a $\Gamma$-boundary map $\boldsymbol{\beta}_{\nu}:G/P\to\mathcal{P}(X)$,
although $G$ does not necessarily act on $X$. We derive the following
from Theorem \ref{invariant2}. 

\begin{theorem}\label{lattice}

Let $\Gamma$ be a lattice in a connected semisimple real Lie group
$G$ with finite center. Equip $\Gamma$ with a Furstenberg measure $\mu_{0}$. Suppose $(X,\nu)$ is an ergodic $(\Gamma,\mu_{0})$-system
where $\nu$ is not $\Gamma$-invariant. Let $\left(G/Q,\nu_{Q}\right)$
be the maximal standard projective $\Gamma$-factor of $(X,\nu)$,
where $Q=P_{I}$. Then the $\Gamma$-boundary map $\boldsymbol{\beta}_{\nu}:G/P\to\mathcal{P}(X)$
factors through $G/P_{\mathsf{f}(I)}$, that is a.e., $\boldsymbol{\beta}_{\nu}(gP)$
depends only on the coset $gP_{\mathsf{f}(I)}$. 

\end{theorem}

Constraints on the Furstenberg entropy spectrum follow from the structure theorems. 
\begin{theorem}\label{constraint}

Let $G$ be a connected semisimple real Lie group with finite center
and denote by $\Delta$ simple restricted roots of $G$. Let $\mu$
be an admissible step distribution on $G$ with finite boundary entropy.
Then the Furstenberg entropy spectrum of $(G,\mu)$ satisfies 
\[
{\rm EntSp}(G,\mu)\subseteq\bigcup_{I\subseteq\Delta}\left[h_{\mu}\left(G/P_{I},\nu_{I}\right),h_{\mu}\left(G/P_{\mathsf{f}(I)},\nu_{\mathsf{f}(I)}\right)\right],
\]
where $\nu_{I}$ is the (unique) $\mu$-stationary measure on $G/P_{I}$.

\end{theorem}

An analogous statement for lattices equipped with Furstenberg discretization measures
is stated in Theorem~\ref{constraint-lattice}.

\subsection{Realisation of entropy values}

For the free group $F_{k}$ on $k$-generators and step distribution
$\mu$ uniform on the generators and their inverses, Bowen shows in
\cite{Bowen} that ${\rm EntSp}\left(F_{k},\mu\right)$ is the full
interval $\left[0,{\bf h}_{\mu}\right]$, where ${\bf h}_{\mu}$ is
the Furstenberg entropy of the Poisson boundary of $\left(F_{k},\mu\right)$,
which is equal to the random walk asymptotic entropy. The proof is
based on a construction of Poisson bundles over invariant random subgroups
(IRSs) and analysis of associated random walks on the coset graphs.

For $G=SL(d,\mathbb{R})$ and its lattices, we realise Furstenberg
entropy values within the constraints of Theorem \ref{constraint}
and \ref{constraint-lattice} via construction of Poisson bundles
over stationary systems. Recall the map $I\mapsto\mathsf{f}(I)$ defined
in Notation~\ref{notationf}. We say a measure $\mu$ is $B_{\infty}$ if it is admissible on $G$, 
of compact support and bounded density with respect to the Haar measure.

\begin{theorem}\label{th:spectrum-sl}

Let $G=SL(d,\mathbb{R})$ and denote by $\Delta=\{1,\ldots,d-1\}$
its simple roots. Suppose 
\begin{itemize}
\item $\mu$ is in the $B_{\infty}$ class on $G$,
\item or $\mu$ is a Furstenberg discretisation on a lattice $\Gamma<G$ of a  $B_{\infty}$ class measure on $G$, and $\mu$ has finite
Shannon entropy. 
\end{itemize}
Write $S=\left\langle {\rm supp}\mu\right\rangle $ for the subgroup
generated by ${\rm supp}\mu$, then 
\[
{\rm EntSp}(S,\mu)=\bigcup_{I\subseteq\{1,\ldots,d-1\}}\left[h_{\mu}\left(G/P_{I},\nu_{I}\right),h_{\mu}\left(G/P_{\mathsf{f}(I)},\nu_{\mathsf{f}(I)}\right)\right],
\]
where $\nu_{I}$ is the $\mu$-stationary measure on $G/P_{I}$.

\end{theorem}

In particular for $B_{\infty}$-measures, $SL(2,\mathbb{R})$ has full entropy
spectrum, and $SL(3,\mathbb{R})$ has entropy spectrum of the form
$\{0\}\cup[h_{\mu}(G/Q),h_{\mu}(G/P)]$ for some non-trivial parabolic
subgroup $P<Q<G$. A particular instance for $SL(4,\mathbb{R})$ is
illustrated in Figure~\ref{fig:sl4}.

\begin{remark}

It is natural to ask what are entropy spectra of other simple
Lie groups. Note that rank one Lie groups $Sp(n,1)$, $n\ge2$, and
$F_{4(-20)}$ have Kazhdan's property (T). Therefore by \cite{Nevo},
$\{0\}$ is an isolated point in their Furstenberg entropy spectra.
For these rank one groups, Theorem \ref{constraint} does not provide
sharp constraints on their entropy spectra.

\end{remark}

It is classical that for admissible $\mu$ considered in Theorem \ref{th:spectrum-sl},
the boundary entropy values $h_{\mu}\left(G/P_{I},\nu_{I}\right)$
can be expressed in terms of the Lyapunov spectrum of the $\mu$-random
walk, see the Furstenberg formula in Subsection \ref{subsec:lyapunov}.

\begin{figure}
\begin{centering}
\begin{tikzpicture}
  \node (max) at (0,4) {$P=P_\emptyset$};
  \node (a) at (-2,2.6) {$P_{\{1\}}$};
  \node (b) at (0,1.8) {$P_{\{2\}}$};
  \node (c) at (2,2.2) {$P_{\{3\}}$};
  \node (d) at (-2,-0.2) {$P_{\{1,2\}}$};
  \node (e) at (0,0.4) {$P_{\{1,3\}}$};
  \node (f) at (2,-0.6) {$P_{\{2,3\}}$};
  \node (min) at (0,-2) {$G$};
  \draw (min) -- (d) -- (a) -- (max) -- (b) -- (f)
  (e) -- (min) -- (f) -- (c) -- (max)
  (d) -- (b);
  \draw[very thick](max) -- (a) -- (max) -- (b) -- (max) -- (c) -- (max);
  \draw[preaction={draw=white, -,line width=6pt}] (a) -- (e) -- (c);
  \draw[very thick](a) -- (e) -- (c);
  
  \draw[thick,->] (3.5,1) -- (4.5,1); 
  
  \draw[very thick] (6,4)--(6,0.4);
  \draw[dashed] (6,0.4)--(6,-2);
  \draw (6,4) node{$\bullet$};
  \draw (6,4) node[right]{$h_\mu(G/P)$};
   \draw (6,2.6) node{$\bullet$};
   \draw (6,2.6) node[right]{$h_\mu(G/P_{\{1\}})$};
   \draw (6,2.2) node{$\bullet$};
    \draw (6,2.2) node[right]{$h_\mu(G/P_{\{3\}})$};
   \draw (6,1.8) node{$\bullet$};
    \draw (6,1.8) node[right]{$h_\mu(G/P_{\{2\}})$};
   \draw (6,0.4) node{$\bullet$};
   \draw (6,0.4) node[right]{$h_\mu(G/P_{\{1,3\}})$};
   \draw (6,-0.2) node{$\bullet$};
   \draw (6,-0.2) node[right]{$h_\mu(G/P_{\{1,2\}})$};
      \draw (6,-0.6) node{$\bullet$};
       \draw (6,-0.6) node[right]{$h_\mu(G/P_{\{2,3\}})$};
   \draw (6,-2) node{$\bullet$};
    \draw (6,-2) node[right]{$0$};
\end{tikzpicture} 
\par\end{centering}
\caption{The poset of parabolic subgroups of $SL(4,\mathbb{R})$. Inclusions
differing only by a rank 1 factor are drawn in bold. The particular
case pictured here occurs when the Lyapunov exponents satisfy $\lambda_{2}>0>\lambda_{3}$.}
\label{fig:sl4} 
\end{figure}
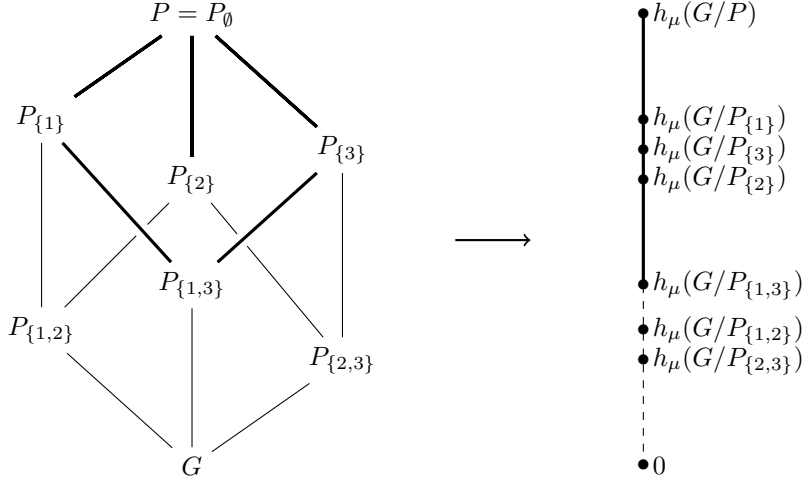

Using Poisson bundles over stationary systems permits to obtain intervals of entropy values.
Such an extension beyond the framework over measure-preserving
systems is necessary. The bundles over IRS considered in \cite{Bowen} can be described as
factors of a system of the form $\left(X\times B,m\times\nu_{B}\right)$,
where $\left(X,m\right)$ is equipped with an ergodic $G$-invariant
measure $m$ and $\left(B,\nu_{B}\right)$ is the Poisson boundary
of the $\mu$-random walk on $G$. For a higher rank connected simple
Lie group $G$, \cite[Theorem 1.2]{NZ4} implies that in this setting, for any $G$-factor
$\left(Z,\nu\right)$ of $\left(X\times G/P,m\times\nu_{B}\right)$,
there is a parabolic subgroup $Q$ of $G$, such that $\left(Z,\nu\right)$
is a measure preserving extension of $\left(G/Q,\bar{\nu}_{B}\right)$.
Therefore in this case Poisson bundles over measure preserving systems
provide at most $2^{r}$ Furstenberg entropy values, where $r$ is
the $\mathbb{R}$-rank of $G$.

For $Q'<Q$, two parabolic subgroups of $G=SL(d,\mathbb{R})$ whose
Levi subgroups differ only by a rank 1 factor, we consider stationary $G$-systems
induced from measure-preserving $Q$-actions. There is a large supply
of such systems where $Q$ acts through the quotient $Q\to PSL(2,\mathbb{R})$.
The associated Poisson bundles will be factors of the stationary joinings
with the Poisson boundary, rather than direct products. We remark that the random walk models that 
appear in the construction, based on stationary joinings,
are inherently different from random walks on stationary random graphs considered in \cite{BC12,CL16}. 

The proof of realisation of the interval $\left[h_{\mu}\left(G/Q,\nu_{Q}\right),h_{\mu}\left(G/Q',\nu_{Q'}\right)\right]$
is based on a continuity argument. As in~\cite{Bowen}, we
construct a family of Poisson bundles $(Z_{p},\lambda_{p})$, parametrised by $p\in[0,1]$.
The key point is to show that the entropy $h_{\mu}(Z_{p},\lambda_{p})$
depends continuously on~$p$. Upper and lower semi-continuity are
treated separately. Upper semi-continuity follows from standard entropy
formulae. More precisely, they provide expressions in terms of infimum of mutual information 
(or Shannon entropy) of time $n$ random walks. This allows to show
that entropy is an infimum of continuous functions. Our approach to lower semi-continuity
is based on identification of the Poisson bundles with concrete models. Then
the KL-divergence occurring in the definition of entropy can be expressed
as the supremum of relative entropies over finite partitions in the model. After verifying 
certain approximation properties, entropy is expressed as a supremum of continuous functions.

In view of an explicit identification of the Poisson bundles, it is
convenient to use measure preserving systems of $SL(2,\mathbb{R})$
induced from IRSs of a lattice $F$, taken for simplicity to be the
Sanov subgroup, free of rank 2. Their respective Poisson boundaries,
the boundary circle $\partial\mathbb{H}$ of the hyperbolic plane
and the space of ends $\partial F$, are $F$-measurably isomorphic, provided $F$
is endowed with a Furstenberg measure, which can be chosen to have
finite entropy and finite log-moment. 
We show that taking Poisson bundles interacts in a compatible way with 
inducing. These two
stages of induction reduce the proof of Theorem~\ref{th:spectrum-sl}
to the case of free groups, for which a key approximation property is shown in 
Proposition \ref{relative-cont}. Along the way we show full
entropy realisation for a large class of step distributions on $F$:

\begin{theo}\label{thm-free}

Let $F$ be a free group of finite rank, endowed with a non-degenerate probability measure
$\mu$ with finite entropy and finite logarithmic moment. Then $\rm{EntSp}(F,\mu)=\left[0,{\bf h}_{\mu}\right]$.

\end{theo}

This generalizes Bowen's original result for the case where $\mu$ is uniform on the generators and their inverses \cite{Bowen}. In the case of finitely supported $\mu$, full entropy realisation was known by \cite{HY}. For virtually free groups, full realisation
is known for symmetric 4th moment measures by \cite{RGY}, and the
existence of a gap is ruled out for first moment measures by \cite{HT}.

Similar arguments can be applied to other groups acting on trees;
for which we will investigate elsewhere.

\subsection{Organization of the article}

Section \ref{sec:Preliminaries} collects necessary preliminaries.
After that the article is divided in three parts, with some additional
details provided in two appendices.

\subsubsection{Part I: Constraints on entropy spectrum}

The first part consists of Section \ref{sec:first-steps} and \ref{sec:consequence}.
In Section \ref{sec:first-steps} we apply the line of arguments in
\cite{NZ2}, in particular an operation on continuous functions using
contracting dynamics and Gauss map considerations, to show a property of maximal projective factor,
stated in Theorem \ref{invariant-lambda}.
Consequences of the structure theorem, namely Theorems \ref{invariant2},
\ref{lattice} and \ref{constraint}, are derived in Section \ref{sec:consequence}.

\subsubsection{Part II: Poisson bundle over a stationary system }

The second part consists of Section \ref{sec:Poissonbundle} to \ref{sec:lower-argu},
where we develop some general theory on Poisson bundles over stationary
systems.

In Section \ref{sec:Poissonbundle}, we explain the definition of
such Poisson bundles, which starts with stationary joinings. The resulting
system $(Z,\lambda)$ is a $G$-factor that fits into 
\[
\left(X\times B,\eta\varcurlyvee\nu_{B}\right)\rightarrow\left(Z,\lambda\right)\rightarrow\left(X,\eta\right),
\]
where $\left(X,\eta\right)$ is a $(G,\mu)$-stationary system, $(B,\nu_{B})$
is the Poisson boundary of $\mu$-random walk, and $\eta\varcurlyvee\nu_{B}$
denotes the stationary joining of the two. For a stationary system
$\left(X,\eta\right)$ which is standard in the sense of Furstenberg-Glasner
\cite{FurstenbergGlasner}, we show that the Poisson bundle over $\left(X,\eta\right)$
can be described as a proximal extension where the fibers are Poisson
boundaries of coset Markov chains whose law is given by suitable Doob
transforms. A typical example of such $(X,\eta)$ is induced from
a measure-preserving action of $Q$, where $Q$ is a parabolic group.
The fiberwise Markov property is the key ingredient in deriving the
entropy formulae in Subsection \ref{subsec:formula}. It is standard
that the entropy formulae imply upper semi-continuity properties of
entropy, as explained in Section \ref{subsec:Upper-semi-continuity}.

There remains to obtain lower semi-continuity. As a starting point,
we formulate in Section~\ref{sec:entropy-criterion} entropy criteria
for identification of Poisson bundles, which are adapted from the
strip and ray criteria originally due to Kaimanovich \cite{KaimanovichHyperbolic}.
Here by identification, we mean explicitly describing a $\left(G,\mu\right)$-stationary
system $\left(M,\bar{\lambda}\right)$ and showing it is $G$-isomorphic
to $\left(Z,\lambda\right)$.

In Section \ref{sec:lower-argu} we formulate an approach to prove
lower semi-continuity of Furstenberg entropy for some specific systems
such as end-compactification bundles. The basic idea is that a symbolic
representation provides a natural sequence of finite partitions into
cylinder sets, which generate the $\sigma$-field on the fiber. We
may then write the fiberwise KL-divergence as the supremum of relative
entropy on the finite partitions. To ensure lower-semicontinuity,
it is sufficient to show that the relative entropy on the chosen finite
partitions varies continuously over the base. In Section \ref{subsec:fiber-uni}
we describe a technical condition, referred to as \textquotedbl locally
constant uniform approximation\textquotedbl , which implies such continuity.

\subsubsection{Part III: entropy realization for free groups, for $SL(d,\mathbb{R})$
and its lattices}

The third part consists of Sections \ref{Sec:identification-free}
and \ref{sec:Poisson-SL}. We apply the general framework of Part
II to free groups and to $SL(d,\mathbb{R})$.

In Section \ref{Sec:identification-free}, we revisit Poisson bundles
of free groups over IRSs constructed in Bowen~\cite{Bowen}. They
are supported on subgroups with \textquotedbl tree-like Schreier graphs\textquotedbl .
We apply the strip criterion Theorem~\ref{strip} to show that these
bundles are isomorphic to end compactification bundles over the same
base system. The finite partitions allowing approximations are simply
the shadows of vertices on a sphere of given radius. We conclude this
section with a proof of Theorem~\ref{thm-free}.

Section~\ref{sec:Poisson-SL} is devoted to the realization part
of Theorem~\ref{th:spectrum-sl}. Given a semisimple Lie group $G$, we consider a factor $L$ of the Levi subgroup of a parabolic subgroup $Q$. Given a lattice $\Gamma_0$ of $L$ and an invariant $\Gamma_0$-system, we induce a stationary $G$-system, in the framework of Part II. When $L$ is isomorphic to $SL(2,\mathbb{R})$, the lattice can be chosen to be a free group, and the locally constant approximations on the free group from Section~\ref{Sec:identification-free} can be transferred to $G$. Roughly speaking, in such a Poisson bundle, a fiber can be described
as the space of ends of a tree-like graph equipped with a suitable
measure, where the graph and measure depend on the base point. One
may view such an identification result as providing a symbolic representation
fiberwise for the Poisson bundle. Together with Section~\ref{subsec:Upper-semi-continuity},
we conclude the continuity argument.

We mention that the SRS in the construction are supported on non-discrete
subgroups of $SL(d,\mathbb{R})$ for $d\ge3$. This is necessary:
by a result of Fraczyk and Gelander \cite{FG}, every discrete SRS
of $SL(d,\mathbb{R})$, $d\ge3$, is an IRS. 

\subsubsection{Appendices}

Appendix A reviews the Nevo-Zimmer operation on continuous functions
from \cite{NZ2} based on contracting dynamics. Operations in the
expanding direction were considered earlier: it is first used by Margulis
in the proof of the Normal Subgroup Theorem \cite[Chapter IV]{MargulisBook};
and by Nevo-Zimmer in structure theorem under $P$-mixing assumption
\cite{NZ1}.

In Appendix B we include proofs of mutual information and entropy
formulae for Furstenberg entropy of Poisson bundles, which are stated
in Subsection \ref{subsec:formula}. These follow from classical arguments
being adapted to our setting.

{\it Acknowledgments.} J.B. acknowledges support of the ANR-22-CE40-0004 GoFR. T.Z. was partially supported by 
a Sloan research fellowship and NSF DMS-2348143. It is a pleasure to thank Yair Hartman and Amos Nevo for interesting discussions at various stages of this work.

{\it AI Usage Statement.} No AI tools were used for this article.

\section{Preliminaries\label{sec:Preliminaries}}

\subsection{Induced actions}

We recall basic facts about induced actions, see Zimmer's book~\cite{ZimmerBook}
for detailed treatment.

Let $G$ be a locally compact group and $H$ a closed subgroup of
$G$. Let $m_{G}$ be a left Haar measure on $G$. Let $(S,\eta)$
be an $H$-space. Let $H$ act from the right on the product $G\times S$
by $(g,s).h=(gh,h^{-1}.s)$. Denote by $X=G\times_{H}S$ the space
of $H$-orbits in $G\times S$ and $p:G\times S\to X$ the natural
projection. There is an action of $G$ on the quotient $X$ induced
from $G\curvearrowright G\times S$ by $g.(g',s)=(gg',s)$. The space
$X$ with the quotient Borel structure and quotient measure from $\left(G\times S,m_{G}\times\eta\right)$
is called the $G$-space induced from $H\curvearrowright(S,\eta)$.
When no ambiguity arises, we write $[g,s]$ for the $H$-orbit $p(g,s)$.

Another way to describe the induced $G$-action is through a cocycle.
Choose a Borel section $\theta:G/H\to G$ of the natural projection
$G\to G/H$, such that $\theta([e])=e$. Let $\alpha:G\times G/H\to H$
be the cocycle defined as $\alpha(g,[g'])=\theta([gg'])^{-1}g\theta([g'])$.
Denote by $\left(G/H\times_{\alpha}S,p_{\ast}m_{G}\times\eta\right)$
the $G$-space where $G$ acts by $g.([g'],s)=\left([gg'],\alpha(g,[g']).s\right)$.
As $G$-spaces, $G/H\times_{\alpha}S$ is isomorphic to $X$ via the
map 
\[
([g],s)\to p(\theta([g]),s),
\]
see \cite[Proposition 2.2]{ZimmerInduced}.

\subsection{$G$-spaces and factor maps }

We follow the preliminaries in \cite[Section 2]{Bader-Shalom}. Let
$G$ be a locally compact second countable group. We say a Lebesgue
space $\left(X,\eta\right)$ is a $G$-space if $G$ acts measurably
on $X$ and the probability measure $\eta$ is quasi-invariant with
respect to the $G$-action. Given a probability measure $\mu$ on $G$,
we say $\left(X,\eta\right)$ is a stationary $\left(G,\mu\right)$-space
if in addition $\mu\ast\eta=\eta$.

A measure $\mu$ on $G$ is admissible if its support generates $G$
as a semi-group and some convolution power is absolutely continuous
with respect to Haar measure. It is in the $B_{\infty}$ class if
it is furthermore absolutely continuous with respect to Haar measure and admits
a bounded density with compact support.

By \cite{Varadarajan}, we may take a \emph{compact model} for $(X,\eta)$,
that is, a compact metric space $Z$ on which $G$ acts continuously,
equipped with a probability measure $\eta_{Z}$ on its Borel $\sigma$-algebra,
such that $\left(X,\eta\right)$ and $\left(Z,\eta_{Z}\right)$ are
measurably isomorphic $G$-spaces. The space of probability measures
on a compact metric space $Z$ is denoted by $P(Z)$. Equip $P(Z)$
with the weak$^{\ast}$-topology.

For a $G$-map $\pi:(X,\eta)\to(Y,\nu)$ between two compact $G$-spaces,
there exists a Borel map $\sigma:X\to Y$ such that $\sigma=\pi$,
$\eta$-a.e.. We say $\pi$ is a $G$-factor map if $Y=\pi(X)$ and
$\nu=\pi_{\ast}\eta$; the set $\pi^{-1}\left(\left\{ y\right\} \right)$
is called the fiber over $y$. Denote by $D_{\pi}:Y\to P(X)$ the
disintegration map, which is the unique map with the property that
for $\nu$-a.e. $y$, $D_{\pi}(y)$ is supported on the fiber $\pi^{-1}\left(\left\{ y\right\} \right)$,
and $\int_{Y}D_{\pi}(y)d\nu(y)=\eta$. We will often write $\eta^{y}:=D_{\pi}(y)$.
A $G$-factor map $\pi:(X,\eta)\to(Y,\nu)$ is called a \emph{measure
preserving extension} if $D_{\pi}$ is $G$-equivariant, that is,
$g.\eta^{y}=\eta^{g.y}$ for all $g\in G$ and a.e. $y\in Y$.

By Mackey's point realization theorem, $G$-factors of $(X,\nu)$
correspond to $G$-invariant sub-$\sigma$-algebras on $X$, modulo
zero measure subsets.

\subsubsection{The boundary map\label{subsec:boundary-map}}

Denote by $\left(B,\nu_{B}\right)$ for the Poisson boundary of $(G,\mu)$
and ${\rm bnd}:\left(G^{\mathbb{N}},\mathbb{P}_{\mu}\right)\to\left(B,\nu_{B}\right)$
the map from the trajectory space to the Poisson boundary.

Let $(X,\eta)$ be a $(G,\mu)$-stationary system, $\mu\ast\eta=\eta$.
By the martingale convergence theorem, $\mathbb{P}_{\mu}$-a.s., 
\[
\eta_{\omega}=\lim_{n\to\infty}\omega_{n}.\eta\mbox{ exists}.
\]
Since the map $\omega\mapsto\eta_{\omega}$ is measurable with respect
to the invariant $\sigma$-field of the random walk, it factorizes
through the Poisson boundary of $(G,\mu)$. That is, we have a $G$-measurable
map 
\begin{align*}
\boldsymbol{\beta}_{\eta}:B & \to P(X)\\
{\rm bnd}(\omega) & \mapsto\eta_{\omega},
\end{align*}
where $P(X)$ is the space of probability measures on $X$. The map
$\boldsymbol{\beta}_{\eta}$ is called the boundary map, it is the
essentially unique measurable $G$-map $B\to P(X)$ which satisfies
the barycenter property that 
\[
\eta=\int_{B}\boldsymbol{\beta}_{\eta}(b)d\nu_{B}(b),
\]
see for instance \cite[Theorem 2.16]{Bader-Shalom}.

Recall the following terminologies. 
\begin{itemize}
\item The $(G,\mu)$-space $(X,\eta)$ is a $\mu$-\emph{boundary} (equivalently
$\mu$-\emph{proximal}) if $\mathbb{P}_{\mu}$-a.s., the measures $\eta_{\omega}\in\mathcal{M}(X)$
are point masses. In other words, if $(X,\eta)$ is a $G$-factor
of the Poisson boundary $(B,\nu_{B})$. 
\item A $G$-factor map $\pi:(X,\eta)\to(Y,\nu)$ is called a $\mu$-\emph{proximal
extension} if $\mathbb{P}_{\mu}$-a.s., the extension $\left(X,\eta_{\omega}\right)\to(Y,\nu_{\omega})$
is a.s. one-to-one. 
\item We call a $(G,\mu)$-stationary system $\left(X,\eta\right)$ \emph{standard}\textbf{
}if there exists a $G$-factor map $\pi:(X,\eta)\to(Y,\nu)$ with
$(Y,\nu)$ a $\mu$-proximal system and $\pi$ a measure preserving
extension. By \cite[Proposition 4.2]{FurstenbergGlasner}, the structure
of a standard system as a measure preserving extension of a proximal
system is unique. 
\end{itemize}

\subsubsection{Furstenberg isomorphism}\label{subsubsec:Furstenberg}

Let $G$ be a lcsc group equipped with an admissible probability measure
$\mu$ on $G$. Assume that the Poisson boundary of the $\mu$-random
walk can be identified with a homogenous space $G/H$. Following \cite{Furstenberg1,Furstenberg2},
in this situation the boundary map can be interpreted as what is now
called the \emph{Furstenberg isomorphism/correspondence}. 

Since $\mu$ is admissible, then the stationary measure $\nu$ is
in the quasi-invariant measure class on $G/H$. Since the action of
$G$ on its Poisson boundary is amenable in the sense of Zimmer, we
have that the subgroup $H$ is necessarily amenable.

Given a locally compact $G$-space $X$, denote by $\mathcal{P}_{\mu}(X)$
the space of $\mu$-stationary probability measures on $X$, and $\mathcal{P}_{H}(X)$
the space of $H$-invariant probability measures on $X$. Then by
\cite[Lemma 2.1]{Furstenberg2}, there is an isomorphism between the
affine spaces $\mathcal{P}_{\mu}(X)$ and $\mathcal{P}_{H}(X)$, implemented
by 
\begin{align*}
\psi_{\mu} & :\mathcal{P}_{\mu}(X)\to\mathcal{P}_{H}(X)\\
\eta & \mapsto\lambda=\boldsymbol{\beta}_{\eta}(H),
\end{align*}
where $\boldsymbol{\beta}_{\eta}:G/H\to\mathcal{P}(X)$ is the $G$-boundary
map associated with $\eta$. Denote by $\nu_{H}$ the $\mu$-harmonic
measure on $G/H$, we have that the barycenter map is implemented
by 
\[
\eta=\int_{G/H}\boldsymbol{\beta}_{\eta}(w)d\nu_{H}(w)=\int_{G/H}g.\lambda d\nu_{H}(gH).
\]

\begin{notation}\label{convolution-lambda}

Let $X$ be a $G$-space, $H$ a closed subgroup of $G$. Given a
measure $\nu_{0}$ on $G/H$ and $H$-invariant measure $\lambda$
on $X$, we write $\nu_{0}\ast\lambda:=\tilde{\nu}_{0}\ast\lambda=\int_{G}g.\lambda d\tilde{\nu}_{0}(g)$,
where $\tilde{\nu}_{0}$ is any lift of $\nu_{0}$ to ${\rm Prob}(G)$.
Since $\lambda$ is $H$-invariant, it does not depend on the choice
of $\tilde{\nu}_0$. In this setting, denote by $X_{0}$ the support
of $\lambda$. We can then view $G\curvearrowright\left(X,\nu\ast\lambda\right)$
as a factor of the induced system $\left(G/H\times_{\alpha}X,\nu\times\lambda\right)$,
see \cite[Proposition 2.5]{NZ1}.

\end{notation}

We will refer to the map $\psi_{\mu}:\mathcal{P}_{\mu}(X)\to\mathcal{P}_{H}(X)$
as the Furstenberg isomorphism. This isomorphism is continuous with
respect to the weak$^{\ast}$ topology. As a consequence of the isomorphism,
we have uniqueness of $\mu$-stationary measure on $X$ is equivalent
to uniqueness of $H$-invariant measure on $X$.

Recall that the action of $G$ on a compact space $X$ is said to
be \emph{strongly proximal} if for any probability measure $\eta$
on $X$, there exists a sequence of elements $\left(g_{n}\right)$
in $G$ such that $g_{n}.\eta$ converges weakly to a $\delta$-mass.
When $G\curvearrowright X$ is strongly proximal, by \cite[Theorem 2.3]{Furstenberg2},
the $G$-space $X\times X$ supports a unique stationary measure for
$\mu$; and this measure is concentrated on the diagonal of $X\times X$.
It follows in particular that if $G\curvearrowright X$ is strongly
proximal, then $H$-invariant measure on $X$ is unique.

\subsubsection{Furstenberg entropy}

Recall that given two probability measures $P$ and $Q$ on the same
space $\Omega$ with respect to the same $\sigma$-algebra $\mathcal{B}$, if $P$ is absolutely continuous with
respect to $Q$, the relative entropy of $P$ with respect to $Q$,
also known as their Kullback-Leibler divergence, is defined as the
integral (possibly infinite): 
\[
\mathrm{D}\left(P\bigparallel Q\right):=\int_{\Omega}\left(\log\frac{dP}{dQ}\right)dP.
\]
The Furstenberg entropy, as defined in (\ref{eq:h-def}) can be written
as $h_{\mu}(X,\nu)=\int_{G}D\left(\nu\parallel g^{-1}\nu\right)d\mu(g)=\int_{G}D\left(g.\nu\parallel\nu\right)d\mu(g)$.
We refer to \cite[Section 1]{NZ3} for a detailed account on Furstenberg
entropy, and only recall here a few well-known properties.

Furstenberg entropy is monotone under factors: when $\pi:(X,\eta)\to(Y,\nu)$
is a $G$-map, we have $h_{\mu}(Y,\nu)\le h_{\mu}(X,\eta)$ with equality
if and only if $\pi$ is measure preserving. The Furstenberg entropy
of a $(G,\mu)$-space is maximal for the Poisson boundary $0\le h_{\mu}(X,\eta)\le h_{\mu}(B,\nu_{B})$.

When $\eta'$ is a probability measure in the measure class of $\eta$,
then by \cite[Lemma 8.9]{Furstenberg2}, we have 
\[
h_{\mu}\left(X,\eta\right)=-\int_{G}\int_{X}\log\frac{dg^{-1}\eta'}{d\eta'}(x)d\eta(x)d\mu(g).
\]
This allows to view the Furstenberg entropy as a cohomology invariant.
It permits to change the stationary measure $\eta$ in the Radon-Nikodym
derivative in order to compute the entropy.

We also record the following property of the barycenter map.

\begin{lemma}\label{barycenter}

Let $\left(C,\nu_{C}\right)$ and $\left(X,\eta\right)$ be two nonsingular
$G$-spaces. Suppose there is a $G$-map $\beta:C\to\mathcal{P}(X)$
such that $\eta$ is the barycenter of $\beta_{\ast}(\nu_{C})$. Then
\[
h_{\mu}\left(C,\nu_{C}\right)\ge h_{\mu}\left(X,\eta\right).
\]

\end{lemma} 
\begin{proof}
Consider the product space $C\times X$ on which $G$ acts diagonally,
and equip it with the measure $\lambda$ such that 
\[
\int_{C\times X}fd\lambda=\int_{C}\int_{X}f(c,x)d\beta_{c}(x)d\nu_{C}(c).
\]
The map $C\times X\to C$ with $(c,x)\mapsto c$ is a measure-preserving
extension as $\beta$ is equivariant. It follows that $h_{\mu}(C\times X,\lambda)=h_{\mu}(C,\nu_{C})$.
Since $\eta=\mathrm{bar}\left(\beta_{\ast}(\nu_{C})\right)$, the
coordinate projection $C\times X\to X$ pushes forward the measure
$\lambda$ to $\eta$. Therefore $h_{\mu}(X,\eta)\le h_{\mu}(C\times X,\lambda)$. 
\end{proof}

\subsection{Structure of parabolic subgroups\label{subsec:structure-of-parabolic}}

Let $G$ be a semisimple real Lie group, we recall some structure
theory, see \cite{KnappBook} and also \cite[Section 2]{NZ2}.
Denote by $\mathfrak{g}$ the Lie algebra of $G$. Let $\theta:\mathfrak{g}\to\mathfrak{g}$
be a Cartan involution on $\mathfrak{g}$. We have the Cartan decomposition
$\mathfrak{g}=\mathfrak{k}\oplus\mathfrak{p}$, where $\mathfrak{k}$
($\mathfrak{p}$ resp.) is the $+1$ ($-1$ resp.) eigenspace of $\theta$.
Let $\mathfrak{a}$ be a maximal commutative subalgebra of $\mathfrak{p}$.
Denote by $\Phi=\Phi(\mathfrak{a},\mathfrak{g})$ the set of restricted
roots. For a fixed ordering on $\Phi$, denote by $\Delta$ simple
roots. For $\alpha\in\Phi$, denote by $\mathfrak{g}_{\alpha}$ the
restricted root space $\mathfrak{g}_{\alpha}=\{x\in\mathfrak{g}:[h,x]=\alpha(h)x\mbox{ for all }h\in\mathfrak{a}\}$, and set $\mathfrak{n}=\oplus_{\alpha >0}\mathfrak{g}_{\alpha}$.
Let $G=KAN$ denote the Iwasawa decomposition, where $K,A,N$ are
the analytic subgroups of $G$ corresponding to $\mathfrak{k},\mathfrak{a},\mathfrak{n}$.

Conjugacy classes of parabolic subalgebras of $\mathfrak{g}$ are
parametrized by subsets of $\Delta$. For $I\subseteq\Delta$, let
\[
\mathfrak{a}_{I}=\bigcap_{\alpha\in I}\ker\alpha.
\]
Let $\mathfrak{z}(\mathfrak{h})$ denote the centralizer of the
subalgebra $\mathfrak{h}$ in $\mathfrak{g}$. Let $\mathfrak{m}_{I}$
be the orthogonal complement of $\mathfrak{a}_{I}$ in $\mathfrak{z}(\mathfrak{a}_{I})$
with respect to the restriction of the Killing form, $\mathfrak{z}(\mathfrak{a}_{I})=\mathfrak{m}_{I}\oplus\mathfrak{a}_{I}$. 

Let $[I]$ denote the set of roots in $\Phi$
expressible as integral linear combination of elements in $I$. Let
\[
\mathfrak{n}^{I}=\oplus_{\alpha>0,\alpha\notin[I]}\mathfrak{g}_{\alpha},\ \mathfrak{n}^{-I}=\oplus_{\alpha<0,\alpha\notin[I]}\mathfrak{g}_{\alpha}.
\]
Then the parabolic subalgebra $\mathfrak{p}_{I}$ admits the decomposition
\[
\mathfrak{p}_{I}=\mathfrak{m}_{I}\oplus\mathfrak{a}_{I}\oplus\mathfrak{n}^{I}.
\]

The parabolic subgroup $P_{I}$ is the normalizer of the parabolic
subalgebra $\mathfrak{p}_{I}$. Denote by $A_{I}=\exp(\mathfrak{a}_{I})$,
$N_{I}=\exp\left(\mathfrak{n}^{I}\right)$, $\bar{N}_{I}=\exp\left(\mathfrak{n}^{-I}\right)$
and $L_{I}=Z_{G}\left(A_{I}\right)$ the centralizer of $A_{I}$ in
$G$. Then $P_{I}$ admits the Levi decomposition $P_{I}=L_{I}\rtimes N_{I}$,
and the Levi subgroup $L_{I}$ is reductive : it is the product of $M_{I}$ with its split component $A_{I}$. The decomposition $P_{I}=M_{I}A_{I}N_{I}$ is
called the Langlands decomposition of $P_{I}$, see \cite[Propositions 7.82 and 7.83]{KnappBook}.

The minimal parabolic subgroup $P$ corresponds to the empty subet
$I=\emptyset$. Write the corresponding decompositions as $\mathfrak{p}=\mathfrak{p}_{\emptyset}=\mathfrak{m}\oplus\mathfrak{a}\oplus\mathfrak{n}$
and $P=P_{\emptyset}=MAN$. By \cite[Proposition 7.82]{KnappBook},
we have that $M_{I}=M_{I}^{0}M$, where $M_{I}^{0}$ is the identity
component of $M_{I}$. It follows that $P_{I}=M_{I}^{0}P$.

For $I\subseteq\Delta$, let 
\begin{align}
R_{I} & :=\left\{ \exp(s):s\in\mathfrak{a},\alpha_{1}(s)\le0,\alpha_{2}(s)<0\mbox{ for all }\alpha_{1}\in\Delta\mbox{ and }\alpha_{2}\in\Delta\setminus I\right\} ,\nonumber \\
D_{I} & :=R_{I}\cap A_{I}=\left\{ \exp(s):s\in\mathfrak{a}_{I},\alpha(s)<0\mbox{ for all }\alpha\in\Delta\setminus I\right\} .\label{eq:D_I}
\end{align}
The set $D_{I}$ is nonempty if and only if $I\neq\Delta$ (\cite[Prop I.2.4.2]{MargulisBook}).
For $I\subseteq\Delta$, $s\in R_{I}$, by \cite[Lemma 3.1]{MargulisBook}
the automorphisms ${\rm Int}(s)|_{N_{I}}$ and ${\rm Int}(s^{-1})|_{\bar{N}_{I}}$
are contracting, where ${\rm Int}(g).x=gxg^{-1}$ is conjugation by
$g$.

\begin{example}
For $G=SL(5,\mathbb{R})$ and $I=\{1,3,4\}$, we have
\[
L_I=\left(\begin{array}{ccccc}
\ast & \ast & 0 & 0 & 0 \\ 
\ast & \ast & 0 & 0 & 0 \\
0 & 0 & \ast & \ast & \ast \\
0 & 0 & \ast & \ast & \ast \\
0 & 0 & \ast & \ast & \ast \\
\end{array} \right), 
N_I=\left(\begin{array}{ccccc}
1 & 0 & \ast & \ast & \ast \\ 
0 & 1 &  \ast & \ast & \ast \\
0 & 0 & 1& 0 & 0 \\
0 & 0 & 0& 1 & 0 \\
0 & 0 & 0& 0 & 1 \\
\end{array} \right),
\bar{N}_I=\left(\begin{array}{ccccc}
1 & 0 & 0 & 0 & 0 \\ 
0 & 1 &  0 & 0 & 0 \\ 
\ast & \ast & 1& 0 & 0 \\
\ast & \ast & 0& 1 & 0 \\
\ast & \ast & 0& 0 & 1 \\
\end{array} \right),
\]
\[
D_I=\left\{\left(\begin{array}{cc}
e^{-t_1} \mathrm{I}_2 & 0  \\ 
0 & e^{-t_2} \mathrm{I}_3 \\ 
\end{array} \right), t_1>t_2\right\}, \quad \textrm{ where } \mathrm{I}_k \textrm{ is the $k\times k$ identity matrix.}
\]
\end{example}

\vspace{0.5cm}

\begin{center}
\textsc{Part I : Constraints on entropy spectrum}
\end{center}

\section{Properties of the Nevo-Zimmer maximal projective factor \label{sec:first-steps}}

\subsection{Statement }

Let $G$ be a connected semisimple real Lie group with finite center
and $P$ be a minimal parabolic subgroup of $G$. Suppose $\nu_{P}$
is a probability measure in the $G$-quasi-invariant measure class
on $G/P$. For a parabolic subgroup $Q>P$, denote by $\nu_{Q}$ the
pushforward of $\nu_{P}$ under the projection $G/P\to G/Q$.

We are given a $G$-system $(X,\nu)$, where $X$ is taken to be a
compact model. Suppose there is a $P$-invariant measure $\lambda$
on $X$ such that $\nu=\nu_{P}\ast\lambda$, where the convolution
is explained in Notation~\ref{convolution-lambda}. 

\begin{example}

This is the setting of \cite[Theorem 1]{NZ2}. The measure $\nu_{P}$ is the $\mu$-harmomic
measure on $G/P$ for some admissible step distribution $\mu$ on
$G$; $(X,\nu)$ is a $(G,\mu)$-stationary system and $\lambda$
is provided by the Furstenberg isomorphism.

\end{example}

Denote by $\Delta$ simple roots of $G$. Recall that a standard parabolic
subgroup $P_{I}$, $I\subseteq\Delta$, admits the Levi decomposition
$P_{I}=L_{I}\rtimes N_{I}$, where $N_{I}$ is the unipotent radical
of $P_{I}$, the Levi subgroup $L_{I}=Z_{G}(A_{I})$ is reductive.
Our goal in this section is to prove the following.

\begin{theorem}\label{invariant-lambda}

Let $G$, $\nu_{P}$, $(X,\nu)$ be as above, where $\nu=\nu_{P}\ast\lambda$
for a $P$-invariant, $P$-ergodic measure $\lambda$ on $X$. Suppose
$\left(G/Q,\nu_{Q}\right)$ is a maximal projective factor of $(X,\nu)$,
where $Q=P_{I}$ is a parabolic subgroup of $G$. Then each connected
simple non-compact factor with $\mathbb{R}$-rank $\ge2$ of the Levi
subgroup $L_{I}$ of $Q$ preserves the measure $\lambda$.

\end{theorem}

The existence of a (unique) maximal projective factor follows from
\cite[Lemma 0.1]{NZ2}. Consequences of Theorem \ref{invariant-lambda}
will be derived in Section \ref{sec:consequence}. Throughout the
rest of this section, we assume the setting of Theorem \ref{invariant-lambda}.

\subsection{$Q$-system and disintegration of the Haar measure $m_{K}$ }

\subsubsection{Notations for Borel sections and cocycles \label{subsec:Notations}}

We first set some notations for the cocycles that appear in the induced
systems. Fix a choice of Borel sections $\theta:Q/P\to Q$ and $\tau:G/Q\to G$
with the property that $\theta([P])=e$ and $\tau([Q])=e$. For later
convenience, we also require that $\theta(Q/P)\subseteq Q\cap K$
and $\tau\left(G/Q\right)\subseteq K$: this is possible because $K$
acts transitively on $G/P$. Denote by $\beta:G/Q\times G\to Q$ the
cocycle associated with the section $\tau$. Then we have a Borel
section 
\begin{align}
\vartheta:G/P & \to G\nonumber \\
gP & \mapsto\tau(gQ)\theta\left(\tau(gQ)^{-1}gP\right).\label{eq:section}
\end{align}
Denote by $\alpha$ the associated cocycle $G/P\times G\to G$ with
this section. Since $\tau(Q)=e$, we have that $\vartheta$ restricted
to $Q/P$ agrees with $\theta$. Then $\alpha$ restricted to $Q/P\times Q$
is the cocycle associated with $\theta$. In the notation for a $Q$-system
$Q/P\times_{\alpha}X_{0}$, it is understood that the cocycle $\alpha$
is restricted to $Q/P\times Q$.

By the \textquotedbl inducing in stages\textquotedbl{} property (see
e.g., \cite[Proposition 2.4]{ZimmerInduced}), the two systems $G/Q\times_{\beta}\left(Q/P\times_{\alpha}X_{0}\right)$
and $G/P\times_{\alpha}X_{0}$ are isomorphic, via a $G$-isomorphism
\begin{align}
j_{P}^{Q}:G/Q\times_{\beta}\left(Q/P\times_{\alpha}X_{0}\right) & \to G/P\times_{\alpha}X_{0}\nonumber \\
\left(gQ,qP,x_{0}\right) & \mapsto\left(\tau\left(gQ\right)\alpha\left(qP\right)P,x_{0}\right).\label{eq:jPQ}
\end{align}

\subsubsection{Change to $K$-invariant measures and decomposition of Haar measure\label{subsec:Preliminary-step}}

For our purposes it is convenient to have that the harmonic measure
on $G/P$ is $K$-invariant, such that disintegration of Haar measures
can be applied. Recall that $\nu_{0}$ denotes the $\mu$-harmonic
measure on $G/P$. Denote by $\lambda$ the $P$-invariant measure
on $X$ given by the Furstenberg isomorphism, $\nu=\nu_{P}\ast\lambda$.
Equivalently, the boundary map $\boldsymbol{\beta}_{\nu}:G/P\to\mathcal{P}(X)$
sends $gP$ to $g.\lambda$.

Denote by $G=KAN$ the Iwasawa decomposition of $G$ and $m_{K}$
the normalized Haar measure on the compact subgroup $K$. Since $\nu_{P}$
is assumed to be in the $G$-quasi-invariant measures on $G/P$, it
follows that $\nu=\nu_{P}\ast\lambda$ and $\nu'=m_{K}\ast\lambda$
are in the same measure class on $X$ as well, see \cite[Corollary 1.5]{NZ1}.
The boundary map associated with $\nu'$ sends $gP$ to $g.\lambda$,
that is, $\boldsymbol{\beta}_{\nu'}=\boldsymbol{\beta}_{\nu}$. Denote
by $\bar{m}_{K}$ the pushforward of $m_{K}$ under the natural projection
$G\to G/Q$. Then $\left(G/Q,\bar{m}_{K}\right)$ is a maximal projective
factor of $\left(X,\nu'\right)$ if and only if $\left(G/Q,\nu_{Q}\right)$
is a maximal projective factor of $(X,\nu)$. Therefore to prove Theorem
\ref{invariant-lambda}, we may replace $(X,\nu)$ by $(X,\nu')$
where $\nu'=m_{K}\ast\lambda$.

Consider the decomposition of the Haar measure on $K$ over the closed
subgroup $K\cap Q$. Since $K$ is compact, both $K$ and $K\cap Q$
are unimodular. Denote by $\bar{m}_K$ the (unique) $K$-invariant
probability measure on $K/K\cap Q$ and $m_{K\cap Q}$ the Haar measure
on $K\cap Q$ normalized to have total mass $1$. Recall that in Subsection
\ref{subsec:Notations} we have chosen a Borel section $\tau:G/Q\to G$
such that $\tau(Q)=e$ and $\tau(G/Q)\subseteq K$. Then the decomposition
of Haar measure (see e.g., \cite[Theorem 8.36]{KnappBook}) implies
the disintegration 
\begin{equation}
m_{K}=\int_{G/Q}\tau(y).m_{K\cap Q}d\bar{m}_{K}(y).\label{eq:disintegration}
\end{equation}

\subsubsection{Structure of the $Q$-system}

We start in the same way as the proof of \cite[Theorem 11.4]{NZ2}.
The assumptions of Theorem \ref{invariant-lambda} imply that $(X,\nu)$
fits into the sequence of $G$-spaces: 
\begin{equation}
G/P\times_{\alpha}X_{0}\overset{\xi}{\to}X\overset{\varphi}{\to}G/Q.\label{eq:G-factors}
\end{equation}
Here $\varphi:X\to G/Q$
is the $G$-map to the maximal projective factor, $X_0$ is the support of the measure $\lambda=\psi_{\mu}(\nu)=\boldsymbol{\beta}_{\nu}(P)$ and  the factor map $\xi:G/P\times_{\alpha}X_{0}\to X$ is given
by $\xi\left(gP,x_{0}\right)=\vartheta(gP).x_{0}$, where $\vartheta$ is the section map defined in (\ref{eq:section}). Also note  $\xi_{\ast}(p_{\ast}m_K \times \lambda)=\nu$ where $p:G\to G/P$ is the quotient map.

By \cite[Theorem 2.5]{ZimmerInduced}
we have that $X$ is induced from an ergodic action of $Q$. Next
we derive some information on the $Q$-system that arises this way.
The main property we will use is that such a $Q$-system is induced
from the $P$-system $\left(X_{0},\lambda\right)$ as well, see Proposition
\ref{Ysystem} below.

The following lemma is well-known, it is based on the fact that the
only $P$-invariant measure on $G/Q$ is the $\delta$-mass at the
identity coset $Q$.

\begin{lemma}\label{Qcoset}

Suppose $(X,\nu)$ fits into the sequence of $G$-spaces (\ref{eq:G-factors}).
Then we have for $\lambda$-a.e. $x_{0}\in X_{0}$, 
\[
\varphi\circ\xi\left(g,x_{0}\right)=gQ\mbox{ for all }g\in G.
\]

\end{lemma} 
\begin{proof}
Since $G/Q$ is a strongly proximal boundary, by \cite[Theorem 2.2]{Furstenberg2},
there is a unique $P$-invariant probability measure on $G/Q$. The
point mass at the identity coset $Q$ is invariant under $P$, thus
it is the unique $P$-invariant measure on $G/Q$.

The measure $\lambda$ on $X_{0}$, also viewed as a measure supported
on the set $\left\{ \left(P,x_{0}\right):x_{0}\in X_{0}\right\} $
in $G/P\times_{\alpha}X_{0}$, is $P$-invariant. Therefore its pushforward
under $\varphi\circ\xi$ is a $P$-invariant measure on $G/Q$, which
must be $\delta_{Q}$. By $G$-equivariance, we have then for any
$g\in G$ and $x_{0}\in X_{0}'$ 
\[
\varphi\circ\xi\left(gP,x_{0}\right)=\varphi\circ\xi\left(g.(P,\alpha(g,P)^{-1}.x_{0})\right)=g.\left(\varphi\circ\xi\left(P,\alpha(g,P)^{-1}.x_{0}\right)\right)=gQ.
\]
\end{proof}
Next we use the disintegration (\ref{eq:disintegration}) to specify
a $Q$-system, which will be denoted as $(Y,\eta)$. Recall
that $\beta$ is the cocycle associated with the section $\tau:G/Q\to G$.

\begin{proposition} \label{Ysystem}

Assume (\ref{eq:G-factors}). Define $Y$ as the subset of $X$ 
\[
Y:=\xi\left(Q/P\times_{\alpha}X_{0}\right),
\]
equipped with the measure $\eta:=m_{K\cap Q}\ast\lambda$. Then the
induced system $\left(G/Q\times_{\beta}Y,\bar{m}_{K}\times\eta\right)$
is $G$-isomorphic to $(X,\nu)$ via the map 
\begin{align*}
\phi & :G/Q\times_{\beta}Y\to X\\
 & \left(gQ,y\right)\mapsto\tau\left(gQ\right).y.
\end{align*}

\end{proposition} 
\begin{proof}
Define $\tilde{\xi}$ to be the map 
\begin{align*}
\tilde{\xi} & :G/Q\times_{\beta}\left(Q/P\times_{\alpha}X_{0}\right)\to G/Q\times_{\beta}Y\\
 & \left(gQ,\left(qP,x_{0}\right)\right)\mapsto\left(gQ,\vartheta(gP).x_{0}\right).
\end{align*}
Recall the notations that $p:G\to G/P$ is the natural projection,
and $\bar{m}_{K}$ denotes the pushforward of $m_{K}$ under the projection
$G\to G/Q$. Write $Z=G/Q\times_{\beta}\left(Q/P\times_{\alpha}X_{0}\right)$.
By inducing in stages, we have a $G$-isomorphism $j_{P}^{Q}:Z\to G/P\times_{\alpha}X_{0}$
as in (\ref{eq:jPQ}). 
\begin{claim}
Let $\varphi:X\to G/Q$ be the $G$-factor map in (\ref{eq:G-factors}).
Then we have a sequence of $G$-factors:

\[
\left(Z,\bar{m}_{K}\times\left(p_{\ast}\left(m_{K\cap Q}\right)\times\lambda\right)\right)\overset{\tilde{\xi}}{\to}\left(G/Q\times_{\beta}Y,\bar{m}_{K}\times\eta\right)\overset{\phi}{\to}\left(X,\nu\right)\overset{\varphi}{\to}\left(G/Q,\bar{m}_{K}\right).
\]
\end{claim}

\begin{proof}[Proof of the claim]

The measurability and $G$-equivariance of the maps $\tilde{\xi}$
and $\phi$ are clear by their definitions. Also by the definitions
of the maps the following diagram commute:

\[
\xymatrix{Z\ar[r]^{\tilde{\xi}}\ar[d]_{j_{P}^{Q}} & G/Q\times_{\beta}Y\ar[d]^{\phi}\\
G/P\times_{\alpha}X_{0}\ar[r]^{\xi} & X.
}
\]
Since $j_{P}^{Q}$ is a $G$-isomorphism and $\xi$ is a $G$-factor
map, we see that $\tilde{\xi}$ and $\phi$ are $G$-factor maps as
well.

We need to verify that the measures follow the maps. The measure $\eta$
is defined as $m_{K\cap Q}\ast\lambda$. Since $k^{-1}\vartheta(kP)\in P$
and $\lambda$ is $P$ invariant, we have $k.\lambda=\vartheta(kP).\lambda$
and then 
\[
\xi_{\ast}\left(p_{\ast}\left(m_{K\cap Q}\right)\times\lambda\right)=\int_{K\cap Q}\vartheta(kP).\lambda dm_{K\cap Q}(k)=\int_{K\cap Q}k.\lambda dm_{K\cap Q}(k)=m_{K\cap Q}\ast\lambda=\eta.
\]
Therefore for the map $\tilde{\xi}$, 
\[
\tilde{\xi}_{\ast}\left(\bar{m}_{K}\times\left(p_{\ast}\left(m_{K\cap Q}\right)\times\lambda\right)\right)=\bar{m}_{K}\times\xi_{\ast}\left(p_{\ast}\left(m_{K\cap Q}\right)\times\lambda\right)=\bar{m}_{K}\times\eta.
\]
Next we verify that $\phi_{\ast}\left(\bar{m}_{K}\times\eta\right)=\nu$:
\begin{align*}
\phi_{\ast}\left(\bar{m}_{K}\times\eta\right) & =\int_{G/Q}\int_{K\cap Q}\tau(y)k.\lambda dm_{K\cap Q}(k)d\bar{m}_{K}(y)\\
 & =\left(\int_{G/Q}\tau(y).m_{K\cap Q}d\bar{m}_{K}(y)\right)\ast\lambda=m_{K}\ast\lambda=\nu.
\end{align*}
In the second line we plugged in the decomposition formula (\ref{eq:disintegration}). 
\end{proof}
For the sequence of $G$-factor maps in the Claim, we have $\phi\circ\tilde{\xi}=\xi\circ j_{P}^{Q}$
and by Lemma \ref{Qcoset}, $\varphi\circ\phi\circ\tilde{\xi}\left(gQ,qP,x_{0}\right)=\tau(gQ)\vartheta(qP)Q=gQ.$
To show that the map $\phi$ is indeed a measurable $G$-isomorphism,
it remains to verify that $\phi$ is injective almost everywhere.
Since $\varphi\circ\phi\left(gQ,y\right)=gQ$, necessarily $\phi^{-1}(\{x\})\subseteq\{\varphi(x)\}\times Y$.
When restricted to this fiber, we have $\phi(\varphi(x),y)=\theta(\varphi(x)).y$
which is injective in $y\in Y$. We conclude that $\phi$ is an isomorphism. 
\end{proof}
As shown in the proof of \cite[Theorem 11.4]{NZ2}, suppose the $Q$-system
$\left(Y,\eta\right)$ admits a projective factor $\left(Q/Q_{1},\bar{\eta}\right)$,
where $Q_{1}<Q$ is a proper closed subgroup of $Q$, then the $G$-system
$\left(X,\nu\right)$ admits $\left(G/Q_{1},\nu_{Q_{1}}\right)$ as
a projective factor. To proceed, we will carry out the inductive step
which applies the Nevo-Zimmer arguments to the $Q$-systems 
\[
\left(Q/P\times_{\alpha}X_{0},\bar{m}_{K\cap Q}\times\lambda\right)\overset{\xi}{\to}(Y,\eta).
\]
The goal is to show that if a higher-rank factor of $L_{I}=Z_{G}(S_{I})$
does not preserve the measure $\lambda$, then we will be able to
find a nontrivial projective factor $\left(Q/Q_{1},\bar{\eta}\right)$
of $(Y,\eta)$, contradicting the assumption that $G/Q$ is the maximal
projective factor. This will be carried out in the next subsections.

\subsection{The Nevo-Zimmer operation applied to parabolic subgroups\label{subsec:The-Nevo-Zimmer-operation}}

Throughout this subsection, we assume the setting of Theorem \ref{invariant-lambda}
and (\ref{eq:G-factors}). We have a $Q$-system $(Y,\eta)$ described
in Proposition \ref{Ysystem}, which fits in the setting of Appendix~\ref{sec:operation}
with $Q=P_{I}$ the lcsc group, its closed subgroup $P=P_{\emptyset}$
and 
\[
\xi_{0}:\left(Q\times_{P}X_{0},m_{K\cap Q}\times\lambda\right)\to\left(Y,\eta\right).
\]

We use notations for parabolic subgroups as in Subsection \ref{subsec:structure-of-parabolic}.
Take the Langlands decomposition $Q=M_{I}A_{I}N_{I}$. Denote by $M_{1},\ldots,M_{\ell}$
the noncompact simple factors of the connected component $M_{I}^{0}$,
and $\mathfrak{m}_{1},\ldots,\mathfrak{m}_{\ell}$ the corresponding
Lie algebras. \textcolor{black}{For each noncompact simple factor
$M_{j}$, let $\mathfrak{a}_{j}=\mathfrak{m}_{j}\cap\mathfrak{a}$.
Denote by $I_{i}$ be the subset of $I$ that consists of $\alpha\in I$
such that $\alpha$ vanishes on all $\mathfrak{a}_{k}$, $k\neq i$.}

Denote by $p:G\to G/P$ the natural projection. Define $\bar{U}_{I}=Q\cap\bar{N}$,
where $\bar{N}=\bar{N}_{\emptyset}$. Restricted to $Q$, we have
that $p$ maps $\bar{U}_{I}$ diffeomorphically onto an open dense
conull set in $\left(Q/P,\bar{m}_{K\cap Q}\right)$.

Suppose $M_{i}$ is a simple factor of $M_{I}^{0}$ with $\mathbb{R}$-rank
at least $2$, which is fixed in what follows. We need to show that
$M_{i}$ preserves the measure $\lambda$. \textcolor{black}{Take
$\varrho\subset I_{i}$ to be a nonempty proper subset of $I_{i}$.
}Take $s\in D_{\varrho}$, where $D_{\varrho}$ is defined as in (\ref{eq:D_I}):

\[
D_{\varrho}=\left\{ s\in\cap_{\alpha\in\varrho}\ker(\alpha):\alpha(s)<0\mbox{ for all }\alpha\in \Delta-\varrho\right\} .
\]
\textcolor{black}{Then $\bar{U}_{I}$ admits a semi-direct product
decomposition as $\bar{U}_{I}=\bar{U}_{\varrho}\ltimes\bar{V}_{\varrho,I}$,
where $\bar{U}_{\varrho}=\bar{N}\cap P_{\varrho}$ and $\bar{V}_{\varrho,I}=\bar{N}_{\varrho}\cap L_{I}$}.
Note that since $\emptyset\neq\varrho\subsetneq I_{i}$, $\bar{U}_{\varrho}$
is a nontrivial subgroup of $M_{i}$.

\begin{example}
For $G=SL(5,\mathbb{R})$, with $I=\{1,3,4\}$ and $\varrho=\{3\}$, we have
\[
\bar{U}_{I}=\left(\begin{array}{ccccc}
1 & 0 & 0 & 0 & 0 \\ 
\ast & 1 & 0 & 0 & 0 \\
0 & 0 & 1& 0 & 0 \\
0 & 0 & \ast & 1 & 0 \\
0 & 0 & \ast & \ast & 1 \\
\end{array} \right),
\bar{U}_{\varrho}=\left(\begin{array}{ccccc}
1 & 0 & 0 & 0 & 0 \\ 
0 & 1 & 0 & 0 & 0 \\
0 & 0 & 1& 0 & 0 \\
0 & 0 & \ast & 1 & 0 \\
0 & 0 & 0& 0 & 1 \\
\end{array} \right),
\bar{V}_{\varrho,I}=\left(\begin{array}{ccccc}
1 & 0 & 0 & 0 & 0 \\ 
\ast & 1 & 0 & 0 & 0 \\
0 & 0 & 1& 0 & 0 \\
0 & 0 & 0 & 1 & 0 \\
0 & 0 & \ast & \ast & 1 \\
\end{array} \right),
\]
\[
N_{\varrho}=\left(\begin{array}{ccccc}
1 & \ast & \ast &\ast & \ast  \\ 
0 & 1 & \ast & \ast & \ast \\
0 & 0 & 1& 0 & \ast \\
0 & 0 & 0 & 1 & \ast \\
0 & 0 & 0 & 0 & 1 \\
\end{array} \right),
D_{\varrho}=\left\{\left(\begin{array}{ccccc}
e^{-t_1} & 0 & 0 & 0 & 0 \\ 
0 & e^{-t_2} & 0 & 0 & 0 \\
0 & 0 & e^{-t_3}& 0 & 0 \\
0 & 0 & 0 & e^{-t_3} & 0 \\
0 & 0 & 0& 0 & e^{-t_4} \\
\end{array} \right), t_1>t_2>t_3>t_4 \right\}.
\]
\end{example}

We have a parametrization 
\begin{align*}
\xi & :\bar{U}_{\varrho}\times\bar{V}_{\varrho,I}\times X_{0}\to Y\\
 & (u,v,x_{0})\mapsto uv.x_{0}.
\end{align*}
By the choice that $s\in D_{\varrho}$, $\left(s,\bar{U}_{\varrho},\bar{V}_{\varrho,I},N_{\varrho}\right)$
satisfies the conditions in Assumption \ref{NZ-assum} for the Nevo-Zimmer
operation as reviewed in Appendix~\ref{sec:operation}. More precisely, the following properties hold:
\begin{description}
\item [{(i)}] the map 
\begin{align*}
p:\bar{U}_{\varrho}\times\bar{V}_{\varrho,I} & \to Q/P\\
(u,v) & \mapsto uvP
\end{align*}
takes $\bar{U}_{\varrho}\times\bar{V}_{\varrho,I}$ homeomorphically
to a $\bar{m}_{K\cap Q}$-conull set in $Q/P$, and moreover the pushforward
$p_{\ast}\left(m_{\bar{U}_{\varrho}}\times m_{\bar{V}_{\varrho,I}}\right)$
is in the same measure class as $\bar{m}_{K\cap Q}$. 
\item [{(ii)}] ${\rm Int}(s)$ acts trivially on $\bar{U}_{\varrho}=\bar{N}\cap P_{\varrho}$. 
\item [{(iii)}] ${\rm Int}(s^{-1})$ acts as a contracting automorphism
on $\bar{V}_{\varrho,I}=\bar{N}_{\varrho}\cap L_{I}$; ${\rm Int}(s)$
acts as a contracting automorphism on $N_{\varrho}$. 
\end{description}
Denote by $\tilde{L}^{\infty}\left(Y\right)$ the lifts of functions
in $L^{\infty}(Y,\eta)$ to $\bar{U}_{\varrho}\times\bar{V}_{\varrho,I}\times X_{0}$
via $\xi$. The Nevo-Zimmer operation provides a map 
\begin{align*}
\mathcal{E}_{\varrho,s}:C(Y) & \to\tilde{L}^{\infty}\left(Y\right),\\
f & \mapsto\mathcal{E}_{\varrho,s}f,\mbox{ where }\left(\mathcal{E}_{\varrho,s}f\right)(\bar{u},\bar{v},\cdot)=\mathbb{E}_{\lambda}\left[\tilde{f}\left(\bar{u},\cdot\right)|\mathcal{F}^{s}\right],
\end{align*}
where $\mathcal{F}^{s}$ is the $s$-invariant sub-$\sigma$-algebra
of $\mathcal{B}(X_{0})$, $\tilde{f}\left(\bar{u},\cdot\right):X_{0}\to\mathbb{R}$
is defined as $\tilde{f}\left(\bar{u},x_{0}\right)=f\left(\bar{u}.x_{0}\right)$.
For more explanation of this operation we refer to \cite[Section 7]{NZ2}.

As discussed in Appendix \ref{subsec:three-cases}, there are three
possible situations for $\mathcal{E}_{\varrho,s}$. 
\begin{description}
\item [{(I)}] The subgroup $\bar{U}_{\varrho}$ preserves the measure $\lambda$.
This is exactly what we are aiming to prove. 
\item [{(II1)}] there exists $f\in C(Y)$ and $\bar{u}\in\bar{U}_{\varrho}$
such that $\int fd\lambda\neq\int fd\bar{u}.\lambda$; and for a.e.
$\bar{u}'\in\bar{U}_{\varrho}$, the function $x_{0}\mapsto\mathcal{E}_{\varrho,s}f\left(\bar{u}',x_{0}\right)$
is $\lambda$-constant. In this case by \cite[Proposition 9.2]{NZ2},
$(Y,\eta)$ has a nontrivial homogeneous factor of the form $\left(Q/Q_{1},\bar{\eta}\right)$,
which can be taken as the Mackey realization of $\tilde{\mathcal{L}}(Y)\cap\tilde{\mathcal{L}}(Q/P)$. 
\item [{(II2)}] The negation of $({\bf I})\vee({\bf II}1)$, see the next
subsection. 
\end{description}

\subsection{The Gauss map argument and Case (II2)\label{subsec:Gauss-map}}

In this subsection we assume that for some $\emptyset\neq\varrho\subseteq I_{i}$
and $s\in D_{\varrho}$, the operation $\mathcal{E}_{\varrho,s}$
is in Case (II2). Following \cite[Section 4]{NZ2}, take $X_{0}'$
to be the Mackey realization of the $N_{\varrho}$-invariant sub-$\sigma$-field
of $\mathcal{B}(X_{0})$, equipped with the measure $\lambda'$ from
the restriction of $\lambda$ to $\mathcal{B}_{N_{\varrho}}\left(X_{0}\right)$.
We have now a $Q$-system $\left(Y',\eta'\right)$ that is the largest
common $Q$-factor of $Y$ and $Q\times_{P}X_{0}'$,

\[
\xymatrix{Q\times_{P}X_{0}\ar[r]^{\xi_{0}}\ar[d] & Y\ar[d]\\
Q\times_{P}X_{0}'\ar[r] & Y'.
}
\]
By constructions, we have the following properties.

\begin{lemma}\label{nontrivial}

Assume that $\mathcal{E}_{\varrho,s}$ is in Case (II2). The $Q$-system
$\left(Y',\eta'\right)$ as above satisfies: 
\begin{description}
\item [{(a)}] The unipotent radical $N_{I}$ of $Q$ acts trivially on
$(Y',\eta')$. 
\item [{(b)}] The measure $\lambda'$, viewed as a measure on $Y'$, is
not preserved by $\bar{U}_{\varrho}$. 
\end{description}
\end{lemma} 
\begin{proof}
To see (a), note that since $\varrho\subset I$, we have that $N_{I}<N_{\varrho}$.
By construction $N_{\varrho}$ acts trivially on $X_{0}'$, thus its
subgroup $N_{I}$ acts trivially on $X_{0}'$. Since $N_{I}$ is contained
in $P$ and it is a normal subgroup of $Q$, we have that for $v\in N_{I}$,
$g\in Q$, $vgP=g(g^{-1}vg)P=gP$, that is, $N_{I}$ acts trivially
on $Q/P$. Then in the cocycle $\beta:Q\times Q/P\to P$, for $v\in N_{I}$,
$\beta(v,gP)=\theta(vgP)^{-1}v\theta(gP)=\theta(gP)^{-1}v\theta(gP)\in N_{I}$.
That is, $\beta(N_{I}\times Q/P)\subseteq N_{I}$. We conclude that
$N_{I}$ acts trivially on $Q/P\times_{\beta}X_{0}'$, and as a consequence
it also acts trivially on the factor $(Y',\eta')$.

Part (b) is Lemma \ref{II2-nonpreserv}. 
\end{proof}
We now recall some facts about Gauss maps from \cite{NZ2,Stuck}.
Let $G$ be a real algebraic group. Denote by $\mathfrak{g}$ or ${\rm Lie}(G)$
the Lie algebra of $G$, and $G^{0}$ the identity component of $G$.
If $V$ is a $k$-dimensional subspace of $\mathfrak{g}$, denote
by $[V]$ the corresponding element of ${\rm Gr}_{k}(\mathfrak{g})$,
where ${\rm Gr}_{k}(\mathfrak{g})$ is the Grassmannian of $k$-planes
in $\mathfrak{g}$. Let ${\rm Gr}(\mathfrak{g})=\cup_{k=1}^{{\rm dim}\mathfrak{g}}{\rm Gr}_{k}(\mathfrak{g})$.
The group $G$ acts on ${\rm Gr}(\mathfrak{g})$ by adjoint action,
and we write $g.[V]=\left[{\rm Ad}(g).V\right]$.

Following the reasoning in \cite[Section 3]{NZ2}, for the action
$Q\curvearrowright(Y',\eta')$, consider the stabilizer map 
\begin{align*}
\psi:Y' & \mapsto{\rm Sub}(Q),\\
y' & \mapsto{\rm Stab}_{Q}(y'),
\end{align*}
and the associated Gauss map 
\begin{align*}
d\psi:Y' & \mapsto{\rm Gr}(\mathfrak{q}),\\
y' & \mapsto\left[{\rm Lie}\left({\rm Stab}_{Q}\left(y'\right)\right)\right].
\end{align*}
The map $d\psi$ is a $Q$-equivariant Borel map.

Since the parabolic subgroup $Q$ is a connected real algebraic group,
its action on ${\rm Gr}\left({\rm Lie}(Q)\right)$ is algebraic. Therefore
every orbit is locally closed (\cite[Proposition 2.1.4]{MargulisBook})
and the measure $\left(d\psi\right)_{\ast}\eta'$ is supported on
a single orbit (\cite[Proposition 2.1.10]{ZimmerBook}). Let $y_{0}\in Y'$
be a point such that $\left(d\psi\right)_{\ast}\eta'$ is supported
on the orbit of $\left[{\rm Lie}\left({\rm Stab}_{Q}\left(y_{0}\right)\right)\right]$.
Since $Y'$ is a $Q$-factor of $Q\times_{P}X_{0}'$, every $Q$-orbit
in $Y'$ meets $X_{0}'$. Replacing
$y_{0}$ by another point on $Q.y_{0}\cap X_{0}'$ if necessary, we
may assume that $y_{0}\in X_{0}'$. By Lemma \ref{nontrivial} (b), $\bar{U}_{\varrho}$ is not
contained in ${\rm Stab}_{Q}(y_{0})$. On the other hand, ${\rm Stab}_{Q}(y_{0})$
contains $N_{\varrho}$, which is positive dimensional.

Recall that we are given a higher-rank factor $M_{i}$ of $L_{I}$
as in the statement of Theorem \ref{invariant-lambda}.

\begin{lemma}\label{nonideal}

Assume that $\mathcal{E}_{\varrho,s}$ is in Case (II2). Write $H={\rm Stab}_{Q}(y_{0})$,
where $y_{0}\in X_{0}'$ is as above. Then we have $Q$-factor maps
\[
\left(Y',\eta'\right)\to Q/N_{Q}(H)\to M_{i}/N_{M_{i}}\left(H\cap M_{i}\right),
\]
where $Q$ acts on $M_{i}/N_{M_{i}}(H\cap M_{i})$ through the factor
$M_{i}$. The normalizer $N_{M_{i}}\left(H\cap M_{i}\right)$ is a
proper subgroup of $M_{i}$.

\end{lemma} 
\begin{proof}
Let $L=L_{I}$ be the Levi subgroup of $Q$. By Lemma \ref{nontrivial}
(a), we have that ${\rm Stab}_{Q}(y)={\rm Stab}_{L}(y)\rtimes N_{I}$.
It follows that $N_{Q}(H)=N_{L}(H\cap L)\rtimes N_{I}$. Therefore
$Q/N_{Q}(H)$ and $L/N_{L}(H\cap L)$ are isomorphic.

Since $M_{i}$ is a normal subgroup of $L$, we have that $N_{L}(H\cap M_{i})>N_{L}(H\cap L)$.
The Levi subgroup $L$ is reductive, it can be written as a product
of $M_{i}$ with other almost simple factors and the split component
$A_{I}$, all the latter components commute with $M_{i}$, thus are
contained in $N_{L}(H\cap M_{i})$. It follows that $M_{i}$ acts
transitively on $L/N_{L}(H\cap M_{i})$, which is isomorphic to $M_{i}/N_{M_{i}}(H\cap M_{i})$
by the second isomorphism theorem.

Recall that the closed subgroup $H\cap M_{i}$ contains $N_{\varrho}\cap M_{i}$,
which is of positive dimension since $\varrho\subsetneq I_{i}$. Recall
also that $\bar{U}_{\varrho}\le M_{i}$. Then $H\cap M_{i}$ does
not contain $\bar{U}_{\varrho}$ by Lemma \ref{nontrivial} (a). Thus
the Lie algebra of $H\cap M_{i}$ cannot be $\{0\}$ or $\mathfrak{m}_{i}$.
Since $M_{i}$ is almost simple, it follows that $H\cap M_{i}$ is
not normal in $M_{i}$. 
\end{proof}
Lemma \ref{nonideal} allows us to apply \cite[Proposition 3.2]{NZ2}
to deduce that in Case (II2), $(Y',\eta')$ admits a nontrivial projective
factor.

\begin{proposition}\label{II2-conclusion}

In Case (II2), $(Y',\eta')$ admits a nontrivial projective factor
$\left(Q/Q_{1},\bar{\eta}\right)$, where $Q_{1}$ is a parabolic
subgroup of $G$ and $Q_{1}$ is a proper subgroup of $Q$.

\end{proposition} 
\begin{proof}
Apply \cite[Proposition 3.2]{NZ2} to the simple Lie group $M_{i}$
and its algebraic subgroup $F=N_{M_{i}}\left(H\cap M_{i}\right)$,
where $F$ is a proper subgroup by Lemma \ref{nonideal}. We deduce
that $F$ is contained in a proper parabolic subgroup of $M_{i}$.
Conjugate $F$ by an element of $M_{i}$ if necessary, there is a
non-empty subset $J\subseteq I_{i}$ such that $F$ is contained in
the parabolic group of $M_{i}$ parametrized by $I_{i}-J$. Let $Q_{1}=P_{I-J}$
be the corresponding parabolic subgroup of $G$. Then $Q_{1}\cap M_{i}=P_{I_{i}-J}$,
$Q/Q_{1}$ is isomorphic to $M_{i}/P_{I_{i}-J}$, thus a factor of
$M_{i}/F$. The statement then follows from Lemma~\ref{nonideal}. 
\end{proof}

\subsection{Concluding that $M_{i}$ preserves the measure $\lambda$}
\begin{proof}[End of the proof of Theorem \ref{invariant-lambda}]

Suppose $\left(G/Q,\nu_{Q}\right)$ is a maximal projective factor
of $\left(X,\nu\right)$, where $Q=P_{I}$ is a proper parabolic subgroup
of $G$. Let $\lambda=\psi_{\mu}(\nu)$. As explained in Subsection
\ref{subsec:Preliminary-step}, we may replace $\mu$ by $\mu'=m_{K}\ast\mu$
and $\nu$ by $m_{K}\ast\lambda$: since $\nu$ and $m_{K}\ast\lambda$
are in the same measure class, we have that $(G/Q,\bar{m}_{K})$ is
the maximal projective factor of $\left(X,m_{K}\ast\lambda\right)$.
Then $\left(X,m_{K}\ast\lambda\right)$ is induced from the $Q$-system
$(Y,\eta)$ specified in Proposition \ref{Ysystem}.

Suppose $M_{i}$ is a simple factor of the Levi subgroup $L_{I}$
with $\mathbb{R}$-rank $\ge2$. \textcolor{black}{Denote by $I_{i}$
the set of simple roots corresponding to $M_{i}$. Take} \textcolor{black}{$\varrho\subset I_{i}$
to be a nonempty proper subset of in $I_{i}$} and take $s\in D_{\varrho}$,
where $D_{\varrho}$ is specified in (\ref{eq:D_I}). Applying the
Nevo-Zimmer operation $\mathcal{E}_{\varrho,s}$ with respect to $\left(s,\bar{U}_{\varrho},\bar{V}_{\varrho,I},N_{\varrho}\right)$
as explained in Subsection \ref{subsec:The-Nevo-Zimmer-operation}.
Then in three cases that can arise, we have 
\begin{itemize}
\item in Case (I), $\bar{U}_{\varrho}$ preserves the $P$-invariant measure
$\lambda$; 
\item in Case (II1), $Q\curvearrowright(Y,\eta)$ admits a nontrivial projective
factor by \cite[Proposition 9.2]{NZ2}; 
\item in Case (II2), $Q\curvearrowright(Y,\eta)$ admits a nontrivial projective
factor by Proposition \ref{II2-conclusion}. 
\end{itemize}
Since $G/Q$ is assumed to be a maximal projective factor of $\left(X,m_{K}\ast\lambda\right)$,
the situations in (II1) or (II2) cannot occur for $\left(\varrho,s\right)$.
Indeed, otherwise $\left(X,m_{K}\ast\lambda\right)$ would admit a
projective factor $\left(G/Q_{1},\bar{m}_{Q_{1}}\right)$ with $Q_{1}$
a proper subgroup of $Q$, contradicting the maximality assumption
on $G/Q$, see the proof of \cite[Theorem 11.4]{NZ2}. It follows
that for any $\varrho$ such that $\emptyset\neq\varrho\subsetneqq I_{i}$,
$\bar{U}_{\varrho}$ preserves $\lambda$. This implies $\lambda$
is invariant under $M_{i}$, see the last paragraph of the proof of
\cite[Theorem 1]{NZ2}. Indeed, since the $\mathbb{R}$-rank of $M_{i}$
is at least $2$, such subgroups $\bar{U}_{\varrho}$ generate the
opposite unipotent subgroup $\bar{N}_{i}$ of $M_{i}$. It follows
that $\bar{N}_{i}$ preserves the measure $\lambda$. Since $M_{i}$
is generated by $\bar{N}_{i}$ and $M_{i}\cap P$, and $P$ preserves
the measure $\lambda$, we conclude that $\lambda$ is invariant under
$M_{i}$. 
\end{proof}

\section{Consequences of the structure theorem\label{sec:consequence}}

\subsection{Proof of Theorem \ref{invariant2} and Corollary \ref{mp-ext}}

By the structure of parabolic subgroups, we derive Theorem \ref{invariant2}
and Corollary \ref{mp-ext} stated in the Introduction from Theorem
\ref{invariant-lambda}. 
\begin{proof}[Proof of Theorem \ref{invariant2}]

Write $I'=\mathsf{f}(I)$ and $Q'=P_{\mathsf{f}(I)}$. Let $Q'=M_{I'}A_{I'}N_{I'}$
be the Langlands decomposition. The disconnectedness of $M_{I'}$
is controlled by that of $P$. More precisely, there is a finite subgroup
$F<P$ such that $M_{I'}=M_{I'}^{0}F$, see \cite[Proposition 7.82]{KnappBook}.
Denote by $E$ the maximal normal compact subgroup of $M_{I'}^{0}$,
and $M_{1},\ldots,M_{\ell'}$ the noncompact simple factors. Then
$M_{I}^{0}$ is an almost direct product of $E,M_{1},\ldots,M_{\ell'}$.
By the definition of $\mathsf{f}(I)$, each $M_{k}$ where $k\in\{1,\ldots,\ell\}$,
has $\mathbb{R}$-rank $\ge2$, and $M_{k}$ is a factor of $M_{I}<Q$.
Theorem \ref{invariant-lambda} then implies that $M_{k}$ preserves
the measure $\lambda$.

Next note that from the decomposition of Lie algebras, see \cite[(7.77) and Proposition 7.78]{KnappBook},
the compact factor $E$ of $M_{I'}^{0}$ is contained in $M=Z_{K}(\mathfrak{s})<P$.
In the decomposition $Q'=M_{I'}^{0}FA_{I'}N_{I'}$, we have $EFA_{I'}N_{I'}<P$
thus preserves the measure $\lambda$; and the noncompact factors
$M_{1},\ldots,M_{\ell'}$ preserves $\lambda$ by Theorem \ref{invariant-lambda}.
We conclude that $Q'$ preserves the measure $\lambda$. 
\end{proof}
In the case that $L_{I}$ has no rank-$1$ non-compact factors, we
have $\mathsf{f}(I)=I$, thus the statement of Corollary \ref{mp-ext}
follows: 
\begin{proof}[Proof of Corollary \ref{mp-ext}]

Under the assumptions, $I'=\mathsf{f}(I)=I$, thus by Theorem \ref{invariant2},
$Q$ preserves the $P$-invariant measure. Recall Proposition~\ref{Ysystem} that $(X,\nu)$
fits into the $G$-factors sequence 
\[
\left(G/Q\times_{\beta}Y,\nu_{Q}\times\eta\right)\to(X,\nu)\to\left(G/Q,\nu_{Q}\right),
\]
and $\eta=\bar{m}_{K\cap Q}\ast\lambda$. Since $\lambda$ is invariant
under $Q$, it follows that $\eta=\lambda$. Thus $\left(G/Q\times_{\beta}Y,\nu_{Q}\times\eta\right)=\left(G/Q\times_{\beta}Y,\nu_{Q}\times\lambda\right)$
is a measure-preserving extension of $\left(G/Q,\nu_{Q}\right)$. 
\end{proof}

\subsection{Proof of Theorem \ref{constraint}: constraints on Furstenberg entropy
spectrum \label{subsec:constraints}}

Theorem \ref{invariant2} implies the following bound on Furstenberg entropy.

\begin{corollary}\label{constraint1}

Suppose $\left(G/P_{I},\bar{\nu}_{0}\right)$ is the maximal standard
projective factor of $(X,\nu)$. Then 
\[
h\left(G/P_{I},\bar{\nu}_{0}\right)\le h(X,\nu)\le h\left(G/P_{\mathsf{f}(I)},\bar{\nu}_{0}\right).
\]

\end{corollary} 
\begin{proof}
Since it is assumed that $\left(G/P_{I},\bar{\nu}_{0}\right)$ is
a $G$-factor of $\left(X,\nu\right)$, the first inequality $h\left(G/P_{I},\bar{\nu}_{0}\right)\le h(X,\nu)$
follows. For the second inequality, by the Fursterberg isomorphism,
we have that $\nu=\nu_{0}\ast\lambda$, where $\lambda$ is $P$-invariant.
By Theorem \ref{invariant2}, the parabolic subgroup $P_{\mathsf{f}(I)}$
preserves the measure $\lambda$. Then we may express $\nu=\nu_{0}\ast\lambda$
as 
\[
\nu=\int_{G/P}g.\lambda d\nu_{0}(g)=\int_{G/P_{\mathsf{f}(I)}}g.\lambda d\bar{\nu}_{0}(g).
\]
The proof of \cite[Proposition 2.5]{NZ1} applied to $P_{\mathsf{f}(I)}$
instead of $P$ shows that $\left(X,\nu\right)$ can be viewed as
a system induced from the measure-preserving $P_{\mathsf{f}(I)}$-system
$\left(X_{0},\lambda\right)$, with a $G$-factor map 
\[
\left(G\times_{P_{\mathsf{f}(I)}}X_{0},\bar{\nu}_{0}\times\lambda\right)\to(X,\nu).
\]
Since $\left(G\times_{P_{\mathsf{f}(I)}}X_{0},\bar{\nu}_{0}\times\lambda\right)$
is a measure-preserving extension of $\left(G/P_{\mathsf{f}(I)},\bar{\nu}_{0}\right)$,
we have that 
\[
h\left(G/P_{\mathsf{f}(I)},\bar{\nu}_{0}\right)=h\left(G\times_{P_{\mathsf{f}(I)}}X_{0},\bar{\nu}_{0}\times\lambda\right)\ge h(X,\nu).
\]
\end{proof}

Theorem \ref{constraint} follows directly from Corollary \ref{constraint1}.

\subsection{Lattices equipped with Furstenberg measures \label{subsec:Lattices}}

In this subsection, we apply the induction procedure for stationary
systems in \cite{BH21,BBHP} to deduce statements on a lattice $\Gamma<G$,
equipped with a Furstenberg measure $\mu_{0}$ as described in Definition~\ref{furstenbergmeas}.
Indeed in what follows, it is sufficient to assume that $\left(G/P,\nu_{P}\right)$
is the Poisson boundary of $\left(\Gamma,\mu_{0}\right)$, and $\nu_P$ is in the $G$-quasi-invariant measure class.

Consider an ergodic $(\Gamma,\mu_{0})$-space $\left(X,\nu\right)$.
 Since $\left(G/P,\nu_{P}\right)$
is the Poisson boundary of $\left(\Gamma,\mu_{0}\right)$, we have
the $\Gamma$-equivariant boundary map $\boldsymbol{\beta}_{\nu}:G/P\to\mathcal{P}(X)$
associated with the stationary measure $\nu$ on $X$. Define for
$g\in G$, the measure $\phi_{g}\in\mathcal{P}(X)$ as the barycenter
of $\left(\boldsymbol{\beta}_{\nu}\right){}_{\ast}(g\nu_{P})$. Note
that when restricted to $\Gamma$, $\gamma\mapsto\phi_{\gamma}$ is
equivariant as $\phi_{\gamma}=\gamma.\nu$.

Fix a measurable section $\tau:G/\Gamma\to G$ and denote by $c:G\times G/\Gamma\to\Gamma$
the corresponding cocycle. Denote by $\tilde{X}=G/\Gamma\times_{c}X$
the induced space. For a function $f\in L^{\infty}\left(G/\Gamma\times_{c}X,m_{G/\Gamma}\times\nu\right)$,
write $f_{z}(\cdot)=f(z,\cdot)$ for $z\in G/\Gamma$, and regard
$f_{z}\in L^{\infty}(X,\nu)$. As in \cite[Theorem 4.3]{BH21}, define
a measure $\tilde{\nu}$ on $\tilde{X}$ as follows: 
\begin{align}
\int_{\tilde{X}}fd\tilde{\nu} & =\int_{G/\Gamma}\int_{X}f_{z}d\phi_{\tau(z)}dm_{G/\Gamma}(z)\nonumber \\
 & =\int_{G/\Gamma}\int_{G/P}\int_{X}f_{z}d\boldsymbol{\beta}_{\nu}(w)d\nu_{P}(\tau(z).w)dm_{G/\Gamma}(z).\label{eq:def-induced}
\end{align}
The formula (\ref{eq:def-induced}) implies that $\tilde{\nu}$ admits
a decomposition as in the Furstenberg isomorphism that $\tilde{\nu}=\nu_{P}\ast\tilde{\lambda}$,
where $\tilde{\lambda}$ is $P$-invariant. Indeed, by \cite[Lemma 4.6 and Remark 4.7]{BBHP},
the map $\tilde{\boldsymbol{\beta}}_{\tilde{\nu}}:G/P\to\mathcal{P}(\tilde{X})$,
which is related to the $\Gamma$-boundary map $\boldsymbol{\beta}_{\nu}:G/P\to\mathcal{P}(X)$
by 
\begin{equation}
\int_{\tilde{X}}fd\tilde{\boldsymbol{\beta}}_{\tilde{\nu}}(w)=\int_{G/\Gamma}\int_{X}f_{z}d\boldsymbol{\beta}_{\nu}\left(\tau(z)^{-1}.w\right)dm_{G/\Gamma}(z),\mbox{ where }w\in G/P,f\in L^{\infty}\left(\tilde{X}\right),\label{eq:Theta}
\end{equation}
is $G$-equivariant; and the measure $\tilde{\nu}$ defined in (\ref{eq:def-induced})
is the barycenter of $\left(\tilde{\boldsymbol{\beta}}_{\tilde{\nu}}\right)_{\ast}\left(\nu_{P}\right)$.
Write $\tilde{\lambda}=\tilde{\boldsymbol{\beta}}_{\tilde{\nu}}(P)$,
then $\tilde{\lambda}$ is $P$-invariant and satisfies $\tilde{\nu}=\nu_{P}\ast\tilde{\lambda}$.

An important property of the induction procedure, shown in \cite[Proof of Theorem B]{BH21} is the following.
\begin{fact}\label{fact-max} $\left(G/Q,\nu_{Q}\right)$
is a $\Gamma$-factor of $\left(X,\nu\right)$ if and only if it is
a $G$-factor of $\left(\tilde{X},\tilde{\nu}\right)$. 
\end{fact}

Here we consider $G$-factors up to measure classes only. As $\nu_P$ and $\bar{m}_K$ are in the same measure class, the measure $\tilde{\nu}_1=\bar{m}_K\ast \tilde{\lambda}$ is in the same measure class as $\tilde{\nu}$, thus $(\tilde{X},\tilde{\nu}_1)$ is a $\mu$-stationary $G$-space for any admissible $K$-invariant probability measure $\mu$ on $G$.

As $(\tilde{X},\tilde{\nu}_1)$ admits a maximal projective factor for $(G,\mu)$, by \cite[Lemma 0.1]{NZ2}, the fact implies that $(X,\nu)$ admits the same maximal projective factor for $(\Gamma,\mu_0)$.


\begin{proof}[Proof of Fact \ref{fact-max}]
If $(G/Q,\nu_Q)$ is a $\Gamma$-factor of $(X,\nu)$, we obtain by induction that $(\tilde{X},\tilde{\nu})$ admits as a $G$-factor the space $G/\Gamma \times_c G/Q$, which is a measure-preserving extension of $(G/Q,\nu_Q)$, where by assumption on $\mu_0$, $\nu_Q$ is in the same measure class as the $K$-invariant probability measure on $G/Q$.

Conversely, if $(G/Q,\nu_Q)$ is a $G$-factor of $\tilde{X}=G/\Gamma \times_c X$, there is a family of maps $(p_z:X\to G/Q)$ for $z \in G/\Gamma$ such that for every $g \in G$ and for $\tilde{\nu}$ almost every $(z,x)$ in $G/\Gamma \times_c X$, one has $p_{gz}(c(g,z)x)=gp_z(x)$. By standard technics (see \cite[Proof of Theorem B]{BH21} for details), one can assume that it actually holds for all $z \in G/\Gamma$ and all $x\in X$. Taking $z=\Gamma$ and restricting to $g=\gamma \in \Gamma$, one gets a $\Gamma$-factor map $p_{\Gamma}:X \to G/Q$, with $\nu_Q$ the unique $\mu_0$-stationary measure.
\end{proof}

\begin{proof}[Proof of Theorem \ref{lattice}]

Suppose $\left(G/Q,\nu_{Q}\right)$ is the maximal standard projective
$\Gamma$-factor of $(X,\nu)$, where $Q=P_{I}$ is a parabolic subgroup.
By Fact~\ref{fact-max}, $\left(G/Q,\nu_{Q}\right)$ is also the maximal standard
projective $G$-factor of the induced $G$-system $\left(\tilde{X},\tilde{\nu}\right)$.
Apply Theorem \ref{invariant2} to $\left(\tilde{X},\tilde{\nu}\right)$,
where $\tilde{\nu}=\nu_{P}\ast\tilde{\lambda}$ as above, we have
that $Q'=P_{\mathsf{f}(I)}$ preserves the measure $\tilde{\lambda}$.

To lighten notations, in what follows write $\boldsymbol{\beta}=\boldsymbol{\beta}_{\nu}$
and $\tilde{\boldsymbol{\beta}}=\tilde{\boldsymbol{\beta}}_{\nu}$.
We claim that the $P$-invariant measure $\tilde{\lambda}$ is $Q'$-invariant
if and only if the $\Gamma$-boundary map $\boldsymbol{\beta}:G/P\to\mathcal{P}(X)$
satisfies that $\boldsymbol{\beta}\left(hqP\right)=\boldsymbol{\beta}\left(hP\right)$
for all $q\in Q'$ and $m$-a.e. $h\in G$. Indeed, for $g\in G$,
by (\ref{eq:Theta}) and $G$-invariance of $m_{G/\Gamma}$ we have
that 
\[
\int_{\tilde{X}}g^{-1}.fd\tilde{\lambda}=\int_{G/\Gamma}\int_{X}f_{z}d\boldsymbol{\beta}(\tau(z)^{-1}gP)dm_{G/\Gamma}(z).
\]
It follows that $g.\tilde{\lambda}=\tilde{\lambda}$ if and only if
for $m_{G/\Gamma}$-a.e. $z$, $\boldsymbol{\beta}(\tau(z)^{-1}gP)=\boldsymbol{\beta}(\tau(z)^{-1}P)$.
By $\Gamma$-equivariance of $\boldsymbol{\beta}$, this is equivalent
to that $\boldsymbol{\beta}(\gamma^{-1}\tau(z)^{-1}gP)=\boldsymbol{\beta}(\gamma^{-1}\tau(z)^{-1}P)$
for all $\gamma\in\Gamma$. Since $\tau:G/\Gamma\to G$ is a section,
$G=\tau(G/\Gamma)\Gamma$, the claim is verified. We conclude that
the property $Q'=P_{\mathsf{f}(I)}$ preserves the measure $\tilde{\lambda}$
implies the $\Gamma$-boundary map $\theta:G/P\to\mathcal{P}(X)$
factors through $\bar{\theta}:G/Q'\to\mathcal{P}(X)$, where $\bar{\theta}(gQ')=\theta(gP)$
is well-defined almost everywhere. 
\end{proof}
Next we derive constraints on the Furstenberg entropy spectrum of
$(\Gamma,\mu)$ in the same manner as in Subsection \ref{subsec:constraints}.

\begin{theorem}\label{constraint-lattice}

Let $\mu_{0}$ be a Furstenberg measure of finite Shannon entropy
on a lattice $\Gamma<G$, where $G$ is a connected semisimple real
Lie group with finite center. Denote by $\Delta$ simple restricted
roots of $G$. The Furstenberg entropy spectrum of $\left(\Gamma,\mu_{0}\right)$
satisfies 
\[
{\rm EntSp}\left(\Gamma,\mu_{0}\right)\subseteq\bigcup_{I\subseteq\Delta}\left[h_{\mu_{0}}\left(G/P_{I},\nu_{I}\right),h_{\mu_{0}}\left(G/P_{\mathsf{f}(I)},\nu_{\mathsf{f}(I)}\right)\right].
\]

\end{theorem} 
\begin{proof}
Suppose $\left(X,\nu\right)$ is an ergodic $\left(\Gamma,\mu_{0}\right)$-space
and $\left(G/Q,\nu_{Q}\right)$ is its maximal projective $\Gamma$-factor,
$Q=P_{I}$. Then $h_{\mu_{0}}(X,\nu)\ge h_{\mu_{0}}\left(G/Q,\nu_{Q}\right)$.
Let $Q'=P_{\mathsf{f}(I)}$. Then by Theorem \ref{lattice}, the $\Gamma$-boundary
map $\theta:G/P\to\mathcal{P}(X)$ factors through the projection
$G/P\to G/Q'$, that is, $\theta(gP)=\theta(gqP)$ for $q\in Q'$.
It follows that $\nu$ is the barycenter of $\bar{\theta}_{\ast}\left(\nu_{Q'}\right)$.
Apply Lemma \ref{barycenter} to the $\left(\Gamma,\mu_{0}\right)$-spaces
$\left(G/Q',\nu_{Q'}\right)$ and $\left(X,\nu\right)$, we conclude
that $h_{\mu_{0}}(X,\nu)\le h_{\mu_{0}}\left(G/Q',\nu_{Q'}\right)$.
We have shown that in this case 
\[
h_{\mu_{0}}\left(G/Q,\nu_{Q}\right)\le h_{\mu_{0}}(X,\nu)\le h_{\mu_{0}}\left(G/Q',\nu_{Q'}\right).
\]
The statement follows. 
\end{proof}
It is worth emphasizing that in the statements above, it is crucial
that the $\mu_{0}$-harmonic measure on $G/P$ is in the quasi-invariant
measure class. For a general step distribution $\mu$ on $\Gamma$,
the $\mu$-harmonic measure may be singular with respect to $\bar{m}_{K}$.
In such a case one can not derive constraints on ${\rm EntSp}\left(\Gamma,\mu\right)$
via the inducing procedure as above.

\vspace{0.5cm}

\begin{center}
\textsc{ Part II : Poisson bundles over stationary systems}
\end{center}

\section{The Poisson bundle over a stationary system\label{sec:Poissonbundle}}

In this section we introduce the $\mu$-Poisson bundle over a stationary
system $(X,\eta)$ and study its basic properties. The construction takes input a 
stationary system $(X,\eta)$ and a $G$-equivariant map $L:X\to \rm{Sub}(G)$, and the output
is a bundle over $(X,\eta)$, where the fiber over $x$ can be viewed as 
the Poisson boundary of certain process on the coset space $L_x\backslash G$. We emphasize that the coset processes 
depend on the system $(X,\eta)$, in general they are not the same as the projection of the $\mu$-random walk on $G$ to 
$L_x\backslash G$.

Throughout, we assume
that $\mu$ is nondegenerate in the sense that ${\rm supp}\mu$ generates
$G$ as a semigroup.

\subsection{Stationary joining\label{subsec:Stationary-joining}}

For a more detailed reference on stationary joining, see \cite{FurstenbergGlasner}.
Suppose we are in the setting of Subsection \ref{subsec:boundary-map}.
Denote by $\left(G^{\mathbb{N}},\mathbb{P}_{\mu}\right)$ the random
walk trajectory space and $(B,\nu_{B})$ the Poisson boundary of $\left(G,\mu\right)$.
For a $(G,\mu)$-stationary system $(X,\eta)$, denote by $\omega\mapsto\eta_{\omega}$
the almost sure limit of $\omega_{n}.\eta$ provided by the martingale
convergence theorem.

Let $(X,\eta)$ and $(Y,\lambda)$ be two $(G,\mu)$-stationary systems.
Let the group $G$ act on the product space $X\times Y$ diagonally.
The \textbf{stationary joining} of the two, denoted by $\left(X\times Y,\eta\varcurlyvee\lambda\right)$
is the system with measure 
\[
\eta\varcurlyvee\lambda=\int_{G^{\mathbb{N}}}\eta_{\omega}\times\lambda_{\omega}d\mathbb{P}_{\mu}(\omega)=\int_{B}\boldsymbol{\beta}_{\eta}(b)\times\boldsymbol{\beta}_{\lambda}(b)d\nu_{B}(b).
\]
By the barycenter property, when $\eta$ is invariant the stationary joining is simply the product measure $\eta \varcurlyvee \lambda=\eta \times \lambda$.

In our notation $\mathbb{P}_{\mu}$ denotes the law of $\mu$-random
walk trajectories from the identity $e$. Then $g.\mathbb{P}_{\mu}$
is the law of the trajectories starting from $g$.  Therefore the stationary joining on $X\times G^{\mathbb{N}}$
is given by 
\[
\eta\varcurlyvee\mathbb{P}_{\mu}=\int_{G^{\mathbb{N}}}\eta_{\omega}\times\delta_{\omega}d\mathbb{P}_{\mu}(\omega).
\]
When $\eta$ is $G$-invariant, $\eta\varcurlyvee\mathbb{P}_{\mu}=\eta\times\mathbb{P}_{\mu}$.

On the space $X\times G^{\mathbb{N}}$ we have a skew transformation
\[
T:\left(x,\left(\omega_{1},\omega_{2},\ldots\right)\right)\mapsto\left(\omega_{1}^{-1}.x,\left(\omega_{1}^{-1}\omega_{2},\omega_{1}^{-1}\omega_{3},\ldots\right)\right).
\]
The arguments of \cite[Theorem 3.1]{Furman} (see also \cite[Theorem I.2.1]{Kifer}) immediately imply that: 
\begin{fact}
\label{pmp} The transformation $T$ preserves the measure $\eta\varcurlyvee\mathbb{P}_{\mu}$
on $X\times G^{\mathbb{N}}$. Assume $G\curvearrowright(X,\eta)$ is ergodic,
then $T$ is an ergodic transformation on $\left(X\times G^{\mathbb{N}},\eta\varcurlyvee\mathbb{P}_{\mu}\right)$, so in particular the $G$-action on the stationary joining $\left(X\times G^{\mathbb{N}},\eta\varcurlyvee\mathbb{P}_{\mu}\right)$ is ergodic too. 
\end{fact}

We recall the proof for the reader's convenience.
\begin{proof}
We have $\eta_{\omega_1^{-1}\omega}(\omega_1^{-1}x)=\lim \omega_1^{-1}\omega_n\eta(\omega_1^{-1}x)=\eta_\omega(x)$ for any $\omega_1$ in $G$. Thus for any $L^1$ function 
\[
\int_{X\times G^{\mathbb{N}}} f\circ Td\eta\varcurlyvee \mathbb{P}_\mu=\int_G\int_{X\times G^{\mathbb{N}}} f(\omega_1^{-1}x,\omega_1^{-1}\omega)d\eta_\omega(x)d \mathbb{P}_\mu(\omega_1^{-1}\omega)d\mu(\omega_1)=\int_{X\times G^{\mathbb{N}}} fd\eta\varcurlyvee \mathbb{P}_\mu.
\]
This shows the first point.
Consider a $T$-invariant function $h\in L^2(X\times G^{\mathbb{N}},\eta\varcurlyvee \mathbb{P}_\mu)$. Denote $\mathcal{B}_n$ the $\sigma$-algebra generated by $(x,\omega_1,\dots,\omega_n)$ and $h_n(x,\omega_1,\dots,\omega_n)$ the orthogonal projection of $h$ onto the subspace $L^2(\mathcal{B}_n)$. We have
\begin{align*}
h_n(x,\omega_1,\dots,\omega_n) &= \mathbb{E}(h(x,\omega)|\omega_1,\dots,\omega_n) = \mathbb{E}(h((\omega_1\dots\omega_n)^{-1}x,(\omega_1\dots \omega_n)^{-1}\omega)|\omega_1,\dots,\omega_n) \\
&= \mathbb{E}(h((\omega_1\dots\omega_n)^{-1}x,\omega'))=h_0((\omega_1\dots\omega_n)^{-1}x).
\end{align*}
It follows that $h_0(x)=\mathbb{E}(\mathbb{E}(h(x,\omega)|\omega_1)=\mathbb{E}h_1(x,\omega_1)=\mathbb{E}(h_0(\omega_1^{-1} x))=\int_G h_0(g^{-1}x)d\mu(g)$.
When the $G$ action on $(X,\eta)$ is ergodic, it implies that $h_0$, and a fortiori $h_n$ for any $n$ hence $h$, is constant.
\end{proof}

\subsection{Definition of the Poisson bundle\label{subsec:Definition}}

Denote by ${\rm Sub}(G)$ the space of closed subgroups of $G$, equipped
with the Chabauty topology. We assume that our stationary system $(X,\eta)$
comes together with a $G$-equivariant measurable map $L:X\to{\rm Sub}(G)$,
denoted by $x\mapsto L_{x}$. We assume throughout the article that 
\begin{align}\label{eq:contain}
L_x < \mathrm{Stab}_G(x). 
\end{align}
For example, $L_{x}={\rm Stab}_{G}(x)$.
The pushforward of $\eta$ under this map is a $\mu$-stationary measure
on ${\rm Sub}(G)$, often referred to as a \emph{stationary random
subgroup} (in short SRS) of $G$. Denote~by 
\[
W_{\Omega}:=\left\{ \left(x,\left(L_{x}\omega_{1},L_{x}\omega_{2},\dots\right)\right):x\in X,(\omega_{1},\omega_{2}\dots)\in G^{\mathbb{N}}\right\} =\bigsqcup_{x\in X}\{x\}\times(L_{x}\backslash G)^{\mathbb{N}}
\]
the space of trajectories in coset spaces. The group $G$ acts on
$W_{\Omega}$ by $g.\left(x,\left(L_{x}\omega_{n}\right)\right)=\left(g.x,\left(L_{g.x}g\omega_{n}\right)\right)$.
Consider the map 
\begin{align*}
\vartheta & :X\times G^{\mathbb{N}}\to W_{\Omega}\\
 & (x,\omega)\mapsto\left(x,\left(L_{x}\omega_{1},L_{x}\omega_{2},\ldots\right)\right),
\end{align*}
which is $G$-equivariant. Write 
\[
\overline{\mathbb{P}}_{\mu}:=\vartheta_{\ast}\left(\eta\varcurlyvee\mathbb{P}_{\mu}\right)
\]
for the pushforward of the measure $\eta\varcurlyvee\mathbb{P}_{\mu}$. We call $(W_{\Omega},\overline{\mathbb{P}}_{\mu})$ the \emph{bundle of trajectories} in coset spaces over $(X,\eta)$.

On the space $W_{\Omega}$ we have a time shift operator $\mathcal{S}$
defined as 
\[
\mathcal{S}\left(x,\left(L_{x}\omega_{1},L_{x}\omega_{2},\ldots\right)\right)=\left(x,\left(L_{x}\omega_{2},L_{x}\omega_{3},\ldots\right)\right),
\]
which commutes with the $G$ action on $W_{\Omega}$. Consider the
invariant $\sigma$-field $\mathcal{I}$ under $\mathcal{S}$, that
is, 
\[
\mathcal{I}:=\left\{ A\in\mathcal{B}(W_{\Omega}):\mathcal{S}^{-1}(A)=A\right\} .
\]

\begin{definition} Let $\left(Z,\lambda\right)$ be the Mackey point
realisation of (the completion of) the invariant $\sigma$-field $\mathcal{I}$
equipped with the measure $\overline{\mathbb{P}}_{\mu}$. We call
$\left(Z,\lambda\right)$ a $(G,\mu)$-Poisson bundle over the stationary
system $\left(X,\eta\right)$. \end{definition}

\begin{remark}\label{ergodic}

The Poisson bundle $\left(Z,\lambda\right)$ is $G$-ergodic if $(X,\eta)$
is $G$-ergodic. Indeed, $\left(Z,\lambda\right)$ is a $G$-factor
of the stationary joining of $\left(X,\eta\right)$ and the Poisson
boundary of $\left(G,\mu\right)$. The ergodicity of the latter follows
from Fact \ref{pmp}.

\end{remark}

Denote by $\theta:(W_{\Omega},\overline{\mathbb{P}}_{\mu})\to(Z,\lambda)$
the factor map which induces an isomorphism of measured $G$-spaces
between $\left(W_{\Omega},\mathcal{I},\overline{\mathbb{P}}_{\mu}|_{\mathcal{I}}\right)$
and $(Z,\mathcal{B}(Z),\lambda)$, where $\overline{\mathbb{P}}_{\mu}|_{\mathcal{I}}$
denotes the restriction of the probability measure $\overline{\mathbb{P}}_{\mu}$
to the invariant $\sigma$-field $\mathcal{I}$.

In each fiber $(L_{x}\backslash G)^{\mathbb{N}}$ we obtain an invariant
$\sigma$-field $\mathcal{I}_{x}$, which almost surely coincides
with the invariant $\sigma$-field of the shift operator restricted
to this fiber. Therefore the fiber over $x$ in $Z$ is (almost surely)
identified with the Poisson boundary $B_{L_{x}\backslash G}$,
i.e. the space of ergodic components of the shift operator. Up to
measure zero, $Z=\bigsqcup_{x\in X}\{x\}\times B_{L_{x}\backslash G}$.
Fiberwise, $\theta(x,\cdot):\left(L_{x}\backslash G\right)^{\mathbb{N}}\to B_{L_{x}\backslash G}$
can be viewed as a map which sends a trajectory on the coset space
to its image in the boundary $B_{L_{x}\backslash G}$. It plays a
role analogous to the map ${\rm bnd}:G^{\mathbb{N}}\to B$.

\begin{figure}[!t]
\centering
\xymatrix{
 & &(X\times G^\mathbb{N},\eta \varcurlyvee \mathbb{P}_\mu) \ar[dd]^{\tilde{\psi}} \ar[ld]_{\vartheta} \ar[rd]^{\mathrm{id}\times {\rm bnd}} \ar@/_3pc/[lldd]_{\xi_n}  & \\
& (W_\Omega,\overline{\mathbb{P}}_\mu) \ar[rd]^\theta \ar[ld]_{\theta_n} & & (X\times B, \eta \varcurlyvee \nu_B) \ar[ld]_\psi \\
 (W,\xi_{n\ast} (\eta \varcurlyvee \mathbb{P}_\mu)) \ar[drr]&& (Z,\lambda)\ar[d]^\varrho &\\
& & (X,\eta)\ar[d]^\pi & \\
& & (Y,\nu) &
}
\caption{Commutative diagram of the $G$-spaces involved in a standard system.}
\label{fig:diag}
\end{figure}
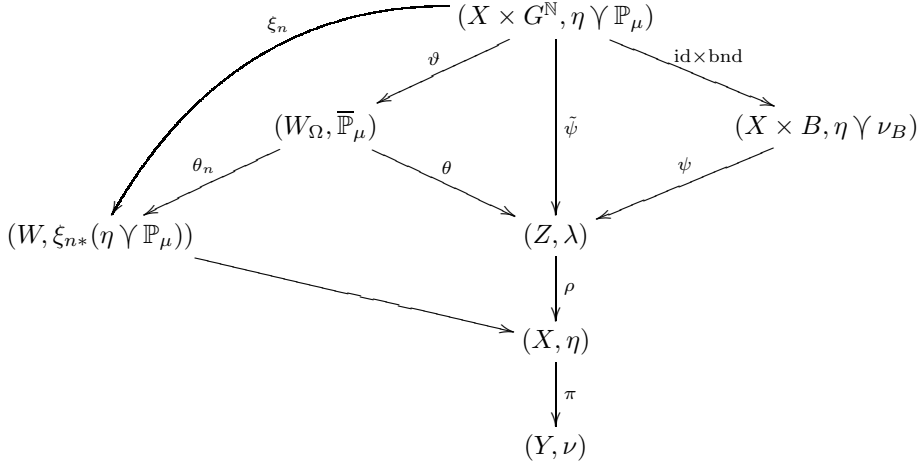

\begin{remark}

When $\eta$ is a $G$-invariant measure, we have $\eta\varcurlyvee\mathbb{P}_{\mu}=\eta\times\mathbb{P}_{\mu}$.
Thus over a measure preserving system $(X,\eta)$, the bundle $\left(Z,\lambda\right)$
is the same Poisson bundle as considered in \cite{Bowen}.

\end{remark}

Denote by $\overline{\mathbb{P}}_{\mu}=\int_{X}\overline{\mathbb{P}}_{\mu,x}d\eta(x)$
the disintegration of the measure $\overline{\mathbb{P}}_{\mu}$ over
the factor map $W_{\Omega}\to X$, that is, for $x\in X$, the distribution
of $\left(L_{x}\omega_{n}\right)_{n=1}^{\infty}$ is $\overline{\mathbb{P}}_{\mu,x}$.
When $\eta$ is not an invariant measure, in general the fiberwise
process $\left(L_{x}\omega_{n}\right)_{n=1}^{\infty}$ is not a Markov
chain. To ensure fiberwise Markov property, we will consider the special
case of standard systems in the next subsection.

\subsection{The special case of standard systems\label{sec:standard}}

Let $(X,\eta)$ be a standard system in the sense of \cite{FurstenbergGlasner},
that is, $\pi:(X,\eta)\to(Y,\nu)$ a measure preserving extension
and $(Y,\nu)$ a $\mu$-boundary. Then $\nu_{\omega}$ is a point
mass $\mathbb{P}_{\mu}$-a.s. In this case write $\nu_{\omega}=\delta_{\boldsymbol{\beta}_{Y}(\omega)}$
where $\boldsymbol{\beta}_{Y}:G^{\mathbb{N}}\to Y$ factors through
the Poisson boundary of $(G,\mu)$. In the notation of the boundary
map we have $\boldsymbol{\beta}_{\nu}({\rm bnd}(\omega))=\delta_{\boldsymbol{\beta}_{Y}(\omega)}$.
We use the same symbol $\boldsymbol{\beta}$ here, understanding that
for a $\mu$-boundary $(Y,\nu)$, $\boldsymbol{\beta}_{Y}$ and $\boldsymbol{\beta}_{\nu}$
are consistent when identifying points with $\delta$-masses. Of course, $(Y,\nu)$ is a Dirac mass when $(X,\eta)$ is invariant.

Denote by $y\mapsto\eta^{y}$ the disintegration of $\eta$ over $\pi:(X,\eta)\to(Y,\nu)$.
A useful property is that for a standard system, disintegration measures
coincide almost surely with conditional measures:

\begin{proposition}[{\cite[Prop 2.19]{Bader-Shalom}}]\label{disintegration}

Let $(X,\eta)$ be a standard system with the structure $\pi:(X,\eta)\to(Y,\nu)$.
Then 
\[
\eta_{\omega}=\eta^{\boldsymbol{\beta}_{Y}(\omega)}\mbox{ for }\mathbb{P}_{\mu}\mbox{-a.e. }\omega.
\]

\end{proposition}

Consider disintegration $\mathbb{P}_{\mu}=\int_{Y}\mathbb{P}_{\mu}^{y}d\nu(y)$
over the map $\boldsymbol{\beta}_{Y}:\left(G^{\mathbb{N}},\mathbb{P}_{\mu}\right)\to(Y,\nu)$.
Since the measurable spaces we consider are all Borel spaces, regular
conditional distributions exist. By uniqueness of disintegration,
we have that $\mathbb{P}_{\mu}^{y}$ is the conditional distribution
of $\omega$ given $\{\boldsymbol{\beta}_{Y}(\omega)=y\}$. It is
known that this conditional measure is the law of the Doob transformed
random walk determined by the Radon-Nikodym derivative $\varphi_{g}(y)=\frac{dg\nu}{d\nu}(y)$,
see \cite[Section 3]{KaimanovichHyperbolic}. Explicitly, a trajectory
$(\omega_{1},\omega_{2},\dots)$ with law $\mathbb{P}_{\mu}^{y}$
is a Markov chain with transition kernel 
\begin{equation}
\mathbb{P}_{\mu}^{y}(g,A)=\int_{G}{\bf 1}_{A}\left(gs\right)\frac{\varphi_{gs}(y)}{\varphi_{g}(y)}d\mu(s).\label{eq:Doob-Markov}
\end{equation}
Recall that $\vartheta_{\ast}\left(\eta\varcurlyvee\mathbb{P}_{\mu}\right)=\overline{\mathbb{P}}_{\mu}=\int_{X}\overline{\mathbb{P}}_{\mu,x}d\eta(x)$
denotes the disintegration over $W_{\Omega}\to X$.

\begin{lemma}\label{standard} When $(X,\eta)$ is standard with
the structure $\pi:(X,\eta)\to(Y,\nu)$, we have 
\[
\eta\varcurlyvee\mathbb{P}_{\mu}=\int_{Y}\eta^{y}\times\mathbb{P}_{\mu}^{y}d\nu(y)=\int_{X}\delta_{\{x\}}\times\mathbb{P}_{\mu}^{\pi(x)}d\eta(x).
\]
The coset process with law $\overline{\mathbb{P}}_{\mu,x}$ can be
sampled as follows. First sample $x\in X$ according to the stationary
measure $\eta$; then take the Doob transformed random walk $\omega$
on $G$ conditioned on $\left\{ \boldsymbol{\beta}_{Y}(\omega)=\pi(x)\right\} $.
The projected trajectory $\left(x,\left(L_{x}\omega_{1},L_{x}\omega_{2},\ldots\right)\right)$
on the coset space $\{x\}\times L_{x}\backslash G$ has distribution $\overline{\mathbb{P}}_{\mu,x}=\vartheta_{\ast}\left(\delta_{\{x\}}\times \mathbb{P}_{\mu}^{\pi(x)}\right)$.
\end{lemma}

\begin{proof}
Take a product set $A\times B\subseteq X\times G^{\mathbb{N}}$. By
Proposition \ref{disintegration} and as $\mathbb{P}_{\mu}^{y}(\boldsymbol{\beta}_{Y}(\omega)=y)=1$,
we have 
\begin{align*}
\eta\varcurlyvee\mathbb{P}_{\mu}(A\times B) & =\int_{G^{\mathbb{N}}}\eta_{\omega}(A)\delta_{\omega}(B)d\mathbb{P}_{\mu}(\omega)=\int_{Y}\int_{G^{\mathbb{N}}}\eta^{\boldsymbol{\beta}_{Y}(\omega)}(A){\bf 1}_{B}(\omega)d\mathbb{P}_{\mu}^{y}(\omega)d\nu(y)\\
 & =\int_{Y}\eta^{y}(A)\int_{G^{\mathbb{N}}}{\bf 1}_{B}(\omega)d\mathbb{P}_{\mu}^{y}(\omega)d\nu(y)=\int_{Y}\eta^{y}(A)\mathbb{P}_{\mu}^{y}(B)d\nu(y).
\end{align*}
It follows that $\overline{\mathbb{P}}_{\mu}=\vartheta_{\ast}\left(\eta\varcurlyvee\mathbb{P}_{\mu}\right)$
can be described as in the statement. 
\end{proof}

Note that the Doob transformed random walk on $G$ is a Markov chain. As we assumed that $L_x < \mathrm{Stab}_G(x)$, the Markov property is retained when projected to $L_{x}\backslash G$. By  Lemma \ref{standard}, we have:

\begin{proposition} \label{markov} 
For any $x\in X$, the coset trajectory $(L_{x}\omega_{1},L_{x}\omega_{2},\dots)\in(L_{x}\backslash G)^{\mathbb{N}}$
of law $\overline{\mathbb{P}}_{\mu,x}$ follows a Markov chain whose
transition kernel is given by 
\[
P_{\mu,x}:\left(L_{x}\backslash G\right)\times\mathcal{B}\left(L_{x}\backslash G\right)\to[0,1]
\]
\begin{equation}
P_{\mu,x}\left(L_{x}g,A\right)=\int_{G}{\bf 1}_{A}\left(L_{x}gs\right)\frac{\varphi_{gs}(\pi(x))}{\varphi_{g}(\pi(x))}d\mu(s),\label{eq:transition}
\end{equation}
where $\varphi_{g}(y)=\frac{dg\nu}{d\nu}(y)$ is the Radon-Nikodym
derivative. 
\end{proposition} 

\begin{proof}
The Doob transformed trajectory $(\omega_{1},\omega_{2},\dots)\in G^{\mathbb{N}}$
with law $\mathbb{P}_{\mu}^{\pi(x)}$ is a Markov chain with transition
kernel given by (\ref{eq:Doob-Markov}) with $y=\pi(x)$. Note that
\[
\frac{\varphi_{gs}(\pi(x))}{\varphi_{g}(\pi(x))}=\frac{dgs\nu}{dg\nu}(\pi(x))=\frac{ds\nu}{d\nu}(g^{-1}\pi(x)).
\]
Since $Y$ is a $G$-factor of $X$, we have $L_{x}<{\rm Stab}_{G}(\pi(x))$, so $L_xg=L_xg'$ implies $g'^{-1}\pi(x)=g^{-1}\pi(x)$. It implies that the transition kernel of the Markov chain $\mathbb{P}_{\mu}^{\pi(x)}$ is constant on each coset $L_xg$.
Therefore the
Markov chain of the group trajectory induces the claimed Markov chain
of the coset trajectory. The proposition follows as Lemma \ref{standard}
gives $\overline{\mathbb{P}}_{\mu,x}=\vartheta_{\ast}\left(\delta_{\{x\}}\times \mathbb{P}_{\mu}^{\pi(x)}\right)$.
\end{proof}

We now explain another description of the Poisson bundle over a standard
system $(X,\eta)$. Consider the Doob-transformed random
walk on $G$ with law $\mathbb{P}_{\mu}^{\pi(x)}$. Recall that $\left(Y,\nu\right)$
is assumed to be a quotient of the Poisson bounary $\left(B,\nu_{B}\right)$
of $(G,\mu)$ and we denote by $\beta_{Y}:(B,\nu_{B})\to(Y,\nu)$
the factor map. Let $\nu_{B}=\int_{Y}\nu_{B}^{y}d\nu(y)$ be the disintegration.
The shift operator $\mathcal{S}$ maps $\left(\omega_{1},\omega_{2},\ldots\right)$
to $\left(\omega_{2},\omega_{3},\ldots\right)$. Take the $\mathcal{S}$-invariant
sub-$\sigma$-field $\mathcal{I}^{y}$ in $\left(G^{\mathbb{N}},\mathbb{P}_{\mu}^{y}\right)$.
By Proposition \ref{disintegration}, we have that the fiber $\left(\beta_{Y}^{-1}(\{y\}),\nu_{B}^{y}\right)$
from the disintegration is a model for $\mathcal{I}^{y}$ equipped
with conditional measure $\mathbb{P}_{\mu}^{y}$. Note that (\ref{eq:contain})
implies that $L_{x}$ preserves $\beta_{Y}^{-1}(\{\pi(x)\})$.

\begin{proposition}\label{ergodiccom}
With the notation introduced above,  in the Poisson bundle
$\left(Z,\lambda\right)\to(X,\eta)$ the fiber over $x\in X$ can
be described as the ergodic components of $L_{x}\curvearrowright\left(\beta_{Y}^{-1}(\{y\}),\nu_{B}^{y}\right)$,
where $y=\pi(x)$. \end{proposition}
\begin{proof}
The proof is the same as the statement for Poisson boundary of random
walks on Schreier graphs, see \cite[Prop 4.1.2.]{Cannizzo}, also
the commutative diagram on \cite[P. 489]{Bowen}. 
\end{proof}

\subsection{Tail $\sigma$-field\label{subsec:tail} }

In preparation for the random walk entropy formulae, we now consider
the relation between the invariant and tail $\sigma$-fields. We denote
the bundle of coset spaces by 
\[
W:=\left\{ \left(x,L_{x}\omega\right):x\in X,\omega\in G\right\} =\bigsqcup_{x\in X}\{x\}\times L_{x}\backslash G.
\]
On the space $\left(X\times G^{\mathbb{N}},\eta\varcurlyvee\mathbb{P}_{\mu}\right)$,
we have a sequence of random variables 
\begin{align*}
\xi_{n} & :X\times G^{\mathbb{N}}\to W\\
 & (x,\omega)\mapsto(x,L_{x}\omega_{n}).
\end{align*}
In other terms, $\xi_{n}=\theta_{n}\circ\vartheta$, where $\theta_{n}$
is taking time $n$ position of a coset trajectory in a fiber of $W_{\Omega}\to X$.
See commutative diagram in Figure~\ref{fig:diag}.

Let $\mathcal{T}$ be the tail $\sigma$-field of $\left(\xi_{n}\right)_{n=1}^{\infty}$:
\[
\mathcal{T}:=\cap_{n=1}^{\infty}\sigma\left(\xi_{n},\xi_{n+1},\xi_{n+2},\ldots\right).
\]
It is clear from the definitions that $\mathcal{I}\subseteq\mathcal{T}$.

The restriction of the map $\xi_{n}$ defined above to the fiber over
a given $x\in X$ is the map 
\[
\xi_{n}(x,\cdot):G^{\mathbb{N}}\to L_{x}\backslash G.
\]
We write $\xi_{n}^{x}$ for the random variable taking values in $L_{x}\backslash G$
with law $\xi_{n}(x,\cdot)_{\ast}\left(\eta\varcurlyvee\mathbb{P}_{\mu}\right)=\xi_{n\ast}\mathbb{P}_{\mu}^{\pi(x)}$
by Lemma~\ref{standard}. (With a slight abuse, we identify $L_{x}\backslash G$
with $\{x\}\times L_{x}\backslash G$.) By Proposition \ref{markov},
$\left(\xi_{n}^{x}\right)_{n=0}^{\infty}$ is the Markov chain with
transition probabilities $P_{\mu,x}$ starting at the identity coset.
Let $\mathcal{T}_{x}$ be the tail $\sigma$-field of this Markov
chain. It is clear that $\mathcal{T}_{x}$ almost surely coincides
with the restriction of the tail $\sigma$-field $\mathcal{T}$ to
the fiber over $x$.

In order to relate the Furstenberg entropy of the Poisson bundle to
the entropy of the random walk, we need to identify the invariant
and tail $\sigma$-fields, up to null sets. For general Markov chains,
the tail $\sigma$-field and the invariant $\sigma$-field do not
necessarily agree modulo null sets, see \cite{Kaimanovich0-2}. To
ensure that fiberwise $\mathcal{I}_{x}\equiv\mathcal{T}_{x}$ modulo
null sets, we assume 
\begin{description}
\item [{\hypertarget{AssumpT}{(T)}}] Tail assumption. The random walk step distribution $\mu$
on $G$ satisfies that there exists $n\in\mathbb{N}$ and $\epsilon>0$
such that 
\[
d_{{\rm TV}}\left(\mu^{(n)},\mu^{(n+1)}\right)<1-\epsilon.
\]
\end{description}
This assumption is satisfied for example by admissible measures on
a locally compact group, and for $\mu$ on a countable group with
$\mu(e)>0$. As in the countable $G$ case, given a $(G,\mu)$-space $(X,\eta)$, consider
the lazy step distribution $\mu_t=(1-t)\delta_{e}+t\mu$ for $t\in(0,1)$,
then $\eta$ is $\mu_t$-stationary as well and $h_{\mu_t}(X,\eta)=th_{\mu}(X,\eta)$. 
For this reason, we will assume without further mention that this tail assumption is
satisfied when we are concerned with Furstenberg entropy realization
problem.

By \cite[Theorem 2.7]{Kaimanovich0-2}, assumption \hyperlink{AssumpT}{(\textbf{T})} is sufficient
to guarantee $\mathcal{I}_{x}\equiv\mathcal{T}_{x}$ modulo null sets:

\begin{corollary}\label{cor:tail} Suppose the step distribution
$\mu$ on $G$ satisfies assumption \hypertarget{AssumpT}{(T)}. Then for any $x\in X$,
the fiberwise invariant and tail $\sigma$-fields are equal up to
null sets : $\mathcal{I}_{x}\equiv\mathcal{T}_{x}$. A fortiori, the
invariant and tail $\sigma$-fields agree $\mathcal{I}\equiv\mathcal{T}$
modulo null sets. \end{corollary}
\begin{proof}
By Proposition \ref{markov}, the fiberwise process $\left(\xi_{n}^{x}\right)$
is a Markov chain. Moreover, the formula for the Markov kernel implies
that 
\[
d_{{\rm TV}}\left(P_{\mu,x}^{n}\left(L_{x}g,\cdot\right),P_{\mu,x}^{n+1}\left(L_{x}g,\cdot\right)\right)\le d_{{\rm TV}}\left(\mu^{(n)},\mu^{(n+1)}\right).
\]
Apply \cite[Theorem 2.7]{Kaimanovich0-2} to $\left(\xi_{n}^{x}\right)$,
we conclude that under assumption \hypertarget{AssumpT}{(T)}, $\mathcal{I}_{x}\equiv\mathcal{T}_{x}$
modulo null sets. 
\end{proof}

\subsection{Formulae for Furstenberg entropy\label{subsec:formula}}

Consider a Poisson bundle $(Z,\lambda)$ over a standard system $(X,\eta)$, together with the fiberwise Markov chain
$(\xi_{n}^{x})$ and the tail $\sigma$-field $\mathcal{T}_{x}$,
as defined above.

We denote the conditional mutual information of $\xi_{n}^{x}$ and
$\mathcal{T}_{x}$ by 
\begin{align*}
\mathrm{I}\left(\xi_{1},\mathcal{T}|X,\eta\right) & :=\int_{X}\mathrm{I}\left(\xi_{1}^{x},\mathcal{T}_{x}\right)d\eta(x).
\end{align*}
Similarly, when $G$ is countable, we denote the conditional Shannon
entropy by 
\[
H\left(\xi_{n}|X,\eta\right)=\int_{X}H\left(\xi_{n}^{x}\right)d\eta(x).
\]
We refer to Appendix~\ref{app:info-ent} for definitions and basic
properties of mutual information and entropy.

The next proposition shows that for a Poisson bundle $(Z,\lambda)$
over a standard stationary system $(X,\eta)$, its Furstenberg entropy
can by expressed as the sum of the Furstenberg entropy of the base
$(X,\eta)$ and mutual information from fiberwise Markov chains. It
will be useful for showing upper-semi continuity properties.

\begin{proposition}\label{entropy-formula} Let $(Z,\lambda)$ be
the Poisson bundle over a standard system $(X,\eta)$ satisfying (\ref{eq:contain}).
Then 
\[
h_{\mu}(Z,\lambda)=h_{\mu}(Y,\nu)+\mathrm{I}\left(\xi_{1},\mathcal{T}|X,\eta\right).
\]
\end{proposition}

This proposition is a particular case of Proposition~\ref{entropy-formula-M}, which follows classical arguments of
Derriennic \cite{Derriennic}, as detailed in Appendix~\ref{subsec:entropymutual}.

In the case where $(X,\eta)=(Y,\nu)$ is a $\mu$-boundary together
with $L:Y\to{\rm Sub}(G)$ the trivial map that $L_{y}=\{id\}$ for
all $y$, then the Poisson bundle over $(Y,\nu)$ is the Poisson boundary
$\left(B,\nu_{B}\right)$ of the $\mu$-random walk. Proposition \ref{entropy-formula}
recovers the known formula: 
\[
h(B,\nu_{B})=h\left(Y,\nu\right)+\int_{Y}{\rm I}(\mathbb{P}_{\mu,1}^{y},\mathcal{T}_{y})d\nu(y)=h\left(Y,\nu\right)+\int_{Y}\inf_{n\in\mathbb{N}}{\rm I}(\mathbb{P}_{\mu,1}^{y},\mathbb{P}_{\mu,n}^{y})d\nu(y).
\]
which follows from combining \cite[Theorem 4.5]{KaimanovichHyperbolic}
and \cite{Derriennic}.

For countable groups, we have the following formulae for Furstenberg
entropy of the Poisson bundle in terms of Shannon entropy of the fiberwise
random walks. This can be viewed as a generalization of the formula
for Poisson bundle over an IRS in \cite{Bowen}.

\begin{theo}[Random walk entropy formula]\label{thm:entropy} Assume
$G$ is a countable group endowed with a probability measure $\mu$
of finite Shannon entropy. Let $(X,\eta)$ be a standard system over
$(Y,\nu)$, with $x\mapsto L_{x}$ satisfying (\ref{eq:contain}). Then the Poisson
bundle $\left(Z,\lambda\right)$ over $\left(X,\eta\right)$ satisfies
\begin{align*}
h_{\mu}(Z,\lambda) 
 =h_{\mu}(Y,\nu)+\lim_{n\to\infty}\frac{1}{n}H\left(\xi_{n}|X,\eta\right)=h_{\mu}(Y,\nu)+\inf_{n\in\mathbb{N}}\frac{1}{n}H\left(\xi_{n}|X,\eta\right).
\end{align*}
\end{theo}

Theorem \ref{thm:entropy} is an immediate corollary of Theorem~\ref{thm:entropy-app} from Appendix~\ref{subsec:random-walk-entropy}.

\section{Upper semi-continuity of Furstenberg entropy}\label{subsec:Upper-semi-continuity}

Given $(G,\mu)$, we consider in this section a fixed $\mu$-boundary $(Y,\nu)$ and a fixed Borel space $X$ such that $L:X\to{\rm Sub}(G)$ is a $G$-equivariant map satisfying our standing assumption~(\ref{eq:contain}). We denote by $\mathcal{M}_\nu$ the subset of all $(G,\mu)$-stationary measures $\eta$ on $X$  for which the $G$-map $\pi:(X,\eta)\to (Y,\nu)$ is a measure preserving extension.
For each choice of $\eta$ in $\mathcal{M}_\nu$, we obtain a Poisson bundle $(Z_\eta,\lambda_\eta)$, which is the Mackey point realisation of the invariant $\sigma$-field of the space $W_\Omega$ equipped with the measure $\vartheta(\eta\varcurlyvee\mathbb{P}_\mu)$, as described in Subsection~\ref{subsec:Definition}. The entropy formulae of the previous section permit to obtain upper semi-continuity results for the map $\eta\mapsto h_\mu(Z_\eta,\lambda_\eta)$ from $\mathcal{M}_\nu$ to the real numbers.

When $\eta$ is invariant, i.e. $Y$ is reduced to a point, and $\mu$ has finite support on a countable group $G$, this is an observation due to Bowen~\cite[Corollary 4.5]{Bowen}. Indeed in this particular case, Theorem~\ref{thm:entropy} recovers the entropy formula~\cite[Theorem 3.1]{Bowen}
\[
h_\mu(Z_\eta,\lambda_\eta)=\inf_n \frac{1}{n} \int_X H(L_x\omega_n)d\eta(x).
\]
For fixed $n$, the map $L\mapsto H(L\omega_n)$ is locally constant on the totally disconnected space $\mathrm{Sub}(G)$, as its value depends only on the ball of radius $n\mathrm{diam}(\mathrm{supp}(\mu)\cup\{e\})$ in the Schreier graph of $L$. Therefore if we assume that $L$ is continuous (in our applications, it can be taken to be the identity map), each map $\eta\mapsto \int_X H(L_x\omega_n)d\eta(x)$ is weak$^{\ast}$-continuous. The Furstenberg entropy of the Poisson bundle is then upper semi-continuous as an infimum of continuous functions.

Our goal in this section is to generalise this observation in two directions. Firstly  to countable groups endowed with a probability measure $\mu$ of finite Shannon entropy.  Secondly to locally compact groups with compactly  supported measures, provided they satisfy a local coincidence assumption with respect to an appropriate intermediate factor $(X,\eta)\to (Y',\nu')\to (Y,\nu)$.

 \subsection{Countable groups with finite Shannon entropy}

In the countable setting, we show upper semi-continuity of the Furstenberg entropy on $\mathcal{M}_\nu$
for the more general class of step distributions $\mu$ with finite Shannon
entropy.


\begin{proposition}\label{prop:upper-semi} Assume $G$ is a countable discrete
group endowed with a probability measure $\mu$ of finite Shannon
entropy. Let $(\eta_k)_{k\ge 0}$ and $\eta$ belong to $\mathcal{M}_\nu$ and assume that on a subset of full $\nu$-measure of $Y$ we have a weak$^\ast$ convergence $\eta_k^y\to \eta^y$. Then
\[
\limsup_{k} h_{\mu}(Z_{\eta_k},\lambda_{\eta_k}) \le h_{\mu}(Z_{\eta},\lambda_{\eta}).
\]
\end{proposition}

The proposition is an easy consequence of the following lemma.

\begin{lemma}\label{lem:continuous} For each $n\ge1$,  we have $\limsup_k\left|H(\xi_{n}|X,\eta_k)- H(\xi_{n}|X,\eta)\right| \le 4$.
\end{lemma}
As we will be interested in cases where the entropy grows linearly with $n$, this lemma essentially states that $H(\xi_{n}|X,\eta)$ depends continuously on $\eta$.
The key point to its proof is to approximate the time $n$ of the Markov chain by a measure with finite
support. In this purpose we record the following lemma.

\begin{lemma}\label{lem:proj}
Let $\zeta$ be a probability measure with finite Shannon entropy on a countable space $G$. There exists $K\subset G$ finite such that for any map $\theta:G\to T$, where $T$ is a set, we have
\[
|H(\theta_\ast\zeta)-H(\theta_\ast\zeta_{|K})| \le 2.
\]
\end{lemma}

Note that $\zeta_{|K}$ is a measure with total mass $\zeta(K)$.

\begin{proof}[Proof of Lemma~\ref{lem:proj}] Denote $K^c=G\setminus K$ and $\kappa(t)=-t\log(t)$. As $\zeta$ has finite Shannon entropy, there exists $K$ finite subset of $G$ with $H(\zeta_{|K^c})=\sum_{g \in K^c}\kappa(\zeta(g)) \le 2$.
We consider the partition $\mathcal{P}$ on $G$ given by all sets $\theta^{-1}(x)\cap K$ and $\theta^{-1}(x)\cap K^c$ for $x \in T$. We have
\begin{align*}H(\theta_\ast\zeta_{}) \le H(\mathcal{P}) &
=H(\theta_\ast \zeta_K)+H(\theta_\ast \zeta_{K^c})=\sum_{x \in T} \kappa(\zeta(\theta^{-1}(x)\cap K))+\kappa(\zeta(\theta^{-1}(x)\cap K^c)) \\&=\sum_{x\in T}\kappa(p_x) +\sum_{x\in H}p_x(\kappa(\frac{\alpha_x}{p_x})+\kappa(\frac{\beta_x}{p_x})) \le H(\theta_\ast \zeta)+2\end{align*}
with notations $p_x=\zeta(\theta^{-1}(x))$, $\alpha_x=\zeta(\theta^{-1}(x)\cap K)$ and $\beta_x=\zeta(\theta^{-1}(x)\cap K^c)$. In the second line, we use the immediate equality $\kappa(\alpha)+\kappa(\beta)=\kappa(\alpha+\beta)+(\alpha+\beta)\left( \kappa(\frac{\alpha}{\alpha+\beta})+ \kappa(\frac{\beta}{\alpha+\beta})\right)$ and $\kappa(t)\le 1$ for the final inequality. 
On the other hand, we have 
\[
H(\mathcal{P})\ge H(\theta_\ast \zeta_{|K}) =H(\mathcal{P})-H(\theta_\ast \zeta_{|K^c})\ge H(\mathcal{P})-2
\]
 by choice of $K$.
\end{proof}

\begin{proof}[Proof of Lemma~\ref{lem:continuous}]Let $\mathbb{P}_{\mu,n}^{y}$ denote the law of step $n$ of the Doob
transformed random walk started at identity. Then $\mu^{(n)}=\int_{Y}\mathbb{P}_{\mu,n}^{y}d\nu(y)$.
By concavity of entropy $\int_{Y}H(\mathbb{P}_{\mu,n}^{y})d\nu(y)\le H(\mu^{(n)})$,
so $\mathbb{P}_{\mu,n}^{y}$ has finite entropy for $\nu$-a.e. $y$.

 Endow $G$ with an exhaustion $(K_j)$ by finite subsets.
For each $y\in Y$, apply Lemma~\ref{lem:proj} to the measure $\mathbb{P}_{\mu,n}^y$ on $G^n$.
We obtain for any  $\varepsilon>0$ a measurable subset $Y_1\subset Y$ and some $K=K_j$ satisfying $\int_{Y\setminus Y_1} H(\mathbb{P}_{\mu,n}^y)d\nu(y) \le \varepsilon$ and that for all $y\in Y_1$ and all $x \in \pi^{-1}(y)$
 \begin{align}\label{eq:entKj}
 \left|H(\theta_{L_x\ast} \mathbb{P}_{\mu,n|K_j}^y) -H(\theta_{L_x\ast} \mathbb{P}_{\mu,n}^y) \right| \le 2.
 \end{align}
 
 Observe that for each $y \in Y$, the map $L\mapsto H(\theta_{L\ast}\mathbb{P}_{\mu,n|K}^{y})$
is continuous. Indeed, if $L,L'\in{\rm Sub}(G)$ are close enough
in Chabauty topology, the images of $K$ under the maps $G\to L\backslash G$ and $L'\backslash G$ agree, and so $\theta_{L\ast}\mathbb{P}_{\mu,n|K}^{y}$ and $\theta_{L'\ast}\mathbb{P}_{\mu,n|K}^{y}$ can be identified, hence have the same entropy.
Then for each $y\in Y$, the map  $\kappa\mapsto\int_{X}H(\theta_{L_{x}\ast}\mathbb{P}_{\mu,n|K}^{y})d\kappa(L)$, defined on the set of probability measures on $G$,
is continuous in weak$^{\ast}$ topology. In particular, we have for $\nu$-almost every $y$ and for all $x \in \pi^{-1}(y)$ that
\begin{align}\label{eq:limitK}
\int_{\pi^{-1}(y)}H(\theta_{L_{x}\ast}\mathbb{P}_{\mu,n|K}^{y})d\eta_{k}^{y}(x) \underset{k\to\infty}{\longrightarrow} \int_{\pi^{-1}(y)}H(\theta_{L_{x}\ast}\mathbb{P}_{\mu,n|K}^{y})d\eta_{}^{y}(x)
\end{align}
As these quantities are dominated by $H(\mathbb{P}_{\mu,n}^{y})$ which belongs to $L^{1}(Y,\nu)$,
the above convergence still holds once integrated over a subset of $Y$.

Now we are able to estimate
\begin{align*}
\left|\int_X H(\xi_n^x)d(\eta_k-\eta)(x)\right| &\le \int_Y \left| \int_{\pi^{-1}(y)} H(\theta_{L_x\ast}\mathbb{P}_{\mu,n}^y)d(\eta_k^y-\eta^y)(x)\right|d\nu(y) \\
&\le \int_{Y\setminus Y_1}2H(\mathbb{P}_{\mu,n}^y)d\nu(y)+\int_{Y_1}\left|\int_{\pi^{-1}(y)} H(\theta_{L_x\ast}\mathbb{P}_{\mu,n}^y)d(\eta_k^y-\eta^y)(x)\right|d\nu(y)
\end{align*}
The first term is bounded above by $2\varepsilon$. In the second term, we can replace $H(\theta_{L_x\ast}\mathbb{P}_{\mu,n}^y)$ by $H(\theta_{L_x\ast}\mathbb{P}_{\mu,n|K}^y)$ up to $4$ by (\ref{eq:entKj}), and then use the limit (\ref{eq:limitK}). We conclude that $|H(\xi_{n}|X,\eta_k)-H(\xi_{n}|X,\eta)|\le 4+3\varepsilon$ when $k$ is large enough.
\end{proof}
\begin{proof}[Proof of Proposition \ref{prop:upper-semi}]
By Theorem~\ref{thm:entropy}, we have for each $n$ that
\begin{align*}
\limsup_k h_\mu(Z_{\eta_k},\lambda_{\eta_k}) &\le \limsup_k h_\mu(Y,\nu)+\frac{1}{n}H(\xi_n|X,\eta_k) \\
&\le h_\mu(Y,\nu)+\frac{H(\xi_n|X,\eta)}{n}+\frac{4}{n}
\end{align*}
whose limsup is $h_\mu(Z_\eta,\lambda_\eta)$ by Theorem~\ref{thm:entropy} again.
\end{proof}

The condition that $\eta_k^y$ converges to $\eta^y$ on a $\nu$-full subset of $Y$ is somewhat technical. It clearly implies that $\eta_k \to \eta$ in weak$^\ast$ topology. We record for completeness that the converse is also true  under mild conditions on $(Y,\nu)$, because the extension $(X,\eta)\to(Y,\nu)$ is measure preserving.

\begin{lemma}\label{lem:Ymfd}
Assume that $Y$ is a manifold with continuous $G$-action and $\nu$ is a volume form. Let $(\eta_k)$ and $\eta$ be probability measures in $\mathcal{M}_\nu$ such that $\eta_k$ converges weak$^\ast$ to $\eta$. Then, for each $y\in Y$, the measures $\eta^y_k$ converge weak$^\ast$ to $\eta^y$.
\end{lemma}

\begin{proof}
Take $y_0\in Y$ and a continuous compactly supported map $f:\pi^{-1}(y_0)\to \mathbb{R}$. Let $\theta:Y=G/\mathrm{stab}(y_0)\to G$  be a section, continuous on a compact neighbourhood $U$ of $y_0$. Then we have $\theta(y)y_0=y$ for any $y$ in $Y$. Since the extension is measure preserving, we also have $\theta(y)\eta^{y_0}=\eta^{\theta(y)y_0}$ for $\nu$-almost every $y$.

 Let $\chi:Y\to [0,1]$ be a function where $\chi|_{U_1}=1$ and $\chi|_{Y\setminus U_2}=0$ for some neighbourhoods $U_1,U_2$ of $y_0$ satisfying $\overline{U}_1 \subset U_2\subset U$ and $\nu(U_1)>0$. Consider the lift $\widetilde{f}(x)=f(\theta(\pi(x))^{-1}x)\chi(\pi(x))$ of $f$. By construction it is a continuous map on $X$, compactly supported. We have
\begin{align*}
\int_X \widetilde{f}d\eta & =\int_Y\int_{\pi^{-1}(y)} f(\theta(y)^{-1}x)\chi(y)d\eta^y(x)d\nu(y) \ge \int_{\overline{U}_1}\int_{\pi^{-1}(y)} \theta(y) f(x)d\eta^{\theta(y)y_0}(x)d\nu(y)\\
& \ge\int_{\overline{U}_1}\int_{\pi^{-1}(y_0)} f(x)d\eta^{y_0}(x)d\nu(y)=\nu(\overline{U}_1)\int_{\pi^{-1}(y_0)} fd\eta^{y_0}
\end{align*}
and a similar upper bound holds with an extra term $\nu(\overline{U}_2\setminus U_1)\max(f)$. We get
\[
\frac{1}{\nu(U_1)}\int_X \widetilde{f}d\eta -\frac{\nu(\overline{U}_2\setminus U_1)}{\nu(U_1)}\max(f)\le \int_{\pi^{-1}(y_0)}fd\eta^{y_0} \le \frac{1}{\nu(U_1)}\int_X \widetilde{f}d\eta.
\]
The same holds for $\eta_k$ in place of $\eta$. Given arbitrary $\varepsilon>0$, take $k$ large enough that $|\int_X\widetilde{f}d\eta_k-\int_X\widetilde{f}d\eta|\le \varepsilon \nu(U_1)$, and  $U_2$ sufficiently close to $U_1$ (here we use the fact that $Y$ is a manifold with a volume form), we get $|\int_{\pi^{-1}(y_0)}fd\eta_k^{y_0}-\int_{\pi^{-1}(y_0)}fd\eta^{y_0}|\le 2\varepsilon$. This shows $\eta_k^{y_0}$ converges to $\eta^{y_0}$.
\end{proof}

Proposition~\ref{prop:upper-semi} and Lemma~\ref{lem:Ymfd} provide a general semi-continuity statement.

\begin{corollary}\label{upper-semi}
Let $G$ be a countable discrete
group endowed with a probability measure $\mu$ of finite Shannon
entropy. Also assume that $Y$ is a manifold with continuous $G$-action and $\nu$ is a volume form.
Then the map $\eta\mapsto h_{\mu}(Z_{\eta},\lambda_{\eta})$ is upper
semi-continuous on the space $\mathcal{M}_{\nu}$. 
\end{corollary}

In particular, the corollary applies for invariant systems when $(Y,\nu)$ is a Dirac mass. 

\subsection{Locally compact groups with local coincidence assumption}

We assume here that $G$ is a locally compact group, endowed with
a probability measure $\mu$ of compact support and finite boundary
entropy $h_{\mu}(B,\nu_{B})<\infty$.  Our goal is to obtain results similar to Proposition~\ref{prop:upper-semi} and Corollary~\ref{upper-semi}. 

For this reason we replace the use of entropy from Theorem~\ref{thm:entropy} by mutual information from Proposition~\ref{entropy-formula}. However in the locally compact setting, as opposed to the discrete setting, it is no longer true that two subgroups close in Chabauty topology have quotients with large isomorphic balls (i.e. local coincidence), which is a crucial ingredient in the previous proofs.

In order to obtain weak$^{\ast}$-continuity
of the maps $\eta\mapsto\mathrm{I}(\xi_{1},\xi_{n}|X,\eta)$, for each $n$, we therefore restrict ourselves to a specific setting, which we describe below, where we can apply local coincidence arguments.
It is possible that these maps are continuous on the whole space of stationary measures on $\mathrm{Sub}(G)$ but we take advantage of the fact that some local coincidence will be satisfied in our applications. 

\subsubsection{Local coincidence in Chabauty topology}

\begin{definition} A subset $S\subset{\rm Sub}(G)$ has the property
of \emph{local coincidence in Chabauty topology} if a sequence $(L_{k})_{k\in\mathbb{N}}$
of subgroups in $S$ converges to $L\in S$ in Chabauty topology
if and only if for any exhaustion $(K_{n})_{n=1}^{\infty}$ of $G$
by compact subsets, we have 
\[
\sup\left\{ n:L_{k}\cap K_{n}=L\cap K_{n}\right\} \underset{k\to\infty}{\longrightarrow}\infty.
\]
\end{definition} 

This property means that two subgroups $L,L'$ in $S$ are
close in Chabauty topology if and only if they actually coincide on large subsets
$L\cap K_{n}=L'\cap K_{n}$.

Clearly ${\rm Sub}(G)$ has this property when $G$ is discrete. It
also holds when $S$ is the collection of subgroups of a fixed discrete
subgroup of $G$. Some non-discrete examples appear as pullbacks : if $p:G\to L$ is a surjective homomorphism of topological groups and $F$ is a discrete subgroup of $L$, the collection of $p$-preimages of subgroups of $F$ has the local coincidence property (such examples will be used in Section
\ref{sec:Poisson-SL}). However local coincidence is not satisfied in general. For instance, the space of one-dimensional subgroups
of $\mathbb{R}^{2}$ does not have local coincidence property.

Local coincidence also implies that the quotient spaces coincide locally. This is well known for finitely generated groups. Let us provide a more general statement. Assume here that $G$ is a topological group equipped with a left invariant distance $d$ generating the topology. Denote $B_G(e,r)$ the ball of radius $r$ centred at identity. For any closed subgroup $L$, the space of cosets $L\backslash G$ inherits a distance $d(Lg_1,Lg_2)=\inf\{d(x,y):x\in Lg_1,y\in Lg_2\}=d(g_1,Lg_2)$. 

\begin{lemma}\label{lem:localcoincide}
Let $G$ be a topological group with a left-invariant distance $d$ generating the topology. Let $L,L'$ be two closed subgroups and $r\ge 0$. Assume that $L\cap B_G(e,5r)=L' \cap B_G(e,5r)$. Then the  coset spaces $L\backslash G$ and $L'\backslash G$ are $r$-locally isometric. 

More precisely, for any $h,g \in B_G(e,r)$, the map $u:Lh \mapsto L'h$ from $B_{L\backslash G}(L,r)$ to $B_{L'\backslash G}(L',r)$ is an isometry satisfying that whenever $Lhg \in B_{L\backslash G}(L,r)$, we have $u(Lhg)=L'hg=u(Lh)g$.
\end{lemma}

\begin{proof}
Let $Lh_1,Lh_2$ belong to $B_{L\backslash G}(L,r)$. We can assume $h_1,h_2$ belong to $B_G(e,r)$. As $d(h_1,Lh_2) \le 2r$, there is $\ell \in L$ with $d(h_1,\ell h_2) \le 2r$. Therefore $d(Lh_1,Lh_2) =\inf\{d(e,h_1^{-1}\ell h_2): \ell \in L, h_1^{-1}\ell h_2\in B_G(e,2r)\}$. Now if $d(e,\ell)\ge 5r$, then $d(e,h_1^{-1}\ell h_2) \ge 3r$. This shows that $d(L h_1,Lh_2)$ depends only on $L\cap B_G(e,5r)$.
\end{proof}

\subsubsection{An upper semi-continuity criterion for locally compact groups}\label{subsec:Y'}

Let us now consider an intermediate factor $(X,\eta)\to(Y',\nu')\to (Y,\nu)$ which is a measure preserving extension of $(Y,\nu)$. For each $y'$ in $Y'$ consider a compact subset $S_{y'}\subset \mathrm{Sub}(G)$ with the local coincidence property. Denote $\mathcal{S}=(S_{y'})_{y'\in Y'}$. We define $\mathcal{M}_\nu(Y',\mathcal{S})$ to be the subset of $\mathcal{M}_\nu$ of probability measures $\eta$ for which the disintegration $\eta=\int_{Y'} \eta^{y'}d\nu(y')$ satisfies that $\mathrm{supp}(L_\ast \eta^{y'}) \subset S_{y'}$ for $\nu'$-almost every $y'$. We say that such measures satisfy the local coincidence property with respect to $(Y',\nu')$.

When the group $G$ is discrete, one can take $S_{y'}=\mathrm{Sub}(G)$, so the local coincidence property holds with respect to $Y'=Y$ on the whole space $\mathcal{M}_\nu$.

In order to obtain upper semi-continuity, we consider use the following fiberwise KL-divergence defined as
\[
D(G,\mathcal{T}|X,\eta):=\int_XD_g(x)dg\eta(x) \quad \textrm{where} \quad D_g(x):=D(\mathbb{P}_{\mu,x,g}|_{\mathcal{T}_x}||\mathbb{P}_{\mu,x,e}|_{\mathcal{T}_x}).
\]
These quantities are defined by (\ref{eq:DgT}) and~(\ref{eq:Dgx}) in the Appendix~\ref{subsec:entropymutual}, in the particular case when the intermediate factor $(M,\bar\lambda)$ is simply $(X,\eta)$. We refer to the appendix for explanations and results concerning KL-divergence.

\begin{proposition}\label{info-conti} Let $G$ be a locally compact
group whose topology is generated by a left invariant distance, equipped with a probability measure $\mu$ of compact support and finite
boundary entropy.  Let $(\eta_k)_{k\ge 0}$ and $\eta$ be probability measures in the space $\mathcal{M}_\nu(Y',\mathcal{S})$ defined above, with $\mathcal{S}$ satisfying the local coincidence property, and assume that on a subset of full $\nu'$-measure of $Y'$ we have a weak$^\ast$ convergence $\eta_k^{y'}\to \eta^{y'}$.  Then for each $n\ge 1$ and $g \in G$, the KL-divergences are upper semi-continuous
\[
\limsup_{k\to \infty}D(g,\mathcal{T}|X,\eta_k) \le D(g,\mathcal{T}|X,\eta).
\]
In particular,
\[
\limsup_{k} h_\mu(Z_{\eta_k},\lambda_{\eta_k}) \le h_\mu(Z_\eta,\lambda_\eta).
\]
\end{proposition}

Let us fix a few notations for the proof. Let  $\xi_{n,g}^L$ denote the time $n$ position of the Markov process induced on $L\backslash G$ and started at $Lg$ as described in Proposition~\ref{markov} (see also Lemma~\ref{lem:MarkovM} and~\ref{translate-1} for the more general setting of intermediate factors).

\begin{proof}  Denote $(X,\eta)\overset{\pi'}{\longrightarrow}(Y',\nu')\overset{\pi''}{\longrightarrow} (Y,\nu)$ the maps of the intermediate factor, with $\pi''\circ\pi'=\pi$. Given $y'\in Y'$, let $L$ belong to $S_{y'}$.
We denote $\xi_{n}^{L}$ the image in $L\backslash G$ of the
time $n$ position of the Doob transformed random walk of law $\mathbb{P}_{\mu}^{y}$ where $y=\pi''(y')$. 
Let $K$ be a compact set containing the set $\{e,g\}(\mathrm{supp}(\mu))^n$. Then the time $n$
positions of the random walk started either at $e$ or at $g$ belong to $K$. Assume $(L_k)_{k\ge 0}$ is a sequence of subgroups in $S_{y'}$, converging to $L$. The local coincidence assumption guarantees that for large $k$, we have $L_k\cap K=L\cap K$.

Using Lemma~\ref{lem:localcoincide} and assuming $k$ is large, the two random variables $(\xi_{n-1,g}^{L_k},\xi_{n,e}^{L_k})$ and $(\xi_{n-1,g}^{L},\xi_{n,e}^{L})$
have the same law, so in particular $D_g^{(n)}(L_k)=D_g^{(n)}(L)$. 
This shows that the map $L\mapsto D_g^{(n)}(L)$ 
is continuous on $S_{y'}$. By compactness of $S_{y'}$, the map  $\kappa\mapsto\int_{S_{y'}}D_g^{(n)}(L)d\kappa(L)$
is weak$^{\ast}$ continuous on the space of probability measures on $S_{y'}$. 

Now consider a converging sequence $\eta_k\to \eta$ in $\mathcal{M}_\nu(Y',\mathcal{S})$. We disintegrate and use Lemma~\ref{derriennic-basic(i)}
\begin{align*}
D(g,\mathcal{T}|X,\eta_k)&=\int_{X}D_g(x)dg\eta_{k}(x)=\int_{Y'}\int_{S_{y'}}D_g(L_x)d(g\eta_{k})^{y'}(x)d\nu'(y') \\
& =\int_{Y'}\int_{S_{y'}}\inf_{n \ge 1} D_g^{(n)}(L_x)d(g\eta_{k})^{y'}(x)d\nu'(y').
\end{align*}
By Lemma~\ref{derriennic-basicY}, $D_g^{(1)}(L_x) <\infty$. We get that 
\[
D(g,\mathcal{T}|X,\eta_k)=\int_{Y'} \inf_{n \ge 1}\int_{S_{y'}} D_g^{(n)}(L_x)d(g\eta_{k})^{y'}(x)d\nu'(y')
\]
Now for  $\nu'$-almost every $y'$, we have $\eta_k^{y'} \to \eta^{y'}$. By continuity of the maps $\kappa\mapsto\int_{S_{y'}}D_g^{(n)}(L)d\kappa(L)$ for each $n$, we derive that for $\nu'$-a.e. $y'$
\[
\limsup_{k \to \infty} \inf_{n \ge 1}\int_{S_{y'}} D_g^{(n)}(L_x)d(g\eta_{k})^{y'}(x) \le \inf_{n \ge 1}\int_{S_{y'}} D_g^{(n)}(L_x)d(g\eta)^{y'}(x)
\]
Since the functions are dominated, see Lemma~\ref{derriennic-basicY}, we can integrate and obtain the required limsup.

For the entropy, Proposition~\ref{entropy-formula} and Lemma~\ref{derriennic-basic} give that
\[
h_\mu(Z_\eta,\lambda_\eta)=h_\mu(Y,\nu)+{\rm I}(\xi_1,\mathcal{T}|X,\eta)=h_\mu(Y,\nu)+\int_GD(g,\mathcal{T}|X,\eta)d\mu(g).
\]
The limsup on entropy follows from the limsup on KL-divergence together with the domination $D(g,\mathcal{T}|X,\eta) \le D(g\nu_B||\nu_B)$ from Lemma~\ref{derriennic-basicY}.
\end{proof}

Now in the setting of Proposition~\ref{info-conti}, consider a lattice $\Gamma<G$, together with a Furstenberg discretization $\mu_0$. This means that the Poisson boundary of $(\Gamma,\mu_0)$ can be realised by the Poisson boundary $(B,\nu_B)$ of $(G,\mu)$, i.e. that $(B,\nu_B)$ is a $(\Gamma,\mu_0)$-stationary space of which all $\mu_0$-boundaries are $\Gamma$-factors. This implies that $(Y,\nu)$ is a $(\Gamma,\mu_0)$-boundary and that $(X,\eta)$ is $\mu_0$-stationary measure preserving extension of $(Y,\nu)$ viewed as $(\Gamma,\mu_0)$-spaces.

\begin{corollary}\label{cor:mu0usc} Let $(G,\mu)$ be a semi-simple Lie group with $B_\infty$ measure together with  a Furstenberg discretisation $(\Gamma,\mu_0)$ as described above. If
$(\eta_k)_{k\ge 0}$ and $\eta$ are probability measures in the space $\mathcal{M}_\nu(Y',\mathcal{S})$ with $\mathcal{S}$ satisfying the local coincidence property, and if we assume that on a subset of full $\nu'$-measure of $Y'$ we have a weak$^\ast$ convergence $\eta_k^{y'}\to \eta^{y'}$.  Then 
\[
\limsup_{k} h_{\mu_0}(Z_{\eta_k},\lambda_{\eta_k}) \le h_{\mu_0}(Z_\eta,\lambda_\eta).
\]
\end{corollary}

\begin{proof}
The setting of Appendix~\ref{subsec:entropymutual} applies with~(\ref{eq:inter-1}) viewed as $(\Gamma,\mu_0)$-spaces. In particular, we have
\[
h_{\mu_0}(Z_\eta,\lambda_\eta)=h_{\mu_0}(Y,\nu)+\sum_{g \in \Gamma} D(g,\mathcal{T}|X,\eta)\mu_0(g).
\]
We apply Proposition~\ref{info-conti} and use the same domination as in the previous proof. Note that $D(g\nu_B||\nu_B)<\infty$ for all $g$ because the Radon-Nikodym derivative $dg\nu_B/d\nu_B$ is bounded and bounded away from $0$.
\end{proof}

\begin{corollary}\label{cor-mutual} Let $G$ be a locally compact
group whose topology is generated by a left invariant distance, equipped with a probability measure $\mu$ of compact support and finite
boundary entropy. Consider a space of the form $\mathcal{M}_{\nu}(Y',\mathcal{S})$ with $\mathcal{S}$ satisfying the local coincidence property. Also assume that $Y'$ is a manifold with continuous $G$-action and $\nu'$ is a volume form. Now let $\mu_0$ be either $\mu$ or a Furstenberg discretization of $\mu$.
Then the map $\eta\mapsto h_{\mu_0}(Z_{\eta},\lambda_{\eta})$ is upper
semi-continuous on the space $\mathcal{M}_{\nu}(Y',\mathcal{S})$. 
\end{corollary}

\begin{proof} We apply Lemma~\ref{lem:Ymfd} with $(Y',\nu')$ in place of $(Y,\nu)$ (this is legitimate since its proof does not use proximality).
In combination with Proposition~\ref{info-conti} or Corollary~\ref{cor:mu0usc}, it gives the corollary.
\end{proof}

\section{Tools for identification of Poisson bundles\label{sec:entropy-criterion}}

The previous Section~\ref{subsec:Upper-semi-continuity} provides upper semi-continuity of the maps $\eta\mapsto h_{\mu}(Z_\eta,\lambda_\eta)$ in the general context of Poisson bundles over standard system. In Section~\ref{sec:lower-argu}, we will obtain another entropy formula that permits to show lower-continuity. However, this second type of entropy formula is valid only when the Poisson bundle $Z_\eta$ can be identified with a fixed topological space independent of the measure $\eta$. For this purpose, we establish in the present section two identification criteria, which are bundle versions of the classical Strip and Ray criteria established by Kaimanovich for identification of Poisson boundaries~\cite{KaimanovichHyperbolic}.

Throughout this section assume that $G$ is a discrete countable group, endowed with a probability measure $\mu$ of finite entropy. We let $(Z,\lambda)$ be the Poisson bundle over a standard system $(X,\eta)$ as in Section~\ref{sec:Poissonbundle}, i.e. a measure preserving extension $\pi:(X,\eta)\to (Y,\nu)$ of a proximal $(G,\mu)$ space, together with  a $G$-equivariant map $L:X\to \mathrm{Sub}(G)$ satisfying the containment condition (\ref{eq:contain}) that $L_x<\mathrm{Stab}_G(x)$.


Suppose we have a system, denoted by $\left(M,\bar{\lambda}\right)$,
which is a $G$-factor of the Poisson bundle $\left(Z,\lambda\right)$
that fits into the sequence of $G$-factors 
\begin{equation}
\left(X\times B,\eta\varcurlyvee\nu_{B}\right)\to\left(Z,\lambda\right)\to\left(M,\bar{\lambda}\right)\to\left(X,\eta\right),\label{eq:M-sequence}
\end{equation}
where the composition $\left(X\times B,\eta\varcurlyvee\nu_{B}\right)\to\left(X,\eta\right)$
is the coordinate projection $X\times B\to X$. 
Since the Poisson
bundle $\left(Z,\lambda\right)$ is a proximal extension of $\left(X,\eta\right)$,
it is a proximal extension of $\left(M,\bar{\lambda}\right)$ as well,
by~\cite[Prop. 4.1]{FurstenbergGlasner}. By~\cite[Prop. 1.9]{NZ3},
it follows that $(Z,\lambda)$ is $G$-measurable isomorphic to $\left(M,\bar{\lambda}\right)$
if and only if $h_{\mu}(Z,\lambda)=h_{\mu}\left(M,\bar{\lambda}\right)$. 
The identification criteria provide geometric conditions that guarantee this equality. 

The setting of intermediate factor of the form (\ref{eq:M-sequence}) is studied in details in the Appendix~\ref{subsec:entropymutual}. In particular, the measure $\overline{\mathbb{P}}_\mu=\vartheta(\eta\varcurlyvee \mathbb{P}_\mu)$ introduced in Subsection~\ref{subsec:Definition} can be disintegrated as $\overline{\mathbb{P}}_{\mu}=\int_M \overline{\mathbb{P}}_{\mu,(x,\zeta)}d\bar\lambda(x,\zeta)$. Then the measure $\overline{\mathbb{P}}_{\mu,(x,\zeta)}$ describes a Markov chain on $L_x \backslash G$ whose transition kernels $P_{\mu,(x,\zeta)}^n$ are given in Lemma~\ref{lem:MarkovM}. In fact, this Markov chain is obtained by a double Doob transform of the random walk, first according to $\pi(x)$ as in Section~\ref{sec:Poissonbundle} and then according to $\theta_M(x,L_x\omega)=(x,\zeta)$ as explained in the appendix.

When the $G$-action is ergodic, we can apply Shannon's theorem, stated in the appendix as Proposition~\ref{shannon}, to the
extension $(Z,\lambda)\to(M,\bar{\lambda})$. It shows that when $(X,\eta)$ is ergodic for the $G$-action, then the difference between Furstenberg
entropy of $\left(Z,\lambda\right)$ and $\left(M,\bar{\lambda}\right)$
is the $\eta\varcurlyvee\mathbb{P}_{\mu}$-a.s. limit 
\begin{equation}
h_\mu\left(Z,\lambda\right)-h_\mu\left(M,\bar{\lambda}\right)=\lim_{n\to\infty}-\frac{1}{n}\log P_{\mu,(x,\zeta)}^{n}\left(L_{x},L_{x}\omega_{n}\right).\label{eq:entropy difference}
\end{equation}
We immediately deduce the following entropy criterion for identification.

\begin{corollary}\label{doob-zero} Consider an intermediate factor $(M,\bar\lambda)$ as in (\ref{eq:M-sequence}). Assume that $(X,\eta)$ is ergodic. Then the following assertions are equivalent:
\begin{enumerate}
\item[(i)] $(M,\bar\lambda)$ is $G$-isomorphic to the Poisson bundle $(Z,\lambda)$,
\item[(ii)] $h_\mu\left(Z,\lambda\right)=h_\mu\left(M,\bar{\lambda}\right)$,
\item[(iii)] for $\bar{\lambda}$-a.e. $(x,\zeta)$, the Doob transformed
coset Markov chain $\left(L_{x}\omega_{n}\right)$ conditioned on
$\theta_{M}(x,L_{x}\omega)=\left(x,\zeta\right)$ has $0$ asymptotic
entropy, i.e
\[
\lim_{n\to\infty} -\frac{1}{n}\log P_{\mu,\theta_M(x,\omega)}^n(L_x,L_x\omega_n)=0.
\]
\end{enumerate}
\end{corollary}

\subsection{Strip approximation for bundles over IRS\label{subsec:strip}}

The strip approximation criterion, due to Kaimanovich \cite[Section 6]{KaimanovichHyperbolic},
is a powerful tool for identification of Poisson boundary in the presence
of some form of hyperbolicity.

Considers bilateral paths in $G^{\mathbb{Z}}$. Given a step distribution
$\mu$ on $G$, denote by $\check{\mu}$ the reflected measure $\check{\mu}(g)=\mu\left(g^{-1}\right)$.
Take the product space $\left(G^{\mathbb{N}},\mathbb{P}_{\mu}\right)\times\left(G^{\mathbb{N}},\mathbb{P}_{\check{\mu}}\right)$
and the map $\left(x,\check{x}\right)\mapsto\omega\in G^{\mathbb{Z}}$,
where $\omega_{0}=id$, $\omega_{n}=x_{n}$, $\omega_{-n}=\check{x}_{n}$
for $n\in\mathbb{N}$. We write $\tilde{\mathbb{P}}_{\mu}$ for the
pushforward of $\mathbb{P}_{\mu}\times\mathbb{P}_{\check{\mu}}$ under
this map and call $\left(G^{\mathbb{Z}},\tilde{\mathbb{P}}_{\mu}\right)$
a bilateral path space. Denote by $\left(B_{+},\nu_{+}\right)$ the
Poisson boundary ($\left(B_{-},\nu_{-}\right)$ resp.) of the $\mu$-random
walk ($\check{\mu}$-random walk resp.) on $G$ and ${\rm bnd}_{+}:\left(G^{\mathbb{N}},\mathbb{P}_{\mu}\right)\to\left(B_{+},\nu_{+}\right)$
the associated boundary map (${\rm bnd}_{-}$ resp.). Then we have
a map from the bilateral paths to the product of the Poisson boundaries
\begin{align*}
{\rm bnd}_{+}\times{\rm bnd}_{-} & :\left(G^{\mathbb{Z}},\tilde{\mathbb{P}}_{\mu}\right)\to\left(B_{+},\nu_{+}\right)\times\left(B_{-},\nu_{-}\right)\\
 & \omega\mapsto\left({\rm bnd}_{+}\left(\left(\omega_{n}\right)_{n\in\mathbb{N}}\right),{\rm bnd}_{-}\left(\left(\omega_{-n}\right)_{n\in\mathbb{N}}\right)\right).
\end{align*}

Bilateral path space does not fit into the general stationary joining
framework considered in Section~\ref{sec:Poissonbundle}. However
when the measure $\eta$ in the base space $(X,\eta)$ is $G$-invariant,
we may take the product space $\left(X\times G^{\mathbb{Z}},\eta\times\tilde{\mathbb{P}}_{\mu}\right)$.
As in Subsection \ref{subsec:Stationary-joining}, it admits skew
transform 
\[
\tilde{T}(x,\omega)=\left(\omega_{1}^{-1}.x,\left(\omega_{1}^{-1}\omega_{n+1}\right)_{n\in\mathbb{Z}}\right).
\]
In the same way as Fact \ref{pmp}, one can verify that if $\left(X,\eta\right)$
is an ergodic p.m.p. $G$-system, then $\tilde{T}\curvearrowright\left(X\times G^{\mathbb{Z}},\eta\times\widetilde{\mathbb{P}}_{\mu}\right)$
is a p.m.p. ergodic transformation.

In this setting, the following version of strip approximation holds.
Suppose in both positive and negative time directions, we have candidates
for the Poisson bundle that fit into 
\[
\left(X\times B_{\pm},\eta\times\nu_{\pm}\right)\to\left(Z,\lambda_{\pm}\right)\to\left(M_{\pm},\bar{\lambda}_{\pm}\right)\to\left(X,\eta\right)
\]
respectively. Denote by 
\begin{align*}
\theta_{M_\pm}:\quad &(X\times G^{\mathbb{Z}},\eta\times\widetilde{\mathbb{P}}_{\mu})  \to(M_{\pm},\bar{\lambda}_{\pm})\\
&(x,\omega)  \mapsto\left(x,\zeta_{\pm}(x,\omega)\right)
\end{align*}
the maps factorising the above. Further assume that $G$ is equipped
with a distance $d$. Denote by $|g|=d(\mathrm{id}_{G},g)$ and $B_{L\backslash G}(R)$
the ball of radius $R$ centred at $L$ in the coset space $L\backslash G$
with induced distance, still denoted by $|Lg|=d(L,Lg)$. For example, when $G$ is finitely generated,
these can be word distances in the group and Schreier graphs.

\begin{theo}\label{strip} Let $G$ be a countable group with $\mu$
of finite entropy. Suppose $(X,\eta)$ is $G$-invariant and ergodic. Assume that
we have a measurable assignment of strips 
\[
S(x,\omega)=S\left(x,\zeta_{+}(x,\omega),\zeta_{-}(x,\omega)\right)\subseteq L_{x}\backslash G
\]
that satisfies 
\begin{description}
\item [{(i)}] compatibility with rerooting: $L_{x}\omega_{1}\in S(x,\omega)$ if and only
if $\omega_{1}^{-1}L_{x}\omega_{1}\in S\left(\tilde{T}\left(x,\omega\right)\right)$, 
\item [{(ii)}] positive probability of containing the root: $\left(\eta\times\widetilde{\mathbb{P}}_{\mu}\right)\left(L_{x}\in S\left(x,\omega\right)\right)>0,$ 
\item [{(iii)}] subexponential size: for any $\epsilon>0$ and $\eta\times\widetilde{\mathbb{P}}_{\mu}$-a.e.
$(x,\omega)$, 
\[
\limsup\frac{1}{n}\log\left|S(x,\omega)\cap B_{L_{x}\backslash G}\left(|L_{x}\omega_{n}|)\right)\right|\le\epsilon.
\]
\end{description}
Then $\left(M_{+},\bar{\lambda}_{+}\right)$ is $G$-isomorphic to
the Poisson bundle $\left(Z_{+},\lambda_{+}\right)$; and $\left(M_{-},\bar{\lambda}_{-}\right)$
is $G$-isomorphic to the Poisson bundle $\left(Z_{-},\lambda_{-}\right)$.
\end{theo}

\begin{remark}

In Theorem \ref{strip}, we require the compatibility condition (i)
and positive probability of the event that the root is on the strip
(ii) as a replacement for $G$-equivariance of strips in the original
Kaimanovich strip criterion. An illustration can be found in Figure~\ref{fig:strip}.

\end{remark} 
\begin{proof}[Proof of Theorem \ref{strip}]

The proof relies on the Birkhoff ergodic theorem applied to the skew
transformation $\tilde{T}$ and Shannon's Theorem as in Proposition~\ref{shannon}.

Denote by $A$ the subset $\left\{ (x,\omega):L_{x}\in S\left(x,\omega\right)\right\} $
of $X\times G^{\mathbb{Z}}$. By (ii), $A$ has positive probability
under $\eta\times\widetilde{\mathbb{P}}_{\mu}$. Apply the Birkhoff
ergodic theorem to $\tilde{T}\curvearrowright\left(X\times G^{\mathbb{Z}},\eta\times\widetilde{\mathbb{P}}_{\mu}\right)$,
we have that for a.e. $(x,\omega)$, the set of times $n$ such that
$\tilde{T}^{n}(x,\omega)\in A$ has positive limiting frequency. Note
that 
\begin{align*}
\left\{ \tilde{T}^{n}(x,\omega)\in A\right\}  & =\left\{ L_{\omega_{n}^{-1}.x}\in S\left(\omega_{n}^{-1}.x,\left(\omega_{n}^{-1}\omega_{m+n}\right)_{m\in\mathbb{Z}}\right)\right\} \\
 & =\left\{ L_{x}\omega_{n}\in S\left(x,\omega\right)\right\} \ \ \mbox{by compatibility (i)}.
\end{align*}
It follows that for a.e. $(x,\omega)$, the set $\left\{ n\in\mathbb{N}:L_{x}\omega_{n}\in S\left(x,\omega\right)\right\} $
has positive density.

Assume by contradiction that $h\left(Z_{+},\lambda_{+}\right)-h\left(M_{+},\bar{\lambda}_{+}\right)=\delta>0$
and take $\epsilon=\frac{\delta}{3}$. By Corollary~\ref{doob-zero},
for any $p>0$ there is a subset $\tilde{V}\subset X\times G^{\mathbb{Z}}$
with $\left(\eta\times\tilde{\mathbb{P}}_{\mu}\right)(\tilde{V})\ge1-p$
and there is $N\in\mathbb{N}$ such that for $(x,\omega)\in\tilde{V}$
and $n\ge N$ 
\begin{align}
P_{\mu,\theta_{M_+}(x,\omega)}^{n}\left(L_{x},L_{x}\omega_{n}\right)\le e^{-2n\epsilon}.\label{0ent2}
\end{align}
As $(X,\eta)$ is invariant, the joining is a product measure, so we have the disintegration
 over $M_{+}$ that $\eta\times\mathbb{P}_{\mu}=\int_{M_{+}}\left(\eta\times\mathbb{P}_{\mu}\right)_{(x,\zeta_{+})}d\bar{\lambda}_{+}$,
and moreover, fiberwise $P_{\mu,(x,\zeta_+)}^{n}$ is the transition
kernel of the Doob transformed random walk conditioned on $\theta_{M_{+}}(x,\omega)=(x,\zeta_{+})$.
We have then 
\begin{align*}
\eta\times\widetilde{\mathbb{P}}_{\mu} & \left(L_{x}\omega_{n}\in S(x,\omega),(x,\omega)\in\tilde{V}\right)\\
 & =\eta\times\widetilde{\mathbb{P}}_{\mu}(L_{x}\omega_{n}\in S(x,\omega)\cap B_{L_{x}\backslash G}\left(|L_{x}\omega_{n}|\right),(x,\omega)\in\tilde{V})\\
 & =\int_{M_{+}}\left(\eta\times\widetilde{\mathbb{P}}_{\mu}\right)_{(x,\zeta_{+})}\left(L_{x}\omega_{n}\in S(x,\omega)\cap B_{L_{x}\backslash G}\left(|L_{x}\omega_{n}|\right),(x,\omega)\in\tilde{V}\right)d\bar{\lambda}_{+}(x,\zeta_{+})\\
 & \le\int_{M_{+}}\left(\sup_{(x,\omega)\in\tilde{V}}P_{\mu,(x,\zeta_+)}^{n}\left(L_{x},L_{x}\omega_{n}\right)\right)\left|S(x,\omega)\cap B_{L_{x}\backslash G}\left(|L_{x}\omega_{n}|\right)\right|d\bar{\lambda}_{+}(x,\zeta_{+})\\
 & \le e^{-n\epsilon},
\end{align*}
where the last line uses the bound (\ref{0ent2}) and the subexponential
size assumption (iii). By the Borel-Cantelli lemma, and as $p$ is
arbitrary, we deduce that 
\[
\eta\times\widetilde{\mathbb{P}}_{\mu}\left(\left\{ L_{x}\omega_{n}\in S(x,\omega)\mbox{ for infinitely many }n\in\mathbb{N}\right\} \right)=0.
\]
contradicting the positive limiting frequency of times spent on the
strip. The statement for $(M_{-},\lambda_{-})$ follows from applying
the same argument to negative indices. 
\end{proof}
\begin{remark}

In the setting of general locally compact groups, one can not apply
the subadditive ergodic theorem to derive an analogue of Proposition
\ref{shannon}. The recent work of Forghani and Tiozzo \cite{FT}
shows a version of Shannon's theorem for random walks on locally compact
groups; and the techniques there could be adapted to our setting.
We will consider the Poisson bundle identification problem for free
groups, then inducing to ${\rm SL}(d,\mathbb{R})$. For this reason
we do not pursue the direction to formulate results for locally compact
groups in this section.

\end{remark}

\subsection{Ray approximation criteria for Poisson bundles over standard systems}

For future reference, we state a version of the ray approximation
criterion for Poisson bundles over standard systems, which is more
generally applicable than the strip criterion. Such a criterion is
originally due to Kaimanovich \cite{KaimanovichHyperbolic}. The version
stated here is adapted from the enhanced criterion of Lyons and Peres
\cite{LyonsPeres}.

\begin{theo}[Ray approximation \cite{KaimanovichHyperbolic,LyonsPeres}]\label{ray}
Let $G$ be a countable group endowed with $\mu$ of finite entropy.
Let $(M,\bar{\lambda})$ be a $G$-system that fits in (\ref{eq:M-sequence}),
denote by $\theta_{M}:W_{\Omega}\to M$ the factor map. Suppose for
any $\epsilon>0$, there is a subset $U\subseteq M$ with positive
measure $\bar{\lambda}(U)>0$ such that there is a sequence of measurable
maps 
\[
(x,\zeta)\mapsto A_{n}^{\epsilon}(x,\zeta),
\]
where $\left(x,\zeta\right)\in U$ and $A_{n}^{\epsilon}(x,\zeta)\subseteq L_{x}\backslash G$
satisfying that 
\begin{description}
\item [{(i)}] $\limsup_{n\to\infty}\overline{\mathbb{P}}_{\mu}\left(\exists m\ge n:\ L_{x}\omega_{m}\in A_{n}^{\epsilon}\left(\theta_{M}(x,L_{x}\omega)\right)|\theta_{M}(x,L_{x}\omega)\in U\right)>0$, 
\item [{(ii)}] $\limsup_{n\to\infty}\frac{1}{n}\log\left|A_{n}^{\epsilon}(x,\zeta)\right|\le\epsilon$
for all $(x,\zeta)\in U$. 
\end{description}
Then $(M,\bar{\lambda})$ is $G$-measurable isomorphic to the Poisson
bundle $\left(Z,\lambda\right)$ over $\left(X,\eta\right)$. \end{theo}

A proof of Theorem \ref{ray} is provided in Subsection \ref{subsec:shannon-ray}.

\section{Sufficient conditions for lower semi-continuity of Furstenberg entropy \label{sec:lower-argu}}

We now return to the general setting of bundles over stationary systems.  Our goal is to provide sufficient conditions under which the map $\eta\mapsto h_\mu(Z_\eta,\lambda_\eta)$ is lower semi-continuous. We obtain such criteria in the case where the Poisson bundles $(Z_\eta,\lambda_\eta)$ can all be identified with a fixed measurable space $(M,\lambda_\eta)$.

Suppose $X$ and $Y$ are locally compact metrizable spaces and $\pi:X\to Y$
is a Borel $G$-factor map where $G$ is a lcsc group. For the remainder
of this section, let $\left(Y,\nu\right)$ be a
fixed $\mu$-stationary system. As in Section~\ref{subsec:Upper-semi-continuity}, we consider the space $\mathcal{M}_{\nu}$ of probability measures $\eta$ on $X$
such that $\pi:(X,\eta)\to (Y,\nu)$ is a relatively measure-preserving extension. We also consider, given a collection $\mathcal{S}=(S_y)_{y\in Y}$ of compact subspaces $S_y\subset \mathrm{Sub}(G)$, the subset $\mathcal{M}_\nu(Y,\mathcal{S}) \subset \mathcal{M}_\nu$ of probability measures whose disintegration $\eta=\int_Y \eta^yd\nu(y)$ satisfies $\mathrm{supp}(L_\ast \eta_y) \subset S_y$.

In this section we do not need to
assume that $\left(Y,\nu\right)$ is a $\mu$-boundary. In particular, the results apply to factors of the form $(Y',\nu')$ as introduced in the beginning of Subsection~\ref{subsec:Y'}.

Throughout this section, we consider a measurable bundle $p:M\to X$ over $X$ where the fiber
over a point $x\in X$ is denoted by $M_{x}$. 
We assume: 
\begin{description}
\item [{\hypertarget{AssumpM}{(M)}}] Fiberwise measures. The space $M_{x}$ is equipped with a
family of probability measures $\left\{ \alpha_{x,g}\right\} _{g\in G}$, all
in the same measure class. Moreover, for every $g\in G$, the Radon-Nikodym
derivative $d\alpha_{x,g}/d\alpha_{x,e}\in L^{\infty}\left(M_{x},\alpha_{x,e}\right)$,
and any fixed $g$, the map $X\to \mathcal{P}(M)$ given by $x\mapsto \alpha_{x,g}$ is measurable.
\end{description}

Denote by $\boldsymbol{\alpha}=\{\alpha_{x,g}\}_{x\in X, g\in G}$ the collection of fiber measures. We will also consider the stronger assumption

\begin{description}
\item [{\hypertarget{AssumpM}{(Mu)}}] Uniform fiberwise measures. Assumption \hyperlink{AssumpM}{(\textbf{M})} holds and for every $g \in G$ there exists $C_g>0$ such that for all $x \in X$, we have $|d\alpha_{x,g}/d\alpha_{x,e}|\le C_g$.
\end{description}

Such conditions will be satisfied when we consider an end-compactification bundle over certain IRS's of a free group in Subsection \ref{subsec:identification}. A more classical instance to keep in mind is the following.
\begin{example}
Let $G=SO(n,1)$, $X$ be the space of torsion free discrete subgroups of $G$, equipped with the induced Chabauty topology, and $Y$ be the trivial space. Let $\mu$ be an admissible probability measure on $G$. In this case one may consider the bundle $M$ over $X$, where for $\Lambda\in X$, the fiber space $M_{\Lambda}$ is the boundary at infinity of the hyperbolic manifold $\Lambda\backslash G$. For each $g\in G$, we equip $M_{\Lambda}$ with the probability measure $\alpha_{\Lambda,g}$ which is the hitting distribution of the induced random walk on $\Lambda\backslash G$ starting from the coset $\Lambda g$. 
\end{example}


We define
the entropy of the bundle $M$ over $(X,\eta)$ equipped with fiber measures $\boldsymbol{\alpha}=\left\{ \alpha_{x,g}\right\} _{x\in X, g\in G}$  as
\begin{align}
h_{\mu}\left(M,\boldsymbol{\alpha},{\eta}\right) & :=h_{\mu}(X,\eta)+\int_{G}\int_{X}D\left(\alpha_{x,g}\parallel\alpha_{x,e}\right)dg\eta(x)d\mu(g)\label{eq:entropy-alpha}
\end{align}
We refer to Appendix~\ref{sec:entropy-formulae}
for definition and basic properties of the KL-divergence $D(\alpha||\beta)$.
This definition of entropy is consistent: when $(M,\lambda)$ is an extension of $(X,\eta)$, denote the disintegration of 
$g\lambda$ over $g\eta$ as $g\lambda=\int_X(g\lambda)_xdg\eta(x)$. Take $\alpha_{x,g}=(g\lambda)_x$ to be the fiber measure of $g\lambda$ at $x$. Then in this case,
\begin{align}
h_\mu(M,\lambda) &=\int_G\int_M \log \frac{dg\lambda}{d\lambda}(m)dg\lambda(m)d\mu(g) \nonumber \\
&= \int_G\int_X\int_{M_x}\left(\log \frac{d(g\lambda)_x}{d\lambda_x}(\zeta)+\log\frac{dg\eta}{d\eta}(x) \right)d(g\lambda)_x(\zeta)dg\eta(x)d\mu(g) \nonumber \\
&=h_\mu(X,\eta)+\int_G\int_X D\left((g\lambda)_x||\lambda_x \right)dg\eta(x)d\mu(g)=h_{\mu}\left(M,\boldsymbol{\alpha},{\eta}\right). \label{eq:h=h}
\end{align}
Note also that when $(M,\lambda)$ is an intermediate factor as in (\ref{eq:M-sequence}), we have $\lambda=\varrho_M\circ \tilde\psi(\eta\varcurlyvee\mathbb{P}_\mu)$ so that $(g\lambda)_x=\varrho_M\circ \tilde\psi(\delta_{\{x\}}\varcurlyvee g\mathbb{P}_\mu^{\pi(x)})$ is independent of $\eta$ by Lemma~\ref{markov}.

 
Our main tool for lower semi-continuity is the following proposition.

\begin{proposition}\label{Fatou}

Under \hyperlink{AssumpM}{(\textbf{M})}, assume moreover that for $\nu$-a.e. $y\in Y$, $\mu$-a.e. $g\in G$,
the map $x\mapsto D\left(\alpha_{x,g}\parallel\alpha_{x,e}\right)$
is lower semi-continuous on $S_y$. Assume that $(\eta_k)_{k\in \mathbb{N}},\eta$ are probability measures in $\mathcal{M}_\nu(Y,\mathcal{S})$ such that for $\nu$-a.e. $y$ the sequence $\eta_k^y$ converges to $\eta^y$ in weak$^\ast$ topology. Then
\[
\liminf_{k\to \infty} h_\mu(M,\boldsymbol{\alpha},\eta_k) \ge h_\mu(M,\boldsymbol{\alpha},\eta).
\]
In particular, if moreover $Y$ is a manifold with continuous $G$-action and $\nu$ is a volume form, then $\eta\mapsto h_{\mu}\left(M,\boldsymbol{\alpha},{\eta}\right)$
is lower semi-continuous on $\mathcal{M}_{\nu}(Y,\mathcal{S})$.
\end{proposition}

Note that with a slight abuse, we identify $S_y$ with its pre-image under $L:X\to \mathrm{Sub}(G)$.

\begin{proof}
The lower semi-continuity assumption on $D\left(\alpha_{x,g}\parallel\alpha_{x,e}\right)$
implies that it can be written as an increasing limit of non-negative
continuous functions $f_{n}$ on $S_{y}$. 
By monotone convergence theorem, we have $\int_{S_y} D(\alpha_{x,g}||\alpha_{x,e})d\eta^y_k(x)=\sup_n \int_{S_y}f_n(x)d\eta_k^y(x)$ and this also holds for $\eta^y$ in place of $\eta_k^y$. Therefore given arbitrary $\varepsilon>0$, there is some $n_0$ with 
\[
\int_{S_y}f_{n_0}(x)d\eta^y(x)\ge \int_{S_y} D(\alpha_{x,g}||\alpha_{x,e})d\eta^y(x)-\varepsilon.
\]
By weak$^\ast$ convergence, for any $k$ large enough,
\[
\int_{S_y} D(\alpha_{x,g}||\alpha_{x,e})d\eta^y_k(x)\ge \int_{S_y}f_{n_0}(x)d\eta_k^y(x) \ge \int_{S_y}f_{n_0}(x)d\eta^y(x) -\varepsilon \ge \int_{S_y} D(\alpha_{x,g}||\alpha_{x,e})d\eta^y(x)-2\varepsilon.
\]
This proves $\liminf_k \int_{S_y} D(\alpha_{x,g}||\alpha_{x,e})d\eta^y_k(x)\ge \int_{S_y} D(\alpha_{x,g}||\alpha_{x,e})d\eta^y(x)$, and holds for $\nu$-a.e. $y\in Y$. By (\ref{eq:entropy-alpha}) we have
\[
h_\mu(M,\boldsymbol{\alpha},\eta_k)=h_\mu(X,\eta_k)+\int_G\int_Y\int_{S_y} D(\alpha_{x,g}||\alpha_{x,e}) d\eta^y_k(x)dg\nu(y)d\mu(g),
\]
so the main proposition follows from Fatou's Lemma and the fact that $h_\mu(X,\eta_k)=h_\mu(Y,\nu)=h_\mu(X,\eta)$ as the extensions are measure preserving. The last assertion follows by Lemma~\ref{lem:Ymfd}.
\end{proof}

There remains to provide sufficient conditions for the KL-divergence $x\mapsto D\left(\alpha_{x,g}\parallel\alpha_{x,e}\right)$ to be lower semi-continuous on $S_y$. 
Assume that our bundle $p:M\to X$ satisfies the following assumption:
\begin{description}
\item [{\hypertarget{AssumpP}{(P)}}] Generating finite partitions. For each $x\in X$, there is
a refining sequence of finite measurable partitions $\mathcal{P}_{x,n}$
of $M_{x}$, for $n\in\mathbb{N}$, such that the union $\cup_{n\in\mathbb{N}}\mathcal{P}_{x,n}$
generates the $\sigma$-field 
of $M_{x}$. 
\end{description}
Then the KL-divergence of two Borel probability measures $\beta_{x,1}$
and $\beta_{x,2}$ on $M_{x}$ is given by 
\[
D\left(\beta_{x,1}\parallel\beta_{x,2}\right)=\sup_{n}H_{\beta_{x,1}\parallel\beta_{x,2}}(\mathcal{P}_{x,n}),\quad\textrm{where}\quad H_{\beta_{x,1}\parallel\beta_{x,2}}(\mathcal{P}_{x,n})=\sum_{A\in\mathcal{P}_{x,n}}\beta_{x,1}(A)\log\frac{\beta_{x,1}(A)}{\beta_{x,2}(A)}.
\]
See Appendix~\ref{sec:entropy-formulae}. As a supremum of continuous functions is lower semi-continuous, our strategy boils down to showing continuity of the
maps $x\mapsto H_{\beta_{x,1}\parallel\beta_{x,2}}(\mathcal{P}_{x,n})$. 

In the two next subsections, we show such continuity  
under assumptions of approximations. We provide two criteria. The first one makes use of the particular situation where $M_x$ is totally disconnected and there are locally constant uniform approximations. It will be used in Subsection~\ref{subsec:Approximation-F}. The second criterion is more general though not used in the present paper.

\subsection{The case of locally constant uniform fiberwise approximations\label{subsec:fiber-uni}}

In this subsection we consider approximations of measures on $M_{x}$, where $x$ belongs to a given closed subset $S$ of $X$ (again, we identify $S$ with $L(S)\subset\mathrm{Sub}(G)$).
We assume  \hyperlink{AssumpP}{(\textbf{P})}, and let $\left(\beta_{x}\right)_{x\in S}$ be a collection
of probability measures with each $\beta_{x}$ supported on $M_{x}$.
We say that 
\begin{itemize}
\item this collection admits \emph{approximations} with respect to the finite generating partitions $\mathcal{P}_{x,n}$, $x\in S$ if for all $x\in S$
and $n\in\mathbb{N}$, there is a sequence of positive measures $(q_{x,n}^{t})_{t\in \mathbb{N}}$
on $\mathcal{P}_{x,n}$ such that 
\[
\varepsilon_{x,n}(t):=\max_{A\in\mathcal{P}_{x,n}}\left|1-\frac{\beta_{x}\left(A\right)}{q_{x,n}^{t}(A)}\right|\quad\textrm{satisfies}\quad\lim_{t\to\infty}\varepsilon_{x,n}(t)=0.
\]
\item Such approximations are \emph{uniform} on $S$ if in addition, 
\[
\lim_{t\to\infty}\sup_{x\in S}\varepsilon_{x,n}(t)=0.
\]
\end{itemize}
For a fixed integer $n$, we say a collection of measured spaces $\left(\mathcal{P}_{x,n},q_{x}\right)_{x\in S}$,
where $q_{x}$ is a positive measure on the partition $\mathcal{P}_{x,n}$,
is \emph{locally constant} on $S$ if for every $x\in S$, there
is an open neighborhood $O(x)$ of $x$ in $S$ such that for any $x'\in O(x)$,
the spaces $\left(\mathcal{P}_{x,n},q_{x}\right)$ and $\left(\mathcal{P}_{x',n},q_{x'}\right)$ are isomorphic, i.e. there is a bijection $f:\mathcal{P}_{x,n}\to \mathcal{P}_{x',n}$ with $f_\ast q_x=q_{x'}$.

The collection $(\beta_x)_{x\in S}$ is said to admit locally constant uniform approximations if it admits uniform approximations $(q^t_{x,n})$ which are locally constant on $S$ for each $n,t$.

\begin{lemma}\label{KL-uni} 
Let $\left(\beta_{x,1}\right)_{x\in S}$
and $\left(\beta_{x,2}\right)_{x\in S}$ be two collections of fiber
probability measures, where $\beta_{x,i}$ is supported on $M_{x}$.
Suppose each $\left(\beta_{x,i}\right)_{x\in S}$ admits locally constant.
uniform approximations on $S$, $i\in\{1,2\}$; and there is a constant
$C>0$ such that the density satisfies $\left\Vert d\beta_{x,1}/d\beta_{x,2}\right\Vert_{\infty} \le C$
for all $x\in S$. Then the following map is continuous: 
\begin{align*}
S & \to\mathbb{R}_{\ge0}\\
x & \mapsto H_{\beta_{x,1}\parallel\beta_{x,2}}\left(\mathcal{P}_{x,n}\right).
\end{align*}
It follows that $x\mapsto D\left(\beta_{x,1}\parallel\beta_{x,2}\right)$
is lower semi-continuous on $S$.
\end{lemma}

For our applications, the subset $S$ will be totally
disconnected and satisfy the local coincidence property. The locally
constant approximation condition is natural in that context. See Proposition~\ref{KL-lower semi} for a formulation with weaker assumptions.

\begin{proof}[Proof of Lemma~\ref{KL-uni}]

Let $q_{x,n,i}^{t}$ be the approximation measures of $\beta_{x,i}$
on the finite partition $\mathcal{P}_{x,n}$ with the corresponding
error bound $\varepsilon_{x,n,i}(t)$. Note that 
\begin{align*}
\max_{A\in\mathcal{P}_{x,n}}\frac{q_{x,n,2}^{t}(A)}{q_{x,n,1}^{t}(A)} & \le\frac{1}{\left(1-\varepsilon_{x,n,1}(t)\right)\left(1-\varepsilon_{x,n,2}(t)\right)}\max_{A\in\mathcal{P}_{x,n}}\frac{\beta_{x,2}(A)}{\beta_{x,1}(A)}\\
 & \le\frac{C}{\left(1-\varepsilon_{x,n,1}(t)\right)\left(1-\varepsilon_{x,n,2}(t)\right)}=:C_{n,t}.
\end{align*}
Lemma \ref{KL-difference-1} implies that 
\begin{align}
H_{q_{x,n,2}^{t}\parallel q_{x,n,1}^{t}}\left(\mathcal{P}_{x,n}\right)-H_{\beta_{x,2}\parallel\beta_{x,1}}\left(\mathcal{P}_{x,n}\right) & \le2C_{n,t}^{1/2}\max_{A\in\mathcal{P}_{x,n}}\left|1-\frac{\beta_{x,2}(A)}{q_{x,n,2}^{t}(A)}\right|+\log\left(\max_{A\in\mathcal{P}_{x,n}}\frac{\beta_{x,1}(A)}{q_{x,n,1}^{t}(A)}\right)\nonumber \\
 & \le2C_{n,t}^{1/2}\varepsilon_{x,n,2}(t)+\varepsilon_{x,n,1}(t);\label{eq:error1}
\end{align}
and 
\begin{align}
H_{\beta_{x,2}\parallel\beta_{x,1}}\left(\mathcal{P}_{x,n}\right)-H_{q_{x,n,2}^{t}\parallel q_{x,n,1}^{t}}\left(\mathcal{P}_{x,n}\right) & \le2C^{1/2}\max_{A\in\mathcal{P}_{x,n}}\left|1-\frac{q_{x,n,2}^{t}(A)}{\beta_{x,2}(A)}\right|+\log\left(\max_{A\in\mathcal{P}_{x,n}}\frac{q_{x,n,1}^{t}(A)}{\beta_{x,1}(A)}\right)\nonumber \\
 & \le2C^{1/2}\frac{\varepsilon_{x,n,2}(t)}{1-\varepsilon_{x,n,2}(t)}+\frac{\varepsilon_{x,n,1}(t)}{1-\varepsilon_{x,n,1}(t)}.\label{eq:error2}
\end{align}
Write $\epsilon_{n}(t)=\sup_{x\in X'}(\varepsilon_{x,n,2}(t)+\varepsilon_{x,n,1}(t))$,
then the uniform assumption states that $\epsilon_{n}(t)\overset{t\to\infty}{\to}0$.
Therefore (\ref{eq:error1}) and (\ref{eq:error2}) show that the
sequence of continuous (actually locally constant) functions $x\mapsto H_{q_{x,n,2}^{t}\parallel q_{x,n,1}^{t}}\left(\mathcal{P}_{x,n}\right)$
converges uniformly to the function $x\mapsto H_{\beta_{x,2}\parallel\beta_{x,1}}\left(\mathcal{P}_{x,n}\right)$
as $t\to\infty$. Thus by the uniform convergence theorem, the limit
function is continuous as well. 

As observed above, the partitions $\mathcal{P}_{x,n}$
generate the $\sigma$-field of $M_{x}$ by assumption  \hyperlink{AssumpP}{(\textbf{P})}.
We have that $\sup_nH_{\beta_{x,2}\parallel\beta_{x,1}}\left(\mathcal{P}_{x,n}\right)=D\left(\beta_{x,1}\parallel\beta_{x,2}\right)$. It follows that $x\mapsto D\left(\beta_{x,1}\parallel\beta_{x,2}\right)$
is lower semi-continuous on $S$. 
\end{proof}

\begin{proposition}\label{cor-lower-uni} 
Assume  \hyperlink{AssumpM}{(\textbf{Mu})} 
and  \hyperlink{AssumpP}{(\textbf{P})}. Let $(\eta_k)_{k\in \mathbb{N}},\eta$ be probability measures in $\mathcal{M}_\nu(Y,\mathcal{S})$. Suppose that for $\nu$-a.e. $y\in Y$, the family of measures
$\left(\alpha_{x,g}\right)_{x\in S_{y}}$ admits locally constant
uniform approximations on $S_{y}$ for each $g$ in $G$, and that we have a weak$^\ast$ convergence $\eta_k^y\to \eta^y$. Then 
\[
\liminf_{k\to \infty} h_\mu(M,\boldsymbol{\alpha},\eta_k) \ge h_\mu(M,\boldsymbol{\alpha},\eta).
\]
\end{proposition}
Again, this can be combine with Lemma~\ref{lem:Ymfd} to get a lower semi-continuity statement.

\begin{proof}

This follows from Proposition~\ref{Fatou} and Lemma~\ref{KL-uni}
for $\beta_{x,1}=\alpha_{x,g}$ and $\beta_{x,2}=\alpha_{x,e}$. 
\end{proof}

\subsection{A more general criterion with integral bounds}

For completeness, we record in this subsection a relaxed version of
Proposition~\ref{cor-lower-uni}, which ensures lower semi-continuity
of $h_{\mu}\left(M,\boldsymbol{\alpha},\eta\right)$ on $\mathcal{M}_\nu(Y,\mathcal{S})$.

Under  \hyperlink{AssumpP}{(\textbf{P})} , we assume that for each $g\in G$ and $n\in\mathbb{N}$,
there is a probability $q_{x,g}^{n}$ defined on the partition $\mathcal{P}_{x,n}$
that approximates $\alpha_{x,g}$ in the sense that 
\[
\varepsilon_{x}(g,n):=\max_{A\in\mathcal{P}_{x,n}}\left|1-\frac{\alpha_{x,g}(A)}{q_{x,g}^{n}(A)}\right|\quad\textrm{ satisfies }\quad\lim_{n\to\infty}\varepsilon_{x}(g,n)=0.
\]
 Under \hyperlink{AssumpM}{(\textbf{M})}, denote $C_{x,g}$ an upper bound for the $L^{\infty}$-norm
of the Radon-Nikodym derivative $d\alpha_{x,g}/d\alpha_{x,e}$. Similar
to the bound in (\ref{eq:error1}), define 
\begin{equation}
\Delta_{x,g}(n):=2\left(\frac{C_{x,g}}{\left(1-\varepsilon_{x}(e,n)\right)\left(1-\varepsilon_{x}(g,n)\right)}\right)^{1/2}\varepsilon_{x}(g,n)+\varepsilon_{x}(e,n).\label{eq:Delta-1}
\end{equation}

\begin{proposition}\label{KL-lower semi} 
Under
\hyperlink{AssumpM}{(\textbf{M})}  and  \hyperlink{AssumpP}{(\textbf{P})}, suppose in addition that for each $y\in Y$, 
\begin{itemize}
\item[-] for all $n\in\mathbb{N}$, the map $x\mapsto H_{q_{x,e}^{n}\parallel q_{x,g}^{n}}\left(\mathcal{P}_{x,n}\right)$
is continuous on $S_{y}$ ,
\item[-] for $\mu\times\eta$-a.e. $(g,x)$, the error terms $\Delta_{x,g}(n)$
defined in (\ref{eq:Delta-1}) is dominated by some function $\psi_{g}(x)$
which satisfies $\int_X \psi_g(x) d\eta(x)<\infty$ for all $\eta \in \mathcal{M}_\nu(Y,\mathcal{S})$,
\item[-] there is a sequence $\left(\delta_{n}(g,y)\right)_{n\in\mathbb{N}}$
that converges to $0$ and 
\[
\int_{S_{y}}\Delta_{x,g}(n)d\eta_{}^{y}(x)\le\delta_{n}(g,y)\mbox{ for all }\eta\in \mathcal{M}_\nu(Y,\mathcal{S}).
\]
\end{itemize}
Then for any $(\eta_k)_{k\in \mathbb{N}},\eta$ probability measures in $\mathcal{M}_\nu(Y,\mathcal{S})$ such that for $\nu$-a.e. $y$ the sequence $\eta_k^y$ converges to $\eta^y$ in weak$^\ast$ topology, we have
\[
\liminf_{k\to \infty} h_\mu(M,\boldsymbol{\alpha},\eta_k) \ge h_\mu(M,\boldsymbol{\alpha},\eta).
\]
\end{proposition} 

\begin{proof}
As in Proposition~\ref{Fatou}, it suffices to show that 
\[
\liminf_k \int_{S_y} D(\alpha_{x,g}||\alpha_{x,e})d\eta^y_k(x)\ge \int_{S_y} D(\alpha_{x,g}||\alpha_{x,e})d\eta^y(x).
\]
As in the proof of Lemma~\ref{KL-uni},
Lemma \ref{KL-difference-1} implies 
\begin{equation}
D\left(\alpha_{x,e}\parallel\alpha_{x,g}\right)=\sup_{n\in\mathbb{N}}\left\{ H_{q_{x,e}^{n}\parallel q_{x,g}^{n}}\left(\mathcal{P}_{x,n}\right)-\Delta_{x,g}(n)\right\} .\label{eq:KL-sup}
\end{equation}
Write $\delta_{n}=\delta_{n}(g,y)$, we have 
\begin{align*}
\int_{S_{y}}D\left(\alpha_{x,g}\parallel\alpha_{x,e}\right)d\eta^{y} & (x)=\int_{S_{y}}\sup_{n}\left\{ H_{q_{x,g}^{n}\parallel q_{x,e}^{n}}\left(\mathcal{P}_{x,n}\right)-\Delta_{x,g}(n)\right\} d\eta^{y}(x)\\
 & \ge\sup_{n}\left\{ \int_{S_{y}}H_{q_{x,g}^{n}\parallel q_{x,e}^{n}}\left(\mathcal{P}_{x,n}\right)d\eta_{p}^{y}-\int_{S_{y}}\Delta_{x,g}(n)d\eta^{y}(x)\right\} \\
 & \ge\sup_{n}\left\{ \int_{S_{y}}H_{q_{x,g}^{n}\parallel q_{x,e}^{n}}\left(\mathcal{P}_{x,n}\right)d\eta^{y}-\delta_{n}(g,y)\right\} .
\end{align*}
In the other direction, (\ref{eq:KL-sup}) and the dominated convergence
theorem imply that 
\[
\int_{S_{y}}D\left(\alpha_{x,g}\parallel\alpha_{x,e}\right)d\eta^{y}(x)=\lim_{n}\int_{S_{y}}H_{q_{x,g}^{n}\parallel q_{x,e}^{n}}\left(\mathcal{P}_{x,n}\right)d\eta^{y}(x).
\]
Since $\delta_{n}(g,y)\to0$ as $n\to\infty$, we have then 
\begin{equation}
\int_{S_{y}}D\left(\alpha_{x,g}\parallel\alpha_{x,e}\right)d\eta^{y}(x)=\sup_{n}\left\{ \int_{S_{y}}H_{q_{x,g}^{n}\parallel q_{x,e}^{n}}\left(\mathcal{P}_{x,n}\right)d\eta^{y}-\delta_{n}(g,y)\right\} .\label{eq:sup-delta}
\end{equation}
The continuity assumption and the weak$^\ast$ convergence provide 
\[
\int_{S_{y}}H_{q_{x,e}^{n}\parallel q_{x,g}^{n}}\left(\mathcal{P}_{x,n}\right)d\eta^{y}-\delta_{n}\underset{k\to \infty}{\longrightarrow} \int_{S_{y}}H_{q_{x,e}^{n}\parallel q_{x,g}^{n}}\left(\mathcal{P}_{x,n}\right)d\eta^{y}-\delta_{n}.
\]
Then (\ref{eq:sup-delta}) implies
the statement. 
\end{proof}

\vspace{0.5cm}

\begin{center}
\textsc{ Part III : Entropy realisations}
\end{center}

\section{Bowen-Poisson bundle for free groups\label{Sec:identification-free}}

In this section, we apply the tools of the previous sections to the
Poisson bundles over IRSs of the free group considered in \cite{Bowen}.

Let $F$ be the free group $\mathbf{F}_{k}$ on $k\ge2$ generators.
Denote $S$ a free generating set.
The Schreier graph of a subgroup $H$ of $F$ has for vertex set the space of cosets $H\backslash F$,  for edge set $\left\{ (Hg,Hgs):g\in F,s\in S\right\} $ and for root the vertex $H$. For any $g \in F$, the Schreier graph of the conjugate subgroup $gHg^{-1}$ can be identified with the graph with same vertex and edge sets but rooted at $Hg$.

Following a terminology of Bowen, we call a subgroup $H$ of $F$,
or rather its Schreier graph $H\backslash F$, \emph{tree-like} if
the only simple loops are self-loops, i.e. have length $1$ -- see
\cite[Section 4]{Bowen}. This precisely means that between any two
vertices of the Schreier graph, there is a unique path without backtrack
nor self-loop from one to the other. Algebraically, a subgroup $H$
is tree-like if and only if it is generated by elements of the form
$gsg^{-1}$ for $g$ in $F$ and $s$ in the generating set (and we
can assume there is a bijection between such pairs $(g,s)$ and the
loops of the Schreier graph). We denote by $\partial\left(H\backslash F\right)$
the space of ends of $H\backslash F$. Denote by ${\rm Tree}_{F}$
the subset of ${\rm Sub}(F)$ which consists of all tree-like
subgroups. It is a conjugacy invariant closed subset of ${\rm Sub}(F)$.

Given a subgroup $H_0 \in {\rm Tree}(F)$, we also consider
\begin{equation}
{\rm Tree}_{F}^{H_{0}}=\left\{ H\in{\rm Tree}_{F}:H<H_{0}^{\gamma}\mbox{ for some }\gamma\in F\right\}.\label{eq:tree-cover}
\end{equation}
This is the set of subgroups $H$ such that, up to rerooting, the Schreier graph
$H\backslash F$ is a tree-like graph that covers $H_{0}\backslash F$.
The main goal of this section is to prove the following proposition, from which we will derive Theorem~\ref{thm-free}.

\begin{proposition}\label{Bowen-conti}
Let $F$ be a finitely generated free group with a measure $\mu$ of finite entropy and finite $\log$-moment. Let $H_0 \in {\rm Tree}(F)$ be a subgroup whose normalizer $N_F(H_0)$ has finite index in $F$ and whose Schreier graph has infinitely many ends. Let $\rho$ be an ergodic IRS of $F$ supported on ${\rm Tree}_{F}^{H_{0}}$. Then the map $\rho \mapsto h_\mu(Z_\rho,\lambda_\rho)$ is continuous.
\end{proposition}

Proposition~\ref{Bowen-conti} fits into the setting of Part II, with  $(X,\eta)=({\rm Tree}(F),\rho)$ and $L:X\to {\rm Sub}(F)$ the identity map. The $\mu$-boundary $(Y,\nu)=(\{\ast\},\delta_{\{\ast\}})$ is a Dirac mass, therefore $\rho \varcurlyvee \mathbb{P}_\mu=\rho \times \mathbb{P}_\mu$ so $\mathbb{P}_\mu^{\pi(x)}=\mathbb{P}_\mu$ for all $x\in X$.  In particular, $(Z_\rho,\lambda_\rho)$ is the Poisson bundle over the IRS $\rho$, as considered in Section~\ref{sec:Poissonbundle} and in~\cite{Bowen}.  We will denote by $H$ rather than $x$ or $L_x$ the elements of ${\rm Tree}(F)$. The space $\mathcal{M}_\nu(Y,S)=\mathcal{M}_{\delta_{\{\ast\}}}(\{\ast\},{\rm Tree}_F^{H_0})$ from Sections~\ref{subsec:Upper-semi-continuity} and~\ref{sec:lower-argu} is simply the space of all IRS's supported on ${\rm Tree}_F^{H_0}$.

The hypothesis that the normalizer of $H_0$ has finite index is unnecessary and can be relaxed to condition (\ref{eq:transB}) described below in Lemma~\ref{crossing}, as appears from the discussions below. In fact, this could be relaxed  to a  uniform transience condition on $H\backslash F$ for $H$ in the support of $\rho$. We do not investigate such generality both for simplicity and by lack of relevant examples.

In Subsection~\ref{subsec:quasi-transitive}, we investigate the induced random walk on the Schreier graph $H_0\backslash F$ and record transience properties. In Subsection~\ref{subsec:identification}, we define the end-compactification bundle and show that it identifies with the Poisson bundle $(Z_\rho,\lambda_\rho)$ over the IRS $\rho$, using the strip criterion of Section~\ref{sec:entropy-criterion}. In Subsection~\ref{subsec:Approximation-F}, we use Proposition~\ref{cor-lower-uni} to obtain lower semi-continuity. Finally we prove Proposition~\ref{Bowen-conti} and derive Theorem~\ref{thm-free} in Subsection~\ref{subsec:Entropy-realization-F}.

\subsection{Quasi-transitive tree-like Schreier graphs\label{subsec:quasi-transitive}}

We consider a subgroup $H_{0}\in{\rm Tree}_{F}$ whose
normalizer $N_{F}(H_{0})$ of finite index in $F$. The normalizer
$N_{F}(H_{0})$ acts from the left on the Schreier graph of $H_{0}\backslash F$
by automorphisms. This action extends to a continuous action on the space of
ends $\partial(H_{0}\backslash F)$. We also assume that the tree-like
Schreier graph $H_{0}\backslash F$ has infinitely many ends.

In this setting, we denote by $\rho_0$ the uniform measure over the finitely many conjugates of $H_0$. It is obviously an ergodic IRS supported in ${\rm Tree}_F^{H_0}$.

\begin{example}
\label{exKell} Following \cite[Section 4.3]{Bowen}, given an integer
$\ell\ge2$, take the subgroup $K_{\ell}$ of $\mathbf{F}_{2}=\langle a,b\rangle$
which is generated by all elements of the form $ghg^{-1}$, where
$g\in\left\langle a^{\ell},b^{\ell}\right\rangle $ and $h\in\left\{ a^{k}ba^{-k},b^{k}ab^{-k}:k=1,2,\ldots,\ell-1\right\} $.
The normalizer $N_{F}(K_{\ell})$ is of finite index in $F$, and
the coset Schreier graph $K_{\ell}\backslash F$ is tree-like. The Schreier graph can be drawn as follows : take a $4$ regular tree with $\{a,b\}$ labels as the Cayley graph of $\mathbb{F}_2$, subdivide each $a$-edge (resp. $b$ edge) into $\ell$ subintervals, attach a $b$-loop (resp. $a$-loop) to each new vertex.
\end{example}

It is known by \cite{CartwrightSoardi} that for a locally finite
infinite tree $\mathsf{T}$, for any random walk step distribution
$\mu$ on ${\rm Aut}(\mathsf{T})$ such that ${\rm supp}\mu$ is not
contained in an amenable subgroup of ${\rm Aut}(\mathsf{T})$, the
sequence $\omega_{n}.v$, where $\left(\omega_{n}\right)$ is a $\mu$-random
walk, converges to an end with probability $1$ for any vertex $v$. This convergence
result uses a martingale argument originally due to Furstenberg. Along
this line of reasoning, we have:

\begin{lemma}\label{quasitransitive-convergence} Let $\mu$ be a
non-degenerate step distribution on $F$ and $\left(\omega_{n}\right)_{n=0}^{\infty}$
be a $\mu$-random walk. Suppose $H_{0}\in{\rm Tree}_{F}$ has a normalizer
$N_{F}(H_{0})$ of finite index in $F$ and that the Schreier graph
$H_{0}\backslash F$ has infinitely many ends. Then the coset random
walk $\left(H_{0}\omega_{n}\right)_{n=0}^{\infty}$ converges to an
end in $\partial\left(H_{0}\backslash F\right)$ with probability
$1$. \end{lemma}

\begin{proof}
Let $\tau_{n}$ be the $n$-th return time of the random walk $\left(\omega_{n}\right)_{n=0}^{\infty}$
to the finite index subgroup $N=N_{F}(H_{0})$. Denote by $\mu_{\tau}$
the distribution of $\omega_{\tau_{1}}$. Note that ${\rm supp}\mu_{\tau}$
generates $N$. Denote by $o$ the identity coset in $H_{0}\backslash F$,
we have that $\omega_{\tau_{n}}.o=H_{0}\omega_{\tau_{n}}$. Apply
the convergence theorem \cite{CartwrightSoardi} to the $\mu_{\tau}$-random
walk $\left(\omega_{\tau_{n}}\right)_{n=0}^{\infty}$ on $N<{\rm Aut}(H_{0}\backslash F)$,
we have that with probability $1$, $\omega_{\tau_{n}}.o$ converges
to an end. On this full measure set of $\omega\in F^{\mathbb{N}}$,
denote by $\lambda_{\omega}^{(\tau)}$ the end where $H_{0}\omega_{\tau_{n}}$
converges to.

Take an infinite reduced word $\xi=x_{1}x_{2}\ldots\in\partial F$
such that for any $\gamma\in F$, in the Schreier graph $H_{0}\backslash F$,
the sequence $\left(H_{0}\gamma x_{1}\ldots x_{n}\right)$ converges
to an end in $\partial\left(H_{0}\backslash F\right)$. Denote the
end as $H_{0}\gamma\xi$. Such infinite words exist: since simple
random walk on $H_{0}\backslash F$ is transient and converges to
an end starting from any vertex, we have that $\nu_{0}$-a.e. $\xi\in\partial F$
has the property required, where $\nu_{0}$ is the harmonic measure
on $\partial F$ of simple random walk on $F$. Here transience follows
from the assumption that the quasi-transitive graph $H_{0}\backslash F$
is not quasi-isometric to $\mathbb{Z}$. Fix a choice of such $\xi$.
Let $\nu_{\gamma}(n)$ be the distribution of $H_{0}\gamma\omega_{n}\xi$
on $\partial\left(H_{0}\backslash F\right)$. By compactness and a
standard diagonal argument, there is a subsequence $\left(n_{i}\right)$
such that $\nu_{\gamma}(n_{i})$ converges in the weak$^{\ast}$ topology
for all $\gamma\in F$. Denote by $\nu_{\gamma}$ the limit of $\nu_{\gamma}(n_{i})$.
The limits satisfy the harmonicity condition $\sum_{s\in F}\nu_{\gamma s}\mu(s)=\nu_{\gamma}$,
$\gamma\in F$. By the martingale convergence theorem, along the $\mu$-random
walk trajectory, $\nu_{\omega_{n}}$ converges to a limit measure
$\nu_{\omega}$ in the weak$^{\ast}$ topology for $\mathbb{P}_{\mu}$-a.e.
$\omega$.

Next we show that for $\mathbb{P}_{\mu}$-a.e. $\omega$, the limit
$\nu_{\omega}$ is the point mass at the end $\lambda_{\omega}^{(\tau)}$.
It suffices to show the subsequence $\left(\nu_{\omega_{\tau_{n}}}\right)$
converges weakly to $\delta_{\lambda_{\omega}^{(\tau)}}$. Since $N$
normalizes $H_{0}$, we have for $g\in N$, 
\begin{equation}
\nu_{g}=\lim_{i\to\infty}\sum_{s\in F}\delta_{\left\{ H_{0}gs\xi\right\} }\mu^{(n_{i})}(s)=\lim_{i\to\infty}\sum_{s\in F}\delta_{\left\{ gH_{0}s\xi\right\} }\mu^{(n_{i})}(s)=g.\nu_{e}.\label{eq:automorphism}
\end{equation}
The $\mu$-harmonicity condition satisfied by $\left\{ \nu_{g}\right\} $
then implies $\nu_{e}$ is $\mu_{\tau}$ stationary. Then the measures
$\nu_{g}$, $g\in N$, are non-atomic, see e.g., the first paragraph
in the proof of \cite[Theorem]{CartwrightSoardi}. Recall that for
$\mathbb{P}_{\mu}$-a.e. $\omega$, along the subsequence $\left(\omega_{\tau_{n}}\right)$,
$\omega_{\tau_{n}}.o$ converges to the end $\lambda_{\omega}^{(\tau)}$.
By \cite[Lemma 2.2]{CartwrightSoardi}, $\omega_{\tau_{n}}.x$ converges
to $\lambda_{\omega}^{(\tau)}$ for all $x\in H_{0}\backslash F\cup\partial\left(H_{0}\backslash F\right)$
except possibly one point. Therefore $\omega_{\tau_{n}}.\nu_{e}$
converges to the point mass at $\lambda_{\omega}^{(\tau)}$. We conclude
that $\nu_{\omega}=\delta_{\lambda_{\omega}^{(\tau)}}$.

Finally, suppose that the sequence $\left(H_{0}\omega_{n}\right)_{n=0}^{\infty}$
has other accumulation points than $\lambda_{\omega}^{(\tau)}$: there
is a subsequence $\left(H_{0}\omega_{m_{j}}\right)$ that converges
to a different end $\lambda'$. Since $N$ is finite index in $F$,
passing to a further subsequence if necessary, we may assume that
there is a $\gamma\in F$ such that $\omega_{m_{j}}\in N\gamma$ for
all $j$. Then $\omega_{m_{j}}\gamma^{-1}.o=H_{0}\omega_{m_{j}}\gamma^{-1}$
converges to the end $\lambda'$ as well. The same calculation as
in (\ref{eq:automorphism}) shows that $\nu_{\omega_{m_{j}}}=\left(\omega_{m_{j}}\gamma^{-1}\right).\nu_{\gamma}$.
Again by \cite[Lemma 2.2]{CartwrightSoardi}, $\left(\nu_{\omega_{m_{j}}}\right)$
converges weakly o $\delta_{\lambda'}$. However we have shown that
$\nu_{\omega_{n}}$ converges weakly to $\delta_{\lambda_{\omega}^{(\tau)}}$.
Therefore $\lambda'=\lambda_{\omega}^{(\tau)}$. We conclude that
there is no other accumulation points and $(H_{0}\omega_{n})_{n=0}^{\infty}$
converges to $\lambda_{\omega}^{(\tau)}$ for $\mathbb{P}_{\mu}$-a.e.
$\omega$. 
\end{proof}

Under the finite entropy and finite log-moment assumption on $\mu$,
we note the following strengthening of the convergence statement in
Lemma \ref{quasitransitive-convergence}. This property will be useful
in the lifting argument in the next subsection. Given a point $x\in H_{0}\backslash F$
and an element $g\in F$ (viewed as a reduced word), for $0\le\ell\le|g|$,
denote by $g_{\ell}$ the length $\ell$ prefix of $g$, and $\left[x;g\right]$
the set $\left\{ x,xg_{1},xg_{2},\ldots,xg\right\} $.

\begin{lemma}\label{crossing} Let $F$ be a finitely generated free group with a measure $\mu$ of finite entropy and finite $\log$-moment. Let $H_0 \in {\rm Tree}(F)$ be a subgroup whose normalizer $N_F(H_0)$ has finite index in $F$ and whose Schreier graph has infinitely many ends.Then
we have for any finite set $B$ in $H_{0}\backslash F$, 
\begin{align}\label{eq:transB}
\mathbb{P}_{\mu}\left(B\cap\left[H_{0}\omega_{n};\omega_{n}^{-1}\omega_{n+1}\right]\neq\emptyset\mbox{ infinitely often}\right)=0.
\end{align}
\end{lemma}
\begin{proof}
It suffices to prove it for the case where $B$ consists of a single
point, $B=\{x_{0}\}$. Let $h=h_\mu(Z_{\rho_0},\lambda_{\rho_0})$ be the entropy of the Poisson bundle
over the finitely supported IRS $\rho_{0}$. By the convergence lemma \ref{quasitransitive-convergence},
we know that the tail $\sigma$-field of $\left(H_{0}\omega_{n}\right)_{n=0}^{\infty}$
is nontrivial, and thus the Furstenberg entropy of the Poisson bundle
is positive: $h>0$. Take any $0<\epsilon<h/3$.
Consider the subset of vertices 
\[
A_{n}=\left\{ x\in H_{0}\backslash F:-\frac{1}{n}\log P_{\mu,H_{0}}^{n}(o,x)\ge h-\epsilon\right\} ,
\]
and the event 
\[
C_{n}=\left\{ \omega:\log\left|\omega_{n}^{-1}\omega_{n+1}\right|\le n\epsilon,H_{0}\omega_{n}\in A_{n},x_{0}\in\left[H_{0}\omega_{n};\omega_{n}^{-1}\omega_{n+1}\right]\right\} .
\]
Given an element $g\in F$, for $\ell\le|g|$, denote by $g_{\ell}$
the length $\ell$ prefix of $g$. Then we have: 
\begin{align*}
\mathbb{P}_{\mu}\left(C_{n}\right) & \le\sum_{g:|g|\le e^{n\epsilon}}\mu(g)\sum_{\ell=0}^{|g|}\mathbb{P}_{\mu}\left(H_{0}\omega_{n}\in A_{n},H_{0}\omega_{n}g_{\ell}=x_{0}\right)\\
 & =\sum_{g:|g|\le e^{n\epsilon}}\mu(g)\sum_{\ell=0}^{|g|}\mathbb{P}_{\mu}\left(H_{0}\omega_{n}=x_{0}g_{\ell}^{-1}\right)\mathbf{1}_{A_{n}}\left(x_{0}g_{\ell}^{-1}\right)\\
 & \le\sum_{g:|g|\le e^{n\epsilon}}\mu(g)\sum_{\ell=0}^{|g|}e^{-n(h-\epsilon)}\quad \mbox{ (by definition of the set }A_{n})\\
 & \le ce^{n\epsilon}e^{-n(h-\epsilon)}\le ce^{-nh/3}.
\end{align*}
By the Borel-Cantelli lemma, we have that $\mathbb{P}_{\mu}\left(\omega\in C_{n}\mbox{ i.o.}\right)=0$.
By the Shannon theorem \ref{shannon}, $\mathbb{P}_{\mu}\left(H_{0}\omega_{n}\notin A_{n}\mbox{ i.o.}\right)=0$,
and recall that finite log-moment implies that $\mathbb{P}_{\mu}\left(\log\left|\omega_{n}^{-1}\omega_{n+1}\right|>n\epsilon\mbox{ i.o.}\right)=0$.
The statement follows from taking a union of these three events. 
\end{proof}

\subsection{Identification of the Poisson bundle with the end-compactification bundle\label{subsec:identification}}

We define the end-compactification bundle and, under transience hypothesis, we identify it with the
Poisson bundle.

\subsubsection{The bundle of end-compactifications}\label{subsec:endbundle}

For $H\in{\rm Tree}_{F}$, we denote by $\partial\left(H\backslash F\right)$ the space of ends of the tree-like Schreier graph of $H$. We denote by $M_F$ the associated end compactification bundle over ${\rm Tree}_{F}$. As a set, we have
\[
M_F=\bigsqcup_{H \in{\rm Tree}_{F}} \left\{H\right\} \times \partial\left(H\backslash F\right).
\]
The group $F$ acts on $M_F$ as follows. For $(H,\zeta)\in M_F$, let
$\xi\in\partial F$ be such that the sequence $H\xi_{n}$ converges
to $\zeta$ on $H\backslash F$, where $\xi_{n}$ is the length $n$
prefix of $\xi$. Then $F$ acts on $M_F$ by $\gamma.\left(H,\zeta\right)=\left(H^{\gamma},\zeta'\right)$,
where $\zeta'$ is the end in $\partial\left(H^{\gamma}\backslash F\right)$
that $\left(H^{\gamma}\gamma\xi_{n}\right)_{n=1}^{\infty}$ converges
to. 
\begin{fact}
\label{F-end}

The $F$-action on $M_F$ described above is well-defined. 
\end{fact}

\begin{proof}
Suppose $\xi,\xi'$ are two infinite reduced words such that $H\xi_{n}$
and $H\xi_{n}'$ converge to the same end $\zeta\in\partial\left(H\backslash F\right)$.
Then on the tree-like Schreier graph $H\backslash F$, the Gromov
product $\left(H\xi_{n}|H\xi_{n}'\right)_{H}\overset{n\to\infty}{\longrightarrow}\infty$.
On the graph $H^{\gamma}\backslash F$, which is related to $H\backslash F$
by rerooting, we have 
\[
\left(H^{\gamma}\gamma\xi_{n}|H^{\gamma}\gamma\xi_{n}'\right)_{H^{\gamma}}\ge\left(H\xi_{n}|H\xi_{n}'\right)_{H}-|\gamma|_{F}.
\]
Thus $H^{\gamma}\gamma\xi_{n}$ and $H^{\gamma}\gamma\xi'_{n}$ converges
to the same end in $\partial\left(H^{\gamma}\backslash F\right)$. 
\end{proof}



Now consider $H$ in ${\rm Tree}_F^{H_0}$ as defined in (\ref{eq:tree-cover}). Notice that for tree-like subgroups of $F$, the property $H<H_0^\gamma$ is equivalent to the fact that the Schreier graph $H\backslash F$ is a (rooted) cover of the Schreier graph $H_0^\gamma \backslash F$.
As the property stated in (\ref{eq:transB}) in Lemma \ref{crossing} naturally lifts to covering
graphs, we get the following convergence to ends lemma.

\begin{lemma}\label{lift} Let $F$ be a finitely generated free group with a measure $\mu$ of finite entropy and finite $\log$-moment. Let $H_0 \in {\rm Tree}(F)$ be a subgroup satisfying (\ref{eq:transB}) in Lemma \ref{crossing}.
For any $H\in{\rm Tree}_{F}^{H_{0}}$ and any finite set $K$ in $H\backslash F$,
$\mathbb{P}_{\mu}$-almost surely $K\cap\left[H\omega_{n};\omega_{n}^{-1}\omega_{n+1}\right]\neq\emptyset$
for only finitely many $n$. In particular, $H\omega_{n}$ converges
to an end in $\partial\left(H\backslash F\right)$ when $n\to\infty$.
\end{lemma}

Denote by $\nu$
the hitting distribution on $\partial F$ of the $\mu$-random walk.
Lemma \ref{lift} implies that there is a measurable $F$-invariant
$\nu$-conull subset $A\subseteq\partial F$ such that the map $\zeta_{H}:A\to\partial\left(H\backslash F\right)$
is defined for all $H\in{\rm Tree}_{F}^{H_{0}}$: if $\omega_{n}$
converges to $\xi\in A$, then $H\omega_{n}$ converges to $\zeta_{H}(\xi)$. It implies that the hitting distribution of the random walk $\left(H\omega_{n}\right)_{n=0}^{\infty}$ is the measure $(\zeta_{H})_{\ast}\nu$
on the ends space $\partial\left(H\backslash F\right)$. 

Let us denote $\alpha_{H,g}:=(\zeta_H)_\ast g\nu$ the hitting distribution of the same random walk started at $Hg$ rather than $H$. 
Suppose that $\rho$ is an $F$-invariant measure supported on ${\rm Tree}_{F}^{}$.
We equip the bundle $M_F\to{\rm Tree}_{F}^{}$ with a measure $\bar{\lambda}$
defined by
\begin{align}\label{barlambda}
\bar{\lambda}=\int_{{\rm Tree}_{F}^{}}\left(\zeta_{H}\right)_{\ast}\nu d\rho(H)=\int_{{\rm Tree}_{F}^{}}\alpha_{H,e} d\rho(H).
\end{align}
Note that as a measurable $F$-space,
$\left(M_F,\bar{\lambda}\right)$ fits into the sequence of $F$-factors
\[
\left({\rm Tree}_{F}^{}\times\partial F,\rho\times\nu\right)\overset{\zeta}{\to}\left(M_F,\bar{\lambda}\right)\to\left({\rm Tree}_{F}^{},\rho\right),
\]
where the first map $\zeta$ sends $(H,\xi)$ to $\left(H,\zeta_{H}(\xi)\right)$;
and the second map is the coordinate projection $\left(H,\zeta_{H}(\xi)\right)\mapsto H$.

\subsubsection{Identification of bundles}

Since $\rho$ is $F$-invariant, we are in the setting of Subsection~\ref{subsec:strip},
with $X={\rm Tree}_{F}^{H_{0}}$ and $x\mapsto L_{x}$ the identity
map. We apply the strip criterion to identify the Poisson bundle with
the end compactification bundle. 

\begin{proposition}\label{identification2} Let $F$ be a finitely generated free group with a measure $\mu$ of finite entropy and finite $\log$-moment. Let $H_0 \in {\rm Tree}(F)$ be a subgroup satisfying (\ref{eq:transB}) in Lemma \ref{crossing}. Let $\rho$ be an ergodic $F$-invariant
probability measure on ${\rm Tree}_{F}^{H_{0}}$. Then the Bowen-Poisson
bundle $(Z_\rho,\lambda_\rho)$ over $\left({\rm Tree}_{F}^{H_{0}},\rho\right)$ is $F$-measurably
isomorphic to the end compactification bundle $\left(M_F,\bar{\lambda}\right)$
defined above.
\end{proposition} 

Before we proceed to the proof, we record the following lower bound on the hitting probabilities of
shadows.

\begin{notation}[Shadows]\label{notation-shd} Let $H\in{\rm Tree}_{F}$,
choose the identity coset $H$ as the base point $o$ in $H\backslash F$.
For a vertex $v\in H\backslash F$, denote by ${\rm Shd}(v)$ the
set of geodesic rays (finite or infinite) based at $o$ that pass
through $v$. We view ${\rm Shd}(v)$ as a subset of $\left(H\backslash F\right)\cup\partial\left(H\backslash F\right)$.
Denote by 
\[
\mho_{H}(v):={\rm Shd}(v)\cap\partial\left(H\backslash F\right).
\]
\end{notation}

Recall that the hitting distribution of the coset random walk on the end-compactification boundary is $\alpha_{H,e}=(\zeta_H)_\ast \nu$.

\begin{lemma}\label{lift-pos} For each $n\in\mathbb{N}$, there
exists a constant $c=c(H_{0},n,\mu)>0$ such that for any $H\in{\rm Tree}_{F}^{H_{0}}$,
the hitting distribution satisfies 
\[
\alpha_{H,e}\left(\mho_{H}(u)\right)\ge c
\]
for any $u\in H\backslash F$ within distance $n$ to the root $o=H$.
\end{lemma}

\begin{proof} Up to replacing $\mu$ by a convolution power, we may and do assume that $\mu$ charges the generators of $F$. We first consider the Schreier graph $H_{0}\backslash F$. Let $B$ be
a connected finite subset of $H_{0}\backslash F$, containing the
root $o=H_{0}$. We claim that for any $v\in H_{0}\backslash F$ such
that $v\notin B$, 
\[
p_{v}(B):=\mathbb{P}_{\mu}\left(B\cap\left[v\omega_{t};\omega_{t}^{-1}\omega_{t+1}\right]=\emptyset\mbox{ for all }t\in\mathbb{N}\right)>0.
\]
Suppose the claim is not true: then the random walk started at some $v$ crosses $B$ almost surely. The connected component
of $v$ in $H_{0}\backslash F-B$ is ${\rm Shd}(v_{0})$ of some vertex
$v_{0}$ at distance $1$ to $B$. Then for each $x\in{\rm Shd}(v_{0})$, we also have $p_x(B)=0$, since $1-p_x(B)$ is the probability that the random walk started at $v$ crosses $B$ conditioned on taking first a path from $v$ to $x$ in ${\rm Shd}(v_{0})$ (event of positive probability). It implies that any trajectory entering ${\rm Shd}(v_{0})$ crosses $B$ almost surely.
 Since the hitting distribution $\left(\zeta_{H_{0}}\right)_{\ast}\nu$
charges the cylinder set $\mho(v_{0})$ with positive probability,
this contradicts with Lemma \ref{crossing}.

Now let $H\in{\rm Tree}_{F}^{H_{0}}$. Recall (\ref{eq:tree-cover}) that
$H<H_{0}^{\gamma}$ for some word $\gamma$ in $F$ representative
of one of the finitely many cosets of $F/N_{H}(F)$. Assume $|\gamma|\le n$.
On the Schreier graph $H\backslash F$, let $u$ be a vertex within
distance $n$ to the root $o=H$, and choose $g\in F$ a representative
such that $u=Hg$ and $|g|\le n$. Then choose an element $g'\in F$
with $|g'|\le2n$ such that on $H\backslash F$, $Hgg'\in{\rm Shd}(u)$;
and on $H_{0}^{\gamma}\backslash F$, $\left|H_{0}^{\gamma}gg'\right|>n$.
By non-degeneracy of $\mu$, there is $m_{0}\in\mathbb{N}$ such that
$\mu^{(m_{0})}$ charges every element in the ball of radius $3n$
around identity in $F$. Consider the event that in $m_{0}$ steps,
the $\mu$-random walk on $F$ is at $gg'$, and after time $m_{0}$,
the induced trajectory on $H_{0}^{\gamma}\backslash F$ never sweeps
cross the ball of radius $n$ around $H_{0}^{\gamma}$. The covering
property implies that the corresponding trajectory $\left(H\omega_{t}\right)$
never sweeps cross the ball of radius $n$ around $H$ after time
$m_{0}$, in particular, it stays in the shadow of $u$. As the $n$-ball
centred at $H_{0}^{\gamma}$ in $H_{0}^{\gamma}\backslash F$ is isometric
the $n$-ball centred at $H_{0}\gamma^{-1}$ in $H_{0}\backslash F$,
it follows that 
\[
\left(\zeta_{H}\right){}_{\ast}\nu\left(\mho(u)\right)\ge\min_{B(e_{F},3n)}\mu^{(m_{0})}\cdot\min\left\{ p_{v}\left(B(H_{0}\gamma^{-1},n)\right):v\notin B(H_{0}\gamma^{-1},n),\ |v|\le3n\right\} .
\]
\end{proof}

\begin{figure}
\begin{centering}
\begin{tikzpicture}[scale=0.8]
\draw[red] (0,0) node {$\bullet$}  node[below right] {$H\omega_0$};
\draw[black] (0,0) node {$\square$};
\draw[black] (1.8,0) node {$\bullet$} node[below right] {$Ha$};
\draw[black] (-1.8,0) node {$\bullet$} node[below] {$Ha^{-1}$};
\draw[black] (3.6,0) node {$\bullet$};
\draw[black] (-3.6,0) node {$\bullet$};
\draw[red] (5.4,0) node {$\bullet$} node[below right] {$H\omega_2$};
\draw[black] (-5.4,0) node {$\bullet$};
\draw[red] (1.8,1.8) node {$\bullet$} node[below right] {$H\omega_1$};
\draw[black] (0,1.8) node {$\bullet$} node[below right] {$Hb$};
\draw[black] (0,3.6) node {$\bullet$};
\draw[red] (-5.4,1.8) node {$\bullet$} node[below right] {$H\omega_{-1}$};
\draw[black] (5.4,1.8) node {$\bullet$};
\draw[blue] (7,3.6) node[right] {$\zeta_+(H,\omega)$};
\draw[blue] (-8.4,1.8) node[above] {$\zeta_-(H,\omega)$};
\draw (-6.6,0)--(6.6,0);
\draw(0,-1.2)--(0,4.2);
\draw(-5.4,-1.2)--(-5.4,2.4);
\draw(1.8,-1.2)--(1.8,2.4);
\draw(1.2,1.8)--(2.4,1.8);
\draw(-6,1.8)--(-4.8,1.8);
\draw(-0.6,3.6)--(0.6,3.6);
\draw (5.4,-1.2)--(5.4,2.4);
\draw[thick] (3.6,0) to[bend right=80] (3.6,0.8);
\draw[thick] (3.6,0.8) to[bend right=80] (3.6,0);
\draw[thick] (-3.6,0) to[bend right=80] (-3.6,0.8);
\draw[thick] (-3.6,0.8) to[bend right=80] (-3.6,0);
\draw[thick] (-1.8,0) to[bend right=80] (-1.8,0.8);
\draw[thick] (-1.8,0.8) to[bend right=80] (-1.8,0);
\draw[thick] (0,1.8) to[bend right=80] (-0.8,1.8);
\draw[thick] (-0.8,1.8) to[bend right=80] (0,1.8);
\draw[thick] (5.4,1.8) to[bend right=80] (4.6,1.8);
\draw[thick] (4.6,1.8) to[bend right=80] (5.4,1.8);
\draw[ultra thick, dashed, blue] (-6.6,0)--(5.4,0);
\draw[ultra thick, dashed, blue] (5.4,2.4)--(5.4,0);
\draw[ultra thick, dashed, blue] (5.4,2.4)--(7,3.6);
\draw[ultra thick, dashed, blue] (-6.6,0)--(-8.4,1.8);
\draw[thick, ->] (0,-1.8)--(0,-2.8);
\draw[black] (0,-2.2) node[right] {$\tilde{T}$};

\begin{scope}[yshift=-7.8cm]
\draw[red] (0,0) node {$\bullet$}  node[below right] {$H'\omega'_{-1}$};
\draw[black] (1.8,1.8) node {$\square$};
\draw[black] (1.8,0) node {$\bullet$} node[below right] {$H'b^{-1}$};
\draw[black] (-1.8,0) node {$\bullet$} node[below] {$H'b^{-1}a^{-2}$};
\draw[black] (3.6,0) node {$\bullet$};
\draw[black] (-3.6,0) node {$\bullet$};
\draw[red] (5.4,0) node {$\bullet$} node[below right] {$H'\omega'_1$};
\draw[black] (-5.4,0) node {$\bullet$};
\draw[red] (1.8,1.8) node {$\bullet$} node[below right] {$H'\omega'_0$};
\draw[black] (0,1.8) node {$\bullet$} node[below right] {\tiny{$H'b^{-1}a^{-1}b$}};
\draw[black] (0,3.6) node {$\bullet$};
\draw[red] (-5.4,1.8) node {$\bullet$} node[below right] {$H'\omega'_{-2}$};
\draw[black] (5.4,1.8) node {$\bullet$};
\draw[blue] (7,3.6) node[right] {$\zeta_+(H',\omega')$};
\draw[blue] (-8.4,1.8) node[above] {$\zeta_-(H',\omega')$};
\draw (-6.6,0)--(6.6,0);
\draw(0,-1.2)--(0,4.2);
\draw(-5.4,-1.2)--(-5.4,2.4);
\draw(1.8,-1.2)--(1.8,2.4);
\draw(1.2,1.8)--(2.4,1.8);
\draw(-6,1.8)--(-4.8,1.8);
\draw(-0.6,3.6)--(0.6,3.6);
\draw (5.4,-1.2)--(5.4,2.4);
\draw[thick] (3.6,0) to[bend right=80] (3.6,0.8);
\draw[thick] (3.6,0.8) to[bend right=80] (3.6,0);
\draw[thick] (-3.6,0) to[bend right=80] (-3.6,0.8);
\draw[thick] (-3.6,0.8) to[bend right=80] (-3.6,0);
\draw[thick] (-1.8,0) to[bend right=80] (-1.8,0.8);
\draw[thick] (-1.8,0.8) to[bend right=80] (-1.8,0);
\draw[thick] (0,1.8) to[bend right=80] (-0.8,1.8);
\draw[thick] (-0.8,1.8) to[bend right=80] (0,1.8);
\draw[thick] (5.4,1.8) to[bend right=80] (4.6,1.8);
\draw[thick] (4.6,1.8) to[bend right=80] (5.4,1.8);
\draw[ultra thick, dashed, blue] (-6.6,0)--(5.4,0);
\draw[ultra thick, dashed, blue] (5.4,2.4)--(5.4,0);
\draw[ultra thick, dashed, blue] (5.4,2.4)--(7,3.6);
\draw[ultra thick, dashed, blue] (-6.6,0)--(-8.4,1.8);
\end{scope}
\end{tikzpicture} 
\par\end{centering}
\caption{A trajectory $H\omega$ in $H\backslash F$, and its image under the
skew tranformation $(H',\omega'):=\tilde{T}(H,\omega)=\left(H^{\omega_{1}^{-1}},(\omega_{1}^{-1}\omega_{n+1})_{n}\right)$.
The root of each Schreier graph is marked by a square. The compatibility
condition (i) of Theorem~\ref{strip} is satisfied as $H\omega_{1}$
is on the strip of $(H,\omega)$ if and only if the root $H'=H^{\omega_{1}^{-1}}$
is on the strip of $(H',\omega')$. }
\label{fig:strip} 
\end{figure}
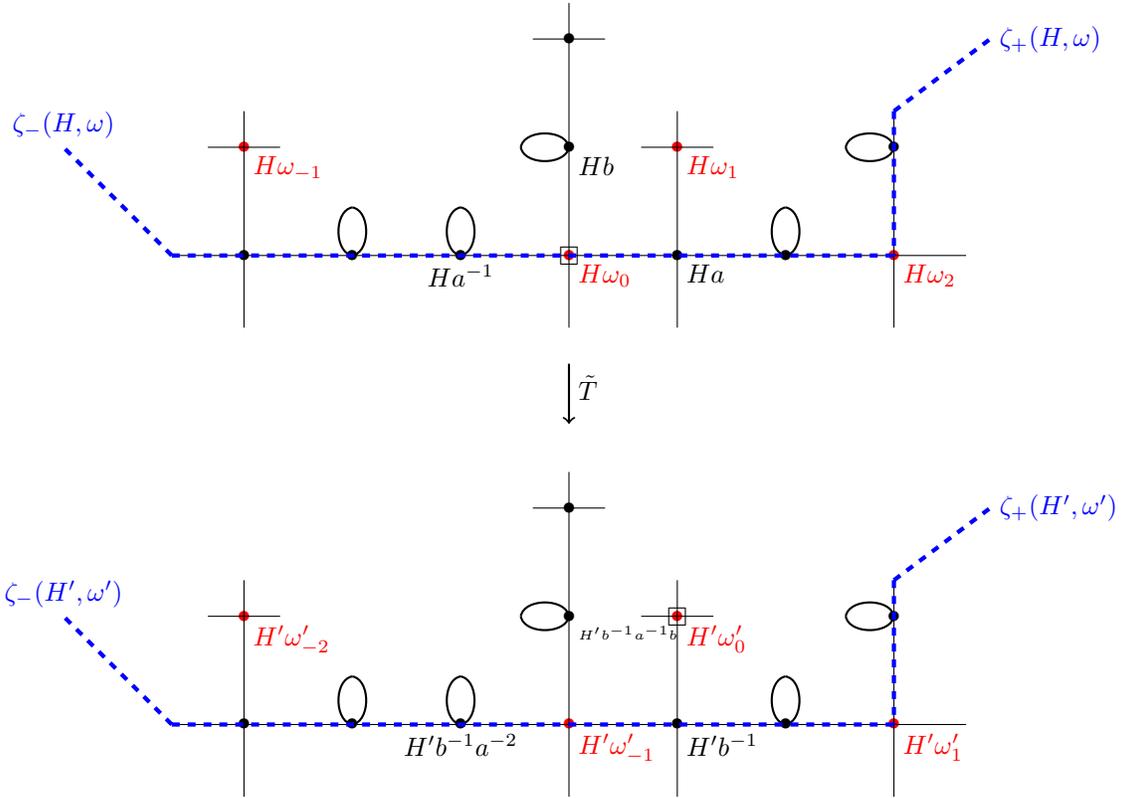

\begin{proof}[Proof of Proposition~\ref{identification2}]
We apply the strip criterion in Theorem \ref{strip}. Consider the
bilaterial path space $\left(F^{\mathbb{Z}},\tilde{\mathbb{P}}_{\mu}\right)$.
Denote by $\zeta_{+}(H,\omega)$ the composition $\zeta_{H}\circ{\rm bnd}_{+}$,
which is the end of $H\backslash F$ that the random walk $H\omega_{n}$
converges to in the positive time direction $n\to\infty$. Similarly,
denote by $\zeta_{-}(H,\omega)$ the composition $\zeta_{H}\circ{\rm bnd}_{-}$
in the negative time direction.

Take the strip $S(H,\omega)$ to be the (unique) geodesic on the tree-like
Schreier graph $H\backslash F$ connecting $\zeta_{+}(H,\omega)$
to $\zeta_{-}(H,\omega)$. Since the geodesic does not depend on the
location of the root, see Figure~\ref{fig:strip}, the choice of
strips satisfies the compatibility condition (i) in Theorem \ref{strip}.

We now verify the positivity condition (ii). Since the graph $H\backslash F$
is tree-like, we have that for $b_{+}\in\mho(Hs)$ and $b_{-}\in\mho(Hs')$
where $s,s'$ are two elements of $F$ such that $Hs,Hs'$ are two
distinct vertices distance $1$ from $H$. the geodesic connecting
$b_{+}$ and $b_{-}$ passes through the identity coset $H$. Therefore
for $Hs\neq Hs'$, we have 
\begin{align*}
\left(\rho\times\widetilde{\mathbb{P}}_{\mu}\right)\left(H\in S\left(H,\omega\right)\right) & \ge\mathbb{\widetilde{P}}_{\mu}\left(\zeta_{+}\left(H,\omega\right)\in\mho(Hs),\zeta_{-}\left(H,\omega\right)\in\mho(Hs')\right)\\
 & \ge c(H_{0},1,\mu)c(H_{0},1,\check{\mu})>0,
\end{align*}
where the positive constants $c(H_{0},1,\mu),c(H_{0},1,\check{\mu})$
are provided by Lemma \ref{lift-pos}. We have verified condition
(ii).

Since $\mu$ is assumed to have finite log-moment, we have that $\log\left|\omega_{n}\right|/n\to0$
when $n\to\infty$ for $\mathbb{P}_{\mu}$-a.e. $\omega$. Since the
strips are chosen to be geodesics, the intersection of a strip with
any ball of radius $r$ is bounded by $2r$. Condition (iii) is verified.
The statement then follows from Theorem~\ref{strip}. 
\end{proof}

\subsubsection{Another interpretation of the hitting distribution\label{subsec:interpretation}}

In the ends compactification bundle $M_F$, the fiberwise measure $\alpha_{H,e}=\left(\zeta_{H}\right)_{\ast}\nu$
is the hitting distribution of the random walk $\left(H\omega_{n}\right)$
on $\partial\left(H\backslash F\right)$. We have the diagram

\[
\xymatrix{\left(F^{\mathbb{N}},\mathbb{P}_{\mu}\right)\ar[r]^{{\rm bnd}}\ar[d] & \left(\partial F,\nu\right)\ar[d]^{\zeta_{H}}\\
\left(\left(H\backslash F\right)^{\mathbb{N}},\overline{\mathbb{P}}_{\mu}\right)\ar[r] & \left(\partial\left(H\backslash F\right),\alpha_{H,e}\right).
}
\]

In the diagram above, a point $\xi\in\partial F$ is viewed as an
end where the random walk $\left(\omega_{n}\right)$ converges to.
For later use, we consider
another way of interpreting the map $\zeta_{H}:\left(\partial F,\nu\right)\to\left(\partial\left(H\backslash F\right),\alpha_{H,e}\right)$.
A point in $\partial F$ is represented uniquely as an infinite reducible
word in the alphabet $S$. Denote
by $\xi_{n}$ the length $n$ prefix of the word $\xi$. We view $\xi_{n}$
as an element in $F$. Then the point $\xi\in\partial F$ induces
a sequence of points $\left(H\xi_{n}\right)$, which form a nearest
neighbor path on the Schreier graph $H\backslash F$.

\begin{proposition}\label{nearest}
Let $F$ be a finitely generated free group with a measure $\mu$ of finite entropy and finite $\log$-moment. Let $H_0 \in {\rm Tree}(F)$ be a subgroup satisfying (\ref{eq:transB}) in Lemma \ref{crossing}. 
For $\nu$-a.e. $\xi\in\partial F$,
the nearest neighbor sequence $\left(H\xi_{n}\right)$ converges to
the end $\zeta_{H}(\xi)\in\partial\left(H\backslash F\right)$.
\end{proposition}   

\begin{proof}
For a $\mu$-random walk trajectory $\omega=\left(\omega_{n}\right)$
on $F$ that converges to a point $\xi\in\partial F$, we claim that
for each $k\in\mathbb{N},$ there is a time $n_{k}$ such that the
prefix $\xi_{k}$ is on the geodesic connecting $\omega_{n_{k}}$
to $\omega_{n_{k}+1}$. Indeed, since $\omega$ converges to $\xi$,
$n_{k}=\max\left\{ n:\xi_{k}\mbox{ is not a prefix of }\omega_{n}\right\} $
is finite. Then $\xi_{k}$ is a prefix of $\omega_{n_{k}+1}$ and
the common prefix of $\omega_{n_{k}}$ and $\omega_{n_{k+1}}$ has
length $<k$. It follows that $\xi_{k}$ is on the geodesic path connecting
$\omega_{n_{k}}$ and $\omega_{n_{k+1}}$.

By Lemma \ref{lift}, we have for any finite set $K$ in $H\backslash F$,
almost surely $K\cap\left[H\omega_{n};\omega_{n}^{-1}\omega_{n+1}\right]\neq\emptyset$
for only finitely many $n$. Since $\xi_{k}$ is on the geodesic connecting
$\omega_{n_{k}}$ to $\omega_{n_{k}+1}$, $H\xi_{k}\in\left[H\omega_{n_{k}};\omega_{n_{k}}^{-1}\omega_{n_{k}+1}\right]$.
It follows that for $\nu$-a.e. $\xi$, the set $\{k\in\mathbb{N}:H\xi_{k}\in K\}$
is finite. Therefore $\left(H\xi_{k}\right)$ converges to the end
in $\partial\left(H\backslash F\right)$. Moreover, since $\xi_{k}$
is the length $k$ prefix of $\omega_{n_{k}+1},$ we have that the
Gromov product of $H\xi_{k}$ and $H\omega_{n_{k}+1}$ goes to infinity
as $k\to\infty$. Therefore the two sequences $\left(H\xi_{k}\right)$
and $\left(H\omega_{n_{k}+1}\right)$ converge to the same end, which
is $\zeta_{H}(\xi)$. 
\end{proof}

\subsection{Approximations on end-compactification bundles of $F$\label{subsec:Approximation-F}}

In this subsection, we show that the end-compactification bundle $M$ constructed above fits into the setting
of Proposition~\ref{KL-uni}, and obtain lower semi-continuity of
entropy.

Consider a tree-like  Schreier graph  $H\backslash F$. We
have a natural sequence of partitions of $\partial\left(H\backslash F\right)$
given by cylinder sets that are shadows of vertices. Denote  $S_{H}(n)=\left\{ v\in H\backslash F:d_{H\backslash F}(o,v)=n\right\} $ the sphere of radius $n$
of the Schreier graph $H\backslash F$. Since the graph is tree-like,
we have that 
\begin{align}\label{eq:shad}
\mathcal{P}_{H,n}=\left\{ \mho(v)\right\} _{v\in S_{H}(n)},
\end{align}
where the shadow $\mho(v)$ is defined in Notation \ref{notation-shd},
forms a partition of $\partial(H\backslash F)$ by clopen subsets.
This sequence of partitions satisfy: 
\begin{itemize}
\item $\mathcal{P}_{H,n+1}$ is a refinement of $\mathcal{P}_{H,n}$, 
\item the Borel $\sigma$-field of $\partial\left(H\backslash F\right)$
is generated by the partitions $\vee_{n=0}^{\infty}\mathcal{P}_{H,n}$. 
\end{itemize}
As shown in Lemma \ref{lift}, for $H\in{\rm Tree}_{F}^{H_{0}}$,
we have a map $\zeta_{H}:\partial F\to\partial\left(H\backslash F\right)$
such that $\alpha_{H,e}=(\zeta_{H})_{\ast}\nu$ is the hitting distribution of
the random walk $H\omega_{n}$ on the Schreier graph $H\backslash F$.

\begin{proposition}\label{relative-cont} 
 Let $F$ be a finitely generated free group with a measure $\mu$ of finite entropy and finite $\log$-moment. Let $H_0 \in {\rm Tree}(F)$ be a subgroup satisfying (\ref{eq:transB}) in Lemma \ref{crossing}.
Let $\beta$ be a probability measure in the measure class of the
$\mu$-harmonic measure $\nu$ on $\partial F$. Moreover suppose
$\left\Vert d\beta/d\nu\right\Vert _{\infty},\left\Vert d\nu/d\beta\right\Vert _{\infty}<\infty$.
Then $\beta_{H}:=(\zeta_{H})_{\ast}\beta$ admits locally constant
uniform approximations on ${\rm Tree}_{F}^{H_{0}}$. 
\end{proposition}

We will apply this proposition with $\beta=g\nu$ to obtain approximations for $\alpha_{H,g}=(\zeta_H)_\ast g\nu$. In this particular case, one option to obtain locally constant uniform approximations is to take the measure of $\mho(v)$
to be the probability that the coset random walk $Hg\omega_{t}$
is in ${\rm Shd}(v)$ and up to time $t$, the random walk never exited
the ball of radius $r_{t}$ around $H$, for a suitable choice
of the radius $r_{t}$. One can indeed verify the conditions to apply
Proposition~\ref{KL-uni} for such approximations. Instead of this
natural choice, for the convenience of inducing to $SL(d,\mathbb{R})$
in the next section, we use the interpretation of the hitting
distributions in Subsection \ref{subsec:interpretation}, which leads
naturally to Proposition~\ref{relative-cont}.

Denote by $\xi_{t}$ the length $t$ prefix of an infinite word $\xi\in\partial F$.
Then a point $\xi\in\partial F$ induces a sequence of points $\left(H\xi_{t}\right)_{t\in\mathbb{N}}$
on the Schreier graph $H\backslash F$. Apply Proposition~\ref{nearest}
to $H\in{\rm Tree}_{F}^{H_{0}}$, we have that for $\nu$-a.e. $\xi$,
the sequence $H\xi_{t}$ converges to an end, denoted by $\zeta_{H}(\xi)$
in $\partial\left(H\backslash F\right)$, when $t\to\infty$.

Suppose $\beta$ is a probability measure on $\partial F$ in the
measure class of the $\mu$-harmonic measure $\nu$. Write $\beta_{H}=\left(\zeta_{H}\right)_{\ast}\beta$.
Define a measure $\phi_{H,\beta}^{t}$ on the partition $\mathcal{P}_{H,n}$
by setting 
\begin{equation}
\phi_{H,\beta}^{t}\left(\mho(v)\right):=\beta\left(\left\{ \xi\in\partial F:H\xi_{t}\in{\rm Shd}(v)\right\} \right).\label{eq:approx}
\end{equation}
For $t>n$, $\phi_{H,\beta}^{t}$ is a measure on $\mathcal{P}_{H,n}$
Moreover, $\phi_{H,\beta}^{t}$ depends only on the ball of radius
$t$ around $H$ in the Schreier graphs $H\backslash F$,
which by definition of the Chabauty topology implies that $H\mapsto\left(\mathcal{P}_{H,n},\phi_{H,\beta}^{t}\right)$
is locally constant on ${\rm Tree}_{F}^{H_{0}}$.

\begin{lemma}[Verifies uniform approximation]\label{error-phi}
Suppose there is a constant $C_{\beta}>0$ such that $1/C_{\beta}\le d\beta/d\nu\le C_{\beta}$.
Then there is a function $\epsilon(t)\underset{t\to\infty}{\longrightarrow}0$,
which only depends on $n,H_{0},\mu$, such that for all $H\in{\rm Tree}_{F}^{H_{0}}$,
\[
\max_{A\in\mathcal{P}_{H,n}}\left|1-\frac{\beta_{H}(A)}{\phi_{H,\beta}^{t}(A)}\right|\le C_{\beta}^{2}\epsilon(t).
\]
\end{lemma}

\begin{proof}
Let $v\in S_{H}(n)$. To ease notations, in this proof we write $\phi_{H,\beta}^{t}(v)$
in place of $\phi_{H,\beta}^{t}(\mho(v))$, similarly for $\beta_{H}$.
The tree-like structure of the Schreier graphs guarantees that the
total variation distance between $\beta_H$ and $\phi_{H,\beta}^{t}$,
both restricted to $\mathcal{P}_{H,n}$, is no more than the measure,
under $\beta$, that there is some $s\ge t$ such that the path $H\xi_{s}$
returns to the ball of radius $n$ in $H\backslash F$. For $H\in{\rm Tree}_{F}^{H_{0}}$,
we have that $H<H_{0}^{\gamma}$ for some $\gamma\in F$, thus
for graph distances, $d(H_{0}^{\gamma},H_{0}^{\gamma}g)\le d(H,Hg)$
for any $g\in F$. Recall also $H_{0}$ has only finitely many conjugates
in $F$. Thus by this covering property, we have 
\begin{align*}
\frac{1}{2}\sum_{v\in S_{H}(n)}|\beta_{H}(v)-\phi_{H,\beta}^{t}(v)| & \le\beta\left(\left\{ \xi\in\partial F:\exists s\ge t,H\xi_{s}\in B_{H\backslash F}(n)\right\} \right)\\
 & \le\left\Vert \frac{d\beta}{d\nu}\right\Vert _{\infty}\nu\left(\left\{ \xi\in\partial F:\exists s\ge t,\exists\gamma\in F,H_{0}^{\gamma}\xi_{s}\in B_{H_{0}^{\gamma}\backslash F}(n)\right\} \right)\\
 & =:\left\Vert \frac{d\beta}{d\nu}\right\Vert _{\infty}\epsilon_{0}(t).
\end{align*}
The term $\epsilon_{0}(t)$ does not depend on $H$, and $\epsilon_{0}(t)\to0$
as $t\to\infty$ by Proposition \ref{nearest}.

The hitting measure $\beta_{H}(v)=(\zeta_{H})_{\ast}\beta\left(\mho(v)\right)$
is equal to the probability that the trajectory $(H\xi_{s})_{s=1}^{\infty}$,
eventually remains in ${\rm Shd}(v)$. By Lemma \ref{lift-pos}, we
have a lower bound 
\[
\beta_{H}(v)\ge\left\|\frac{d\nu}{d\beta}\right\|_{\infty}^{-1}\alpha_{H,e}(v) \ge C_\beta^{-1} c(H_{0},n,\mu).
\]
It follows that
\begin{align*}
\max_{A\in\mathcal{P}_{H,n}}\left|1-\frac{\phi_{H,\beta}^{t}(A)}{\beta_{H}(A)}\right| & \le\frac{1}{\min_{v\in\mathcal{S}_{H}(n)}\beta_{H}(v)}\sum_{v\in S_{H}(n)}|\beta_{H}(v)-\phi_{H,\beta}^{t}(v)|\le\frac{2C_\beta^2}{c(H_{0},n,\mu)}\epsilon_{0}(t).
\end{align*}
\end{proof}
\begin{proof}[Proof of Proposition \ref{relative-cont}]
The approximations $\phi_{H,\beta}^{t}$ defined in (\ref{eq:approx})
are locally constant and uniform by Lemma \ref{error-phi}. 
\end{proof}

\subsection{Entropy realization for the free group\label{subsec:Entropy-realization-F}}

In this subsection we conclude the proofs of Proposition~\ref{Bowen-conti} and Theorem~\ref{thm-free}.

\begin{proof}[Proof of Proposition~\ref{Bowen-conti}]
Recall that  $\mu$ is an admissible probability measure on $F$ with finite
entropy and finite log-moment. Take $H_0$ to be a tree-like subgroup of $F$ with finite index normalizer and whose Schreier graph has infinitely many ends. Let $\rho$ denote an IRS supported on ${\rm Tree}_F^{H_0}$. 

This is the setting of Part II with  $(X,\eta)=({\rm Tree}_F,\rho)$ and $L:X\to {\rm Sub}(F)$ the identity map. The $\mu$-boundary $(Y,\nu)=(\{\ast\},\delta_{\{\ast\}})$ is a Dirac mass, so in particular a manifold with volume form. Corollary~\ref{upper-semi} applies with $\mathcal{M}_\nu$ the set of IRS's supported on ${\rm Sub}(F)$, it gives upper semi-continuity of $\rho\mapsto h_\mu(Z_\rho,\lambda_\rho)$. 

There remains to prove lower semi-continuity. As $\rho$ is ergodic, Proposition~\ref{identification2} applies and shows that the Bowen-Poisson bundle $(Z_\rho,\lambda_\rho)$ identifies with the end-compactification bundle $(M_F,\bar\lambda_\rho)$. In particular $h_\mu(Z_\rho,\lambda_\rho)=h_\mu(M_F,\bar\lambda_\rho)$. Recall (\ref{barlambda}) that $\bar \lambda=\int_{{\rm Tree}_F}\alpha_{H,e}d\rho(H)$. Denote $\boldsymbol{\alpha}=(\alpha_{H,g})_{H \in {\rm Tree}_F^{H_0},g \in F}$. By the computation above Proposition~\ref{Fatou}, we have $h_\mu(M_F,\bar\lambda_\rho)=h_\mu(M_F,\boldsymbol{\alpha},\rho)$.

For each $H\in {\rm Tree}_F^{H_0}$, any two measures in $(\alpha_{H,g})_{g \in F}$ have uniformly bounded Radon-Nikodym derivatives (as  $\zeta_H$ push forwards of $g\nu$) so condition  \hyperlink{AssumpM}{(\textbf{Mu})} is satisfied. The end-compactification bundle $M$ also admits generating partitions as in condition \hyperlink{AssumpP}{(\textbf{P})}, by shadows of spheres (\ref{eq:shad}), and $\boldsymbol{\alpha}$ admits locally constant uniform approximations by Proposition~\ref{relative-cont}.
We can apply Proposition~\ref{cor-lower-uni} to get that the map $\rho \mapsto h_\mu(M_F,\boldsymbol{\alpha},\rho)=h_\mu(Z_\rho,\lambda_\rho)$ is lower semi-continuous on $\mathcal{M}_\nu(Y,\mathcal{S})$ which is the set of IRSs supported on ${\rm Tree}_F^{H_0}$, as $Y$ is a point.
\end{proof}

Now we turn to the proof of Theorem~\ref{thm-free}. We use Proposition~\ref{Bowen-conti} for some family of IRSs $\left(\rho_{\ell,p}\right)_{\ell\ge 2,p\in[0,1]}$ introduced by Bowen~\cite[p.505]{Bowen}.

\begin{proof}[Proof of Theorem \ref{thm-free}]
For an integer $\ell\ge2$, let $H_{0}=K_{\ell}$ be the quasi-transitive
subgroup of Example~\ref{exKell}. Its normalizer has finite index and its Schreier graph is tree-like with infinitely many ends. 
Informally, an $\rho_{\ell,p}$ sample is obtained by taking a random covering of the Schreier graph
$K_{\ell}\backslash F$ where each loop is ``opened'' independently
according to a $p$ Bernoulli distribution (or equivalently, the generating
pair $(g,s)$ of a loop is removed from the set of generators of $K_{\ell}$).
One can check directly from this description that for $p\in(0,1)$, $F\curvearrowright({\rm Tree}_{F}^{K_\ell},\rho_{\ell,p})$
is a weakly mixing extension of the finite transitive system $({\rm Tree}_{F}^{K_\ell},\rho_{\ell,0})$ consisting of the conjugates of $K_\ell$. The properties we need are: the
map $p\mapsto\rho_{\ell,p}$ is weak$^{\ast}$ continuous, $\rho_{\ell,0}$
is the uniform measure on conjugates of $K_{\ell}$, and $\rho_{\ell,1}$
is the $\delta$-mass on the trivial group $\{e\}$. These are shown
in \cite[Section 4.3]{Bowen}. 

By Proposition~\ref{Bowen-conti}, the composed map $p \mapsto h_\mu(Z_{\rho_{\ell,p}},\lambda_{\rho_{\ell,p}})$ is continuous. The proof concludes in the
same way as in \cite{Bowen}: by the intermediate value theorem, each
Furstenberg entropy value between $h_{\mu}(Z_{0},\lambda_{\ell,0})$
and $h_{\mu}(Z_{1},\lambda_{\ell,1})$ is attained. Now by \cite[Lemma 4.7]{Bowen},
$\rho_{\ell,1}$ is the trivial subgroup, so $h_{\mu}(Z_{1},\lambda_{\ell,1})=h_{\mu}(B,\nu_{B})$
is maximal.

Finally by \cite[Proof of Lemma 4.7]{Bowen}, the sequence
of measures $(\rho_{\ell,0})_{\ell\ge2}$ converges in weak$^{\ast}$
topology towards a measure $\kappa=\sum_{i=1}^{{\rm rank}(F)}\delta_{A_{i}}$,
where $A_{i}$ is the normal closure of the cyclic group $\left\langle a_{i}\right\rangle $.
Since $A_{i}\backslash F$ is isomorphic to $\mathbb{Z}$, any coset
random walk has trivial Poisson boundary. It follows that the Poisson
bundle over $(X,\kappa)$ has zero Furstenberg entropy. We conclude
that 
\[
\limsup_{\ell\to\infty}h_{\mu}(Z_{\rho_{\ell,0}},\lambda_{\rho_{\ell,0}})=0
\]
by the upper semi-continuity Corollary~\ref{upper-semi}. 
\end{proof}

\section{Poisson bundles for semi-simple Lie groups}\label{sec:Poisson-SL}

Throughout this section, let $G$ be a semi-simple Lie group and $\mu$ be either
an admissible measure or a Furstenberg discretization measure supported
on an irreducible lattice. In both cases, the Poisson boundary of the $\mu$-random
walk can be identified as $\left(G/P,\nu_{P}\right)$, where $P$ is a minimal parabolic subgroup and $\nu_{P}$
is in the measure class of the unique $K$-invariant measure. 

\subsection{Stationary systems induced from IRS's of Levi factors}\label{subsec:Gbundle}

Consider a semi-simple Lie group $G$ with a parabolic subgroup $Q$. Denote $\tau:G/Q \to G$ a measurable section and $\beta: G \times G/Q \to Q$ the associated cocycle. Let $\nu_Q$ denote the unique $\mu$-stationary measure on $G/Q$. Consider a surjective homomorphism of Lie groups $p:Q \twoheadrightarrow L$. Consider also a lattice $\Gamma_0$ of $L$. The choice of a measurable section $\tau_L:L/\Gamma_0 \to L$ provides a cocycle $\sigma:L \times L/\Gamma_0 \to \Gamma_0$. Denote $m_{L/\Gamma_0}$ the push-forward of the Haar measure on $L$ to $L/\Gamma_0$, normalised to be a probability measure.  Assume further that we have a measure preserving action $\Gamma_0 \curvearrowright (X_0,\rho)$. Consider the induced $G$-space 
\begin{equation}\label{eq:eta-tilt}
X=G/Q\times_{\beta}\left(L/\Gamma_0\times_{\sigma}X_{0}\right)\mbox{ equipped with measure }\eta_{\rho}^{\mu}=\nu_{Q}\times m_{L/\Gamma_0}\times\rho.
\end{equation}

Take the stabilizer map $X\to{\rm Sub}(G)$, $x\mapsto{\rm Stab}_{G}(x)$,
then the pushforward of $\eta_{\rho}^{\mu}$ is a $\mu$-stationary
random subgroup (SRS). 
This SRS may be viewed as co-induced from $\left(X_0,\rho\right)$
in the specific way described above. Here as customary, we identify $\rho$ and the IRS $\textrm{Stab}_\ast \rho$, where $\textrm{Stab}:X_0\to {\rm Sub}(\Gamma_0)$ maps a point to its stabiliser. Note that this operation is different
from the canonical co-induction of IRSs in the setting of countable
groups \cite{KQ19}.

\begin{notation}\label{notation2}
Given a measure preserving action $\Gamma_0 \curvearrowright (X_0,\rho)$  and a step distribution $\mu$ on $G$,
denote by $\left(Z,\lambda_{\rho}^{\mu}\right)$ the Poisson bundle
associated with the $\mu$-random walk, over the standard system $\left(X,\eta_{\rho}^{\mu}\right)\overset{\pi}{\to}\left(G/Q,\nu_{Q}\right)$
as in (\ref{eq:eta-tilt}). 
\end{notation}
The Poisson bundle also depends on $Q$ and the map $p:Q \to L$,
even though we do not explicit the dependence in the notation.

This situation fits into the setting of standard systems of Part II as described in Section~\ref{sec:Poissonbundle}. Indeed, $\left(X,\eta_{\rho}^{\mu}\right)$
is a relative-measure preserving extension of the $\mu$-boundary $(Y,\nu)=\left(G/Q,\nu_{Q}\right)$, and the standing assumption~(\ref{eq:contain}) is satisfied by $L_x=\mathrm{Stab}_G(x)$.

\subsubsection{Ergodicity of the Poisson bundle}
In this subsection, we assume that $G=\prod G_i$ is a finite product, where each $G_i$ is a connected non-compact simple Lie group with finite center.  Recall that $S:=\left\langle {\rm supp}\mu\right\rangle<G$ is assumed to be either $G$ or an irreducible lattice in $G$.  

\begin{lemma}\label{lem:ergodicZ}
In the setting of Notation~\ref{notation2}, assume that the action $\Gamma_0 \curvearrowright (X_0,\rho)$ is ergodic, then the group $S $ acts ergodically on the Poisson bundle
$\left(Z,\lambda_{\rho}^{\mu}\right)$.
\end{lemma}

\begin{proof}
It is enough to prove that $S$ acts ergodically on
$\left(X,\eta_{\rho}^{\mu}\right)$.
Indeed by Remark~\ref{ergodic}, it implies that
$S:=\left\langle {\rm supp}\mu\right\rangle $ acts ergodically on
$\left(Z,\lambda_{\rho}^{\mu}\right)$. We distinguish two cases. Firstly when $\mu$ is an admissible measure, then $S=G$, and ergodicity follows from classical induction properties, see e.g.~\cite[Proposition 2.7]{ZimmerInduced}. Secondly when $\mu$ is a discretisation measure we have that $S$ is an irreducible lattice of $G$. Then by~\cite[Proposition 2.2.2]{ZimmerBook}, $S$ acts ergodically on $X$ if and only if $G$ acts ergodically on $X\times G/S$, where $G/S$ is equipped with the $G$-invariant probability measure. The Howe-Moore property, see e.g.~\cite[Theorem 2.2.20]{ZimmerBook} implies that the $G$ action on $G/S$ is mixing, combined with the fact that $G$ acts ergodically on $X$, we have that $G$ acts ergodically on $X\times G/S$, and thus $S$ acts ergodically on $X$.
\end{proof}

\subsubsection{Description of the  $G$-Poisson bundle relatively to the $\Gamma_0$-Poisson bundle}

Our next task is to identify the $G$-Poisson bundle $(Z,\lambda_\rho^\mu)$ described in Notation~\ref{notation2}.  
Let us further assume that $L$ is a quotient of the Levi subgroup of $Q$. The projection $p:Q\to L$ induces a map still denoted $p:Q/P \to L/(L\cap P)$. Denote by $\nu_{P}=\int_{Y}\nu_{P}^{y}d\nu_{Q}(y)$ the disintegration
of the harmonic measure $\nu_{P}$ over the quotient map $G/P\to Y=G/Q$.
Note that the support of $\nu_{P}^{y}$ is $\tau(y)Q/P$ and $\mathrm{Stab}_G(x)$
acts on $\tau(y)Q/P$ where $y=\pi(x)$. Since our assumptions on $\mu$ imply that $\nu_P$ is in the measure class of the $K$-invariant measure $\bar{m}_K$, it follows that the push-forward measure $p_\ast(\tau(y)^{-1}\nu_P^y)$, which is a measure on $L/(L\cap P)$, is in the measure class of the $K\cap L$-invariant measure $\bar{m}_{L\cap K}$. Note that $(L/(L\cap P),\bar{m}_{L\cap K})$ is the Poisson boundary of any bi-$L\cap K$-invariant admissible measure on $L$ by~\cite{Furstenberg2}.

We consider the $\Gamma_0$-bundle $(E_{\Gamma_0},{m}_{\rho})\to (X_0,\rho)$ where the fiber over $x_0 \in X_0$ is the ergodic decomposition $E(x_0):=\mathrm{Stab}_{\Gamma_0}(x_0) \sslash (L/(L\cap P),\bar{m}_{L\cap K})$, with measure $m_\rho^{x_0}$ and disintegration $m_\rho=\int_{X_0}m_\rho^{x_0}d\rho(x_0)$. We denote $\psi_{x_0}: L/(L\cap P) \to E(x_0)$ the quotient $\Gamma_0$-map.

We can know describe the $G$-Poisson bundle $(Z,\lambda_\rho^\mu)$, using the $\Gamma_0$-bundle $E_{\Gamma_0}$.


\begin{lemma}[Description of the $G$-Poisson bundle relatively to the $\Gamma_0$-Poisson bundle]\label{lem:GrelGamma0}
With notations above, there is a measurable $G$-isomorphism:
\[
\Psi_0:(Z,\lambda_\rho^\mu) \rightarrow \left(G/Q \times_\beta (L/\Gamma_0 \times_\sigma E_{\Gamma_0}), \alpha_\rho\right),
\]
where $G/Q \times_\beta (L/\Gamma_0 \times_\sigma E_{\Gamma_0})$ is a bundle over $(X,\eta_\rho^\mu)$ given by (\ref{eq:eta-tilt}).

The disintegration of $\alpha_\rho$ over $(X,\eta_\rho^\mu)$ is given by   $ \alpha_\rho:=\int_{X}(\psi_{x_0})_\ast\left(\tau_L^{-1}(r) p_\ast (\tau(y)^{-1}\nu_P^y)\right)d\eta_\rho^{\mu}(x)$ for $x=(y,r,x_0) \in X=G/Q\times_{\beta}\left(L/\Gamma_0\times_{\sigma}X_{0}\right)$.
\end{lemma}
In the case where $\nu_P=\bar{m}_K$ is the $K$-invariant measure on $G/P$,
we simply have $(\Psi_0)_\ast \lambda_\rho^\mu=\bar{m}_{G/Q} \times m_{L/\Gamma_0} \times \bar{m}_{L\cap K}$ by $K$-invariance.

\begin{proof}
Let $x=(y,r,x_0) \in X=G/Q \times L/\Gamma_0 \times X_0$. By definition (\ref{eq:eta-tilt}) of the $G$-action on $X$, we have that
\begin{align}\label{eq:stabG}
\mathrm{Stab}_G(y,r,x_0)=\tau(y) p^{-1}\left( \tau_L(r) \mathrm{Stab}_{\Gamma_0}(x_0) \tau_L(r)^{-1}\right)\tau(y)^{-1}.
\end{align}
By Proposition \ref{ergodiccom},
in the Poisson bundle $\left(Z,\lambda_{\rho}^{\mu}\right)$ over
$\left(X,\eta_{\rho}^{\mu}\right)$, the fiber over $x\in X$ is the
space of ergodic components $\mathrm{Stab}_G(x)\sslash\left(\tau(y)Q/P,\nu_{P}^{y}\right)$,
where $y=\pi(x)$. Therefore an element of $Z$ is a couple $(x,A)$ with $x \in X$ and $A$ a $\mathrm{Stab}_G(x)$-invariant subset in $\tau(y)Q/P$. Equivalently, $\tau(y)^{-1}A$ is a subset of $Q/P$ that is invariant under $p^{-1}\left( \tau_L(r) \mathrm{Stab}_{\Gamma_0}(x_0) \tau_L(r)^{-1}\right)$. It means that $p(\tau(y)^{-1}A)$ is a $\tau_L(r)\mathrm{Stab}_{\Gamma_0}(x_0)\tau_L(r)^{-1}$ invariant subset of $L/(L\cap P)$, i.e that $\tau_L(r)^{-1}p(\tau^{-1}(y)A)$ is a $\mathrm{Stab}_{\Gamma_0}(x_0)$ invariant subset of $L/(L\cap P)$. This shows that the map $\Psi_0:Z\rightarrow G/Q \times_\beta (L/\Gamma_0 \times_\sigma E_{\Gamma_0})$ is a $G$-isomorphism. 

There remains to identify the measure. 
Our assumptions on $\mu$ guarantee that $\nu_P$ is in the measure class of $\bar{m}_K$, whence the disintegration $\tau(y)^{-1}\nu_P^y$ is in the measure class of $\bar{m}_{Q \cap K}$ on $Q/P$ for $\nu_Q$-a.e. $y$. A fortiori $p_\ast (\tau(y)^{-1}\nu_P^y)$ is in the measure class of $\bar{m}_{L\cap K}$ and $\alpha_\rho^{y,r}:=\tau_L(r)^{-1}p_\ast (\tau(y)^{-1}\nu_P^y)$ as well for $m_{L/\Gamma_0}$-a.e. $r$.

Consider first the case where $(X_0,\rho)$ is a Dirac mass, i.e. the case of a Poisson bundle $(Z',\lambda')$ over the standard system $(Y':=G/Q \times_\beta L/\Gamma_0,\nu_Q\times m_{L/\Gamma_0})$. Then the fibers over $Y'$ have the form $L/(L\cap P)$, and $\alpha_\rho^{y,r}$ as above is the measure over  the fiber at $(y,r) \in Y'$ in the disintegration of $\lambda'$ over $Y'$, which proves the formula in this particular case. In general, over the fiber at $(y,r,x_0) \in X$, one has to further take the ergodic decomposition $\psi_{x_0}$, and push forward $\alpha_\rho^{y,r}$ accordingly, so that $\alpha_\rho=\int_{Y'} \int_{X_0} (\psi_{x_0})_\ast(\alpha_\rho^{y,r})d\rho(x_0)d \nu_Q\times m_{L/\Gamma_0}(y,r)$.
\end{proof}

\subsection{Locally constant approximations and continuity of entropy}

Our goal is to obtain continuity of the entropy map $\rho \mapsto h_\mu(Z,\lambda_\rho^\mu)$. We do this under assumptions on the lattice $\Gamma_0$.

Let us first fix the space $X_0=\mathrm{Sub}(\Gamma_0)$ together with a compact subset $S_0 \subset \mathrm{Sub}(\Gamma_0)$. We denote $\mathcal{S}_0$ the set of IRSs supported on $S_0$. Consider the $G$ space $Y':=G/Q\times_\beta {L/\Gamma_0}$ with measure $\nu':=\nu_Q\times m_{L/\Gamma_0}$ which is a measure preserving extension of the $\mu$-boundary $(G/Q,\nu_Q)$. For $(y,r) \in Y'$, denote $S_{(y,r)}:=\{\mathrm{Stab}_G(y,r,x_0):\mathrm{stab}_{\Gamma_0}(x_0)\in S_0\}$, and $\mathcal{S}_0:=(S_{(y,r)})_{(y,r)\in Y'}$. Using the notations of Sections~\ref{subsec:Upper-semi-continuity} and~\ref{sec:lower-argu}, we consider the set $\mathcal{M}_{\nu'}(Y',\mathcal{S})$. 
It consists of all measures $\eta$ on the space 
$X=G/Q \times_\beta (L/\Gamma_0 \times_\sigma X_0)$
that project to $\nu_Q \times m_{L/\Gamma_0}$ under the map $X\to Y'$, and whose disintegrations $\eta=\int_{Y'}\eta^{(y,r)}d\nu_Q \times m_{L/\Gamma_0}(y,r)$ satisfy $\mathrm{supp}(\eta^{(y,r)}) \subset S_{(y,r)}$ almost surely. 


The main proposition is the following criterion providing continuity of entropy under some assumptions on the $\Gamma_0$-Poisson bundles $(E_{\Gamma_0},m_\rho^K)$.

\begin{proposition}\label{prop:lowerscG}
Consider $G$-Poisson bundles $(Z,\lambda_\rho^\mu)$ as in Subsection~\ref{subsec:Gbundle}. Assume that all $\Gamma_0$-Poisson bundles $(E_{\Gamma_0},m_\rho)$ over $(X_0,\rho)$ can be realised with the same topological model $M_{\Gamma_0}$ independent of $\rho \in \mathcal{S}_0$, and which admits finite generating partitions. Assume moreover that the spaces of ergodic components $(E(x_0),m_\rho^{x_0})_{x_0 \in S_0}$ admit locally constant uniform approximations on $S_0$. Then the map $\rho \mapsto h_\mu(Z,\lambda_\rho^\mu)$ is continuous on the space $\mathcal{S}_0$.
\end{proposition}

We recall that finite generating partitions are defined in assertion \hyperlink{AssumpP}{(\textbf{P})} in Section~\ref{sec:lower-argu}.

\begin{proof} 
By~(\ref{eq:eta-tilt}),  when $\rho_k \to \rho$ we have $\eta_{\rho_k}^\mu \to \eta_\rho^\mu$.  We first prove upper semi-continuity using the setting of Subsection~\ref{subsec:Y'}. Denote $\eta_\rho^\mu=\int_{Y'} (\eta_\rho^\mu)^{y'}d\nu'(y')$ the disintegration with $y'=(y,r)$. Then by~(\ref{eq:stabG}) the fiber measure  $(\eta_\rho^\mu)^{(y,r)}$ is supported on a set $S_{(y,r)}$ of triples $(y,r,x_0)$ such that for $\rho$-a.e. $x_0$ in $X_0$ we have
\[
\mathrm{Stab}_G(y,r,x_0)\in \tau(y)p^{-1}\left( \tau_L(r)\mathrm{Sub}(\Gamma_0)\tau_L(r)^{-1}\right) \tau(y)^{-1}.
\]
Since the lattice $\Gamma_0$ is discrete, each subset $S_{(y,r)}$ has the local coincidence property so the measure $\eta_\rho^\mu$ has the local coincidence property, and this holds for any~$\rho$ supported on $S_0$. Moreover $(Y',\nu')=(G/Q\times_{\beta} L/\Gamma_0,\nu_Q\times m_{L/\Gamma_0})$ is a manifold with a volume form. Corollary~\ref{cor-mutual} applies to ensure that $\rho \mapsto h_\mu(Z,\lambda_\rho^\mu)$ is upper semi-continuous.

For the lower semi-continuity, we use  the setting of Section~\ref{subsec:fiber-uni}. Our assumption that $(E_{\Gamma_0},m_\rho)$ can be realised with the same topological model $M_{\Gamma_0}$ permits, by Lemma~\ref{lem:GrelGamma0}, to identify $(Z,\lambda_{\rho}^{\mu})$ with 
\[
M_G:=\left(G/Q \times_\beta L/\Gamma_0 \times_\sigma M_{\Gamma_0},\alpha_\rho\right)
\]
where $\alpha_\rho=\int_X \alpha_{\rho,x,e} d\eta_\rho^\mu(x)$ is defined by its disintegration with $\alpha_{\rho,x,g}:=(\psi_{x_0})_\ast(g\nu_P^{y,r})$, where $x=(y,r,x_0)$ and $\nu_P^{y,r}:=\tau_L^{-1}(r)p_\ast(\tau(y)^{-1} \nu_P^y)$. Our assumption on $\mu$ that its Poisson boundary is $(G/P,\nu_P)$ with $\nu_P$ in the measure class of the $K$-invariant measure guarantees that $\alpha_{\rho,x,g}$ is in the measure class of $m_\rho^{x_0}$. It follows that
\[
\frac{d\alpha_{\rho,x,g}}{d\alpha_{\rho,x,e}} \in L^\infty(M_x, \alpha_{\rho,x,e}), \quad \textrm{with} \quad \left| \frac{d\alpha_{\rho,x,g}}{d\alpha_{\rho,x,e}}\right| \le \sup\left|\frac{dg\nu_P^{y,r}}{d\nu_P^{y,r}} \right|=:C_g<\infty.
\]
Therefore condition  \hyperlink{AssumpM}{(\textbf{Mu})}  from Section~\ref{sec:lower-argu} is satisfied by each fiber $M_x$ over $x\in X$ of the bundle $M_G$. Now the fiber $(M_x,\alpha_{\rho,x,g})$, which is isomorphic as measured space to $E(x_0)$, admits locally constant uniform approximations for each $x$ by assumption. Finally, 
Lemma~\ref{lem:Ymfd} gives that $(\eta_{\rho_k}^\mu)^{(y,r)} \to (\eta_{\rho}^\mu)^{(y,r)}$ for $\nu'$-a.e. $(y,r)$. Therefore Proposition~\ref{cor-lower-uni} applies to give that $\rho \mapsto h_\mu(M_G,\{\alpha_{\rho,x,g}\}_{x,g},\eta_\rho^\mu)$ is lower semi-continuous.
By Lemma~\ref{lem:GrelGamma0} and the computation (\ref{eq:h=h}) in Section~\ref{sec:lower-argu}, we have $h_\mu(Z,\eta_\rho^\mu)=h_\mu(M_G,\{\alpha_{\rho,x,g}\}_{x,g},\eta_\rho^\mu)$ which finishes the proof.
\end{proof}

\subsection{Entropy realization for $SL(d,\mathbb{R})$ and its lattices\label{subsec:Entropy-realization-S}}

Let us now focus on the case of $G=SL(d,\mathbb{R})$ and its lattices, with the aim to prove Theorem~\ref{th:spectrum-sl}.

\subsubsection{Parabolic subgroups and their rank $1$ Levi factors}\label{subsec:parabolicrank1}

We first recall the structure of parabolic subgroups.
The set of simple roots of $G=SL(d,\mathbb{R})$ is $\Delta=\{e_{1}-e_{2},\ldots,e_{d-1}-e_{d}\}$,
which is naturally identified with the set $\{1,\ldots,d-1\}$. Let
$I\subseteq\Delta$ and list $\Delta-I=\left\{ i_{1},\ldots,i_{\ell-1}\right\} $
in increasing order. Then associated with $I$ is the partition $d=d_{1}+\ldots+d_{\ell}$,
where $d_{j}=i_{j}-i_{j-1}$, $i_{0}=0$ and $i_{\ell}=d$. The parabolic subgroup $Q=P_{I}$ is stabilizer of the standard
flag $V_{1}\subset V_{2}\subset\ldots\subset V_{\ell}$, where $V_{j}$
is spanned by the $i_{j}$-first standard basis vectors. The Levi
subgroup of $P_{I}$ consists of block diagonal matrices 
\begin{equation}
L_{I}=\left\{ \left[\begin{array}{cccc}
M_{1}\\
 & M_{2}\\
 &  & \ddots\\
 &  &  & M_{\ell}
\end{array}\right],\ M_{j}\in GL(d_{j},\mathbb{R}),\ \det(M_{1})\ldots\det(M_{\ell})=1\right\} .\label{eq:LI}
\end{equation}
The minimal
parabolic subgroup $P=P_\emptyset=P(n,\mathbb{R})$ is the subgroup of upper triangular
matrices. 

Throughout this subsection we consider the situation where $I\subseteq\Delta$
is such that $SL(2,\mathbb{R})$ is a factor of $L_{I}$, i.e., one
of the blocks in~(\ref{eq:LI}) is $2\times2$. Take such a block, that is, $k\in\{1,\ldots,\ell\}$
with $d_{k}=2$. Regard $SL(2,\mathbb{R})$ as a subgroup of $L_{I}$,
embedded in $M_{k}$ and all the other blocks are identities. Denote
by 
\[
p:Q\to L:=\left\{ M \in GL(2,\mathbb{R}), \left|\det(M)\right|=1 \right\}
\]
the quotient map which is the composition of $Q=L_{I}\ltimes V_{I}\to L_{I}$
and $L_{I}\to L$ which sends the $2\times2$-block $M_{k}$ to $\frac{1}{\sqrt{|\det M_{k}|}}M_{k}$. Note that $L$ is isomorphic to $SL(2,\mathbb{R})\rtimes\mathbb{Z}/2\mathbb{Z}$ (except for $G=SL(2,\mathbb{R})$ with $Q=G$ and $p$ the identity map).

By the structure
of parabolic subgroups, we have $Q/P=\prod_{j=1}^{\ell}SL(d_{j},\mathbb{R})/P(d_{j},\mathbb{R})$, where $P(n,\mathbb{R})$
is the minimal parabolic subgroup of $SL(n,\mathbb{R})$.
The unique $K\cap Q$ invariant measure on $Q/P$ is $\bar{m}_{K\cap Q}=\prod_{j=1}^{\ell}\bar{m}_{SO(d_{j})}$,
where $\bar{m}_{SO(n)}$ denotes the unique $SO(n)$-invariant probability
measure on $SL(n,\mathbb{R})/P(n,\mathbb{R})$. Denote by
\[
\bar{p}:\left(Q/P,\bar{m}_{K\cap Q}\right)\to\left(SL(2,\mathbb{R})/P(2,\mathbb{R}),\bar{m}_{SO(2)}\right)
\]
 the projection to the $k$-th
component in the product, induced by the projection $p$.


\subsubsection{The free group as a lattice}

For clarity of later arguments, it is convenient to fix an embedding of
$F=F_{2}$ as a lattice in $SL(2,\mathbb{R})$, hence also a lattice $\Gamma_0=F$ in $L$. Take 
\[
A=\left(\begin{array}{cc}
1 & 2\\
0 & 1
\end{array}\right),\ B=\left(\begin{array}{cc}
1 & 0\\
2 & 1
\end{array}\right).
\]
The group $\left\langle A,B\right\rangle $ is called the Sanov subgroup,
it is a free group of rank $2$. Denote by $\mathbb{H}$ the upper
half plane model of the $2$-dimensional hyperbolic space, where $SL(2,\mathbb{R})$
acts by Mobius transforms. Take the ideal rectangle $R_{0}$ with
vertices $-1,0,1,\infty$ on $\mathbb{H}$. It is the union of two
adjacent ideal triangles with vertices $-1,0,\infty$ and $0,1,\infty$
in the Farey tessellation by ideal triangles. The orbit of $R_{0}$
under $\left\langle A,B\right\rangle $ forms a tessellation of the
hyperbolic plane. The dual graph of the tessellation is a tree, it
can be identified as the standard Cayley graph of the free group $F=\left\langle A,B\right\rangle $.

We follow the classical method to code hyperbolic geodesics with the
tessellation. Take the map $\psi:\mathbb{H}\to F$ by sending a point
$x$ in the tile $\gamma R_{0}$ to $\gamma$. Given a base point
$x_{0}\in\mathbb{H}$ and an irrational point $z\in\partial\mathbb{H}$,
for the geodesic from $x_{0}$ to $z$, record the sequence of tiles
that it passes through: $\left(\gamma_{0}R_{0},\gamma_{1}R_{0},\gamma_{2}R_{0},\ldots\right)$,
where each $\gamma_{n}\in F$. Since the sequence $\left(\gamma_{n}\right)$
comes from a geodesic, which in particular can not backtrack, we have
that $\gamma_{n}$ converges to an infinite reduced word $\gamma_{\infty}\in\partial F$
as $n\to\infty$. By basic properties of the Farey tessellation, see
e.g., \cite{Series}, we have:

\begin{lemma}\label{farey}

The map $\psi$ extends continuously to 
\begin{align*}
\psi:\partial\mathbb{H}\setminus\mathbb{Q} & \to\partial F\\
z & \mapsto\gamma_{\infty},
\end{align*}
which is $F$-equivariant, injective on $\partial\mathbb{H}-\mathbb{Q}$.

\end{lemma}

In particular, up to a countable subset of measure zero, we will identify $\partial \mathbb{H}=SL(2,\mathbb{R})/P(2,\mathbb{R})$ with the tree boundary $\partial F$.

\subsubsection{Entropy realization}

We can now obtain an interval of entropy values when one factor of the Levi subgroup is commensurable to $SL(2,\mathbb{R})$.

\begin{proposition}\label{prop-realization} 
Let $G$ be a semi-simple Lie group. Let $\mu$ be either in the $B_\infty$ class of $G$ or a Furstenberg measure on a lattice $\Gamma <G$ with finite Shannon entropy. Denote $S$ the subgroup generated by the support of $\mu$. Let $Q$ be a parabolic subgroup such that $SL(2,\mathbb{R})$ is a finite index subgroup of a factor $L$ of the Levi subgroup $L_Q$. Let $Q'$ be the parabolic subgroup obtained by replacing the factor $L$ by $\left\{\left(\begin{matrix}\ast & \ast \\ 0 & \ast \end{matrix} \right) \right\}$. Then
\[
\left[h_\mu(G/Q),h_\mu(G/Q') \right] \subset \mathrm{EntSp}(S,\mu).
\]
\end{proposition}

\begin{proof}
We let $\Gamma_0=F$ be the Sanov subgroup of $SL(2,\mathbb{R})$, which is a subgroup of finite index of $L$, so $\Gamma_0$ is a lattice in $L$. From an ergodic invariant action of $F=\Gamma_0$ on $(X_0=\mathrm{sub}(F),\rho)$, we induce a $G$-action over the space $X$ given by~(\ref{eq:eta-tilt}) and consider the associated Poisson bundle as in Notation~\ref{notation2}. Such actions are ergodic by Lemma~\ref{lem:ergodicZ}. 


We now prove continuity of the map $\rho \mapsto h_\mu(Z_{\rho},\lambda_{\rho}^{\mu})$ using Proposition~\ref{prop:lowerscG}. For this purpose, we assume that the input IRS $\rho$ is supported in $S_0=\mathrm{Tree}_F^{H_0}\subset S$ for some subgroup $H_0$ as in Proposition~\ref{Bowen-conti}. Under this condition, Proposition~\ref{identification2} shows that the $\Gamma_0$-Poisson bundles $(E_{\Gamma_0},m_\rho)$ are always measurably isomorphic to the end-compactification bundle $(M_F,\bar\lambda)$ introduced in Subsection~\ref{subsec:endbundle}. Moreover for $\rho$-a.e. $x_0$ in $S_0$, the spaces of ergodic components $(E(x_0),m_\rho^{x_0})_{x_0\in S_0}$, which coincide with the ends of the associated tree-like Schreier graphs, admit locally constant uniform approximations by Proposition~\ref{relative-cont}. 
Indeed, as $\nu_P$ is assumed to be in the measure class of the $K$-invariant measure on $G/P$, we have that almost all $\nu^{y'}=\alpha^{y,r}_\rho$ where $\nu=\int_{Y'} \nu^{y'}d\nu_Q\times m_{L/\Gamma_0}(y')$ are in the measure class of $\bar m_{SO(2)}$, which we identify with a harmonic measure on $\partial F$ as explained below. It follows that $\alpha_{\rho,(y,r,x_0),e}=(\psi_{x_0})_\ast(\alpha_\rho^{y,r})$  
is in the measure class of $m_\rho^{x_0}$ and $\alpha_{\rho,(y,r,x_0),g}$ as well for any $g \in G$. Proposition~\ref{prop:lowerscG} applies and gives continuity. An important fact used above is the identification between $\bar m_{SO(2)}$ and a harmonic measure on $\partial F$. This is possible since $F$ is a lattice in $L$, so there exists a probability measure $\kappa_{F}$ on $F$, which can be chosen fully supported, with finite logarithmic moment and finite entropy, such that $(L/(L\cap P),\bar{m}_{L\cap K})\simeq (SL(2,\mathbb{R})/P(2\mathbb{R}),\bar m_{SO(2)})$ is also the Poisson boundary of $(F,\kappa_F)$ by~\cite[Theorem 0.3]{CP07} (see also~\cite[Theorem 15.2]{Furstenberg4} or~\cite[Prop VI. 4.1]{MargulisBook}). 

Now we apply continuity of entropy to Bowen's families of tree-like IRSs of the form $\rho_{\ell,p}$ for $\ell\ge 2$ and $p \in (0,1)$ from~\cite[Section 4.3]{Bowen} as described in Section~\ref{subsec:Entropy-realization-F} in the proof of Theorem~\ref{thm-free}. For $p=1$, the IRS $\rho_{\ell,1}=\delta_{\{e\}}$ is supported on
the trivial subgroup. By Lemma~\ref{lem:GrelGamma0}, the $G$-bundle $(Z_{\rho_{\ell,1}},\lambda^{\mu}_{\rho_{\ell,1}})$ is simply isomorphic to $G/Q\times_{\beta}\left(L/\Gamma_0\times_{\sigma}\partial F\right)$. By assumption the $Q$-space $L/\Gamma_0\times_{\sigma}\partial F$ is $Q$-isomorphic to $Q/Q'$. It follows that the $G$-bundle $(Z_{\rho_{\ell,1}},\lambda^{\mu}_{\rho_{\ell,1}})$ is $G$-isomorphic to $G/Q\times_\beta Q/Q'=G/Q'$, which gives the upper bound of the interval.

Next for $p=0$, as in the proof of Theorem~\ref{thm-free} at the
end of Section~\ref{Sec:identification-free}, we have the weak$^{\ast}$
convergence $\rho_{\ell,0}\to\kappa=\frac{1}{2}\sum_{i=1}^{2}\delta_{A_{i}}$,
where $A_{i}$ is the normal closure of the cyclic subgroup $\left\langle a_{i}\right\rangle $
of $F$. Corollary~\ref{cor-mutual} on upper semi-continuity gives
\[
\limsup_{\ell\to\infty}h_{\mu}(Z_{\rho_{\ell,0}},\lambda^{\mu}_{\rho_{\ell,0}})\le h_{\mu}(Z_\kappa,\lambda^\mu_{\kappa})=\frac{1}{2}\sum_{i=1}^{2}h_{\mu}\left(Z_{\delta_{A_{i}}},\lambda^\mu_{\delta_{A_{i}}}\right).
\]
We have that $A_{i}$ is a normal subgroup of $F$ with quotient $A_{i}\setminus F$
isomorphic to $\mathbb{Z}$. Since any random walk on $\mathbb{Z}$
has trivial Poisson boundary, Lemma~\ref{lem:GrelGamma0} implies that the Poisson bundle $(Z_{\delta_{A_{i}}},\lambda^\mu_{\delta_{A_{i}}})$
has trivial fibers over the corresponding base system $(X,\eta^\mu_{\delta_{A_{i}}})$.
In particular, $(Z_{\delta_{A_{i}}},\lambda^\mu_{\delta_{A_{i}}}$ is a measure-preserving
extension of $\left(G/Q,\nu_Q\right)$. We conclude $\limsup_{\ell\to\infty}h_{\mu}(Z,\lambda_{\ell,0})\le h_{\mu}(G,Q)$
and so by the intermediate value theorem
\[
\bigcup_{\ell\in\mathbb{N}}\left\{ h_{\mu}\left(Z,\lambda_{\ell,p}\right):p\in[0,1]\right\} =\left(h_{\mu}\left(G/Q\right),h_{\mu}\left(G/Q'\right)\right].
\]
\end{proof}

\begin{proof}[Proof of Theorem \ref{th:spectrum-sl}]
For $SL(d,\mathbb{R})$, the parabolic subgroup $P_{\mathsf{f}(I)}$ is obtained from $P_I$ by removing all isolated roots in $I$, i.e. by replacing the $2\times 2$ blocks of the Levi subgroup by triangular blocks. By successive iteration, it is enough to consider the case where $Q'$ is obtained from $Q=P_I$ by removing only one such block, say the block $k$ with $d_k=2$ as in Subsection~\ref{subsec:parabolicrank1}. In this case, Proposition~\ref{prop-realization} applies, and we obtain
\[
{\rm EntSp}(S,\mu)\supseteq\bigcup_{I\subseteq\{1,\ldots,d-1\}}\left[h_{\mu}\left(G/P_{I},\nu_{I}\right),h_{\mu}\left(G/P_{\mathsf{f}(I)},\nu_{\mathsf{f}(I)}\right)\right],
\]
The reverse inclusion is Theorem~\ref{constraint}.
\end{proof}

\subsection{Interpretation in terms of Lyapunov exponents\label{subsec:lyapunov}}

Equip $\mathbb{R}^{n}$ with the standard Euclidean norm $\left\Vert \cdot\right\Vert $
and $SL(d,\mathbb{R})$ with the operator norm $\left\Vert \cdot\right\Vert _{{\rm op}}$.
For a step distribution $\mu$ on $G=SL(d,\mathbb{R})$ satisfying
the first moment condition $\int\log^{+}\left\Vert g\right\Vert _{{\rm op}}d\mu(g)<\infty$,
by the Osceledets multiplicative ergodic theorem, there exists exponents
$\lambda_{1}\ge \lambda_{2}\ge\ldots\ge\lambda_{d}$ and for a.e. $\omega$
a flag $V^{\le\lambda_{d}}\subset V^{\le\lambda_{d-1}}\subset\ldots\subset V^{\le\lambda_{1}}=\mathbb{R}^{d}$ such that for all $k$
\[
\forall v \in V^{\le \lambda_{k}}, \quad \limsup_{n\to \infty} \frac{1}{n}\log\|g_n\dots g_1 v\| \le \lambda_k.
\]
The flag is almost surely unique if the exponents are pairwise distinct.

When the harmonic measure $\nu_{P}$ is in the quasi-invariant measure
class, it is well-known that the Furstenberg entropy $h_{\mu}(G/Q,\nu_{Q})$ can be expressed
in terms of the Lyapunov exponents, see \cite{Ledrappier84} and references therein. Then in the setting of Theorem \ref{th:spectrum-sl}, 
the Furstenberg entropy spectrum is determined by the Lyapunov spectrum of the $\mu$-random walk. 
We include the explicit formulae below for the convenience of the reader.

\begin{proposition}

For the minimal parabolic subgroup $P$, 
\[
h_{\mu}\left(G/P,\nu_{P}\right)=\sum_{1\le i<j\le d}\lambda_{i}-\lambda_{j}.
\]
For a standard parabolic subgroup $Q$ such that the corresponding
flag in $G/Q$ is of type $\left(r_{1},r_{2},\ldots,r_{k}\right)$,
\[
h_{\mu}\left(G/Q,\nu_{Q}\right)=\sum_{\ell=1}^{k}\sum_{r_{\ell-1}<i\le r_{\ell},j>r_{\ell}}\lambda_{i}-\lambda_{j}=h_{\mu}\left(G/P,\nu_{P}\right)-\sum_{\ell=1}^{k}\sum_{r_{\ell-1}\le i<j\le r_{\ell}}\lambda_{i}-\lambda_{j}.
\]

\end{proposition} 
\begin{proof}[Sketch of proof.]

Given a parabolic subgroup $Q=P_{I}$, let $\bar{V}=\bar{V}_{I}$
be the corresponding opposite unipotent subgroup. denote by $\eta$
the Haar measure on $\bar{V}$. Recall that the projection $G\to G/Q$
maps $\bar{V}$ diffeomorphically onto a set of full $\bar{m}_{K}$-measure.
Still denote by $\eta$ the pushward of $\eta$ to $G/Q$. Write $g=n_{g}\sigma_{g}k_{g}\in\bar{N}AK$
for the Iwasawa decomposition. Note that $\bar{N}$ preserves the
measure $\eta$ on $G/Q$. By \cite[Proposition 1.15]{NZ3} we may
change $\bar{m}_{K}$ to the measure $\eta$ in the Radon-Nikodym
derivative and 
\begin{align*}
t\cdot  h_{\mu}\left(G/Q,\bar{m}_{K}\right) & =-\int_{G}\int_{G/Q}\log\frac{dg^{-1}.\bar{m}_{K}}{d\bar{m}_{K}}(x)d\bar{m}_{K}(x)d\mu^{(t)}(g)\\
 & =-\int_{G}\int_{G/Q}\log\frac{dg^{-1}.\eta}{d\eta}(x)d\bar{m}_{K}(x)d\mu^{(t)}(g)\\
 & =-\int_{G}\int_{G/Q}\log\frac{d\sigma_{g}^{-1}n_{g}^{-1}.\eta}{d\eta}(x)d\bar{m}_{K}(x)d\mu^{(t)}(g)\\
 & =-\int_{G}\int_{G/Q}\log\frac{d\sigma_{g}^{-1}\eta}{d\eta}(x)d\bar{m}_{K}(x)d\mu^{(t)}(g).
\end{align*}
Denote a $\mu$-random walk by $(\omega_{t})_{t=0}^{\infty}$. By
the almost sure convergence \cite[Theorem 2.6]{Ledrappier84} and
the equivalence between Iwasawa and polar decompositions \cite[Corollary 2.8]{GuivarchRaugi81},
it is known that $\frac{1}{t}\log\sigma_{\omega_{t}}$ converges almost
surely to the deterministic diagonal matrix $\Lambda$ with Lyapunov
exponents $\left(\lambda_{1},\lambda_{2},\ldots,\lambda_{d}\right)$,
repeated with multiplicity when the spectrum is not simple. So an
appropriate $i,j$ entry of the corresponding matrix in $\bar{V}$
is essentially multiplied by $e^{t(\lambda_{i}-\lambda_{j})}$. The
formula follows. 
\end{proof}

\appendix

\section{The Nevo-Zimmer operation on functions\label{sec:operation}}

In this section we review some arguments from \cite{NZ2} for the
use in Section \ref{sec:first-steps}. Let $G$ be a lcsc group and
$H$ be a closed subgroup of $G$. Throughout this section we assume
that an ergodic $G$-system $(X,\nu)$ has the following structure:
there is an $H$-invariant measure $\lambda$ on $X$ and a measure
$\nu_{0}$ on $G/H$ such that the map 
\begin{align*}
\xi_{0} & :\left(G\times_{H}X_{0},\nu_{0}\times\lambda\right)\to\left(X,\nu\right)\\
 & [g,x_{0}]\mapsto g.x_{0}
\end{align*}
is a $G$-factor map, where $X_{0}={\rm supp}\lambda\subseteq X$.
We identify $X_{0}$ as a subset of $G\times_{H}X_{0}$ via $x_{0}\mapsto[e,x_{0}]$.

\subsection{Conditional expectation}

Let $s$ be an element in $H$. Denote by $\mathcal{F}^{s}$ the sub-$\sigma$-field
of $\left(X_{0},\lambda\right)$ which consists of $s$-invariant
measurable sets, that is, $\mathcal{F}^{s}=\{A\in\mathcal{B}(X_{0}):s.A=A\}$.
Denote by $\mathbb{E}_{\lambda}\left[\cdot|\mathcal{F}^{s}\right]:L^{2}(X_{0},\lambda)\to L^{2}\left(X_{0},\lambda\right)$
the conditional expectation on the space of $s$-invariant functions.

Given a function $f\in C(X)$, lift it to a function $\tilde{f}$
on $G\times X_{0}$ by $\tilde{f}(g,x_{0})=f\left(g.x_{0}\right)$.
With a fixed element $g\in G$, we view $\tilde{f}(g,\cdot)$ as a
function in $L^{2}(X_{0},\lambda)$ and take its conditional expectation
given $\mathcal{F}^{s}$.

\begin{lemma}\label{invariant}

Suppose for an element $g\in G$, for any $f\in C(X)$, 
\begin{equation}
\mathbb{E}_{\lambda}\left[\tilde{f}(g,\cdot)|\mathcal{F}^{s}\right]=\mathbb{E}_{\lambda}\left[\tilde{f}(e,\cdot)|\mathcal{F}^{s}\right]\mbox{ }\lambda\mbox{-a.e.},\label{eq:condition=00003D00003D00003D}
\end{equation}
then the measure $\lambda$ (viewed as a measure on $X$) is invariant
under $g$.

\end{lemma} 
\begin{proof}
By the definition of conditional expectation, we have that 
\begin{equation}
\int_{X_{0}}\mathbb{E}_{\lambda}\left[\tilde{f}(g,\cdot)|\mathcal{F}^{s}\right]d\lambda=\int_{X_{0}}\tilde{f}\left(g,x_{0}\right)d\lambda(x_{0})=\int_{X}f\left(g.x_{0}\right)d\lambda(x_{0})=\int_{X}f\left(x\right)dg.\lambda(x).\label{eq:liftintegral}
\end{equation}
In the second equality we used the fact that $X_{0}={\rm supp}(\lambda)$.
Then (\ref{eq:condition=00003D00003D00003D}) implies $\int_{X}f\left(x\right)dg.\lambda(x)=\int_{X}f\left(x\right)d\lambda(x)$,
since it holds for all $f\in C(X)$, $g.\lambda=\lambda$. 
\end{proof}

\subsection{Dynamics of the $\left\langle s\right\rangle $-action\label{subsec:s-contraction}}

When the $\left\langle s\right\rangle $-action on $G/H$ has certain
contracting properties, in \cite{NZ2} it is shown that the resulting
conditional expectation given $\mathcal{F}^{s}$ factors through $\xi_{0}$.
We now describe the conditions.

Recall that ${\rm Int}(s).g=sgs^{-1}$.

\begin{assumption}\label{NZ-assum}

Suppose we have an element $s\in H$, subgroups $U,V$ of $G$ and
a normal subgroup $W$ of $H$ satisfying 
\begin{description}
\item [{(i)}] the map 
\begin{align*}
p:U\times V & \to G/H\\
(u,v) & \mapsto uvH
\end{align*}
takes $U\times V$ homeomorphically to a $\nu_{0}$-conull set in
$G/H$, and moreover $p_{\ast}\left(m_{U}\times m_{V}\right)$ is
in the same measure class as $\nu_{0}$. 
\item [{(ii)}] ${\rm Int}(s)$ acts trivially on $U$, that is, $s$ commutes
with elements in $U$, 
\item [{(iii)}] ${\rm Int}(s^{-1})$ acts as a contracting automorphism
on $V$; ${\rm Int}(s)$ acts as a contracting automorphism on $W$. 
\end{description}
\end{assumption}

Note that (iii) implies that $U\cap V=\{e\}$ and (i) implies that
$U\cap H=V\cap H=\{e\}$. Also note that since $s$ is an element
in $H$, the measure $\lambda$ is invariant under $s$.

\begin{example}
For $G=SL(3,\mathbb{R})$, these assumptions are satisfied for example for
\[
H=P=\left(\begin{array}{ccc}
 \ast & \ast & \ast \\
0  & \ast & \ast \\
0 & 0 & \ast  \\
\end{array} \right),
U=\left(\begin{array}{ccc}
1 & 0 & 0 \\
\ast  & 1 & 0 \\
0 & 0 & 1  \\
\end{array} \right),
V=\left(\begin{array}{ccc}
1 & 0 & 0 \\
0  & 1 & 0 \\
\ast & \ast & 1  \\
\end{array} \right),
W=\left(\begin{array}{ccc}
1 & 0 & \ast \\
0  & 1 & \ast \\
0 & 0 & 1  \\
\end{array} \right), 
\]
\[
\textrm{and }
s=\left(\begin{array}{ccc}
e^{-t_1} & 0 & 0 \\
0  & e^{-t_1} & 0 \\
0 & 0 & e^{-t_2}  \\
\end{array} \right), \quad \textrm{ for } t_1<t_2.
\]
\end{example}

When (i) - (iii) are satisfied, consider the following continuous
map 
\begin{align*}
\xi & :U\times V\times X_{0}\to X\\
 & (u,v,x_{0})\mapsto uv.x_{0},
\end{align*}
which is related to the map $\xi_{0}$ by $\xi(u,v,x_{0})=\xi_{0}\left(\left[uv,x_{0}\right]\right)$.
Since $\nu=\nu_{0}\ast\lambda$, by condition (i) we have that the
image $\xi\left(U\times V\times X_{0}\right)$ is a $\nu$-conull
set in $X$. Denote by $\tilde{\mathcal{L}}(X)$ the sub-$\sigma$-field
of the Borel $\sigma$-field of $U\times V\times X_{0}$ that consists
of (classes modulo null sets of ) lifts by $\xi$ of measurable subsets
of $X$. Denote by $\tilde{L}^{\infty}(X)=\tilde{L}^{\infty}(X,\nu)$
the subspace of $L^{\infty}\left(U\times V\times X_{0},m_{U}\times m_{V}\times\lambda\right)$
that consists of functions that are measurable with respect to $\tilde{\mathcal{L}}(X)$
. In other words, $\tilde{L}^{\infty}(X)$ consists of lift functions
in $L^{\infty}(X)$.

The following actions of the groups $U$ and $\left\langle s\right\rangle $
on $U\times V\times X_{0}$ preserve the product structure: 
\begin{align*}
u_{1}.(u,v,x_{0}) & =(u_{1}u,v,x_{0}),\ u_{1}\in U,\\
s.(u,v,x_{0}) & =\left(sus^{-1},svs^{-1},s.x_{0}\right)=(u,{\rm Int}(s).v,s.x_{0}).
\end{align*}
The map $\xi$ is equivariant under $U$ and $\left\langle s\right\rangle $.
Note also that since $U$ commutes with $\left\langle s\right\rangle $,
$U$ preserves the $s$-invariant $\sigma$-field $\mathcal{F}^{s}$.

For a continuous function $f\in C(X)$, consider the composition $f\circ\xi$,
which is a continuous function on $U\times V\times X_{0}$. The proof
of \cite[Proposition 7.1]{NZ2} applies verbatim to the current setting
and implies that the conditional expectation $\mathbb{E}_{\lambda}\left[\tilde{f}\left(u,\cdot\right)|\mathcal{F}^{s}\right]$
is the limit of Cesaro averages under the $\left\langle s\right\rangle $-action.\textcolor{blue}{{} }

\begin{proposition}\label{contract1}

Under conditions (i), (ii) and (iii), we have

\[
\lim_{N\to\infty}\frac{1}{N+1}\sum_{n=0}^{N}s^{n}.\left(f\circ\xi\right)\left(u,v,x_{0}\right)=\mathbb{E}_{\lambda}\left(\tilde{f}\left(u,\cdot\right)|\mathcal{F}^{s}\right)(x_{0}),
\]
where the converge is a.s. and in $L^{2}$. Moreover, the function
$\left(u,v,x_{0}\right)\mapsto\mathbb{E}_{\lambda}\left(\tilde{f}\left(u,\cdot\right)|\mathcal{F}^{s}\right)(x_{0})$
is in\textcolor{black}{{} $\tilde{L}^{\infty}(X)$. }

\end{proposition}

This property allows to define a map 
\begin{align*}
\mathcal{E}_{s} & :C(X)\to\tilde{L}^{\infty}(X)\\
\mathcal{E}_{s}f\left(u,v,x_{0}\right) & =\lim_{N\to\infty}\frac{1}{N+1}\sum_{n=0}^{N}s^{n}.\left(f\circ\xi\right)\left(u,v,x_{0}\right)=\mathbb{E}_{\lambda}\left(\tilde{f}\left(u,\cdot\right)|\mathcal{F}^{s}\right)(x_{0}).
\end{align*}
We will refer to the map $\mathcal{E}_{s}$ as the Nevo-Zimmer operation
with $\left(s,U,V,W\right)$ on $C(X)$. The resulting function does
not depend on the $v$-coordinate, that is, $\mathcal{E}_{s}f\left(u,v,x_{0}\right)=\mathcal{E}_{s}f\left(u,e,x_{0}\right)$.
Note the following immediate properties:

\begin{lemma}\label{contract2}

In the setting above, we have that for $f\in C(X)$, 
\begin{description}
\item [{(i)}] $u.\mathcal{E}_{s}\left(f\right)=\mathcal{E}_{s}\left(u.f\right)$
for $u\in U$. It follows that $\int_{X}fdu.\lambda=\int_{X_{0}}\mathcal{E}_{s}f(u,e,x_{0})d\lambda(x_{0}).$ 
\item [{(ii)}] If $u\in U$ is such that $u.X_{0}=X_{0}$, then $\mathcal{E}_{s}f(u,e,x_{0})=\mathcal{E}_{s}f(e,e,u.x_{0})$. 
\item [{(iii)}] Suppose $h\in H$ is an element such that either ${\rm Int}(s)$
or ${\rm Int}(s^{-1})$ contracts $h$ to $e$, then $\mathcal{E}_{s}f\left(u,v,h.x_{0}\right)=\mathcal{E}_{s}f\left(u,e,x_{0}\right)$. 
\end{description}
\end{lemma} 
\begin{proof}
(i). Since $s$ is assumed to commute with $U$, we have that for
$u\in U$, $us^{n}.(f\circ\xi)=s^{n}.\left((u.f)\circ\xi\right)$.
The statement then follows from Proposition \ref{contract1}.

(ii). This follows from $\left(f\circ\xi\right)\left(u,e,x_{0}\right)=\left(f\circ\xi\right)\left(e,e,u.x_{0}\right)$.

(iii). For a given $u\in U$, the function $\mathcal{E}_{s}f(u,e,\cdot)$
viewed as an element in $L^{2}(X_{0},\lambda)$, is invariant under
$\left\langle s\right\rangle $ in the sense that $\mathcal{E}_{s}f(u,e,x_{0})=\mathcal{E}_{s}f(u,e,s.x_{0})$.
Apply the generalized Mautner lemma (\cite[Lemma 3.2]{MargulisBook})
to the unitary representation of $H$ on $L^{2}(X_{0},\lambda)$,
we have that $\mathcal{E}_{s}f(u,e,\cdot)$ is invariant under $h$
as well. 
\end{proof}

\subsection{Three cases\label{subsec:three-cases}}

Suppose $\left(s,U,V,W\right)$ are such that conditions (i) - (iii)
are satisfied. When we apply the Nevo-Zimmer operation $\mathcal{E}_{s}$,
one of the following scenarios occurs. The first case is: 
\begin{description}
\item [{(I)}] The subgroup $U$ preserves the measure $\lambda$. 
\end{description}
As in the \cite[Section 8]{NZ2}, in the negation of (I), there are
two situations to consider, depending on whether $\left\langle s\right\rangle \curvearrowright\left(X_{0},\lambda\right)$
is ergodic. 
\begin{description}
\item [{(II1)}] There exists a function $f\in C(X)$ and $u\in U$ such
that $\int_{X}fd\lambda\neq\int_{X}fd\left(u.\lambda\right)$; and
for $m_{U}$-a.e. $u'\in U$, the function $x_{0}\mapsto\mathcal{E}_{s}f\left(u',e,x_{0}\right)$
is $\lambda$-constant. 
\end{description}
The remaining case is 
\begin{description}
\item [{(II2)}] $=\neg({\bf I}\vee{\bf II}1)$. There exists a function
$f\in C(X)$ such that $\int_{X}fd\lambda\neq\int_{X}fd\left(u.\lambda\right)$
for some $u\in U$. Moreover, for every such $f$, there is a $m_{U}$-positive
set of $u\in U$ where the function $x_{0}\mapsto\mathcal{E}_{s}f\left(u,e,x_{0}\right)$
is not $\lambda$-constant. 
\end{description}
\textcolor{black}{Case (II1) is treated in \cite[Proposition 9.2]{NZ2}.
We briefly describe the argument to show existence of a nontrivial
homogeneous factor in this case. Take a function $f\in C(X)$ }as
in the description of (II1). Since $x_{0}\mapsto\mathcal{E}_{s}f\left(u',e,x_{0}\right)$
is $\lambda$-constant, we have $\mathcal{E}_{s}f(u',e,x_{0})=\int_{X_{0}}\mathcal{E}_{s}f(u',e,x_{0})d\lambda(x_{0})=\int fdu'.\lambda$.
Recall that we have the map $\xi:U\times V\times X_{0}\to X$ where
$\xi\left(u,v,x_{0}\right)=uv.x_{0}$; and the projection map $p:U\times V\times X_{0}\to G/H$
given by $(u,v,x_{0})\mapsto uvH$. Similar to the sub-$\sigma$-field
$\tilde{\mathcal{L}}(X)$, denote by $\tilde{\mathcal{L}}(G/H)$ the
sub-$\sigma$-field of the Borel $\sigma$-field of $U\times V\times X_{0}$
that consists of lifts by $p$ of measurable subsets of $G/H$.
Then the function $(u,v,x_{0})\mapsto\mathcal{E}_{s}f(u,e,x_{0})=\int fdu.\lambda$
can be viewed as a non-constant function measurable with respect to
the intersection of $\sigma$-fields $\tilde{\mathcal{L}}(X)\cap\tilde{\mathcal{L}}(G/H)$.
Take the Mackey realization of $\tilde{\mathcal{L}}(X)\cap\tilde{\mathcal{L}}(G/H)$,
we obtain a nontrivial common factor of $(X,\nu)$ and $\left(G/H,m_{G/H}\right)$.
Thus in this case we conclude that $(X,\nu)$ has a nontrivial homogeneous
factor. 

Case (II2) is treated by considerations of the Gauss map in \cite{NZ2}.
By the second part of condition (iii), ${\rm Int}(s)$ acts as a contracting
automorphism on $W\vartriangleleft H$. As in the proof of \cite[Proposition 4.1]{NZ2},
apply the generalized Mautner lemma to the unitary representation
of $H$ on $L^{2}(X_{0},\lambda)$, we have that $s$-fixed functions
in $L^{2}(X_{0},\lambda)$ are fixed by $W$ as well. In $\mathcal{B}(X_{0})$
take the $W$-invariant sub-$\sigma$-algebra: 
\[
\mathcal{B}_{W}(X_{0}):=\{A\in\mathcal{B}(X_{0}):g.A=A\mbox{ for all }g\in W\}.
\]
In particular, $\mathcal{F}^{s}\subseteq\mathcal{B}_{W}(X_{0})$ by
the Mautner lemma. Since $W$ is assumed to be normal in $H$, we
have that if $A\in\mathcal{B}_{W}(X_{0})$, then for $h\in H$, $h.A\in\mathcal{B}_{W}(X_{0})$
as well. Denote by $X_{0}'$ the Mackey realization of the $\sigma$-algebra
$\mathcal{B}_{W}(X_{0})$, equipped with the measure $\lambda'$ which
corresponds to the restriction of $\lambda$ to $\mathcal{B}_{W}(X_{0})$.
Then $X_{0}'$ is an $H$-space and $W$ acts trivially on $X_{0}'$.
Note that by construction $X_{0}'={\rm supp}\lambda'$.

As in \cite[Section 6]{NZ2}, take the space $\left(X',\nu'\right)$,
which is the largest common $G$-factor of $(X,\nu)$ and $\left(G\times_{H}X_{0}',\nu_{0}\times\lambda'\right)$

\[
\xymatrix{G\times_{H}X_{0}\ar[r]\ar[d] & X\ar[d]\\
G\times_{H}X_{0}'\ar[r] & X'.
}
\]
The existence of a function $f$ as described in (II2) implies that
the sub-$\sigma$-field $\mathcal{F}^{s}$ is nontrivial, it follows
that $\left(X_{0}',\lambda'\right)$ and $\left(X',\nu'\right)$ are
nontrivial, see \cite[Section 8, Case I]{NZ2}. Note the following
property.

\begin{lemma}\label{II2-nonpreserv}

In Case (II2), the subgroup $U$ does not preserve the measure $\lambda'$,
which is viewed as a measure on $X'$.

\end{lemma} 
\begin{proof}
Take a function $f\in C(X)$ such that $\int_{X}fd\lambda\neq\int_{X}fd\left(u.\lambda\right)$
for some $u\in U$. The function $x_{0}\mapsto\mathcal{E}_{s}f(e,e,x_{0})$,
which is measurable with respect to $\mathcal{F}^{s}$, thus also
$\mathcal{B}_{W}(X_{0})$, can be viewed as the lift of a function
$\phi\in L^{2}(X_{0}',\lambda')$ to $L^{2}(X_{0},\lambda)$. Suppose
on the contrary $u.\lambda'=\lambda'$. Then in particular, $u.{\rm supp}\lambda'={\rm supp}\lambda'$
which implies $u.X_{0}=X_{0}$. Apply Lemma \ref{contract2}, we have
that $\mathcal{E}_{s}f(u,e,\cdot)$ is the lift of $u^{-1}.\phi$
to $L^{2}(X_{0},\lambda)$. Then 
\begin{align*}
\int_{X}f(x)d\lambda(x) & =\int_{X_{0}}\mathcal{E}_{s}f(e,e,x_{0})d\lambda(x_{0})=\int_{X'}\phi d\lambda'\\
 & =\int_{X'}\phi\left(x'\right)du.\lambda'(x')=\int_{X_{0}}\mathcal{E}_{s}f(u,e,x_{0})d\lambda(x_{0})=\int_{X}f(x)du.\lambda(x),
\end{align*}
which is a contradiction. 
\end{proof}
Since $W$ acts trivially on $X_{0}'$, we have that in the induced
system $G\times_{H}X_{0}'$, a point stabilizer contains a conjugate
of $W$. Passing to the factor $X'$, it follows that a stabilizer
${\rm Stab}_{G}(x')$ contains a conjugate of $W$ for $x'\in X'$.

\section{Mutual information and entropy formulae}

\label{app:info-ent}

In this second appendix, we record useful facts about KL-divergence,  mutual information and entropy.

\subsection{Kullback-Leibler divergence and mutual information}
\label{sec:entropy-formulae}

Let us first review the definitions of KL-divergence and mutual information. We also record a useful inequality regarding entropies of a finite partition, which is used in the proof of Proposition~\ref{KL-uni}.

\subsubsection{Relative entropy and KL-divergence}

Given a
finite measurable partition $\mathcal{P}$ of $X$, denote the relative
entropy of $\mathcal{P}$ with measure $\alpha$ with respect to $\beta$
as 
\[
H_{\alpha\parallel\beta}\left(\mathcal{P}\right):=\sum_{A\in\mathcal{P}}\alpha(A)\log\frac{\alpha(A)}{\beta(A)}.
\]
Suppose now $\alpha,\beta$ are probability measures on a measurable space
$(\Omega,\mathcal{B})$ and $\mathcal{P}_{n}$ is a refining sequence
of finite partitions whose union generates $\mathcal{B}$. 
The Kullback-Leibler divergence of probability measures $\alpha$
and $\beta$ can be defined by (see e.g., \cite[Corollary 7.3]{Gray})
\[
D\left(\alpha\parallel\beta\right)=\sup_{n}H_{\alpha\parallel\beta}\left(\mathcal{P}_{n}\right).
\]
When $\alpha$ is absolutely continuous with respect to $\beta$,
we have $D(\alpha\parallel\beta)=\int_{\Omega}\log\left(\frac{d\alpha}{d\beta}\right)d\alpha$
by \cite[Lemma 7.4]{Gray}. We refer to \cite[Chapter 7]{Gray} for
a detailed account.

In the proof of Proposition~\ref{KL-uni}, we will use the following
inequality, which is a consequence of the reverse Pinsker inequality
in \cite[Theorem 7]{Verdu}.

\begin{lemma}\label{KL-difference-1}

Let $\alpha,\beta,\alpha',\beta'$ be probability measures on the
space $\Omega$ and let $\mathcal{P}$ be a finite partition of $\Omega$. Denote
by $C=C_{\alpha,\beta}(\mathcal{P})=\max_{A\in\mathcal{P}}\left\{ \alpha(A)/\beta(A)\right\} $.
Then 
\[
H_{\alpha\parallel\beta}(\mathcal{P})-H_{\alpha'\parallel\beta'}(\mathcal{P})\le\log\left(\max_{A\in\mathcal{P}}\frac{\beta'(A)}{\beta(A)}\right)+2C^{1/2}\max_{A\in\mathcal{P}}\left|1-\frac{\alpha'(A)}{\alpha(A)}\right|.
\]
\end{lemma}
\begin{proof}
The difference in relative entropies is

\begin{align*}
H_{\alpha\parallel\beta}(\mathcal{P})-H_{\alpha'\parallel\beta'}(\mathcal{P}) & =\sum_{A\in\mathcal{P}}\alpha(A)\log\frac{\alpha(A)}{\beta(A)}-\alpha'(A)\log\frac{\alpha'(A)}{\beta'(A)}\\
 & =\left(\sum_{A\in\mathcal{P}}\alpha(A)\log\frac{\alpha(A)}{\beta(A)}-\alpha'(A)\log\frac{\alpha'(A)}{\beta(A)}\right)+\sum_{A\in\mathcal{P}}\alpha'(A)\log\frac{\beta'(A)}{\beta(A)}\\
 & =I+II.
\end{align*}
Rewrite $I$ as 
\begin{align*}
I & =\left(\sum_{A\in\mathcal{P}}\left(1-\frac{\alpha'(A)}{\alpha(A)}\right)\alpha(A)\log\frac{\alpha(A)}{\beta(A)}\right)-H_{\alpha'\parallel\alpha}\left(\mathcal{P}\right)\\
 & \le\sum_{A\in\mathcal{P}}\left(1-\frac{\alpha'(A)}{\alpha(A)}\right)\alpha(A)\log\frac{\alpha(A)}{\beta(A)}.
\end{align*}
Split the sum into two parts: $A$ is in $\mathcal{P}_{+}$ ($\mathcal{P}_{-}$
resp.) if $\alpha(A)/\beta(A)\ge1$ ($<1$ resp.). By the reverse
Pinsker inequality we have 
\begin{equation}
\sum_{A\in\mathcal{P}_{+}}\alpha(A)\log\frac{\alpha(A)}{\beta(A)}\le\sqrt{C}d_{{\rm TV}}(\alpha,\beta).\label{eq:reverse-pinsker}
\end{equation}
Note that \cite[Theorem 7]{Verdu} is stated with $D(\alpha\parallel\beta)$
on the left-hand side in the inequality. For the convenience of the
reader, we briefly repeat its proof here to show (\ref{eq:reverse-pinsker}).
For $1\le z\le C$, $\frac{1-C^{-1}}{\log C}z\log z\le z-1$, summing
over $\mathcal{P}_{+}$, we have 
\[
\frac{1-C^{-1}}{\log C}\sum_{A\in\mathcal{P}_{+}}\beta(A)\frac{\alpha(A)}{\beta(A)}\log\frac{\alpha(A)}{\beta(A)}\le\sum_{A\in\mathcal{P}_{+}}\beta(A)\left(\frac{\alpha(A)}{\beta(A)}-1\right)\le d_{{\rm TV}}(\alpha,\beta).
\]
For $x\in(0,1)$, $\sqrt{x}\le(x-1)/\log x$, then the inequality
(\ref{eq:reverse-pinsker}) follows.

Since $H_{\alpha\parallel\beta}(\mathcal{P})\ge0$, it follows that
$-\sum_{A\in\mathcal{P}_{-}}\alpha(A)\log\frac{\alpha(A)}{\beta(A)}\le\sum_{A\in\mathcal{P}_{+}}\alpha(A)\log\frac{\alpha(A)}{\beta(A)}$.
Plugging back in $I$, we have then 
\begin{align*}
I & \le\max_{A\in\mathcal{P}}\left|1-\frac{\alpha'(A)}{\alpha(A)}\right|\left(\sum_{A\in\mathcal{P}_{+}}\alpha(A)\log\frac{\alpha(A)}{\beta(A)}-\sum_{A\in\mathcal{P}_{-}}\alpha(A)\log\frac{\alpha(A)}{\beta(A)}\right)\\
 & \le2\sqrt{C}\max_{A\in\mathcal{P}}\left|1-\frac{\alpha'(A)}{\alpha(A)}\right|.
\end{align*}
The second part is bounded by 
\[
II=\sum_{A\in\mathcal{P}}\alpha'(A)\log\frac{\beta'(A)}{\beta(A)}\le\sum_{A\in\mathcal{P}}\alpha'(A)\log\left(\max_{A\in\mathcal{P}}\frac{\beta'(A)}{\beta(A)}\right)=\log\left(\max_{A\in\mathcal{P}}\frac{\beta'(A)}{\beta(A)}\right).
\]
\end{proof}

\subsubsection{Mutual information}

The mutual information of two random variables $U$ and $V$ of laws
$P(U)$ and $P(V)$ is the KL-divergence of their joint law with respect
to their product law: 
\[
\mathrm{I}(U,V):=\mathrm{D}\left(P(U,V)\bigparallel P(U)\otimes P(V)\right).
\]

Given a probability space $(\Omega,\mathcal{B},\mathbb{P})$, a random
variable $U:\Omega\to S$ and a sub-$\sigma$-field $\mathcal{F}$
of $\mathcal{B}$, denote by $P(U|\mathcal{F})$ the conditional law
of $U$ given $\mathcal{F}$. The mutual information of $U$ and $\mathcal{F}$
is given by 
\[
\mathrm{I}(U,\mathcal{F})=\int_{\Omega}\int_{S}\log\frac{dP\left(U|\mathcal{F}\right)}{dP(U)}dP\left(U|\mathcal{F}\right)d\mathbb{P}.
\]
These formulae are consistent: denote by $\sigma(V)$ the $\sigma$-field
generated by $V$, then $I(U,V)=I(U,\sigma(V))$.

\subsection{Entropy formulae for intermediate factors}\label{subsec:entropymutual}

The remainder of this appendix  slightly extends the setting of Section~\ref{sec:Poissonbundle}.
Let $\mu$ be an admissible measure on an lcsc group $G$. Consider
$(B,\nu_{B})$ the Poisson boundary, and $(X,\eta)$ a standard $(G,\mu)$-system, i.e. a measure preserving extension of a $\mu$-proximal space $(Y,\nu)$. Given an equivariant map $L:X\to \mathrm{Sub}(G)$ with $L_x<\mathrm{Stab}_G(x)$ for all $x \in X$, we defined in Section~\ref{sec:Poissonbundle} the associated Poisson bundle $(Z,\lambda)$. 
Here, we further consider an intermediate factor
\begin{equation}
\left(X\times B,\eta\varcurlyvee\nu_{B}\right)\overset{\psi}{\longrightarrow}\left(Z,\lambda\right)\overset{\varrho_M}{\longrightarrow} (M,\bar{\lambda}) \overset{\varrho}{\longrightarrow} \left(X,\eta\right).\label{eq:inter-1}
\end{equation}
where the composition is the natural coordinate
projection $X\times B\to X$. We refer to Figure~\ref{fig:diag} for an extended diagram. Note that all these factors are ergodic once $(X,\eta)$ is ergodic by Fact~\ref{pmp}.


Let us write $\tilde{\psi}$ for the map $X\times G^{\mathbb{N}}\to(M,\bar{\lambda})$
given by $\tilde{\psi}(x,\omega)=\varrho_M\circ \psi(x,{\rm bnd}(\omega))$. The Furstenberg entropy of $(M,\bar{\lambda})$ can be expressed in terms of the joining measure $\eta\varcurlyvee \mathbb{P}_\mu$.

\begin{lemma}\label{standard-entropy} Let $(M,\bar\lambda)$ be a factor of $(X\times G^{\mathbb{N}},\eta \varcurlyvee \mathbb{P}_\mu)$, we have for every $g\in G$,
\[
D(g\bar{\lambda}||\bar{\lambda})=\int_{X\times G^{\mathbb{N}}}\log\frac{d g\bar\lambda}{d\bar\lambda}\left(g.\tilde{\psi}(x,\omega)\right)d\left(\eta\varcurlyvee\mathbb{P}_{\mu}\right),
\]
and the Furstenberg entropy is
\[
h_{\mu}(M,\bar\lambda)=\int_{X\times G^{\mathbb{N}}}\log\frac{d\omega_{1}\bar\lambda}{d\bar\lambda}\left(\tilde{\psi}(x,\omega)\right)d\left(\eta\varcurlyvee\mathbb{P}_{\mu}\right).
\]
\end{lemma}
This lemma applies in particular to $(Z,\lambda)$ and $(X,\eta)$.

\begin{proof}

The first identity follows from the definition of KL-divergence, and that 
$\bar{\lambda}=\tilde{\psi}_{\ast}\left(\eta\varcurlyvee\mathbb{P}_{\mu}\right)$.

The Furstenberg entropy of $(M,\bar\lambda)$ is defined as 
\begin{align*}
h_{\mu}(M,\bar\lambda) &=\int_{G} D(g\bar{\lambda}||\bar{\lambda})d\mu(g)\\
 & =\int_{G}\int_{G^{\mathbb{N}}}\int_{X}\log\frac{dg\bar\lambda}{d\bar\lambda}\left(g.\tilde{\psi}\left(x,\omega\right)\right)d\eta_{\omega}(x)d\mathbb{P}_{\mu}(\omega)d\mu(g).
\end{align*}
Now for $\omega=(\omega_{1},\omega_{2},\dots)$ in $G^{\mathbb{N}}$,
set $T'\omega=(\omega_{1}^{-1}\omega_{2},\omega_{1}^{-1}\omega_{3},\dots)$.
When $\omega$ has law $\mathbb{P}_{\mu}$, then $(\omega_{1},T'\omega)$
has the same law $\mu\times\mathbb{P}_{\mu}$ as $(g,\omega)$.
It follows that 
\begin{align*}
h_{\mu}(M,\bar\lambda)= & \int_{G^{\mathbb{N}}}\int_{X}\log\frac{d\omega_{1}\bar\lambda}{d\bar\lambda}\left(\omega_{1}.\tilde{\psi}\left(x,T'\omega\right)\right)d\eta_{T'\omega}(x)d\mathbb{P}_{\mu}(\omega)\\
= & \int_{G^{\mathbb{N}}}\int_{X}\log\frac{d\omega_{1}\bar\lambda}{d\bar\lambda}\left(\tilde{\psi}\left(\omega_{1}.x,\omega'\right)\right)d\eta_{\omega'}(\omega_{1}.x)d\mathbb{P}_{\mu}(\omega)\quad(\textrm{as }\eta_{\omega_{1}^{-1}\omega'}=\omega_{1}^{-1}\eta_{\omega'})\\
= & \int_{G^{\mathbb{N}}}\int_{X}\log\frac{d\omega_{1}\bar\lambda}{d\bar\lambda}\left(\tilde{\psi}\left(x,\omega\right)\right)d\eta_{\omega}(x)d\mathbb{P}_{\mu}(\omega).
\end{align*}
\end{proof}

\subsubsection{Doob transform with respect to an intermediate factor}\label{subsec:DoobM}

Proposition~\ref{markov} ensures that over $x\in X$, the trajectory $\left(L_{x}\omega_{n}\right)$
is the projection of the Doob transformed random walk $\mathbb{P}_{\mu}^{\pi(x)}$
to the coset space $L_{x}\backslash G$. Since $(M,\bar{\lambda})$
fits into (\ref{eq:inter-1}), the fiber of $M$ over a point $x\in X$
is covered by the Poisson boundary of the coset random walk $\left(L_{x}\omega_{n}\right)$. Recall that $(W_\Omega,\overline{\mathbb{P}}_\mu)$ is the bundle of coset trajectories defined in
subsection~\ref{subsec:Definition}, see also Figure~\ref{fig:diag}.
Consider the disintegration 
\begin{align}\label{eq:desintM}
\vartheta_\ast (\eta\varcurlyvee \mathbb{P}_\mu)=:\overline{\mathbb{P}}_{\mu}=\int_M \overline{\mathbb{P}}_{\mu,(x,\zeta)}d\bar{\lambda}(x,\zeta).
\end{align}
In this setting,
the fiber measure $\overline{\mathbb{P}}_{\mu,(x,\zeta)}$,
considered as a distribution on $(L_x\backslash G)^{\mathbb{N}}$, is the law of a
Markov chain $(L_{x}\omega_{n})$ conditioned on $\theta_{M}(x,L_{x}\omega)=\left(x,\zeta\right)$, where  $\theta_{M}:W_{\Omega}\to M$ is the lift of the map $Z\to M$.

\begin{lemma}\label{lem:MarkovM}
The Markov chain of law $\overline{\mathbb{P}}_{\mu,(x,\zeta)}$, 
has transition kernel 
\[
P_{\mu,{(x,\zeta)}}\left(L_{x}g,A\right)=\int_{G}{\bf 1}_{A}\left(L_{x}gs\right)\frac{dgs\bar{\lambda}}{dg\bar{\lambda}}(x,\zeta)d\mu(s).
\]
\end{lemma}

\begin{proof}
Denote $\bar{\lambda}=\int_X \bar\lambda_xd\eta(x)$ the disintegration of $\bar\lambda$ over $(X,\eta)$. By (\ref{eq:desintM}), we get
\[
\overline{\mathbb{P}}_\mu=\int_X\int_{\rho^{-1}(x)}\overline{\mathbb{P}}_{\mu,(x,\zeta)} d\bar\lambda_x(\zeta)d\eta(x).
\]
By Lemma~\ref{standard}, it follows that $\mathbb{P}_\mu^{\pi(x)}=\int_{\rho^{-1}(x)}\overline{\mathbb{P}}_{\mu,(x,\zeta)} d\bar\lambda_x(\zeta)$. By~\cite[Section 3]{KaimanovichHyperbolic}, since $\mathbb{P}_\mu^{\pi(x)}$ is a Markov chain whose transition kernel is given by Proposition~\ref{markov}, we can further take a Doob transform and get a Markov chain on $L_x\backslash G$ with transition kernel
\[
P_{\mu,(x,\zeta)}(L_xg,A)=\int_{G}{\bf 1}_A(L_x gs)\frac{dgs\nu}{dg\nu}(\pi(x))\frac{d(gs\bar\lambda)_x}{dg\bar\lambda_x}(\zeta)d\mu(s).
\]
Here $g\bar\lambda_x(A)$ is the probability that the Markov chain $(L_x\omega_n)$ of law $\mathbb{P}_\mu^{\pi(x)}$ belongs to $\theta_M^{-1}(x,A)$ provided that $L_x \omega_1=L_xg$. We get the proposition by the Radon-Nikodym derivative
\[
\frac{dgs\bar\lambda}{dg\bar\lambda}(x,\zeta)=\frac{dgs\eta}{dg\eta}(x)\frac{d(gs\bar\lambda)_x}{dg\bar\lambda_x}(\zeta).
\]
and because $(X,\eta)\to(Y,\nu)$ is measure preserving, so $\frac{dgs\eta}{dg\eta}(x)=\frac{dgs\nu}{dg\nu}(\pi(x))$.
\end{proof}

On the Poisson boundary $\left(B,\nu_{B}\right)$ of the $\mu$-random
walk, since $g.\nu_{B}={\rm bnd}_{\ast}\left(g.\mathbb{P}_{\mu}\right)$
for $g\in G$, the measure $g.\nu_{B}$ can be regarded as the harmonic
measure of the random walk starting from $g$. A similar property
holds in the current setting. Denote by $\overline{\mathbb{P}}_{\mu,(x,\zeta),g}$
the law of the Markov chain on $L_{x}\backslash G$ with transition
kernel $P_{\mu,x}^\zeta$ given by Lemma~\ref{lem:MarkovM}, starting from the
coset $L_{x}g$. 

\begin{lemma}\label{translate-1}
We have for disintegration over $W_{\Omega}\to M$, 
\[
\left(g\overline{\mathbb{P}}_{\mu}\right)_{(x,\zeta)}=\overline{\mathbb{P}}_{\mu,(x,\zeta),g}.
\]
\end{lemma}
 
\begin{proof}
We first verify that the transition probabilities satisfy that for
$A\subseteq\{x\}\times L_{x}\backslash G\subseteq W$, 
\begin{equation}
P_{\mu,(x,\zeta)}^{}\left(L_{x}g^{-1},A\right)=P_{\mu,g.(x,\zeta)}^{}\left(L_{g.x},g.A\right).\label{eq:invariance}
\end{equation}
By Lemma~\ref{lem:MarkovM}, 
\begin{equation}
P_{\mu,(x,\zeta)}\left(L_{x}g^{-1},A\right)=\int_{G}{\bf 1}_{A}\left(x,L_{x}g^{-1}s\right)\frac{dgs\bar{\lambda}}{dg\bar{\lambda}}(x,\zeta)d\mu(s).\label{eq:step n}
\end{equation}
 If $\left(x,L_{x}g^{-1}s\right)\in A$ then $g.\left(x,L_{x}g^{-1}s\right)\in g.A$,
where $g.\left(x,L_{x}g^{-1}s\right)=\left(g.x,gL_{x}g^{-1}s\right)=\left(g.x,L_{g.x}s\right).$
Therefore ${\bf 1}_{A}\left(x,L_{x}g^{-1}s\right)={\bf 1}_{g.A}\left(g.x,L_{g.x}s\right)$.
Moreover $\frac{dgs\bar{\lambda}}{dg\bar{\lambda}}(x,\zeta)=\frac{ds\bar{\lambda}}{d\bar{\lambda}}(g(x,\zeta))$.
Plugging back in (\ref{eq:step n}), we have that 
\[
P_{\mu,(x,\zeta)}\left(L_{x}g^{-1},A\right)=\int_{G}{\bf 1}_{g.A}\left(g.x,L_{g.x}s\right)\frac{ds\bar{\lambda}}{d\bar{\lambda}}(g(x,\zeta))d\mu(s)=P_{\mu,g(x,\zeta)}\left(L_{g.x},g.A\right).
\]
It follows then from the Markov property that for $A\subseteq\{x\}\times\left(L_{x}\backslash G\right)^{\mathbb{N}}\subseteq W_{\Omega}$,
we have 
\begin{equation}
\mathbb{\overline{P}}_{\mu,(x,\zeta),g^{-1}}(A)=\overline{\mathbb{P}}_{\mu,g(x,\zeta),id}\left(g.A\right).\label{eq:P-translate}
\end{equation}
Take a subset $C\subseteq W_{\Omega}$,
we have 
\begin{align*}
g\overline{\mathbb{P}}_\mu(C)=\overline{\mathbb{P}}_\mu(g^{-1}C)  & =\int_{M}\overline{\mathbb{P}}_{\mu,(x,\zeta),e}\left(g^{-1}C\right)d\bar\lambda(x,\zeta)\\
 & =\int_{M}\overline{\mathbb{P}}_{\mu,g(x,\zeta),g}\left(C\right)d\bar\lambda(x,\zeta) \quad 
\mbox{ by }(\ref{eq:P-translate})\\
& = \int_M \overline{\mathbb{P}}_{\mu,(x,\zeta),g}dg\bar{\lambda}(x,\zeta),
\end{align*}
where we made the change of variable $g(x,\zeta)=(x',\zeta')$. We conclude by uniqueness of disintegration. 
\end{proof}

\subsubsection{KL-divergence of the shifted chain restricted to the tail $\sigma$-field}
Recall the map $\xi_n:X\times G^{\mathbb{N}} \to W$ given by $\xi_n(x,\omega)=(x,L_x\omega_n)$ defined in Subsection~\ref{subsec:tail}. Let us denote $\xi_{n,g}^{x,\zeta}$ for the random variable of law $\xi_n(x,\cdot)_\ast \overline{\mathbb{P}}_{\mu,(x,\zeta),g}$, which is the time $n$ step of the Markov chain starting at $L_xg$ described in previous section~\ref{subsec:DoobM}. In the notation of the Markov transition kernel in Lemma \ref{lem:MarkovM}, the law of $\xi_{n,g}^{x,\zeta}$ is $P^n_{\mu,(x,\zeta)}(L_xg,\cdot)$.

We will use the following two classical properties of KL divergence.
\begin{enumerate}
    \item  Disintegration of KL divergence:
Let $p:(Z,\mathcal C)\to(M,\mathcal M)$ be a measurable map, let $\alpha,\beta$ be
probability measures on $Z$ with $\bar\alpha:=p_*\alpha$, $\bar\beta:=p_*\beta$
and disintegrations $\alpha=\int_M\alpha_m\,d\bar\alpha(m)$,
$\beta=\int_M\beta_m\,d\bar\beta(m)$. Then, 
\[
D(\alpha||\beta)=D(\bar\alpha||\bar\beta)+\int_M D(\alpha_m||\beta_m)d\bar\alpha(m).
\]
See \cite[Chapter 7]{Gray} or \cite[\S 3.2]{Pinsker}. 

\item  Monotonicity:
If $K$ is a Markov kernel on $(Z,\mathcal C)$ then $D(\alpha K||\beta K)\le D(\alpha||\beta)$.
More generally $D\left(\alpha|_{\mathcal{G}}||\beta|_{\mathcal{G}})\le D(\alpha|_{\mathcal{H}}||\beta|_{\mathcal{H}}\right)$
for $\mathcal G\subseteq\mathcal H$, where $\alpha|_{\mathcal{F}}$ denote the restriction of a
measure $\alpha$ to a sub-$\sigma$-field $\mathcal{F}$.
See \cite[Corollary 7.2]{Gray}.
\end{enumerate}

We may consider the tail $\sigma$-field $\mathcal{T}_{x,\zeta}$ associated with the sequence of random variables $\xi_n^{x,\zeta}$. When $(M,\bar\lambda)=(X,\eta)$, we have $\zeta=x$ and we recover the random variable $\xi_n^x$ introduced in Subsection~\ref{subsec:tail}. The tail assumption \hyperlink{AssumpT}{(\textbf{T})} guarantees again that it fiberwise coincides with the invariant $\sigma$-field modulo null sets $\mathcal{T}_{x,\zeta}=\mathcal{I}_{x,\zeta}$.

In \cite{Derriennic}, mutual information $\mathrm{I}(\xi_{1}^{x,\zeta},\mathcal{T}_{x,\zeta})$
between the chain at time one and its tail $\sigma$-field $\mathcal{T}_{x,\zeta}$ is considered. 
We will show that the difference of Furstenberg entropy between $(Z,\lambda)$ and $(M,\bar\lambda)$ can be expressed as the conditional mutual information between $\xi_1^{x,\zeta}$ and the tail $\sigma$-fields $\mathcal{T}_{x,\zeta}$.
For $M=X$ we recover the formula of Subsection~\ref{subsec:formula}.

To clarify the relation between mutual information and KL-divergence, we have the following classical formula, reproduced directly from \cite[Section V]{Derriennic} for the convenience of the reader. 

\begin{lemma}\label{lem:DerV}
Let $(E,\mathcal{E},m)$ be a probability space. Let $F$ be a space equipped with a $\sigma$-field $\mathcal{F}$. Given for each $a\in E$,  a probability measure  $Q(a,\cdot)$ on $(F,\mathcal{F})$, consider the probability measures  $m':=\int_E Q(a,\cdot)dm(a)$ on $(F,\mathcal{F})$ and $\rho(A,B)=\int_A Q(a,B)dm(a)$ on $(E\times F, \mathcal{E}\otimes \mathcal{F})$.

Then the measure $\rho$ is absolutely continuous with respect to $m\otimes m'$ if and only if for $m$-almost every $a \in E$, the measure $Q(a,\cdot)$ is absolutely continuous with respect to $m'$, in which case 
\[
\frac{d\rho}{d m\otimes m'}(a,y)=\frac{dQ(a,\cdot)}{dm'}(y), \quad \textrm{for }m\otimes m'\textrm{ almost every }(a,y).
\]
Moreover, one has
\[
D(\rho|| m\otimes m')=\int_E D( Q(a,\cdot )||m')dm(a),
\]
i.e. both sides are finite and equal if and only if one of the sides is finite.
\end{lemma}

\begin{proof}[Sketch of proof of Lemma~\ref{lem:DerV}]
Let $\mathcal{F}_k$ be an increasing sequence of finite $\sigma$-fields containing sets of positive $m'$-measure, whose union generates $\mathcal{F}$ modulo $m$. Denote $F_k(y)$ the atom of $\mathcal{F}_k$ containing $y$. The map  $f_k(a,y):=\frac{Q(a,F_k(y))}{m'(F_k(y))}$ is a density of $\rho$ with respect to $m\otimes m'$ restricted to $\mathcal{E}\otimes \mathcal{F}_k$. The probability $\rho$ is absolutely continuous against $m\otimes m'$ if and only if the martingale $f_k$ converges in $L^1(m\otimes m')$. On the other hand $Q(a,\cdot)$ is absolutely continuous against $m'$ if and only if $f_k(a,\cdot)$ converges in $L^1(m')$ for $m$-almost every $a$. This gives the first statement. The second is a consequence of the Gelfand-Yaglom-Perez theorem~\cite[Theorem 2.4.2]{Pinsker}.
\end{proof}

 Let us define
\[
\mathrm{I}(\xi_1,\mathcal{T}|M,\bar\lambda):=\int_M\mathrm{I}(\xi_1^{x,\zeta},\mathcal{T}_{x,\zeta})d\bar\lambda(x,\zeta).
\]
For $g\in G$ and $(x,\zeta)\in M$,  set the notation
\begin{align}\label{eq:Dgx}
D_g(x,\zeta):=D\left(\bar{\mathbb{P}}_{\mu,(x,\zeta),g}|_{\mathcal{T}_{x,\zeta}}|| \bar{\mathbb{P}}_{\mu,(x,\zeta),e}|_{\mathcal{T}_{x,\zeta}}\right),
\end{align}
for the KL-divergence, restricted to the tail $\sigma$-field, of the fiberwise Markov chain started at the coset $L_xg$
with respect to the same chain started at $L_x$. And the notation for the integrated version:
\begin{align}\label{eq:DgT}
D(g,\mathcal{T}|M,\bar\lambda):=\int_M D_g(x,\zeta)d(g\bar\lambda)(x,\zeta).
\end{align}
As will appear in the next lemma, we have that 
$\mathrm{I}(\xi_1,\mathcal{T}|M,\bar\lambda)=\int_G D(g,\mathcal{T}|M,\bar\lambda) d\mu(g).$

For the proof of upper semicontinuity of Furstenberg entropy, we need the monotonicity in following lemma. 
Let us denote $\lambda=\int_M \lambda_{(x,\zeta)}d\bar\lambda(x,\zeta)$ and $\eta\varcurlyvee \mathbb{P}_\mu=\int_{M}\delta_{\{x\}}\times \mathbb{P}_\mu^{\pi(x),(x,\zeta)}d\bar\lambda(x,\zeta)$ the disintegrations with respect to the intermediate factor $(M,\bar\lambda)$. Set the notation for $n\ge 1$,
\[
D_g^{(n)}(x,\zeta)=D(P_{\mu,(x,\zeta)}^{n-1}(L_xg,\cdot)||P_{\mu,(x,\zeta)}^n(L_x,\cdot))=D\left( \xi_{n-1,g}^{x,\zeta}||\xi_{n,e}^{x,\zeta}\right).
\]

\begin{lemma}\label{derriennic-basic}
Let $(M,\bar\lambda)$ be an intermediate factor of the Bowen-Poisson bundle over a standard system as in (\ref{eq:inter-1}). Then 
For $\mu$-a.e. $g\in G$ and $\bar\lambda$-a.e. $(x,\zeta)$,
\[
D_g(x,\zeta)=\int_{G^{\mathbb{N}}}\log\frac{d(g\lambda)_{(x,\zeta)}}{d\lambda_{(x,\zeta)}}\left(\psi(x,\omega)\right)d\mathbb{P}_{\mu}^{\pi(x),(x,\zeta)}(\omega|\omega_1=g).
\]
And the mutual information satisfies
\[
\mathrm{I}(\xi_{1}^{x,\zeta},\mathcal{T}_{x,\zeta})=\int_{G^{\mathbb{N}}}\log\frac{d(\omega_{1}\lambda)_{(x,\zeta)}}{d\lambda_{(x,\zeta)}}\left(\psi(x,\omega)\right)d\mathbb{P}_{\mu}^{\pi(x),(x,\zeta)}(\omega).
\]
In particular, $\mathrm{I}(\xi_1,\mathcal{T}|M,\bar\lambda)=\int_G D(g,\mathcal{T}|M,\bar\lambda) d\mu(g).$
\end{lemma}

\begin{proof}
As $(Z,\lambda)$ is the Mackey point realisation of $(W_\Omega,\mathcal{I},\overline{\mathbb{P}}_\mu)$ we have for each $(x,\xi) \in W_{\Omega}$
\begin{align}\label{eq:invariantzeta}
\frac{d\overline{\mathbb{P}}_{\mu,(x,\zeta),g}|_{\mathcal{I}_{x,\zeta}}}{d\overline{\mathbb{P}}_{\mu,(x,\zeta),e}|_{\mathcal{I}_{x,\zeta}}}(x,\xi)=\frac{d(g\lambda)_{(x,\zeta)}}{d(\lambda)_{(x,\zeta)}}(\theta_M(x,\xi)).
\end{align}
By Corollary~\ref{cor:tail} (see also comment at the beginning of the present subsection),
we may replace in (\ref{eq:invariantzeta}) the invariant $\sigma$-field $\mathcal{I}_{x,\zeta}$
 by the tail $\sigma$-field
$\mathcal{T}_{x,\zeta}$ of $\left(\xi_{n}^{x,\zeta}\right)$. 
The formula for $D_g(x,\zeta)$ follows.

To derive the formula for mutual information, we apply Lemma~\ref{lem:DerV} with $E=L_x\backslash G$, $m$ the law of $\xi_1^{x,\zeta}$, $F$ the trajectory space $(L_x\backslash G)^{\mathbb{N}}$, $\mathcal{F}$ the tail $\sigma$-field $\mathcal{T}_{x,\zeta}$ and $Q(g,\cdot)=\overline{\mathbb{P}}_{\mu,(x,\zeta),g}|_{\mathcal{T}_{x,\zeta}}$ the restriction to $\mathcal{T}_{x,\zeta}$ of the law $\overline{\mathbb{P}}_{\mu,(x,\zeta),g}$ of a trajectory of the Markov chain on $L_x\backslash G$ started at $L_x g$, defined in Subsection~\ref{subsec:DoobM}. 
In this case, by Lemma~\ref{translate-1}, $m'=\overline{\mathbb{P}}_{\mu,(x,\zeta),e}|_{\mathcal{T}_{x,\zeta}}$ is the restriction to the tail $\sigma$-field of the law of the Markov chain trajectory started at $L_x$, i.e. the law of $\theta_M(x,L_x\omega)$, the random variable generating $\mathcal{T}_{x,\zeta}$, and $\rho$ is the joint law of $\xi_1^{x,\zeta}$ and $\theta_M(x,L_x\omega)$. In summary, $\rho=P(\xi_1^{x,\zeta},\theta_M(x,L_x\omega)$, $m=P(\xi_1^{x,\zeta})$ and $m'=P(\theta_M(x,L_x\omega))$. We get
\begin{align}
\mathrm{I}(\xi_{1}^{x,\zeta},\mathcal{T}_{x,\zeta})=\int_{(L_{x}\backslash G)^{\mathbb{N}}}\log\frac{d\overline{\mathbb{P}}_{\mu,(x,\zeta),\xi_{1}
}|_{\mathcal{T}_{x,\zeta}}}{d\overline{\mathbb{P}}_{\mu,(x,\zeta),e}|_{\mathcal{T}_{x,\zeta}}}d\overline{\mathbb{P}}_{\mu,(x,\zeta),e}.\label{DerriennicVprelim}
\end{align}

Finally we have $\mathbb{\overline{P}}_{\mu,(x,\zeta),e}=\vartheta\ast\mathbb{P}_\mu^{\pi(x),(x,\zeta)}$, where $\mathbb{P}_\mu^{\pi(x),(x,\zeta)}$ is the law of the twice Doob transformed trajectory on the group from identity, conditioned by $\eta_{\omega}=\eta^{\pi(x)}$ and $\psi_M(x,\omega)=(x,\zeta)$.
Together with~(\ref{eq:invariantzeta}) and the comment below, it follows that (\ref{DerriennicVprelim}) can be rewritten in the form as claimed. 
The last equality is obvious.
\end{proof}

We also record the following divergence and entropy formulae.

\begin{proposition}\label{entropy-formula-M}
Let $(M,\bar\lambda)$ be an intermediate factor of the Bowen-Poisson bundle over a standard system as in (\ref{eq:inter-1}). Then for every $g\in G$,
\[
D(g\lambda||\lambda)=D(g\bar\lambda||\bar\lambda)+D(g,\mathcal{T}|M,\bar\lambda). 
\]
It follows that
\[
h_\mu(Z,\lambda)=h_\mu(M,\bar{\lambda})+I(\xi_1,\mathcal{T}|M,\bar{\lambda}).
\]
\end{proposition} 

In the particular case of $M=X$, we get Proposition~\ref{entropy-formula}, as $h_{\mu}(X,\eta)=h_{\mu}(Y,\nu)$ since $\pi:X\to Y$ is measure preserving. 
\begin{proof}[Proof of Proposition \ref{entropy-formula-M}]
Using disintegration over $(M,\bar\lambda)$, the
Radon-Nikodym derivative can be written as 
\[
\frac{dg\lambda}{d\lambda}(\psi(x,\omega))=\frac{dg\bar\lambda}{d\bar\lambda}(x,\zeta)\frac{d\left(g\lambda\right)_{(x,\zeta)}}{d\lambda_{(x,\zeta)}}\left(\psi(x,\omega)\right).
\]
By the definition of $(Z,\lambda)\to (M,\bar\lambda)$, the fiber $(Z_{x,\zeta},\lambda_{x,\zeta})$ over a point $(x,\zeta)\in M$ can be viewed as a realization of $\bar{\mathbb{P}}_{\mu,(x,\zeta),e}$ restricted to the tail $\sigma$-field $\mathcal{T}_{x,\zeta}$. Then the expression for $D(g\lambda||\lambda)$ follows from the disintegration formula of KL-divergence.

By Lemma~\ref{standard-entropy} and the last equality in Lemma~\ref{derriennic-basic}, the formula for Furstenberg entropy follows from integrating both sides of the identity with respect to $\mu$. 
\end{proof}

\begin{lemma}\label{derriennic-basic(i)}
Let $(M,\bar\lambda)$ be an intermediate factor of the Bowen-Poisson bundle over a standard system as in (\ref{eq:inter-1}). Then 
 for every $g\in G$, the sequence of KL-divergence $D_g^{(n)}(x,\zeta)$ is
non-increasing. In particular  if $D_g^{(n_0)}(x,\zeta)<\infty$ for some $n_0$, we have
\begin{align}
D_g(x,\zeta)=\inf_{n}D_g^{(n)}(x,\zeta)=\lim_{n}D_g^{(n)}(x,\zeta).\label{DerriennicIII}
\end{align}
If moreover for some $n_0$ we have $\int_M D_g^{(n_0)}(x,\zeta)d(g\bar\lambda)(x,\zeta)<\infty$, then
\[
D(g,\mathcal{T}||M, \bar\lambda)=\lim_{n\to\infty} \int_M D_g^{(n)}(x,\zeta)d(g\bar\lambda)(x,\zeta),
\]
and $\lim_n$ may be replaced by $\inf_n$. 
\end{lemma}

\begin{proof}
The proof follows \cite[Section III]{Derriennic}, we sketch it here. Denote by $\mathcal{T}_{x,\zeta}^{(n)}$ the $\sigma$-field generated
by time $m\ge n$ of the Markov chain, $\mathcal{T}_{x,\zeta}^{(n)}=\sigma(\xi_m^{x.\zeta},m\ge n)$.  By the Markov property, we have 
\[
D_g^{(n)}(x,\zeta)=D\left(\bar{\mathbb{P}}_{\mu,(x,\zeta)}(\cdot|\omega_1=g)|_{\mathcal{T}^{(n)}_{x,\zeta}}|| \bar{\mathbb{P}}_{\mu,(x,\zeta)}(\cdot)|_{\mathcal{T}^{(n)}_{x,\zeta}}\right).
\]
The reason for writing $\bar{\mathbb{P}}_{\mu,(x,\zeta)}(\cdot|\omega_1=g)$ instead of $\bar{\mathbb{P}}_{\mu,(x,\zeta),g}$ is that in the latter notation, the Markov chain is at $L_xg$ at time $0$.
The claim that $D_g^{(n)}(x,\zeta)$ is non-increasing follows from monotonicity property of KL-divergence. To see the limit,  $\mathcal{T}_{x,\zeta}^{(n)}\downarrow \mathcal{T}_{x,\zeta}$, the densities 
$f_n=\frac{d\bar{\mathbb{P}}_{\mu,(x,\zeta)}(\cdot|\omega_1=g)|_{\mathcal{T}^{(n)}_{x,\zeta}}}{d\bar{\mathbb{P}}_{\mu,(x,\zeta)}(\cdot)|_{\mathcal{T}^{(n)}_{x,\zeta}}}$ form a reverse martingale under $\bar{\mathbb{P}}_{\mu,(x,\zeta)}$, which converges a.s.\ and in $L^1$ to the density on $\mathcal{T}_{x,\zeta}$.
Under the assumption $D_g^{(1)}(x,\zeta)<\infty$, the sequence $(f_n\log f_n)$ is uniformly integrable, and $D_g^{(n)}(x,\zeta)\downarrow D_g(x,\zeta)$. 

The integral formula follows from Lemma~\ref{entropy-formula-M} and the monotone convergence theorem.
\end{proof}

The following lemma is used only in the proof of Proposition~\ref{info-conti}. Let $(Y,\nu)$ be a $\mu$-boundary and $(X,\eta)\to (Y,\nu)$ a measure-preserving extension. Let $L:X\to \mathrm{Sub}(G)$ be a $G$-equivariant map with $L_x<\mathrm{Stab}_G(x)$. Using the same notation as at the beginning of this Subsection, in the special case where $M=X$, denote by $\xi_{n,g}^x$ the time $n$ step of the Markov chain starting at $L_xg$. Denote by $\mathbb{P}_{\mu,n}^{y}(\cdot|\omega_1=g)$ the law of $\omega_n$ under the Doob-transformed random walk $\mathbb{P}_{\mu}^{y}$ on $G$, conditioned on the first step being $g$. 

\begin{lemma}\label{derriennic-basicY}
In this setting,
\begin{description}
\item [{(i)}] 
$D\left(\xi_{n-1,g}^{x}||\xi_{n,e}^{x}\right)\le D\left(\mathbb{P}_{\mu,n}^{\pi(x)}(\cdot|\omega_1=g)||\mathbb{P}_{\mu,n}^{\pi(x)}\right)$. 
\item [{(ii)}] Integrated over $(Y,\nu)$, we have 
\begin{equation}
\int_{Y}D\left(\mathbb{P}_{\mu,n}^{y}(\cdot|\omega_1=g)||\mathbb{P}_{\mu,n}^{y}\right)dg\nu(y)
=D\left(\delta_g\ast\mu^{\ast(n-1)}||\mu^{\ast (n)}\right)-D(g\nu||\nu).\label{eq:doob-I}
\end{equation}
\item [{(iii)}] 
$D(g,\mathcal{T}|X,\eta)\le D(g\nu_B||\nu_B)-D(g\nu||\nu)$.
\end{description}
\end{lemma}

\begin{proof}
(i). Recall that $\mathbb{P}_{\mu}^{\pi(x)}$ is the law of the Doob
transformed  Markov chain on $G^{\mathbb{N}}$
conditioned by $\{\boldsymbol{\beta}_{Y}(\omega)=\pi(x)\}$. The law
of $\left(\xi_{n-1,g}^{x},\xi_{n}^{x}\right)$ can be viewed as the restriction
of the joint law $\left(\mathbb{P}_{\mu,n}^{\pi(x)}(\cdot|\omega_1=g),\mathbb{P}_{\mu,n}^{\pi(x)}\right)$
on $G\times G$ to the sub-$\sigma$-field of $L_{x}$-invariant subsets,
that is, measurable sets $A\subseteq G\times G$ such that $L_{x}A=A$.
The inequality follows from the monotonicity properties of KL-divergence. 

(ii). In what follows $g\in G$ is a fixed element. Denote by $P$ the law $\mathbb{P}_{\mu,n}^{y}(\cdot|\omega_1=g)$
and by $Q$ the law $\mathbb{P}_{\mu,n}$. 
Similarly for $y$ in $Y$, denote by $P^{y}$ the Doob transformed law $\mathbb{P}_{\mu,n}\left(\cdot|\omega_1=g)^{y}\right)$
and by $Q^{y}$ the Doob transformed law $\mathbb{P}_{\mu,n}^{y}$.

We have $\frac{dP^{y}}{dQ^{y}}(h)=\frac{dP}{dQ}(h)\frac{1}{\varphi_{g}(y)}$ with $\varphi_g(y)=\frac{dg\nu}{d\nu}(y)$.
Also note that $\mathbb{P}_{\mu,n}(\cdot|\omega_1=g)=\delta_g\ast\mu^{(n-1)}$
Thus 
\begin{align*}
\int_{Y}D\left(\mathbb{P}_{\mu,n}^{y}(\cdot|\omega_1=g)||\mathbb{P}_{\mu,n}^{y}\right)dg.\nu(y) & =\int_{Y}\int_{G}\log\frac{dP}{dQ}dP^{y}dg.\nu(y)-\int_{Y}\int_{G}\log\varphi_{g}(y)dP^{y}(h)dg.\nu(y)\\
 & =\int_{G}\log\frac{dP}{dQ}dP-\int_{Y}\int_{G}\log\varphi_{g}(y)d\mu^{n-1}(g^{-1}h)\varphi_{h}(y)dg.\nu(y)\\
 & =D\left(\delta_g\ast\mu^{\ast(n-1)}||\mu^{\ast (n)}\right)-\int_{Y}\log\varphi_{g}(y)dg.\nu(y) \\
 & =D\left(\delta_g\ast\mu^{\ast(n-1)}||\mu^{\ast (n)}\right)-D(g.\nu||\nu).
\end{align*}
In the second last line we used the harmonicity of $\varphi_{h}$, which
implies $\int_{G}\varphi_{h}(y)d\mu^{n-1}(g^{-1}h)=\varphi_{g}(y)$.
We conclude by Lemma~\ref{standard-entropy}.

(iii). We have, using Lemma~\ref{derriennic-basic(i)} and the previous points
\begin{align*}
D(g,\mathcal{T}|X,\eta) & = \int_XD_g(x)dg\eta(x)= \int_X \inf_n D_g^{(n)}(x)dg\eta(x) \\
& \le  \int_X \inf_n D\left(\mathbb{P}_{\mu,n}^{\pi(x)}(\cdot|\omega_1=g)||\mathbb{P}_{\mu,n}^{\pi(x)}\right) dg\eta(x) \\ 
&\le \inf_n \int_{Y}D\left(\mathbb{P}_{\mu,n}^{y}(\cdot|\omega_1=g)||\mathbb{P}_{\mu,n}^{y}\right)dg\nu(y) \\
& \le \inf_n D\left(\delta_g\ast\mu^{\ast(n-1)}||\mu^{\ast (n)}\right)-D(g\nu||\nu)=D(g\nu_B||\nu_B)-D(g\nu||\nu).
\end{align*}
\end{proof}

\subsubsection{A random walk entropy formula for countable groups\label{subsec:random-walk-entropy}}

In this subsection we assume that $G$ is a countable group. 
We recall that $\xi_n^{x,\zeta}$ is the time $n$ position of the Markov chain of law $\overline{\mathbb{P}}_{\mu,(x,\zeta)}$.  It is obtained from the random walk by taking Doob transform according both to $\eta_\omega=\eta^{\pi(x)}$ and $\psi_M(x,\omega)=(x,\zeta)$, as explained in Subsection~\ref{subsec:DoobM}.
Let us denote 
\[
I(\xi_{1},\xi_{n}|M,\bar\lambda):=\int_{M}I(\xi_{1}^{x,\zeta},\xi_{n}^{x,\zeta})d\bar\lambda(x)\quad\textrm{and}\quad H(\xi_{n}|M,\bar\lambda):=\int_{M}H(\xi_{n}^{x,\zeta})d\bar\lambda(x,\zeta),
\]
the average conditional mutual information and average conditional entropy. We have the following entropy formulae, generalising Theorem~\ref{thm:entropy}.

\begin{theo}\label{thm:entropy-app} Assume
$G$ is a countable group endowed with a probability measure $\mu$
of finite Shannon entropy. Let $(M,\bar\lambda)$ be an intermediate factor as in (\ref{eq:inter-1}). Then the entropy of the Poisson
bundle $\left(Z,\lambda\right)$ satisfies 
\begin{align*}
h_{\mu}(Z,\lambda) & =h_{\mu}(M,\bar\lambda)+\lim_{n\to\infty}\int_{M}\left(H\left(\xi_{n}^{x,\zeta}\right)-H\left(\xi_{n-1}^{x,\zeta}\right)\right)d\bar\lambda(x,\zeta)\\
 & =h_{\mu}(M,\bar\lambda)+\lim_{n\to\infty}\frac{1}{n}H\left(\xi_{n}|M,\bar\lambda\right).
\end{align*}
In both lines $\lim$ can be replaced by $\inf_{n\in\mathbb{N}}$.
\end{theo}

\begin{proof}[Proof of Theorem~\ref{thm:entropy}]
Apply Theorem~\ref{thm:entropy-app} in the case $(M,\bar\lambda)=(X,\eta)$, and notice that $h_\mu(X,\eta)=h_\mu(Y,\nu)$ since $(X,\eta)\to (Y,\nu)$ is a measure preserving extension.
\end{proof}

For a
convolution $\mu$-random walk $\left(\xi_{n}\right)_{n=0}^{\infty}$
on a discrete group, it is classical that 
\[
{\rm I}(\xi_{1},\xi_{n})=H\left(\mu^{(n)}\right)-H\left(\mu^{(n-1)}\right)
\]
where $H(p)=-\sum_{g\in G}p(g)\log p(g)$ is the Shannon entropy of
the discrete probability measure $p$ on $G$. In the case of a countable
group $G$ endowed with a finite entropy probability measure $\mu$,
we obtain a bundle version of this formula. 

\begin{lemma}\label{lem-mutual}Assume $G$ is a countable group
endowed with a probability measure $\mu$ of finite Shannon entropy.
Let $(M,\bar\lambda)$ be an arbitrary intermediate factor as in (\ref{eq:inter-1}). Then
\[
\mathrm{I}\left(\xi_{1},\xi_{n}|M,\bar\lambda\right)=H\left(\xi_{n}|M,\bar\lambda\right)-H\left(\xi_{n-1}|M,\bar\lambda\right).
\]
\end{lemma} The proof follows standard calculations as in \cite[Section I]{Derriennic}. 
\begin{proof}
The sequence $(\xi_{n}^{x,\zeta})_{n=0}^{\infty}$ is the Markov chain in
the coset space $L_{x}\backslash G$, started at the identity coset
$L_{x}$ and with transition probabilities given by Proposition~\ref{lem:MarkovM}.
By definition, we have

\[
\mathrm{I}\left(\xi_{1}^{x,\zeta},\xi_{n}^{x,\zeta}\right)=\sum_{L_xh}\sum_{L_xg}\log\left(\frac{P_{\mu,(x,\zeta)}^{n-1}(L_xh,L_xg)}{P_{\mu,(x,\zeta)}^{n}(L_{x},L_xg)}\right)P_{\mu,(x,\zeta)}^{1}(L_{x},L_xh)P_{\mu,(x,\zeta)}^{n-1}(L_xh,L_xg)
\]
We can split the logarithm to get $\mathrm{I}\left(\xi_{1}^{x,\zeta},\xi_{n}^{x,\zeta}\right)={\rm I}_{x,\zeta}+{\rm II}_{x,\zeta}$,
where ${\rm I}_{x,\zeta}=H(\xi_{n}^{x,\zeta})$, so we get $\int_{M}{\rm I}_{x,\zeta}d\bar\lambda(x,\zeta)=H\left(\xi_{n}|M,\bar\lambda\right)$.
There remains to show $\int_{M}{\rm II}_{x,\zeta}d\bar\lambda(x,\zeta)=-H\left(\xi_{n-1}|M,\bar\lambda\right)$,
where 
\[
{\rm II}_{x,\zeta}=\sum_{s\in G}\mu(s)\frac{ds\bar\lambda}{d\bar\lambda}(x,\zeta)\sum_{L_xg\in L_{x}\backslash G}\log\left(P_{\mu,(x,\zeta)}^{n-1}(L_{x}s,L_xg)\right)P_{\mu,(x,\zeta)}^{n-1}(L_{x}s,L_xg)
\]
because the transition probabilities satisfy 
\[
P_{\mu,(x,\zeta)}^{1}(L_{x},L_xh)=\sum_{\{s:L_xh=L_{x}s\}}\mu(s)\frac{ds\bar\lambda}{d\bar\lambda}(x,\zeta).
\]
Now by equivariance, see (\ref{eq:invariance}) in the proof of Lemma~\ref{translate-1},
\[
P_{\mu,(x,\zeta)}^{n-1}(L_{x}s,L_{x}h)=P_{\mu,s^{-1}(x,\zeta)}^{n-1}\left(L_{s^{-1}x},L_{s^{-1}x}s^{-1}h\right).
\]
Then $s^{-1}h$ has the law of $\xi_{n-1}^{s^{-1}(x,\zeta)}$, in other terms of the
position at time $n-1$ of a twice Doob transformed trajectory in the coset space $L_{s^{-1}x}\backslash G$ conditioned on $s^{-1}(x,\zeta) \in \varrho^{-1}(s^{-1}x)$.
We get 
\[
{\rm II}_{x,\zeta}=-\sum_{s\in G}\mu(s)\frac{ds\bar\lambda}{d\bar\lambda}(x,\zeta)H\left(\xi_{n-1}^{s^{-1}(x,\zeta)}\right).
\]
Once integrated over $(M,\bar\lambda)$ we obtain
\[
\int_M{\rm II}_{x,\zeta}d\bar\lambda(x,\zeta)=-\sum_{s\in G}\mu(s)\int_M \frac{ds\bar\lambda}{d\bar\lambda}(x,\zeta)H\left(\xi_{n-1}^{s^{-1}(x,\zeta)}\right)d\bar\lambda(x,\zeta)=-\int_M H\left(\xi_{n-1}^{x,\zeta}\right) d\bar\lambda(x,\zeta)
\]
since $(M,\bar\lambda)$ is $\mu$-stationary.
\end{proof}

\begin{proof}[Proof of Theorem \ref{thm:entropy-app}]
By Lemma~\ref{derriennic-basic} (i), we have 
\[
{\rm I}(\xi_{1},\mathcal{T}|M,\bar\lambda)=\int_{M}{\rm I}(\xi_{1}^{x,\zeta},\mathcal{T}_{x,\zeta})d\bar\lambda(x,\zeta)=\int_{M}\inf_{n}{\rm I}(\xi_{1}^{x,\zeta},\xi_{n}^{x,\zeta})d\bar\lambda(x,\zeta)=\lim_{n\to\infty}\int_{M}{\rm I}(\xi_{1}^{x,\zeta},\xi_{n}^{x,\zeta})d\bar\lambda(x,\zeta)
\]
as the pointwise convergence is dominated by the integrable function 
\[
H(\xi_{1}^{x,\zeta})={\rm I}(\xi_{1}^{x,\zeta},\xi_{1}^{x,\zeta})\ge\inf_{n}{\rm I}(\xi_{1}^{x,\zeta},\xi_{n}^{x,\zeta})=\lim_{n\to\infty}{\rm I}(\xi_{1}^{x,\zeta},\xi_{n}^{x,\zeta}).
\]
Then Proposition \ref{entropy-formula-M} and Lemma \ref{lem-mutual}
give the first formula. The second formula follows by Cesaro average
since $H(\xi_{n}|M,\bar\lambda)=\sum_{k=1}^{n}H(\xi_{k}|M,\bar\lambda)-H(\xi_{k-1}|M,\bar\lambda)$ and
$H(\xi_{k}|M,\bar\lambda)-H(\xi_{k-1}|M,\bar\lambda)={\rm I}(\xi_{1},\xi_{k}|M,\bar\lambda)$ is a non-increasing
sequence tending to ${\rm I}(\xi_{1},\mathcal{T}|M,\bar\lambda)$. 
\end{proof}

\subsubsection{A Shannon-Breiman-McMillan type theorem for standard Poisson bundles}

In the case of a countable group $G$ with ergodic action on $(X,\eta)$, Kingman's subadditive ergodic
theorem provides a Shannon-Breiman-McMillan type theorem for the Markov chain $\overline{\mathbb{P}}_{\mu,(x,\zeta)}$.

\begin{proposition}[Shannon's theorem]\label{shannon} Assume $G$
is a countable group endowed with a probability measure $\mu$ of
finite Shannon entropy, acting ergodically on a probability space $(X,\eta)$. Let $(M,\bar\lambda)$ be an arbitrary intermediate factor as in (\ref{eq:inter-1}). Then 
\[
\lim_{n\to\infty}-\frac{1}{n}\log P_{\mu,{(x,\zeta)}}^{n}\left(L_{x},\omega_{n}\right)=\lim_{n\to\infty}\frac{1}{n}H(\xi_{n}|M,\bar\lambda)=h_{\mu}(Z,\lambda)-h_{\mu}(M,\bar\lambda),
\]
$\eta\varcurlyvee\mathbb{P}_{\mu}$-a.s. and in $L^{1}$ limits.

\end{proposition} 
\begin{proof}
By Fact \ref{pmp}, the skew transformation 
\[
T:\left(x,\left(\omega_{1},\omega_{2},\ldots\right)\right)\mapsto\left(\omega_{1}^{-1}.x,\left(\omega_{1}^{-1}\omega_{2},\omega_{1}^{-1}\omega_{3},\ldots\right)\right)
\]
is ergodic p.m.p. on the space $\left(X\times G^{\mathbb{N}},\eta\varcurlyvee\mathbb{P}_{\mu}\right)$.
Consider the functions $f_{n}:X\times G^{\mathbb{N}}\to\mathbb{R}$
given by $f_{n}(x,\omega)=P_{\mu,(x,\zeta)}^{n}\left(L_{x},L_{x}\omega_{n}\right)$.
Then 
\begin{align*}
f_{n+m}(x,\omega) & \ge P_{\mu,{(x,\zeta)}}^{n}\left(L_{x},L_{x}\omega_{n}\right)P_{\mu,{(x,\zeta)}}^{m}\left(L_{x}\omega_{n},L_{x}\omega_{n+m}\right)\\
 & =P_{\mu,{(x,\zeta)}}^{n}\left(L_{x},L_{x}\omega_{n}\right)P_{\mu,{\omega_{n}^{-1}(x,\zeta)}}^{m}\left(L_{\omega_{n}^{-1}.x},L_{\omega_{n}^{-1}.x}\omega_{m}\right)\ \mbox{by }(\text{\ref{eq:invariance}})\\
 & =f_{n}(x,\omega)f_{m}\left(T^{n}\left(x,\omega\right)\right).
\end{align*}
Kingman's subadditive ergodic theorem applied to $-\log f_{n}$, gives
a constant $h$ with $\lim_{n\to\infty}-\frac{1}{n}\log f_{n}=h$,
both $\mbox{\ensuremath{\eta\varcurlyvee\mathbb{P}_{\mu}}}$-a.s.
and in $L^{1}$. Moreover, disintegration (\ref{eq:desintM}) gives 
\begin{align*}
\int_{X\times G^{\mathbb{N}}}-\log f_{n}d\eta\varcurlyvee\mathbb{P}_{\mu} & =\int_{X}\int_{\varrho^{-1}(x)}\int_{G^{\mathbb{N}}}-\log\left(P_{\mu,(x,\zeta)}^{n}(L_{x},L_{x}\omega_{n})\right)d\overline{\mathbb{P}}_{\mu,(x,\zeta)}(\omega)d\bar\lambda_x(\zeta)d\eta(x)\\
 & =\int_{X}\int_{\varrho^{-1}(x)}\sum_{\omega_{n}\in G}-\log\left(P_{\mu,(x,\zeta)}^{n}(L_{x},L_{x}\omega_{n})\right)P_{\mu,(x,\zeta)}^{n}(L_{x},L_{x}\omega_{n})d\bar\lambda_{x}(\zeta)d\eta(x)\\
 & =\int_{M}H(\xi_{n}^{x,\zeta})d\bar\lambda(x,\zeta)=H(\xi_{n}|M,\bar\lambda).
\end{align*}
Therefore 
\[
h=\lim_{n\to\infty}\frac{1}{n}\int_{X\times G^{\mathbb{N}}}-\log f_{n}d\eta\varcurlyvee\mathbb{P}_{\mu}=\lim_{n\to\infty}\frac{1}{n}H(\xi_{n}|M,\bar\lambda)=h_{\mu}(Z,\lambda)-h_{\mu}(M,\bar\lambda).
\]
by Theorem~\ref{thm:entropy-app}. 
\end{proof}
Note that the proof above shows subadditivity of the sequence $n\mapsto H(\xi_{n}|X)$. For completeness, we also record the following proposition. Recall that $\tilde\psi:X\times G^{\mathbb{N}} \to M$ is the map $\tilde\psi(x,\omega)=\varrho_M\circ \psi(x,{\rm bnd}(\omega))$.

\begin{proposition}[Consequence of Birkhoff's pointwise ergodic
theorem]\label{birkhoff}

Let $(M,\bar\lambda)$ be an intermediate factor in (\ref{eq:inter-1}).
If $(X,\eta)$ is an ergodic stationary $(G,\mu)$-system, then for $\eta\varcurlyvee\mathbb{P}_{\mu}$-a.e.
$\left(x,\omega\right)\in X\times G^{\mathbb{N}}$, we have 
\[
\lim_{n\to\infty}\frac{1}{n}\log\frac{d\omega_{n}\bar\lambda}{d\bar\lambda}\left(\tilde\psi\left(x,\omega\right)\right)=h_{\mu}(M,\bar\lambda).
\]

\end{proposition} 

\begin{proof}
The telescoping argument is adapted from the proof of \cite[Lemma 4.2]{KaimanovichHyperbolic}.
Given $\omega$, let $g_{k}=\omega_{k-1}^{-1}\omega_{k}$ be its $k$-th
increment, with $\omega_{0}=id$. The Radon-Nikodym derivative can
be rewritten as a product: 
\begin{align*}
\frac{d\omega_{n}\bar\lambda}{d\bar\lambda}\left(\tilde\psi\left(x,\omega\right)\right) & =\prod_{k=1}^{n}\frac{d\omega_{k}\bar\lambda}{d\omega_{k-1}\bar\lambda}\left(\tilde\psi\left(x,\omega\right)\right)\\
 & =\prod_{k=1}^{n}\frac{dg_{k}\bar\lambda}{d\bar\lambda}\left(\tilde\psi\left(\omega_{k-1}^{-1}x,\omega_{k-1}^{-1}\omega\right)\right)\quad \mbox{(by equivariance)}\\
 & =\prod_{k=1}^{n}\frac{d\left(T^{k-1}(x,\omega)\right){}_{1}\bar\lambda}{d\bar\lambda}\left(\tilde{\psi}\left(T^{k-1}(x,\omega)\right)\right)
\end{align*}
where $T(x,\omega)=(\omega_{1}^{-1}x,(\omega_{1}^{-1}\omega_{2},\omega_{1}^{-1}\omega_{3},\dots))$
is the skew transformation defined in Section~\ref{subsec:Stationary-joining}.
Let $f(x,\omega):=\log\frac{d\omega_{1}\bar\lambda}{d\bar\lambda}\left(\tilde{\psi}(x,\omega)\right)$.
By Fact \ref{pmp}, we can apply Birkhoff's pointwise ergodic theorem
to $T$ and deduce almost surely, 
\begin{align*}
\lim_{n\to\infty}\frac{1}{n}\log\frac{d\omega_{n}\bar\lambda}{d\bar\lambda}\left(\tilde\psi\left(x,\omega\right)\right)=\lim_{n\to\infty}\frac{1}{n}\sum_{k=1}^{n}f\left(T^{k-1}\left(x,\omega\right)\right)=\int_{X\times G^{\mathbb{N}}}f(x,\omega)d\eta\varcurlyvee\mathbb{P}_{\mu}=h_{\mu}(M,\bar\lambda)
\end{align*}
by Lemma~\ref{standard-entropy}.
\end{proof}

\subsection{A proof of the ray criterion\label{subsec:shannon-ray}}
\begin{proof}[Proof of Ray approximation Theorem \ref{ray} ]
It suffices to show that $h(M,\bar{\lambda})=h(Z,\lambda)$. Suppose
on the contrary $h\left(Z,\lambda\right)-h\left(M,\bar{\lambda}\right)=\delta>0$.
Take $\epsilon=\frac{1}{3}\delta$, we will omit the reference to
$\epsilon$ in the notation for $A_{n}^{\epsilon}$ in what follows.
By Corollary \ref{doob-zero}, for any $p>0$, there is a subset $V$
of $W_{\Omega}$ with measure $\overline{\mathbb{P}}_{\mu}(V)>1-p$
and a constant $N=N_{\epsilon,p}<\infty$ such that for all $\left(x,L_{x}\omega\right)\in V$
and $n\ge N$, we have 
\begin{align}
P_{\mu,(x,\zeta)}^{n}\left(L_{x},L_{x}\omega_{n}\right)\le e^{-(\delta-\epsilon)n}=e^{-2n\epsilon},\label{0entropy}
\end{align}
where we denoted $(x,\zeta)=\theta_{M}(x,L_{x}\omega)$. Take
a union bound over time $m$ in the expression of probability in (i),
we have 
\begin{multline*}
\overline{\mathbb{P}}_{\mu}\left(\exists m\ge n:\ L_{x}\omega_{m}\in A_{n}(x,\zeta)|(x,\zeta)\in U\right)\\
\le\frac{1}{\bar{\lambda}(U)}\overline{\mathbb{P}}_{\mu}\left(\exists m\ge n:\ L_{x}\omega_{m}\in A_{n}\left(x,\zeta\right),(x,\zeta)\in U,(x,L_{x}\omega)\in V\right)\\
+\frac{1}{\bar{\lambda}(U)}\overline{\mathbb{P}}_{\mu}\left(\left(x,\zeta\right)\notin V\right)\\
\le\frac{p}{\bar{\lambda}(U)}+\sum_{m=n}^{\infty}\overline{\mathbb{P}}_{\mu}\left(L_{x}\omega_{m}\in A_{n}\left(x,\zeta\right),(x,L_{x}\omega)\in V|\left(x,\zeta\right)\in U\right).
\end{multline*}

To estimate the conditional probabilities that appear as summands,
we use the disintegration $\overline{\mathbb{P}}_{\mu}=\int_{M}\overline{\mathbb{P}}_{\mu,(x,\zeta)}^{}d\bar{\lambda}(x,\zeta)$.
We have for $(x,\zeta)\in U$ and $m\ge N$, 
\begin{align*}
\overline{\mathbb{P}}_{\mu} & \left(L_{x}\omega_{m}\in A_{n}\left(x,\zeta\right),(x,L_{x}\omega)\in V|(x,\zeta)\in U\right)\\
 & \le\frac{1}{\bar{\lambda}(U)}\int_{M}{\bf 1}_{U}(x,\zeta)\overline{\mathbb{P}}_{\mu,(x,\zeta)}^{}\left(L_{x}\omega_{m}\in A_{n}\left(x,\zeta\right),(x,L_{x}\omega)\in V\right)d\bar{\lambda}(x,\zeta)\\
 & \le\left|A_{n}\left(x,\zeta\right)\right|e^{-2m\epsilon}\quad\textrm{by }(\ref{0entropy}).
\end{align*}
Condition (ii) gives the upper bound $\left|A_{n}\left(x,\zeta\right)\right|\le e^{\frac{3}{2}n\epsilon}$
for large $n$. Then 
\[
\sum_{m=n}^{\infty}\overline{\mathbb{P}}_{\mu}\left(L_{x}\omega_{m}\in A_{n}\left(x,\zeta\right),(x,L_{x}\omega)\in V|\left(x,\zeta\right)\in U\right)\le ce^{-\frac{1}{2}\epsilon n}.
\]
Since $p$ can be arbitrarily small and $\bar{\lambda}(U)>0$, we
get the limit 
\[
\lim_{n\to\infty}\overline{\mathbb{P}}_{\mu}\left(\exists m\ge n:\ L_{x}\omega_{m}\in A_{n}(x,\zeta)|(x,\zeta)\in U\right)=0.
\]
This contradicts the strictly positive limit in (i). 
\end{proof}
 \bibliographystyle{alpha}
\bibliography{IRS}

@book{Pinsker,
	author = {Pinsker, M. S.},
	publisher = {Holden-Day},
	title = {Information and information stability of random variables and processes},
	year = {1964}}

@incollection{Furman,
	author = {Alex Furman},
	doi = {https://doi.org/10.1016/S1874-575X(02)80014-5},
	editor = {B. Hasselblatt and A. Katok},
	issn = {1874-575X},
	pages = {931-1014},
	publisher = {Elsevier Science},
	series = {Handbook of Dynamical Systems},
	title = {Chapter 12 Random walks on groups and random transformations},
	url = {https://www.sciencedirect.com/science/article/pii/S1874575X02800145},
	volume = {1},
	year = {2002}}

@unpublished{RGY,
	author = {Ron-George, L. and Yadin, A.},
	note = {arXiv:2106.10172},
	title = {Full realization of ergodic IRS entropy in $SL(2,\mathbb{Z})$ and free groups},
	year = {2021}}

@inproceedings{Ledrappier84,
	address = {Berlin, Heidelberg},
	author = {Ledrappier, F.},
	booktitle = {{\'E}cole d'{\'E}t{\'e} de Probabilit{\'e}s de Saint-Flour XII - 1982},
	editor = {Hennequin, P. L.},
	isbn = {978-3-540-39109-8},
	pages = {305--396},
	publisher = {Springer Berlin Heidelberg},
	title = {Quelques proprietes des exposants caracteristiques},
	year = {1984}}

@article{GuivarchRaugi81,
	author = {Guivarc'h, Yves and Raugi, Albert},
	eid = {2},
	journal = {Publications des s\'eminaires de math\'ematiques et informatique de Rennes},
	language = {fr},
	number = {1},
	publisher = {D\'epartement de Math\'ematiques et Informatique, Universit\'e de Rennes},
	title = {Fronti\`ere de {Furstenberg,} propri\'et\'es de contraction et th\'eor\`emes de convergence},
	url = {http://www.numdam.org/item/PSMIR_1981___1_A2_0/},
	year = {1981}}

@incollection{FurstenbergGlasner,
	author = {Furstenberg, Harry and Glasner, Eli},
	booktitle = {Dynamical numbers -- interplay between dynamical systems and number theory},
	editor = {Amer. Math. Soc., Providence, RI},
	pages = {1-28},
	publisher = {Contemp. Math.},
	title = {Stationary dynamical systems},
	year = {2010}}

@article{Bader-Shalom,
	author = {Bader, Uri and Shalom, Yehuda},
	date = {2006/02/01},
	doi = {10.1007/s00222-005-0469-5},
	id = {Bader2006},
	isbn = {1432-1297},
	journal = {Inventiones mathematicae},
	number = {2},
	pages = {415--454},
	title = {Factor and normal subgroup theorems for lattices in products of groups},
	url = {https://doi.org/10.1007/s00222-005-0469-5},
	volume = {163},
	year = {2006}}

@article{BH21,
	author = {Boutonnet, R\'{e}mi and Houdayer, Cyril},
	doi = {10.1007/s10240-021-00122-8},
	fjournal = {Publications Math\'{e}matiques. Institut de Hautes \'{E}tudes Scientifiques},
	issn = {0073-8301},
	journal = {Publ. Math. Inst. Hautes \'{E}tudes Sci.},
	mrclass = {22E40 (22D10 22D25 37A55 46L30)},
	mrnumber = {4292738},
	mrreviewer = {Anton Deitmar},
	pages = {1--46},
	title = {Stationary characters on lattices of semisimple {L}ie groups},
	url = {https://doi.org/10.1007/s10240-021-00122-8},
	volume = {133},
	year = {2021}}

@article{BBHP,
	author = {Bader, Uri and Boutonnet, R\'{e}mi and Houdayer, Cyril and Peterson, Jesse},
	doi = {10.1007/s00222-022-01117-w},
	fjournal = {Inventiones Mathematicae},
	issn = {0020-9910},
	journal = {Invent. Math.},
	mrclass = {22D10 (20C15 22D25 22E40 37A55 43A35 46L10)},
	mrnumber = {4462622},
	number = {3},
	pages = {929--985},
	title = {Charmenability of arithmetic groups of product type},
	url = {https://doi.org/10.1007/s00222-022-01117-w},
	volume = {229},
	year = {2022}}

@article{KQ19,
	author = {Kechris, Alexander S. and Quorning, Vibeke},
	doi = {10.4171/ggd/517},
	fjournal = {Groups, Geometry, and Dynamics},
	issn = {1661-7207},
	journal = {Groups Geom. Dyn.},
	mrclass = {37A15 (20E06 20E22 20F06 22F10 28D15)},
	mrnumber = {4033500},
	mrreviewer = {Francesco G. Russo},
	number = {4},
	pages = {1151--1193},
	title = {Co-induction and invariant random subgroups},
	url = {https://doi.org/10.4171/ggd/517},
	volume = {13},
	year = {2019}}

@book{Kifer,
	author = {Kifer, Yuri},
	doi = {10.1007/978-1-4684-9175-3},
	isbn = {0-8176-3319-7},
	mrclass = {58F11 (28Dxx 60J27)},
	mrnumber = {884892},
	mrreviewer = {Karl Petersen},
	pages = {x+210},
	publisher = {Birkh\"{a}user Boston, Inc., Boston, MA},
	series = {Progress in Probability and Statistics},
	title = {Ergodic theory of random transformations},
	url = {https://doi.org/10.1007/978-1-4684-9175-3},
	volume = {10},
	year = {1986}}

@article{NZ4,
	author = {Nevo, Amos and Zimmer, Robert J.},
	date = {2002/12/01},
	doi = {10.1007/BF02786645},
	id = {Nevo2002},
	isbn = {1565-8538},
	journal = {Journal d'Analyse Math{\'e}matique},
	number = {1},
	pages = {93--104},
	title = {A generalization of the intermediate factors theorem},
	url = {https://doi.org/10.1007/BF02786645},
	volume = {86},
	year = {2002}}

@article{Cannizzo,
	auther = {Cannizzo, J.},
	author = {Cannizzo, Jan},
	journal = {Dissertation University of Ottawa},
	title = {Schreier Graphs and Ergodic Properties of Boundary Actions},
	year = {2014}}

@article{CartwrightSoardi,
	author = {Cartwright, Donald I. and Soardi, P. M.},
	doi = {10.2307/2048184},
	fjournal = {Proceedings of the American Mathematical Society},
	issn = {0002-9939},
	journal = {Proc. Amer. Math. Soc.},
	mrclass = {60J50 (05C05 43A05 60J15)},
	mrnumber = {984784},
	mrreviewer = {Wolfgang Woess},
	number = {3},
	pages = {817--823},
	title = {Convergence to ends for random walks on the automorphism group of a tree},
	url = {https://doi.org/10.2307/2048184},
	volume = {107},
	year = {1989}}

@book{Derriennic,
	author = {Derriennic, Y.},
	mrclass = {60J15 (60B15 60F05)},
	mrnumber = {880018},
	mrreviewer = {M. Iosifescu},
	pages = {i+44},
	publisher = {Universit\'{e} de Rennes I, Institut de Recherche Math\'{e}matique de Rennes, Rennes},
	series = {Publications de l'Institut de Recherche Math\'{e}matique de Rennes. [Publications of the Rennes Institute of Mathematical Research]},
	title = {Entropie, th\'{e}or\`emes limite et marches al\'{e}atoires},
	year = {1985}}

@article{Furstenberg1,
	author = {Furstenberg, Harry},
	doi = {10.2307/1970220},
	fjournal = {Annals of Mathematics. Second Series},
	issn = {0003-486X},
	journal = {Ann. of Math. (2)},
	mrclass = {22.65},
	mrnumber = {146298},
	mrreviewer = {S. Helgason},
	pages = {335--386},
	title = {A {P}oisson formula for semi-simple {L}ie groups},
	url = {https://doi.org/10.2307/1970220},
	volume = {77},
	year = {1963}}

@article{Furstenberg2,
	author = {Furstenberg, Harry},
	doi = {10.2307/1993589},
	fjournal = {Transactions of the American Mathematical Society},
	issn = {0002-9947},
	journal = {Trans. Amer. Math. Soc.},
	mrclass = {60.08 (60.66)},
	mrnumber = {163345},
	pages = {377--428},
	title = {Noncommuting random products},
	url = {https://doi.org/10.2307/1993589},
	volume = {108},
	year = {1963}}

@inproceedings{Furstenberg4,
	author = {Furstenberg, Harry},
	booktitle = {Harmonic analysis on homogeneous spaces ({P}roc. {S}ympos. {P}ure {M}ath., {V}ol. {XXVI}, {W}illiams {C}oll., {W}illiamstown, {M}ass., 1972)},
	mrclass = {22E40 (60J50)},
	mrnumber = {0352328},
	mrreviewer = {R. G. Laha},
	pages = {193--229},
	title = {Boundary theory and stochastic processes on homogeneous spaces},
	year = {1973}}

@incollection{Furstenberg3,
	author = {Furstenberg, Harry},
	booktitle = {Advances in {P}robability and {R}elated {T}opics, {V}ol. 1},
	mrclass = {28.75 (22.00)},
	mrnumber = {0284569},
	mrreviewer = {H. Kesten},
	pages = {1--63},
	publisher = {Dekker, New York},
	title = {Random walks and discrete subgroups of {L}ie groups},
	year = {1971}}

@article{HY,
	author = {Hartman, Y. and Yadin, A.},
	journal = {Compos. Math.},
	number = {10},
	pages = {2239-2265},
	title = {Furstenberg entropy of intersectional invariant random subgroups},
	volume = {154},
	year = {2018}}

@article{HT,
	author = {Hartman, Y. and Tamuz, O.},
	journal = {J. Anal. Math.},
	number = {1},
	pages = {227-257},
	title = {Furstenberg entropy realizations for virtually free groups and lamplighter groups},
	volume = {126},
	year = {2015}}

@article{FG,
	author = {Fraczyk, Mikolaj and Gelander, Tsachik},
	doi = {10.4007/annals.2023.197.1.6},
	fjournal = {Annals of Mathematics. Second Series},
	issn = {0003-486X},
	journal = {Ann. of Math. (2)},
	mrclass = {22E40 (22F30 32M15 53C25 53C35)},
	mrnumber = {4515682},
	number = {1},
	pages = {389--421},
	title = {Infinite volume and infinite injectivity radius},
	url = {https://doi.org/10.4007/annals.2023.197.1.6},
	volume = {197},
	year = {2023}}

@article{NZ1,
	author = {Nevo, Amos and Zimmer, Robert J.},
	doi = {10.1007/s002220050377},
	fjournal = {Inventiones Mathematicae},
	issn = {0020-9910},
	journal = {Invent. Math.},
	mrclass = {22D40 (22F10 57S20)},
	mrnumber = {1720183},
	mrreviewer = {Scot Adams},
	number = {2},
	pages = {229--252},
	title = {Homogenous projective factors for actions of semi-simple {L}ie groups},
	url = {https://doi.org/10.1007/s002220050377},
	volume = {138},
	year = {1999}}

@article{NZ2,
	author = {Nevo, A. and Zimmer, R.},
	journal = {Ann. Math.},
	number = {2},
	pages = {565-594},
	title = {A structure theorem for actions of semi-simple Lie groups},
	volume = {156},
	year = {2002}}

@article{NZ3,
	author = {Nevo, A. and Zimmer, R.},
	journal = {Ann. Sci. {\'E}c. Norm. Sup{\'e}r.},
	number = {3},
	pages = {321-343},
	title = {Rigidity of {F}urstenberg entropy for semi-simple {L}ie group actions},
	volume = {33},
	year = {2000}}

@article{Nevo,
	author = {Nevo, A.},
	journal = {Geom. Dedic.},
	pages = {187-218},
	title = {The spectral theory of amenable actions and invariants of discrete groups},
	volume = {100},
	year = {2003}}

@incollection{GS,
	author = {Grigorchuk, R. and Sunik, Z.},
	booktitle = {Group St. Andrews 2005},
	editor = {Cambridge Univ. Press},
	pages = {36-95},
	publisher = {London Math. Soc. Lecture Note},
	title = {Self-similarity and branching in group theory},
	volume = {339},
	year = {2007}}

@article{GZ,
	author = {Grigorchuk, R. and Zuk, A.},
	journal = {Int. J. Algebra Comput.},
	number = {1},
	pages = {1-24},
	title = {On a torsion-free weakly branch group defined by a three state automaton},
	volume = {12},
	year = {2002}}

@article{BC12,
	author = {Benjamini, Itai and Curien, Nicolas},
	doi = {10.1214/EJP.v17-2401},
	fjournal = {Electronic Journal of Probability},
	journal = {Electron. J. Probab.},
	mrclass = {05C80 (28D20)},
	mrnumber = {2994841},
	pages = {no. 93, 20},
	title = {Ergodic theory on stationary random graphs},
	url = {https://doi.org/10.1214/EJP.v17-2401},
	volume = {17},
	year = {2012}}

@article{CL16,
	author = {Carrasco Piaggio, Mat\'{\i}as and Lessa, Pablo},
	doi = {10.1214/16-EJP4650},
	fjournal = {Electronic Journal of Probability},
	journal = {Electron. J. Probab.},
	mrclass = {05C80 (28D20 60B05)},
	mrnumber = {3546392},
	mrreviewer = {D. Yogeshwaran},
	pages = {Paper No. 55, 24},
	title = {Equivalence of zero entropy and the {L}iouville property for stationary random graphs},
	url = {https://doi.org/10.1214/16-EJP4650},
	volume = {21},
	year = {2016}}

@article{CP07,
	author = {Connell, Chris and Muchnik, Roman},
	doi = {10.1007/s00039-007-0608-9},
	fjournal = {Geometric and Functional Analysis},
	issn = {1016-443X},
	journal = {Geom. Funct. Anal.},
	mrclass = {60J50 (30F40 30F60 37A35 41A58 41A65)},
	mrnumber = {2346273},
	number = {3},
	pages = {707--769},
	title = {Harmonicity of quasiconformal measures and {P}oisson boundaries of hyperbolic spaces},
	url = {https://doi.org/10.1007/s00039-007-0608-9},
	volume = {17},
	year = {2007}}

@article{KaimanovichHyperbolic,
	author = {Kaimanovich, Vadim A.},
	doi = {10.2307/2661351},
	fjournal = {Annals of Mathematics. Second Series},
	issn = {0003-486X},
	journal = {Ann. of Math. (2)},
	mrclass = {60J35 (20F67 28D20 31C05 60G50 60J50)},
	mrnumber = {1815698},
	mrreviewer = {Wolfgang Woess},
	number = {3},
	pages = {659--692},
	title = {The {P}oisson formula for groups with hyperbolic properties},
	url = {https://doi.org/10.2307/2661351},
	volume = {152},
	year = {2000}}

@book{KnappBook,
	author = {Knapp, Anthony W.},
	edition = {Second},
	isbn = {0-8176-4259-5},
	mrclass = {22-01},
	mrnumber = {1920389},
	pages = {xviii+812},
	publisher = {Birkh\"{a}user Boston, Inc., Boston, MA},
	series = {Progress in Mathematics},
	title = {Lie groups beyond an introduction},
	volume = {140},
	year = {2002}}

@book{MargulisBook,
	author = {Margulis, G. A.},
	doi = {10.1007/978-3-642-51445-6},
	isbn = {3-540-12179-X},
	mrclass = {22E40 (20Hxx 22-02 22D40)},
	mrnumber = {1090825},
	mrreviewer = {Gopal Prasad},
	pages = {x+388},
	publisher = {Springer-Verlag, Berlin},
	series = {Ergebnisse der Mathematik und ihrer Grenzgebiete (3) [Results in Mathematics and Related Areas (3)]},
	title = {Discrete subgroups of semisimple {L}ie groups},
	url = {https://doi.org/10.1007/978-3-642-51445-6},
	volume = {17},
	year = {1991}}

@article{Bowen,
	author = {Bowen, L.},
	journal = {Invent. Math.},
	number = {2},
	pages = {485-510},
	title = {Random walks on coset spaces with applications to {F}urstenberg entropy},
	volume = {196},
	year = {2014}}

@article{FT,
	author = {Forghani, Behrang and Tiozzo, Giulio},
	doi = {10.1214/21-aop1529},
	fjournal = {The Annals of Probability},
	issn = {0091-1798},
	journal = {Ann. Probab.},
	mrclass = {60J50 (22D40 60B15 60G50)},
	mrnumber = {4385123},
	number = {1},
	pages = {61--89},
	title = {Shannon's theorem for locally compact groups},
	url = {https://doi.org/10.1214/21-aop1529},
	volume = {50},
	year = {2022}}

@incollection{Kaimanovich0-2,
	author = {Kaimanovich, Vadim A.},
	booktitle = {Harmonic analysis and discrete potential theory ({F}rascati, 1991)},
	mrclass = {60J10 (28D20 31C20 60J50)},
	mrnumber = {1222456},
	mrreviewer = {Donald I. Cartwright},
	pages = {145--180},
	publisher = {Plenum, New York},
	title = {Measure-theoretic boundaries of {M}arkov chains, {$0$}-{$2$} laws and entropy},
	year = {1992}}

@book{Gray,
	author = {Gray, Robert M.},
	doi = {10.1007/978-1-4419-7970-4},
	edition = {Second},
	isbn = {978-1-4419-7969-8; 978-1-4419-7970-4},
	mrclass = {94-02 (94A15 94A17)},
	mrnumber = {3134681},
	mrreviewer = {Oliver Johnson},
	pages = {xxviii+409},
	publisher = {Springer, New York},
	title = {Entropy and information theory},
	url = {https://doi.org/10.1007/978-1-4419-7970-4},
	year = {2011}}

@article{LyonsPeres,
	author = {Lyons, Russell and Peres, Yuval},
	doi = {10.4171/jems/1030},
	fjournal = {Journal of the European Mathematical Society (JEMS)},
	issn = {1435-9855},
	journal = {J. Eur. Math. Soc. (JEMS)},
	mrclass = {20F69 (20P05 43A05 60B15 60J50)},
	mrnumber = {4228277},
	number = {4},
	pages = {1133--1160},
	title = {Poisson boundaries of lamplighter groups: proof of the {K}aimanovich-{V}ershik conjecture},
	url = {https://doi.org/10.4171/jems/1030},
	volume = {23},
	year = {2021}}

@article{Series,
	author = {Series, Caroline},
	doi = {10.1112/jlms/s2-31.1.69},
	fjournal = {Journal of the London Mathematical Society. Second Series},
	issn = {0024-6107},
	journal = {J. London Math. Soc. (2)},
	mrclass = {58F17 (11J06 30F35 57N05)},
	mrnumber = {810563},
	mrreviewer = {Peter J. Nicholls},
	number = {1},
	pages = {69--80},
	title = {The modular surface and continued fractions},
	url = {https://doi.org/10.1112/jlms/s2-31.1.69},
	volume = {31},
	year = {1985}}

@article{Stuck,
	author = {Stuck, Garrett},
	doi = {10.1017/S0143385700009123},
	fjournal = {Ergodic Theory and Dynamical Systems},
	issn = {0143-3857},
	journal = {Ergodic Theory Dynam. Systems},
	mrclass = {57S20 (22E46)},
	mrnumber = {1406436},
	mrreviewer = {A. L. Onishchik},
	number = {4},
	pages = {821--831},
	title = {Minimal actions of semisimple groups},
	url = {https://doi.org/10.1017/S0143385700009123},
	volume = {16},
	year = {1996}}

@article{Varadarajan,
	author = {Varadarajan, V. S.},
	doi = {10.2307/1993903},
	fjournal = {Transactions of the American Mathematical Society},
	issn = {0002-9947},
	journal = {Trans. Amer. Math. Soc.},
	mrclass = {28.70},
	mrnumber = {159923},
	mrreviewer = {G. W. Mackey},
	pages = {191--220},
	title = {Groups of automorphisms of {B}orel spaces},
	url = {https://doi.org/10.2307/1993903},
	volume = {109},
	year = {1963}}

@inproceedings{Verdu,
	author = {Verd\'u, Sergio},
	booktitle = {2014 Information Theory and Applications Workshop (ITA)},
	doi = {10.1109/ITA.2014.6804281},
	pages = {1-3},
	title = {Total variation distance and the distribution of relative information},
	year = {2014}}

@book{ZimmerBook,
	author = {Zimmer, Robert J.},
	doi = {10.1007/978-1-4684-9488-4},
	isbn = {3-7643-3184-4},
	mrclass = {22E40 (22D40 28D15)},
	mrnumber = {776417},
	mrreviewer = {S. G. Dani},
	pages = {x+209},
	publisher = {Birkh\"{a}user Verlag, Basel},
	series = {Monographs in Mathematics},
	title = {Ergodic theory and semisimple groups},
	url = {https://doi.org/10.1007/978-1-4684-9488-4},
	volume = {81},
	year = {1984}}

@article{ZimmerInduced,
	author = {Zimmer, Robert J.},
	fjournal = {Annales Scientifiques de l'\'{E}cole Normale Sup\'{e}rieure. Quatri\`eme S\'{e}rie},
	issn = {0012-9593},
	journal = {Ann. Sci. \'{E}cole Norm. Sup. (4)},
	mrclass = {22D40 (28D15)},
	mrnumber = {521638},
	mrreviewer = {Pierre Eymard},
	number = {3},
	pages = {407--428},
	title = {Induced and amenable ergodic actions of {L}ie groups},
	url = {http://www.numdam.org/item?id=ASENS_1978_4_11_3_407_0},
	volume = {11},
	year = {1978}}
\textsc{\newline J\'er\'emie Brieussel --- Universit\'e de Montpellier 
} --- jeremie.brieussel@umontpellier.fr

\textsc{\newline Tianyi Zheng --- UC San Diego
} --- tzheng2@math.ucsd.edu

\end{document}